\documentclass[twoside,12pt,reqno]{amsart}
\usepackage{amsmath,amsthm, amsfonts,amssymb}
\usepackage{color,soul}
\usepackage[dvipsnames]{xcolor}
\usepackage[all]{xy}
\usepackage{graphicx}
\usepackage{indentfirst}
\usepackage{bm}
\usepackage{mathrsfs}
\usepackage{latexsym}
\usepackage{hyperref}
\usepackage{microtype}	
\usepackage{bookmark}
\usepackage{tikz}[2010/10/13]

\newcommand{\graa}[1]{\raisebox{-.6cm}{\includegraphics[height=1.5cm]{PPA#1.pdf}}}
\newcommand{\grb}[1]{\raisebox{-.8cm}{\includegraphics[height=2cm]{PPA#1.pdf}}}
\newcommand{\grc}[1]{\raisebox{-1.3cm}{\includegraphics[height=3cm]{PPA#1.pdf}}}
\newcommand{\grd}[1]{\raisebox{-1.8cm}{\includegraphics[height=4cm]{PPA#1.pdf}}}
\newcommand{\gre}[1]{\raisebox{-2.3cm}{\includegraphics[height=5cm]{PPA#1.pdf}}}

\def\ket#1{\left|#1\right\rangle}
\def\bra#1{\left\langle#1\right|}
\def\braket#1#2{\left\langle#1\right.\left|#2\right\rangle}

\theoremstyle{remark}

\newtheorem{Thm}{Theorem}[section]

\newtheorem{proposition}[Thm]{\bf Proposition}
\newtheorem{corollary}[Thm]{\bf Corollary}
\newtheorem{lemma}[Thm]{\bf Lemma}

\theoremstyle{remark}
\newtheorem{remark}[Thm]{Remark}
\newtheorem{definition}[Thm]{\bf Definition}

\newtheorem{theorem}[Thm]{\bf Theorem}
\newtheorem{notation}[Thm]{Notation~}

\def\bv#1#2{\begin{pmatrix}  {#1}\\{#2}  \end{pmatrix}}

\def\calt{{\mathcal T}}

\def\Z{{ \mathbb  Z}}

\def\pslash{\hbox{$\partial$\kern-1.2ex \raise.14ex\hbox{/}\kern.5ex}}
\def\Epslash{\hbox{{\rm {\pslash\kern-.6ex \lower.28ex \hbox{$_E$}}}}}
\def\Epslashi{\hbox{{\rm {\pslash\kern-.6ex \lower.28ex \hbox{$_{E,i}$}}}}}
\def\Epslashii#1{\hbox{{\rm {\pslash\kern-.6ex \lower.28ex \hbox{$_{E,#1}$}}}}}
\def\Epslasht{\hbox{{\rm {\pslash\kern-.6ex \lower.28ex \hbox{$_{E,\calt}$}}}}}
\def\Epslashtstar{\hbox{{\rm {\pslash\kern-.6ex \lower.28ex \hbox{$_{E,\calt^*}$}}}}}
\def\Epslashti{\hbox{{\rm {\pslash\kern-.6ex \lower.28ex \hbox{$_{E,\calt_i}$}}}}}

\def\l{\left}
\def\r{\right}
\def\la{\lambda}

\def\be{\begin{equation}}
\def\ee{\end{equation}}
\def\beq{\begin{eqnarray}}
\def\eeq{\end{eqnarray}}
\def\beqs{\begin{eqnarray*}}
\def\eeqs{\end{eqnarray*}}

\def\bep{\begin{proof}}
\def\eep{\end{proof}}
\def\lra#1{\left\langle #1\right\rangle}

\def\abs#1{{\l\vert #1 \r\vert}}

\newif\ifMarginNotes \MarginNotestrue
\def\mrgn#1{\ifMarginNotes\setbox0=\vtop{\hsize 6.75pc
   {\noindent\relax #1\par}}\leavevmode
   \vadjust{\dimen0=\dp0 \dimen1=\ht0\advance\dimen1 by .5ex
 \advance\dimen0 by -.5ex
  \kern-\dimen1\hbox{\kern\hsize\kern.5pc$\leftarrow$
  \box0}\kern-\dimen0}\fi}
\catcode`\@=11

\font\ninemsb=msbm10 scaled 900
\font\eightmsb=msbm10 scaled 800

\font\tenmsb=msbm10 scaled 950
\font\ninemsb=msbm7 scaled 1200
\newfam\msbfam
\textfont\msbfam=\tenmsb  \scriptfont\msbfam=\ninemsb
  \scriptscriptfont\msbfam=\eightmsb
\def\msb@{\hexnumber@\msbfam}

\font\teneufm=eufm10
\font\nineeufm=eufm7 scaled 1200
\font\seveneufm=eufm7

\newfam\eufmfam
\textfont\eufmfam=\teneufm  \scriptfont\eufmfam=\nineeufm
  \scriptscriptfont\eufmfam=\seveneufm
\def\frak{\relax\ifmmode\let\next\frak@\else
 \def\next{\errmessage{Use \string\frak\space only in math mode}}\fi\next}
\def\frak@#1{{\frak@@{#1}}}
\def\frak@@#1{\fam\eufmfam#1}

\catcode`\@=12

\usepackage{fancyhdr}

\makeatletter
\def\l@section{\@tocline{1}{0pt}{1pc}{}{}}
\def\l@subsection{\@tocline{2}{0pt}{1pc}{4.6em}{}}
\def\l@subsubsection{\@tocline{3}{0pt}{1pc}{7.6em}{}}
\renewcommand{\tocsection}[3]{%
  \indentlabel{\@ifnotempty{#2}{\makebox[2.3em][l]{%
    \ignorespaces#1 #2.\hfill}}}#3}
\renewcommand{\tocsubsection}[3]{%
  \indentlabel{\@ifnotempty{#2}{\hspace*{2.3em}\makebox[2.3em][l]{%
    \ignorespaces#1 #2.\hfill}}}#3}
\renewcommand{\tocsubsubsection}[3]{%
  \indentlabel{\@ifnotempty{#2}{\hspace*{4.6em}\makebox[3em][l]{%
    \ignorespaces#1 #2.\hfill}}}#3}
\makeatother

\usepackage{datetime}
\usdate%\longdate
\def\:{{:}}

\renewcommand{\theequation}{\Roman{section}.\arabic{equation}}

\makeatletter
\renewcommand{\author}[2][]{%
  \def\@tempa{#1}
  \ifx\@empty\authors
    \ifx\@tempa\@empty
      \gdef\shortauthors{#2}%
    \else
      \gdef\shortauthors{#1}%
    \fi
    \gdef\authors{\author{#2}}%
  \else
    \ifx\@tempa\@empty
      \g@addto@macro\shortauthors{\and#2}%
    \else
      \g@addto@macro\shortauthors{\and#1}%
    \fi
    \g@addto@macro\authors{\and\author{#2}}%
  \fi
}
\renewcommand{\address}[2][]{\g@addto@macro\authors{\address{#1}{#2}}}
\def\@setauthors{%
  \begin{center}%
    \footnotesize
    \vspace{20pt}
    \let\and\@empty
    \def\author##1{\advance\@tempcnta\@ne}%
    \def\address##1##2{\advance\@tempcntb\@ne}%
    \@tempcnta=\z@  \@tempcntb=\z@
    \authors
    \ifnum\@tempcnta>\@ne \ifnum\@tempcntb=\@ne
        \oneaddress
      \else
        \sepaddresses
      \fi
    \else
      \oneaddress
    \fi
  \end{center}%
}
\def\oneaddress{%
  \begingroup
  \let\author\@iden \let\address\@gobbletwo
  \renewcommand{\andify}{%
    \nxandlist{\unskip, }{\unskip{} and~}{\unskip, and~}}%
  \uppercasenonmath\authors
  \andify\authors
  \authors
  \endgroup
  \begingroup \let\and\relax \let\author\@gobble
  \def\address##1##2{\unskip\\[10pt] \itshape##2}%
  \authors
  \endgroup
}
\def\sepaddresses{%
  \begingroup
    \baselineskip10\p@\relax
    \def\address##1##2{ ({\itshape##2}\/)}
    \def\author##1{\def\temp{##1}\leavevmode\uppercasenonmath\temp\temp}%
    \nxandlist
      {,\\[\baselineskip]}
      {\\[\baselineskip] \textsc{\lowercase{and}}\\[\baselineskip]}
      {,\\[\baselineskip]\textsc{\lowercase{and}}\\[\baselineskip]}
      \authors % macro to operate on
    \authors
  \endgroup
}
\def\maketitle{\par
  \@topnum\z@
  \@setcopyright
  \thispagestyle{firstpage}%
  \uppercasenonmath\shorttitle
  \ifx\@empty\shortauthors \let\shortauthors\shorttitle
  \else
    \newcommand{\@xuppercasenonmath}[1]{\toks@\@emptytoks
      \@xp\@skipmath\@xp\@empty##1$$%
      \edef##1{\@nx\protect\@nx\@upprep\the\toks@}}%
    \@xuppercasenonmath\shortauthors
    \def\@@and{AND}
    \renewcommand{\andify}{%
      \nxandlist{\unskip, }{\unskip{ }\@@and{ }}{\unskip, \@@and{ }}}%
    \andify\shortauthors
  \fi
  \@maketitle@hook
  \begingroup
  \@maketitle
  \endgroup
  \c@footnote\z@
  \@cleartopmattertags
}
\def\@maketitle{%
  \normalfont\normalsize
  \let\@makefntext\noindent
  \@adminfootnotes
  \ifx\@empty\addresses\else \@footnotetext{\@setotheraddresses}\fi
  \global\topskip68\p@\relax
  \@settitle
  \ifx\@empty\authors \else \@setauthors \fi
  \ifx\@empty\@dedicatory
  \else
    \baselineskip26\p@
    \vtop{\centering{\footnotesize\itshape\@dedicatory\@@par}%
      \global\dimen@i\prevdepth}\prevdepth\dimen@i
  \fi
  \toks@\@xp{\shortauthors}\@temptokena\@xp{\shorttitle}%
  \edef\@tempa{\@nx\markboth{\the\toks@}{\the\@temptokena}}\@tempa
  \@setabstract
  \normalsize
  \if@titlepage
    \newpage
  \else
    \dimen@34\p@ \advance\dimen@-\baselineskip
    \vskip\dimen@\relax
  \fi
} % end \@maketitle
\renewcommand{\thanks}[1]{%
  \ifx\@empty\thankses
    \gdef\thankses{\thanks{#1}}%
  \else
    \g@addto@macro\thankses{\endgraf\thanks{#1}}%
  \fi}
\def\@setthanks{\def\thanks##1{\noindent##1\@addpunct.}\thankses}
\renewcommand{\curraddr}[2][]{%
  \ifx\@empty\addresses
    \gdef\addresses{\curraddr{#1}{#2}}%
  \else
    \g@addto@macro\addresses{\endgraf\curraddr{#1}{#2}}%
  \fi}
\renewcommand{\email}[2][]{%
  \ifx\@empty\addresses
    \gdef\addresses{\email{#1}{#2}}%
  \else
    \g@addto@macro\addresses{\endgraf\email{#1}{#2}}%
  \fi}
\renewcommand{\urladdr}[2][]{%
  \ifx\@empty\addresses
    \gdef\addresses{\urladdr{#1}{#2}}%
  \else
    \g@addto@macro\addresses{\endgraf\urladdr{#1}{#2}}%
  \fi}
\def\@setotheraddresses{%
  \def\curraddr##1##2{\noindent
    \emph{Current address\@ifnotempty{##1}{ of ##1}}:\space
      ##2\@addpunct.}%
  \def\email##1##2{\noindent
    \emph{E-mail address\@ifnotempty{##1}{ of ##1}}:\space
      \texttt{##2}}%
  \def\urladdr##1##2{\noindent
    \emph{WWW address\@ifnotempty{##1}{ of ##1}}:\space
      \texttt{##2}}%
  \addresses
}
\let\enddoc@text\relax
\makeatother
\newcommand{\mn}{/3}
\newcommand{\nn}{/3}

\newcommand{\fbraid}[4]{
\draw (#1,#2)--(#3,#4);
\draw (#1,#4)--(2/3*#1+1/3*#3,2/3*#4+1/3*#2);
\draw (#3,#2)--(2/3*#3+1/3*#1,2/3*#2+1/3*#4);
}

\newcommand{\fmeasure}[5]{
\node at (#1+#3,#2-#4) {\size{$#5$}};
\draw (#1,#2) --(#1,#2-#4)  arc (-180:0:#3) -- (#1+#3+#3,#2);
}

\newcommand{\fdoublemeasure}[6]
{
\fmeasure{#1}{#2}{3*#3}{#4}{}
\fmeasure{#1+2*#3}{#2}{#3}{#4}{#5}
\node at (#1+3*#3, #2-#4-2*#3) {\size{$#6$}};
}

\newcommand{\fqudit}[5]{
\node at (#1+#3,#2+#4) {\size{$#5$}};
\draw (#1,#2) --(#1,#2+#4)  arc (180:0:#3) -- (#1+#3+#3,#2);
}

\newcommand{\fdoublequdit}[6]
{
\fqudit{#1}{#2}{3*#3}{#4}{}
\fqudit{#1+2*#3}{#2}{#3}{#4}{#5}
\node at (#1+3*#3, #2+#4+2*#3) {\size{$#6$}};
}

\newcommand{\size}[1]{\fontsize{10pt}{\baselineskip}\selectfont{#1}}

\newcommand{\FS}{\mathfrak{F}_{\text{{s}}}}

\title{Planar Para Algebras, Reflection Positivity}
\author{Arthur Jaffe and Zhengwei Liu}
\address{Harvard University, Cambridge, MA 02138}

\begin{document}
\maketitle

\begin{abstract}
We define a \textit{planar para algebra}, which arises naturally from combining planar algebras with the idea of $\Z_{N}$ para symmetry in physics.
A subfactor planar para algebra is a Hilbert space representation of planar tangles with parafermionic defects, that are invariant under para isotopy.  For each $\Z_{N}$, we construct a family of subfactor planar para algebras which play the role of Temperley-Lieb-Jones planar algebras. The first example in this family is the parafermion planar para algebra (PAPPA).
Based on this example, we introduce parafermion Pauli matrices,  quaternion relations, and braided relations for parafermion algebras which one can use in the study of  quantum information.
An important ingredient in planar para algebra theory is the string Fourier transform (SFT), that  we use on the matrix algebra generated by the Pauli matrices.
Two different reflections play an important role in the theory of planar para algebras. One is the adjoint operator; the other is the modular conjugation in Tomita-Takesaki theory. We use the latter one to define the double algebra and to introduce reflection positivity. We give a new and geometric proof of reflection positivity, by relating the two reflections through the string Fourier transform.
\end{abstract}

\tableofcontents

\section{Introduction}
We introduce the notion of a {\em planar para algebra} (PPA), which generalizes the concept of a planar algebra introduced by Jones \cite{JonPA}.  The idea arises naturally from considering a grading with planar algebras.  In physics a $\Z_{2}$ grading (charge) arises from fermions, and a $\Z_{N}$ grading from parafermions, a generalization of fermions.  One might think of planar para algebras as a topological quantum field theory with parafermionic defects \cite{Ati88,Wit88}.

PPA is a natural algebraic object to study.  But the applications of PPA  enhance their value.  As they relate naturally to the physics of parafermions, it is natural to expect that PPA have application in physics.  We encounter some of these relations in Pauli $X,Y,Z$ matrices and in the property of reflection positivity.   But shortly after writing the original version of this paper, we were surprised to find that PPA are also relevant in the field of quantum information.  So we have rewritten this paper and we elaborate on this interesting and  newly-discovered connection   that we treat in detail in other work, see \S\ref{ApplicationsOutside}.

\subsection{Fundamental Properties of PPA}
The partition function of a planar para algebra is a representation of planar tangles with parafermionic defects on a vector space invariant under isotopy. Usually we require that those tangles without boundary are presented by a scalar multiple of the vacuum vector, defined by the empty diagram.
When the partition function has the standard positivity property in planar algebra theory with respect to the vertical reflection, we call the planar para algebra a \textit{subfactor} planar para algebra; those planar para algebras are closely related to subfactor theory.

The fundamental planar algebra is known as the  Temperley-Lieb-Jones planar algebra. The positivity condition for the Temperley-Lieb-Jones planar algebra was proved by Jones' remarkable result on the rigidity of indices \cite{Jon83}.
In Theorem \ref{Theorem: construction}, we show that for each group $\mathbb{Z}_N$ one can construct a planar para algebra which plays the role of the Temperley-Lieb-Jones planar algebra in the theory of planar para algebras. We prove a similar rigidity result for the positivity condition in Theorem \ref{Theorem: positivity general} and thereby obtain a family of subfactor planar para algebras.

For each $\mathbb{Z}_N$, the subfactor planar algebra in the family that has the smallest index is called a parafermion planar para algebra (PAPPA), since it is algebraically isomorphic to the parafermion algebra in physics, with infinitely many generators. We explore other properties of PAPPAs in Sections \ref{Pauli}, \ref{sec:PauliDiagrams}, \ref{Pictorial presentation}, and \ref{braided relations}.  In the first two of these sections we discuss local properties of the algebras.  In the later sections we discuss global properties.

There are two different natural states on a PAPPA.   One is the Markov trace, usually used in subfactor theory.
We can realize the underlying Hilbert spaces by the Gelfand-Naimark-Segal construction, and this gives a braid-group representation and well-known knot invariants.

The other natural state is the expectation in the zero-particle vector state, arising from a standard Fock-space construction, see for example \cite{CO89}.
The corresponding GNS representation is different, and this state is especially suitable for applications in quantum information, see \S\ref{ApplicationsOutside}.
The PAPPA not only gives a pictorial representation of a parafermion algebra, but it also gives a picture of the underlying Hilbert space.

Furthermore, in \S\ref{braided relations} we extend the isotopy to the three-dimensional space by introducing braids. We prove that parafermion planar para algebras are {\it half-braided}. The diagrammatic representation of the underlying Hilbert space is compatible with the braided isotopy.

%
%{\color{red}Move these bullets:}
%\begin{itemize}
%\item{} A major difference between PAPPA and the para-fermion algebra in physics is the existence of the string Fourier transform (SFT), defined as a partial rotation of diagrams.  This transform agrees in a special case with the usual Fourier series on $\Z_{N}$, as shown in Proposition \ref{Prop:SFT=FT}.
%
%\item{}
%We apply our diagrammatic picture to give a geometric interpretation to the twisted product, introduced in \cite{JafJan2015,JafJan2016}, as an intermediate state in ``para-isotopy,'' that is also invariant under a horizontal reflection homomorphism $\Theta$.
%
%\end{itemize}
%

\subsection{Pauli Matrices $X,Y,Z$}\label{Sect:PauliIntro}
These unitary matrices are important in physics. With Fourier transform and the Gaussian, they generate an interesting projective, linear group $\Z_{N}^{2}\times SL(2,\Z_{N})$.
In \S\ref{sec:PauliDiagrams} we give a variety of diagrammatic representations of the ``parafermion Pauli matrices'' $X,Y,Z$ of dimension $N$, which are given algebraically  in \S\ref{Pauli}.\footnote{Our diagrammatic interpretation illustrates that certain models of $X,Y,Z$ that are quadratic in parafermions have a special  advantage: they are neutral (charge-zero), so they preserve charge.}

\subsection{The String Fourier Transform}
The action of rotation on various defects are described by the para degree.  For usual planar algebras, a $2\pi$ rotation equals the identity.   In the case of planar fermion algebras, the $2\pi$ rotation on a fermion has the eigenvalue $-1$. In the general case of planar para algebras, the $2\pi$ rotation of a $\mathbb{Z}_N$ parafermion has the eigenvalue $e^{\frac{2\pi i}{N}}$.

In terms of planar para algebras, we define the string Fourier transform (SFT) $\mathfrak{F}_{\text{s}}$ on parafermion algebras as a (one-string) rotation of the diagrams.  Acting on two-box diagrams, the SFT amounts to a rotation by $\frac{\pi}{2}$. It turns an ordinary product of diagrams into a convolution product of diagrams.

\begin{itemize}
\item{} We give an elementary proof in  Proposition~\ref{Prop:SFT=FT}, that in the case of zero-graded two-string diagrams, the SFT reduces to the usual Fourier transform on $\Z_{N}$.
\end{itemize}

\subsection{Reflection Positivity}
The graded commutant of the parafermion algebra on the GNS representation can be represented pictorially in the parafermion planar algebra. The modular conjugation $\Theta$ in Tomita-Takisaki theory turns out to be a horizontal reflection.
In \S\ref{sect:RP} we study reflection-doubled algebras, leading to the study of the \textit{reflection-positivity} property \cite{OstSch73a,OstSch73b}. This property is quite important in the context of particle physics and statistical physics, where it has wide use in establishing existence results in quantum field theory, as well as in the study of phase transitions.  Reflection positivity of parafermion algebras had been proved in a different context \cite{JafPed,JafJan2015,JafJan2016}, where one finds further references to other papers on reflection positivity.  However here we apply our diagrammatic picture to give a geometric interpretation to the twisted product used in these proofs, as an intermediate state in ``para-isotopy.'' This product is also invariant under a horizontal reflection homomorphism $\Theta$.

In Theorem \ref{RPgeneral} we give a new and geometric proof of reflection positivity that applies to parafermion algebras as a special case, and in general to subfactor planar para algebras. In particular, we relate the two notions of positivity mentioned above: $C^{*}$ positivity and reflection positivity.  We show that reflection positivity of a Hamiltonian in a subfactor planar para algebra is a consequence of the $C^*$ positivity of the string Fourier transform of the Hamiltonian.

The underlying mechanism that leads to reflection positivity relies on the relation between two different reflections, one is the rotation of the other.  In the planar para algebra, a horizontal reflection $\Theta$ defines the double.  On the other hand  a vertical reflection defines the adjoint~$^{*}$.  These two reflections are related by a $\frac{\pi}{2}$ rotation, which is how the string Fourier transform enters.  We combine rotation and reflection with the isotopy invariance of the partition function, in order to obtain the reflection positivity property.
For parafermion algebras, we show in Theorem \ref{RPparafermion} that reflection positivity is equivalent to the positivity of the coupling constant matrix $J^{0}$ of the Hamiltonian for interaction across the reflection plane.
%
%\begin{itemize}
%\item{}
%In \S\ref{sect:RP}  we use para-isotopy to relate properties of the SFT with the action of the  reflection $\Theta$.  The reflection property (RP) property plays a central role in mathematical physics, especially in quantum field theory and in the study of phase transitions in statistical physics.
%\end{itemize}
%
%
%\begin{itemize}
%\item{}
%In \S\ref{sect:RP}  we use para-isotopy to relate properties of the SFT with the action of the  reflection $\Theta$.  The reflection property (RP) property plays a central role in mathematical physics, especially in quantum field theory and in the study of phase transitions in statistical physics.
%
%
%
%\item{}
%In Theorem \ref{RPparafermion} we give a geometric proof of the reflection-positivity property for PAPPA's that satisfy $C^{*}$ positivity.  This is the first time that the two positivity conditions ($C^{*}$ positivity in operator algebras and reflection positivity for interacting systems in statistical physics)  have been shown to be directly related to one-another.
%\end{itemize}
%

\subsection{Quantum Information\label{ApplicationsOutside}}
%We concentrate much of our paper on one single example of PPA, namely PAPPA. 
We mention briefly the surprising connections  that we discovered linking PAPPA with quantum information. We have  explored this relation in detail in related works.  The overlap with quantum information is based on the use of diagrams arising in PAPPA.  We call this work  \textit{holographic software}, since we present a dictionary to translate between quantum information protocols and PAPPA diagrams \cite{JafLiuWoz}. In fact central concepts in PAPPA map onto central concepts in protocols for communication; this includes the diagrammatic representation of the resource state in addition to diagrammatic protocols.

\begin{itemize}
\item{}
The resource state is central in quantum information as a means to enable entanglement between different parties. We find that the SFT yields a maximally-entangled resource state by  acting on the zero-particle state of a quantum system \cite{JafLiuWoz}. This state has maximal entanglement entropy.

\item{}
Previously, entanglement had been thought to be engendered by the topological properties of the braid.  As a consequence of the relationship between quantum information and PAPPA, we believe that it is more natural to think of the resource state as arising from the SFT.

\item{}
Our diagram for the teleportation protocol  in quantum information is extremely intuitive,  conveying the idea of transporting the qudit $\phi_{A}$, see Figure \ref{Pic:Teleportation1-Intro}.
\begin{figure}[h]
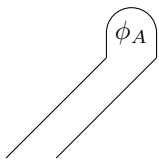

\raisebox{-1.5cm}{
\tikz{
\fqudit{0}{2\nn}{1/3}{1\nn}{\phi_A}
\draw (-2/-3,2\nn)--(2/-3,-2\nn);
\draw (0/-3,2\nn)--(4/-3,-2\nn);
}}\caption{Holographic protocol for teleportation of qudit $\phi_{A}$.} \label{Pic:Teleportation1-Intro}
\end{figure}
Motivated by insights arising from PAPPA, we introduced compressed transformations in quantum  information~\cite{JafLiuWoz-Tele}. We also found a new, more efficient, and more general protocol for teleportation that allows for multiple parties, each with multiple persons~\cite{JafLiuWoz-Tele}.  We illustrate this example in Figure~\ref{CTX-Protocol-Pic}. In the special case of two persons, this protocol covers many known protocols.

\begin{figure}[!htb]
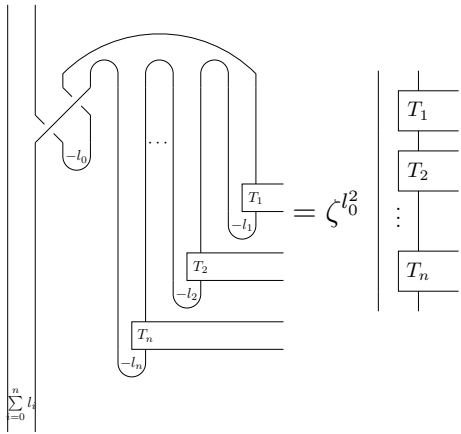

\begin{align*}
\scalebox{0.55}{\raisebox{-5.1 cm}{\tikz{
\fbraid{6\mn}{0}{8\mn}{2\nn}
\fbraid{10\mn}{4\nn}{8\mn}{2\nn}
%%%%
\draw (8\mn,4\nn) -- (8\mn,5\nn) to [bend left=45] (22\mn,5\nn)--(22\mn,4\nn);
\fqudit{18\mn}{4\nn}{1\mn}{1\nn}{}{}
\fqudit{14\mn}{4\nn}{1\mn}{1\nn}{}{}
\fqudit{10\mn}{4\nn}{1\mn}{1\nn}{}{}
\fmeasure{8\mn}{0\nn}{1\mn}{1\nn}{}
\node at (9\mn,-1\nn) {\size{$-l_0$}};
\draw (10\mn,0\nn) -- (10\mn,2\nn);
\node at (15\mn,0\nn) {$\cdots$};
%%%%
\draw (22\mn,-5\nn+2\nn) -- (22\mn,4\nn);
\draw (20\mn,-7\nn+2\nn) -- (20\mn,4\nn);
\draw (24\mn,-5\nn+2\nn)--(21\mn,-5\nn+2\nn) -- (21\mn,-7\nn+2\nn) --(24\mn,-7\nn+2\nn);
\node at (22\mn,-6\nn+2\nn) {\size{$T_1$}};
\fmeasure{20\mn}{-7\nn+2\nn}{1\mn}{1\nn}{-l_1}
%%%%
\draw (-4\mn+22\mn,-5\nn-3\nn) -- (-4\mn+22\mn,4\nn);
\draw (-4\mn+20\mn,-7\nn-3\nn) -- (-4\mn+20\mn,4\nn);
\draw (24\mn,-5\nn-3\nn)--(-4\mn+21\mn,-5\nn-3\nn) -- (-4\mn+21\mn,-7\nn-3\nn) --(24\mn,-7\nn-3\nn);
\node at (-4\mn+22\mn,-6\nn-3\nn) {\size{$T_2$}};
\fmeasure{-4\mn+20\mn}{-7\nn-3\nn}{1\mn}{1\nn}{-l_2}
%%%%
\draw (-8\mn+22\mn,-5\nn-8\nn) -- (-8\mn+22\mn,4\nn);
\draw (-8\mn+20\mn,-7\nn-8\nn) -- (-8\mn+20\mn,4\nn);
\draw (24\mn,-5\nn-8\nn)--(-8\mn+21\mn,-5\nn-8\nn) -- (-8\mn+21\mn,-7\nn-8\nn) --(24\mn,-7\nn-8\nn);
\node at (-8\mn+22\mn,-6\nn-8\nn) {\size{$T_n$}};
\fmeasure{-8\mn+20\mn}{-7\nn-8\nn}{1\mn}{1\nn}{-l_n}
%%%%
\node at (5\mn,-19\nn) {\size{$\sum\limits_{i=0}^n l_i$}};
\draw (4\mn,10\nn) -- (4\mn,-21\nn);
\draw (6\mn,0) -- (6\mn,-21\nn);
\draw (6\mn,10\nn) -- (6\mn,2\nn);
}}}
&=\zeta^{l_0^2}\;
\scalebox{0.8}{\raisebox{-1.5 cm}{
\tikz{
\draw (22\mn,-5\nn) -- (22\mn,-4\nn);
\draw (20\mn,-7\nn) -- (20\mn,-4\nn);
\draw (24\mn,-5\nn)--(21\mn,-5\nn) -- (21\mn,-7\nn) --(24\mn,-7\nn);
\node at (22\mn,-6\nn) {\size{$T_1$}};
\draw (22\mn,-5\nn-3\nn) -- (22\mn,-4\nn-3\nn);
\draw (20\mn,-7\nn-3\nn) -- (20\mn,-4\nn-3\nn);
\draw (24\mn,-5\nn-3\nn)--(21\mn,-5\nn-3\nn) -- (21\mn,-7\nn-3\nn) --(24\mn,-7\nn-3\nn);
\node at (22\mn,-6\nn-3\nn) {\size{$T_2$}};
\node at (21\mn,-11\nn) {\size{$\vdots$}};
\draw (22\mn,-5\nn-8\nn) -- (22\mn,-4\nn-6\nn);
\draw (20\mn,-7\nn-8\nn) -- (20\mn,-4\nn-6\nn);
\draw (22\mn,-5\nn-11\nn) -- (22\mn,-4\nn-11\nn);
\draw (20\mn,-5\nn-11\nn) -- (20\mn,-4\nn-11\nn);
\draw (24\mn,-5\nn-8\nn)--(21\mn,-5\nn-8\nn) -- (21\mn,-7\nn-8\nn) --(24\mn,-7\nn-8\nn);
\node at (22\mn,-6\nn-8\nn) {\size{$T_n$}};
}}}
\end{align*}
%\begin{figure}[h]
\caption{ CT Protocol for teleportation of compressed transformations.} \label{CTX-Protocol-Pic}
\end{figure}
\end{itemize}
\goodbreak

\setcounter{equation}{0}
\section{Planar Para Algebras}
\subsection{Planar tangles}
Our definition of planar para algebras involves planar tangles. These tangles are similar to the planar tangles in Jones' original definition of planar algebras \cite{Jon12}.  However, for readers who are not familiar with planar algebras, we give the definitions here, indicating some main distinctive features in color or boldface\footnote{A main difference between planar and planar para algebras is that we mark a distinguished point on the boundary of each disc, within a distinguished interval. This change is necessary, in order to describe the precise height of Jones' symbol $\$$. This height is significant in the definition of our twisted tensor product.}.

A planar $k$-tangle $T$ will consist of a smooth closed output disc $D_0$ in $\mathbb{C}$ together with a finite (possibly empty)  set $\mathcal{D}=\mathcal{D}_T$ of disjoint smooth input discs in the interior of $D_0$. Each input disc $D\in\mathcal{D}$ and the output disc $D_0$, will have an even number $2k_D\geq 0$ of marked points on its boundary with $k=k_{D_0}$. Inside $D_0$, but outside the interiors of the $D\in\mathcal{D}$, there is also a finite set of disjoint smoothly embedded curves called strings, which are either closed curves, or the end points of the strings are different marked points of $D_0$ or of the $D$'s in $\mathcal{D}$. Each marked point is the end-point of some string, which meets the boundary of the corresponding disc transversally.

The connected components of the complement of the strings in $\overset{\circ}{D_0}\backslash\bigcup_{D\in\mathcal{D}}D$ are called regions. The connected component of the boundary of a disc, minus its marked points, will be called the intervals of that disc. Regions of the tangle are shaded (say in gray), or they are unshaded (say in white).   Shading is done in a way that regions whose boundaries meet have different shading.  Intervals have a unique shading, as only one side of any interval lies in a region. The shading will be considered to extend to the intervals which are part of the boundary of a region.

To each disc in a tangle there is a distinguished ${\boldsymbol{point}}$ on its boundary that is not an end point of a string. The distinguished point is marked by a dollar sign $\$$, placed to the left of the input disc, or to the right of the output disc. This distinguished point defines a distinguished interval for each disc.

We denote the set of all planar $k$-tangles for $k\geq0$ by $\mathcal{T}_k$, and let $\mathcal{T}=\cup_k\mathcal{T}_k$. If the distinguished interval of $D_0$ for $T\in\mathcal{T}$ is unshaded, $T$ will be called positive; if it is shaded, $\mathcal{T}$ will be called negative. Thus $\mathcal{T}_k$ is the disjoint union of sets of positive and negative planar para tangles: $\mathcal{T}_k=\mathcal{T}_{k,+}\cup\mathcal{T}_{k,-}$.
\begin{figure}[h]
$$\grd{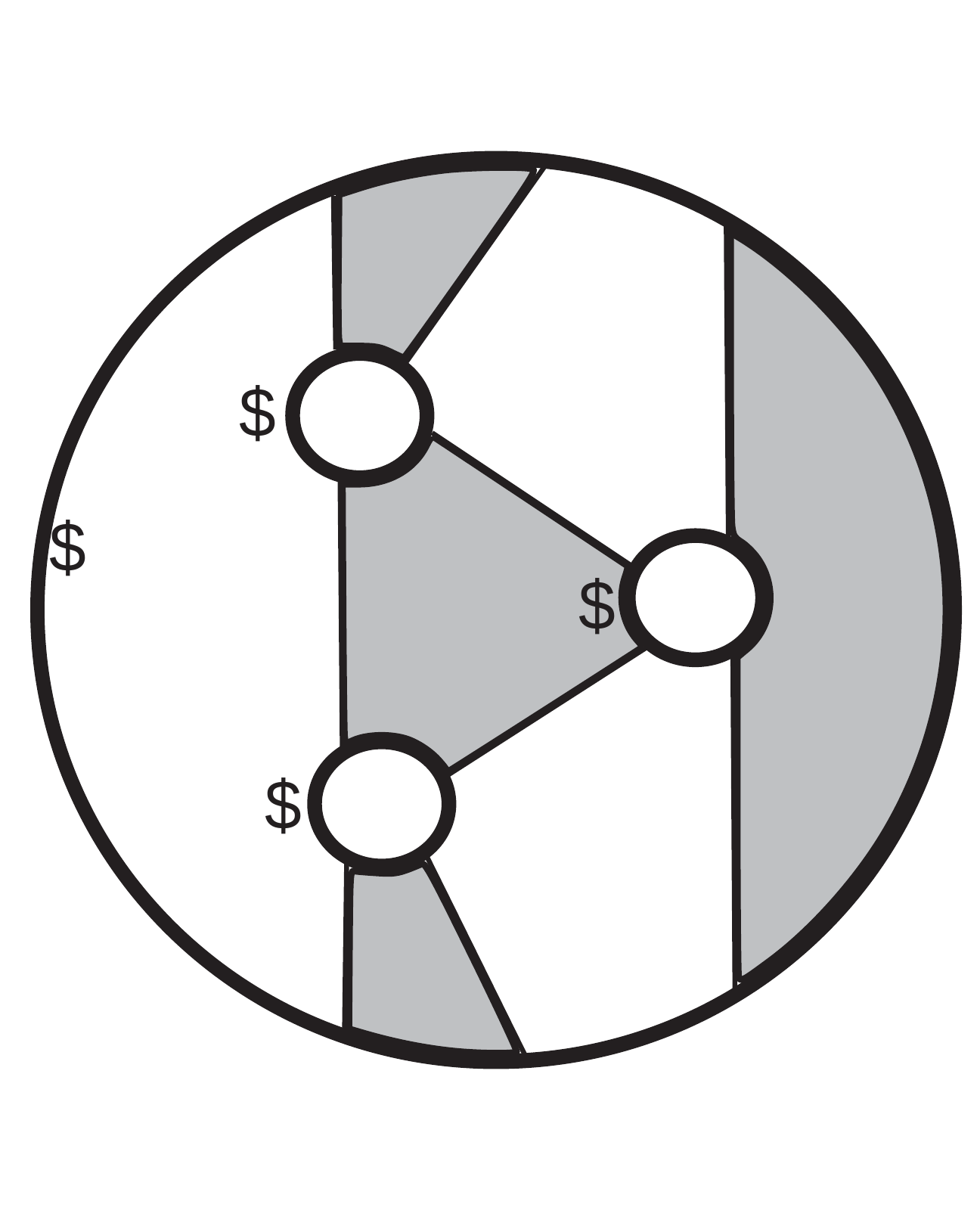}$$
\caption{A regular planar 3-tangle}
\end{figure}

\setcounter{Thm}{0}
\begin{definition}
A planar tangle will be called \textit{regular} if the distinguished point of each disc is on the left, and the distinguished points of the input discs are ordered vertically.  Let RT denote the set of \textit{regular planar tangles}.
\end{definition}
This means that the $x$ coordinate of each disc is the smallest one among all points on the boundary of the disc, and the $y$ coordinates of the input discs are pairwise different. Let $y(D)$ denote the $y$ coordinate of the distinguished point of an input disc $D$.

In certain situations one can compose two tangles $T$ and $S$ to obtain a tangle $T\circ_D S\in$ RT. To make this possible, the output disc of $S\in$ RT  must be identical to one input disc $D$ of a  $T\in$  RT. Furthermore $D$ must be lower than all the $\$$'s above the $\$$ of $D$, and it must be higher than the $\$$'s under the $\$$ of $D$.  This makes it possible to find a diffeomorphism of the plane that moves each disc in $T$, other than $D$, to be completely higher or lower than $D$.
Using this representative of the planar tangle $D$, one can then define the composition $T\circ_D S$ in the usual way: match the intervals and points of $D$ in $S$ with those of $T$.  Also replace any closed, contractible string formed in this composition by a scalar $\delta$, which we denote as  the \textit{circle parameter}.

\subsection{Planar para algebras}\label{section: planar para algebra}
Let $G$ be a finite abelian group and $\chi$ be a bicharacter of $G$.
\begin{definition}
A (shaded) $(G,\chi)$ planar \textcolor{blue}{para} algebra $\mathscr{P}_{\bullet}$ will be a family of $\mathbb{Z}/2\mathbb{Z}$-graded vector spaces indexed by the set $\mathbb{N}\cup\{0\}$, having the following properties:
	\begin{itemize}
	\item{} Let $\mathscr{P}_{n,\pm}$ denote the $\pm$ graded space indexed by $n$.
	\item{} To each regular planar $n$-tangle $T$ for $n\geq 0$ and $\mathcal{D}_T$ non-empty input discs, there will be a multilinear map
	\be
	Z_T:\times_{i\in\mathcal{D}_T}\mathscr{P}_{D_{i}}\to \mathscr{P}_{D_0}\;,
	\ee
where $\mathscr{P}_D$ is the vector space indexed by half the number of marked boundary points of $i$.
	\item{} The $\mathbb{Z}_{2}$ grading of each $\mathscr{P}_{i}$ is taken to be $+$ if the distinguished interval of $D_{i}$ is unshaded, or $-$ if it is shaded,  and similarly for $\mathscr{P}_{D_{0}}$.
\end{itemize}
\end{definition}

\begin{definition}
The map $Z_T$ is called the ``partition function'' of $T$ and is subject to the following \textcolor{blue}{five} requirements:
\end{definition}

\begin{itemize}
\item[(i)](\textcolor{blue}{RT} \textcolor{blue}{ isotopy invariance}) If $\varphi$ is a continuous map from $[0,1]$ to orientation preserving diffeomorphisms of $\mathbb{C}$, such that $\varphi_0$ is the identity map and $\varphi_t(T) \in$ RT,      then
	\[
		\mathbb{Z}_T= Z_{\varphi_1(T)}\;,
	\]
where the sets of internal discs of $T$ and $\varphi_t(T)$ are identified using $\varphi_t$, for $t\in[0,1]$.

\item[(ii)](Naturality) If $T\circ_D S$ exists and $\mathcal{D}_S$ is non-empty
$$Z_{T\circ_D S}=Z_T\circ_D Z_S$$
where $D$ is an internal disc in $T$.

\item[(iii)](\textcolor{blue}{Grading})
Each vector space $\mathscr{P}_{n,\pm}$ is $G$ graded,
	\[
	\mathscr{P}_{n,\pm}=\oplus_{g\in G} \mathscr{P}_{n,\pm,g}\;,
	\quad\text{and}\quad
	Z_T:\otimes_{i\in \mathcal{D}_T} \mathscr{P}_{i,g_i} \to \mathscr{P}_{D_0, \sum_ig_i}\;.
	\]

\item[(iv)](\textcolor{blue}{Para isotopy})
Take $\mathscr{P}_g=\oplus_{n,\pm} \mathscr{P}_{n,\pm,g} $ for $g\in G$.  We have
$$\grb{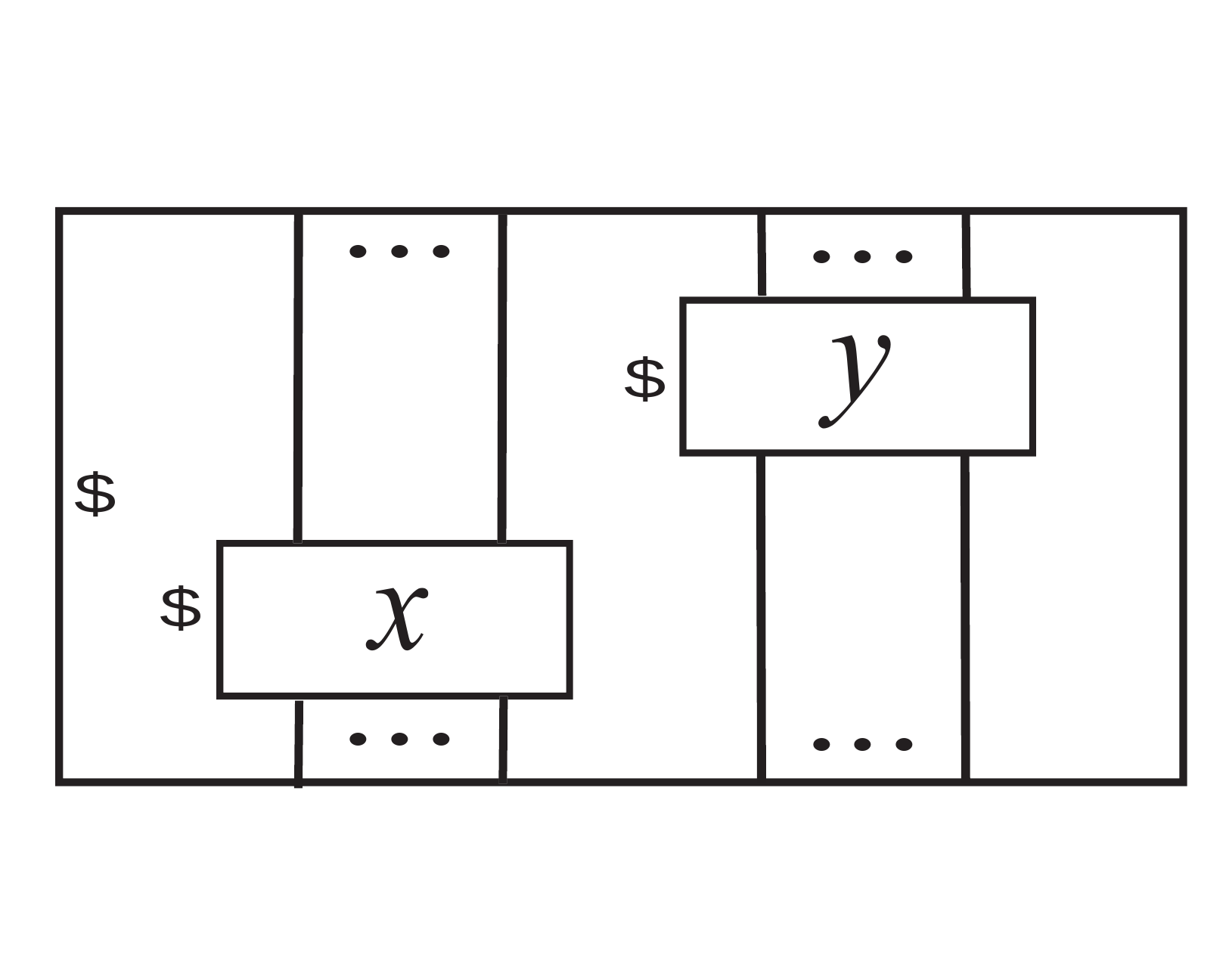}=\chi(g,h) \grb{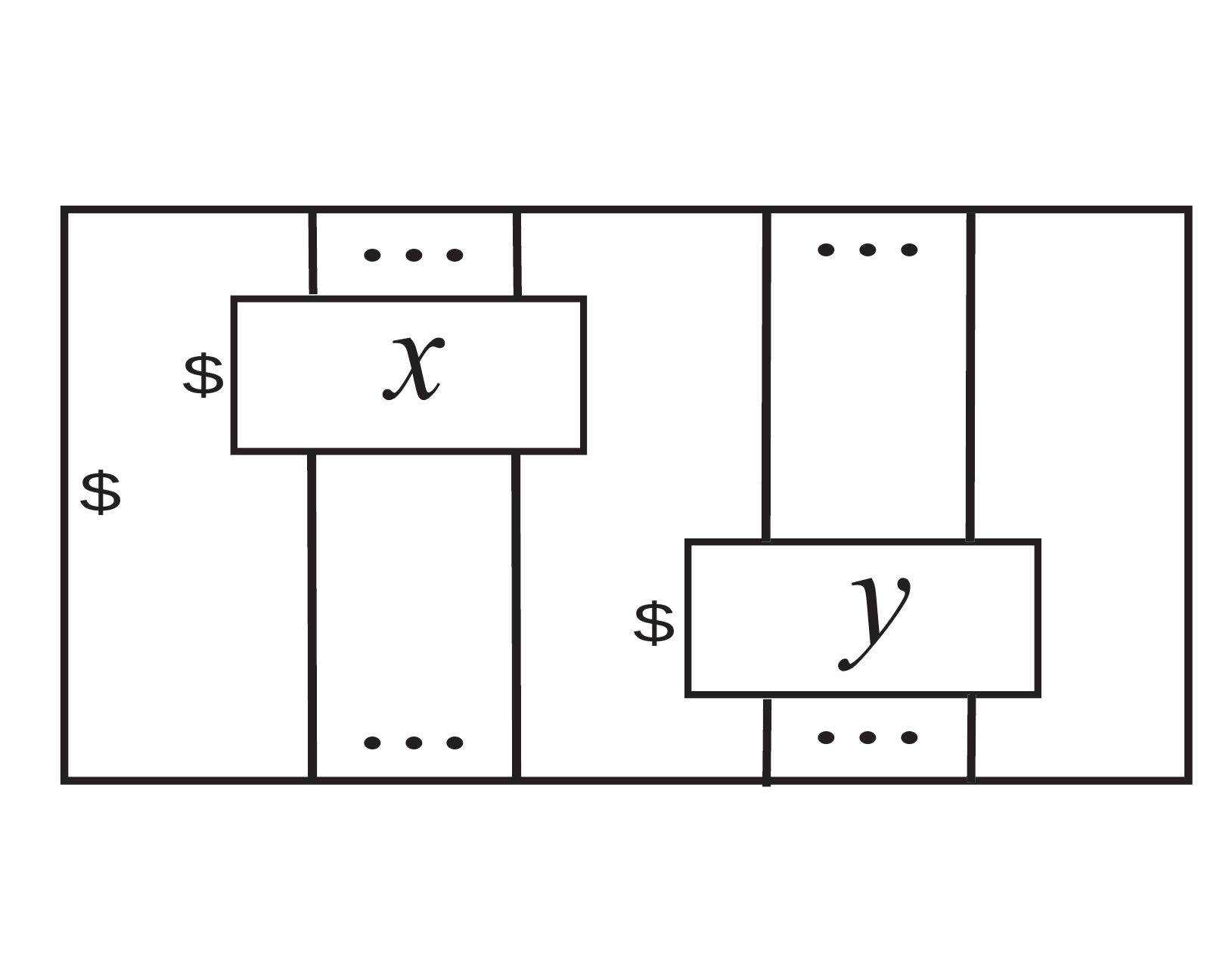}\;,$$
for any $x\in \mathscr{P}_g$ and $y \in \mathscr{P}_h$
\item[(v)] (\textcolor{blue}{Rotation}) The clockwise $2\pi$ rotation of any $g$ graded vector $x$ is $\chi(g,g)x$, i.e.,
 $$\grb{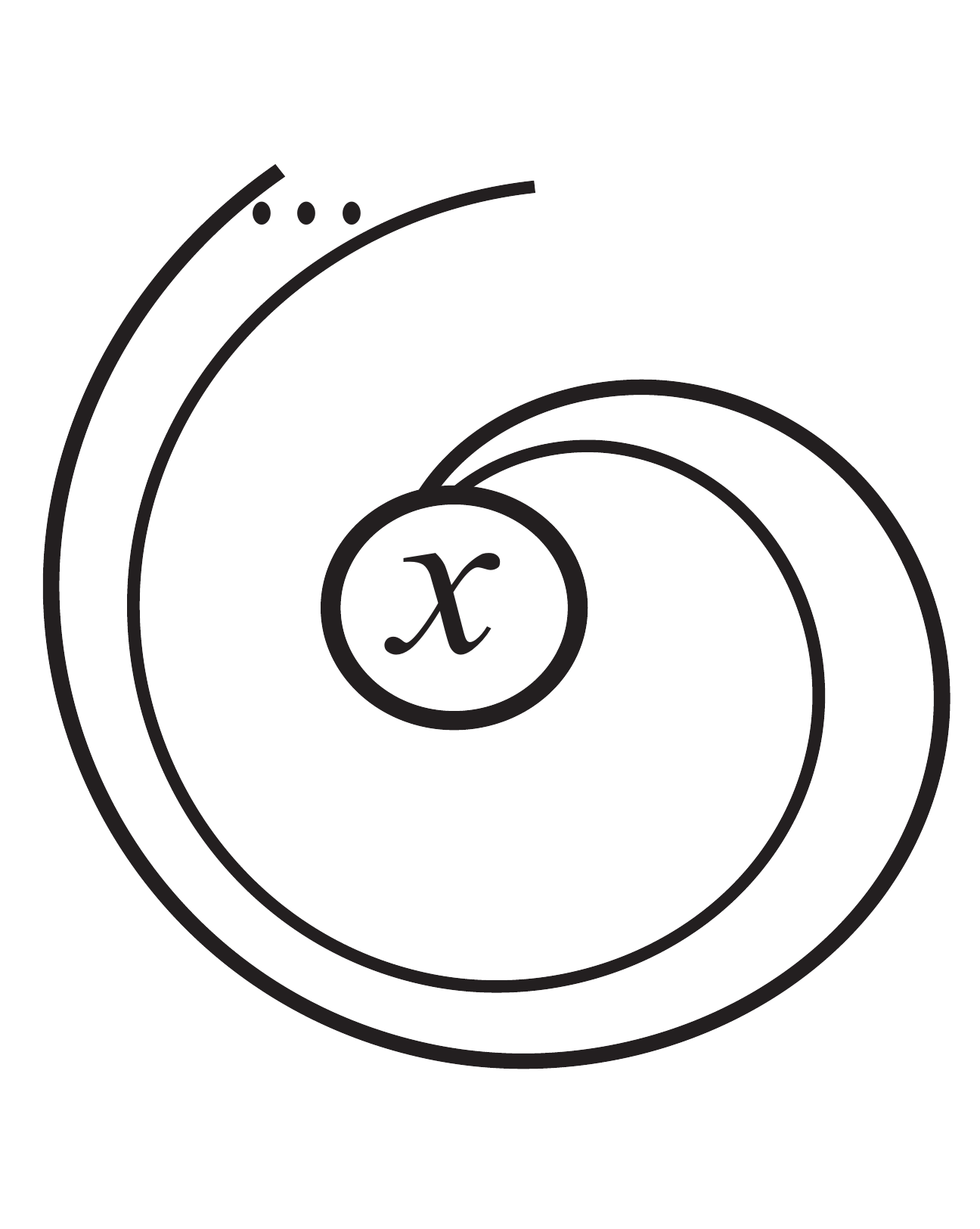}=\chi(g,g)x.$$
\end{itemize}

\begin{remark}
``Planar algebras" satisfying conditions (i) and (ii) have their own interests.  Conditions (iii), (iv) and (v) are motivated by the discussion of parafermion algebras in \cite{JafPed,JafJan2015,JafJan2016}.
\end{remark}

\begin{remark}
One can remove the condition that the $\$$ is on the left, and introduce the rotation isotopy for arbitrary angle, not only $2\pi$.  However, this makes the definition and computation more complicated. For convenience, we choose a representative of planar tangles in the isotopy class by fixing the $\$$ sign on the left.
\end{remark}

\begin{remark}
When $\chi$ is the constant $1$, the planar para algebra is a planar algebra.
The zero graded planar para subalgebra is a planar algebra.
\end{remark}

\begin{definition}
A vector $x$ in $\mathscr{P}_g$ is called homogenous. The grading of $x$ is defined to be $g$, denoted by $|x|_{G}$, or $|x|$,  if it causes no confusion.
\end{definition}

\begin{notation}
Furthermore in case it cannot cause confusion, we omit the output disc and the $\$$ signs.
A vector in $\mathscr{P}_{m,\pm}$ is called an $m$-box. Usually we put $m$ strings on the top and $m$ strings on the bottom. Then the $m$-box space $\mathscr{P}_{m,\pm}$ forms an algebra, where we denote the multiplication of $x,y\in \mathscr{P}_{m,\pm}$ diagrammatically by
$$\graa{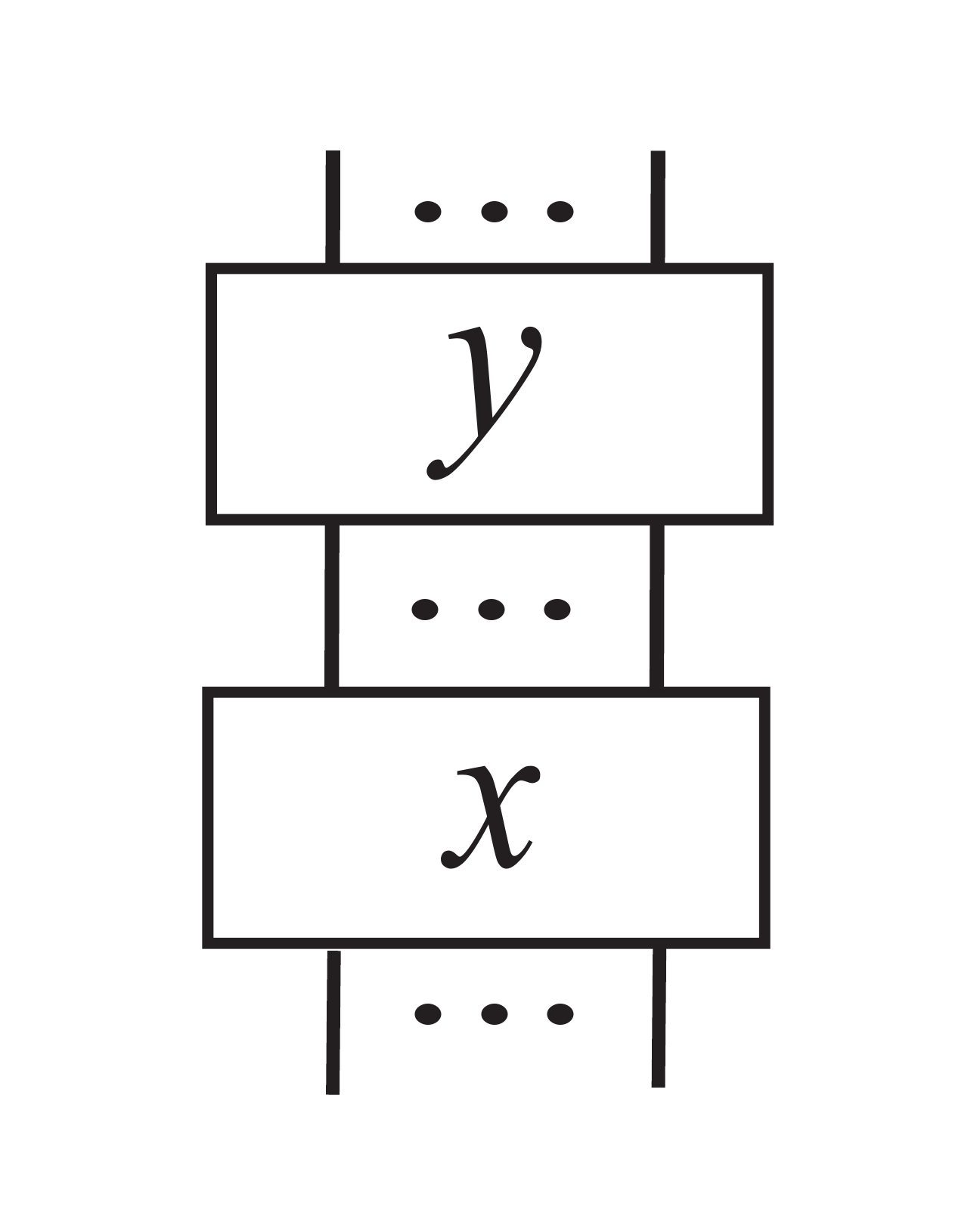}.$$
The identity is given by the diagram with $m$ vertical strings, denoted by $I_m$.
\end{notation}

\begin{definition}
We denote the graded tensor product as follows:
\[
x\otimes_+y=\graa{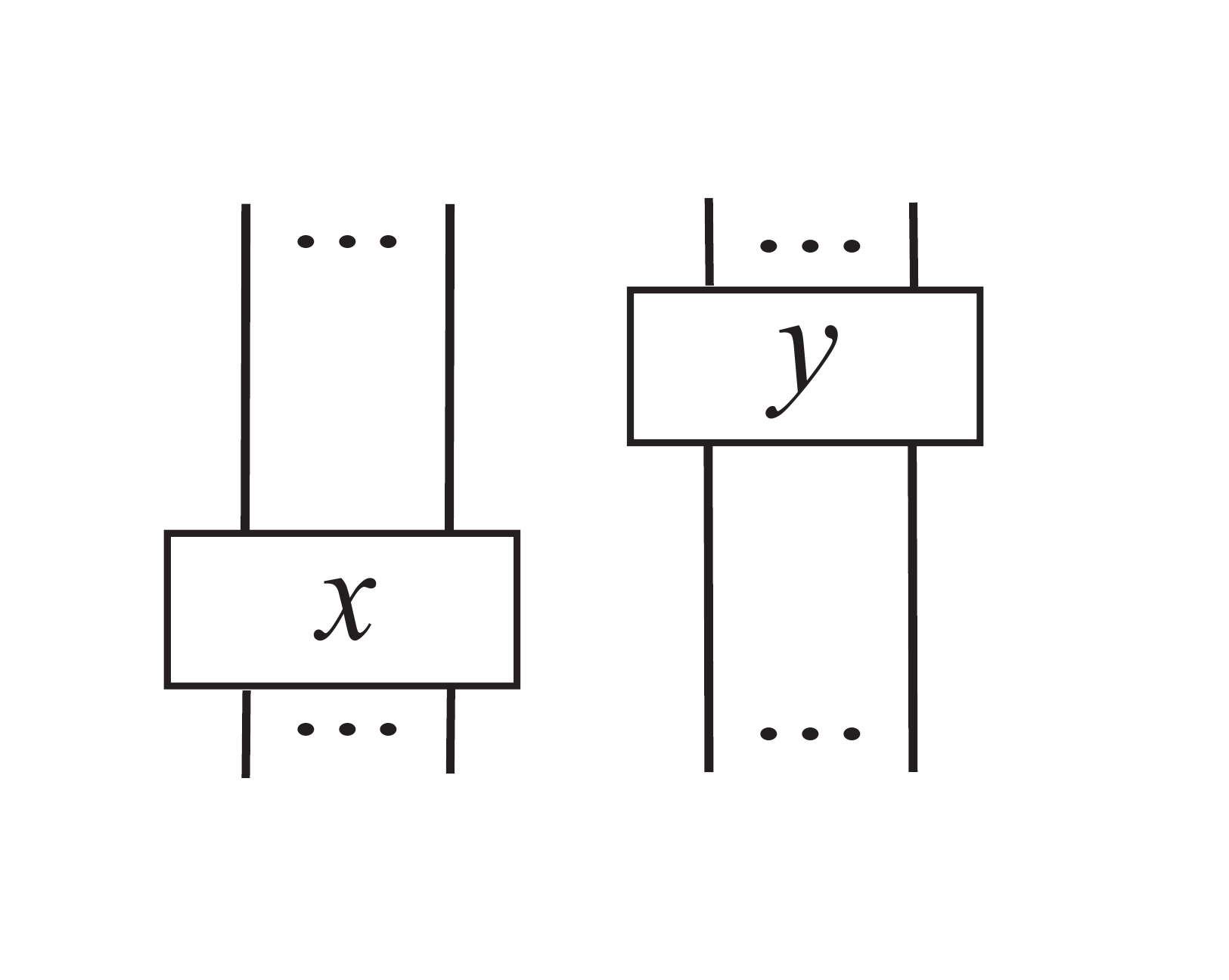}\;,
\quad
x\otimes_-y=\graa{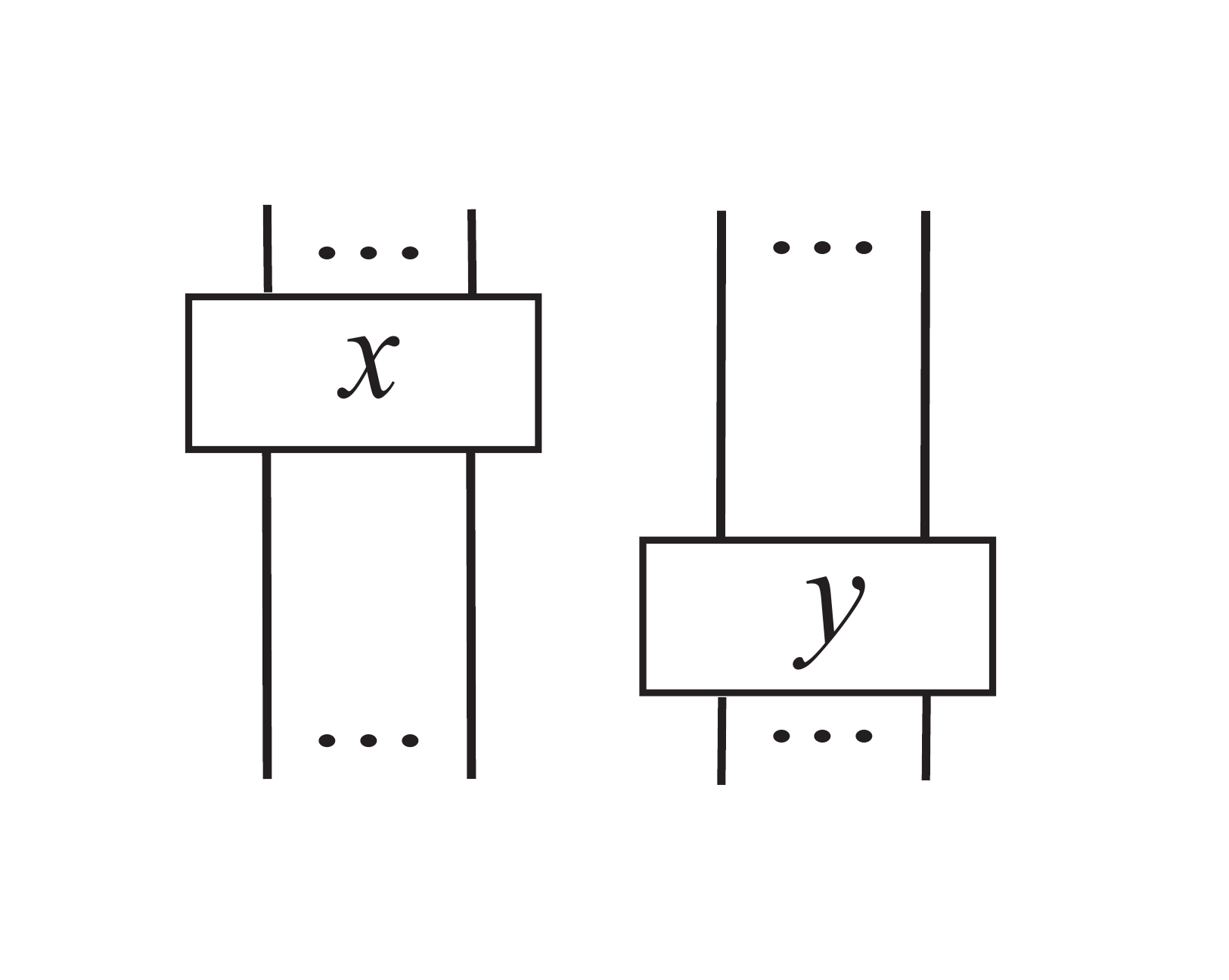}\;.
\]
\end{definition}
If $x$ and $y$ are homogenous, then we infer from para isotopy that $x\otimes_+y=\chi(|x|,|y|)\,x\otimes_- y$.
Under the multiplication and the graded tensor product $\otimes_+$, one obtains a $(G,\chi)$ graded tensor category. The objects are given by zero graded idempotents and the morphisms are given by maps from idempotents to idempotents.  We refer the readers to \cite{ENO,MPSD2n} for the planar algebra case.

\begin{definition}
A planar para algebra is called unital if the empty disc is a vector in $\mathscr{P}_{0,\pm,0}$, called the vacuum vector.
\end{definition}

\begin{definition}
A unital planar para algebra is called spherical, if \underline{$\dim \mathscr{P}_{0,\pm}=1$} and
$$\graa{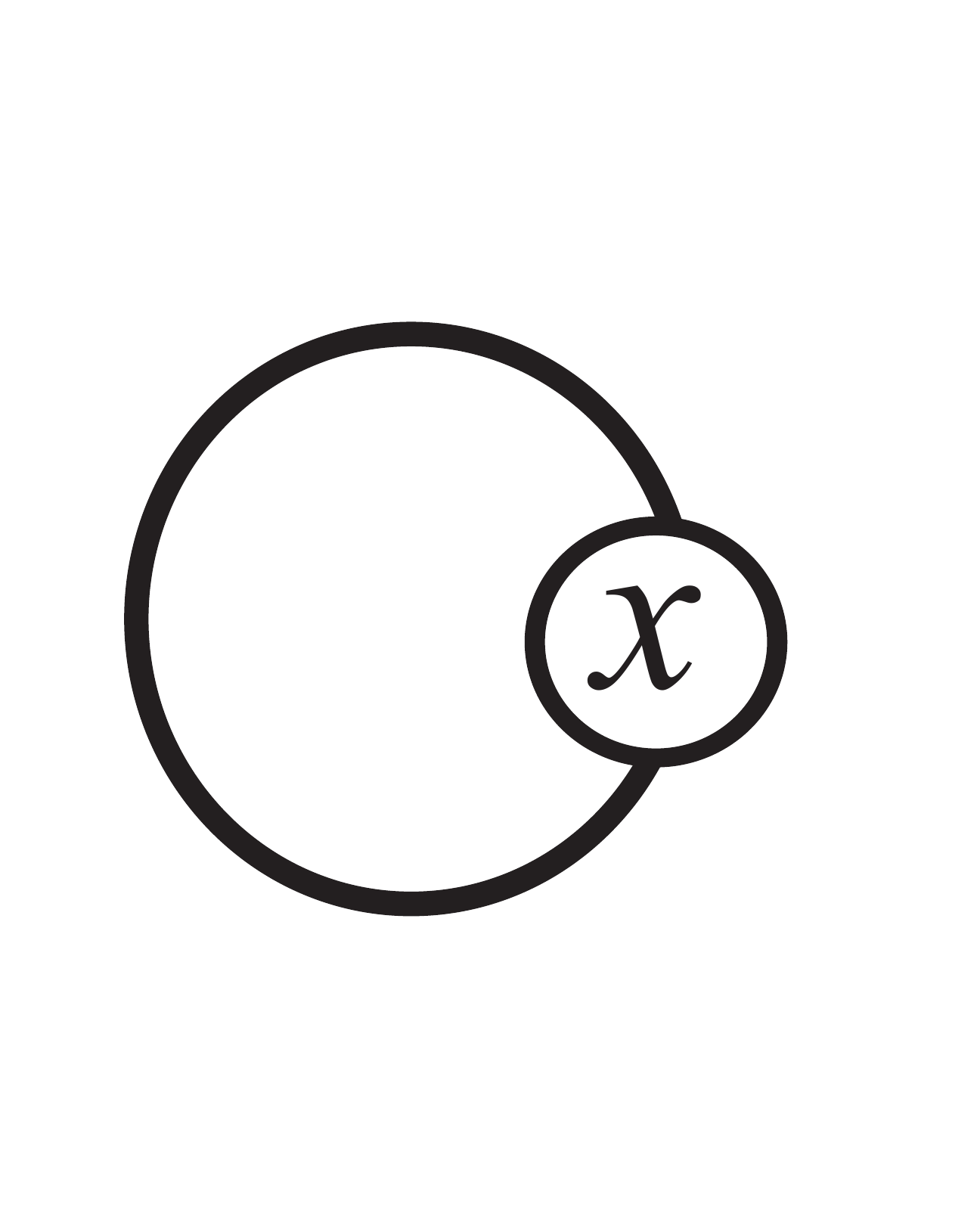}=\graa{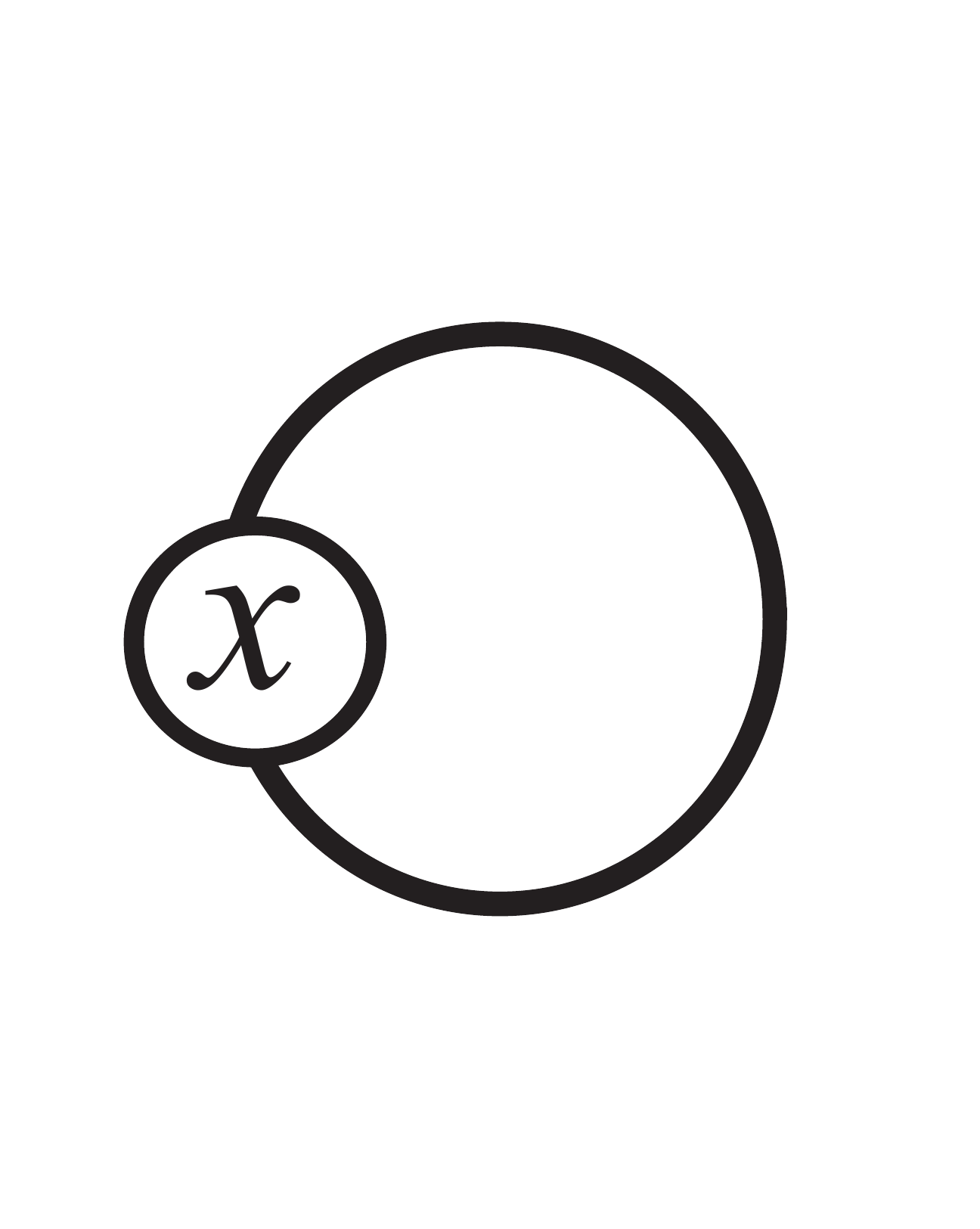},$$
for any 1-box $x$.
 Both $\mathscr{P}_{0,+}$ and $\mathscr{P}_{0,-}$ are identified as the ground field.
\end{definition}

\begin{proposition}
The linear functional $\grb{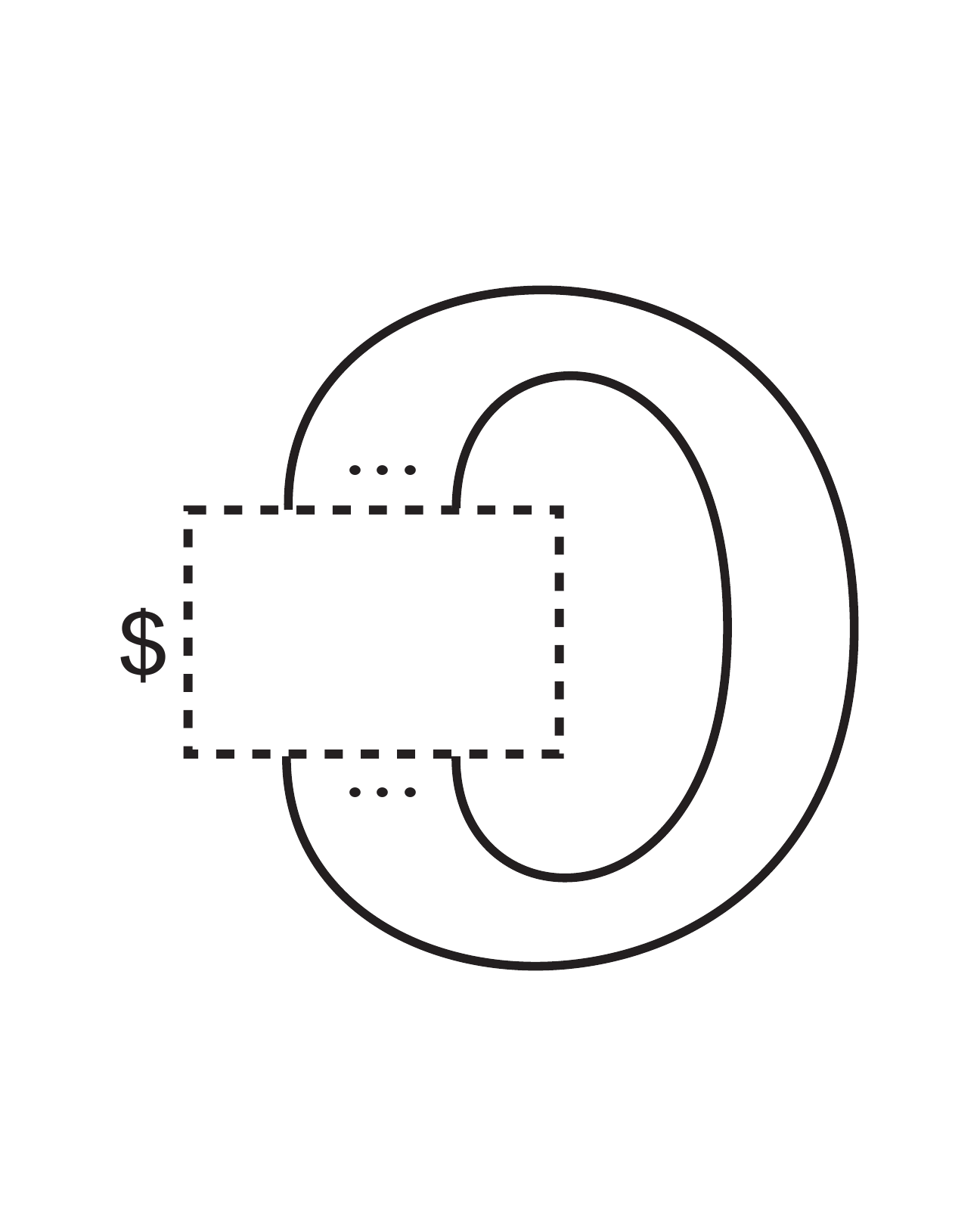}$ on $m$-boxes is a trace, i.e., {}
  $$\grb{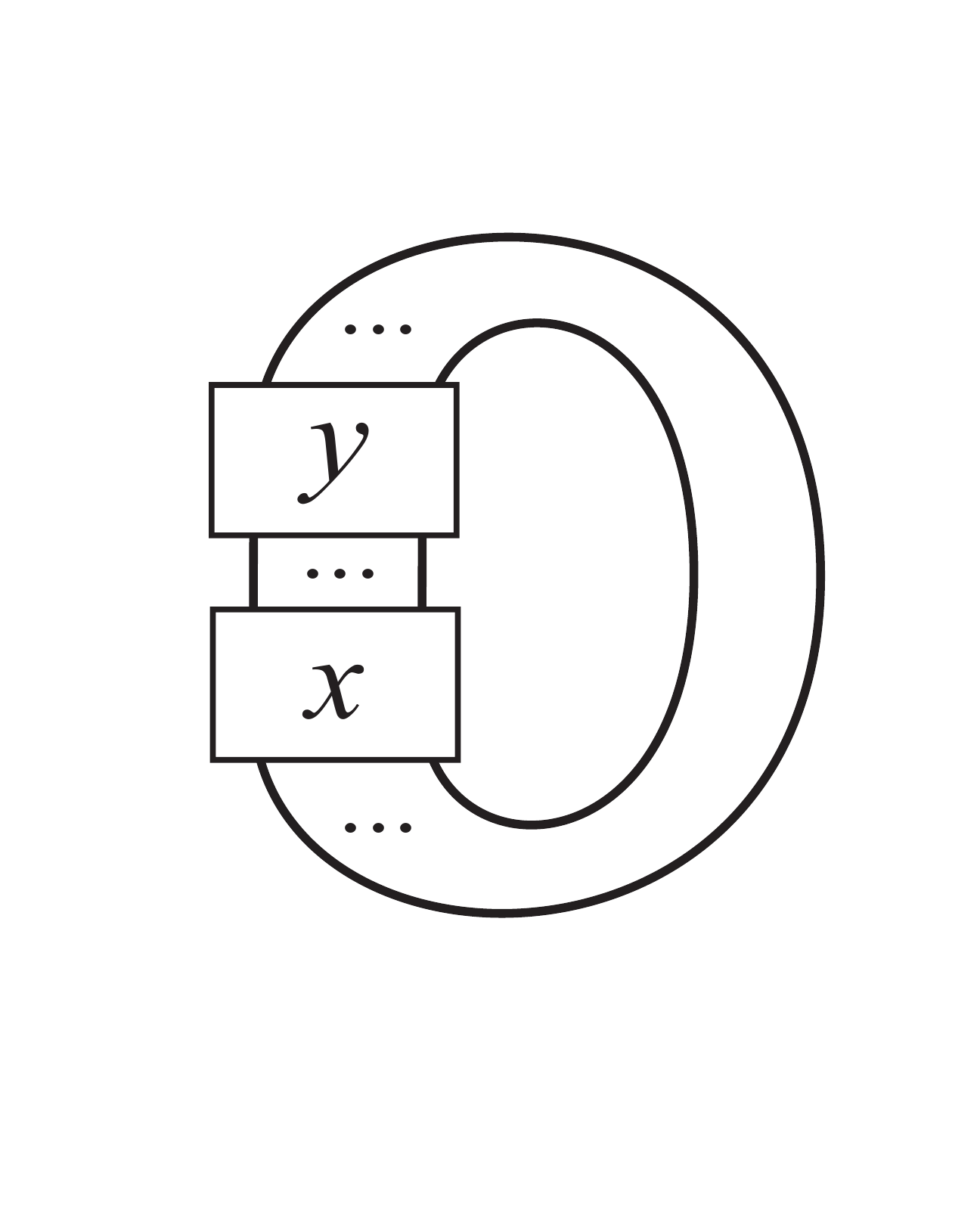}=\grb{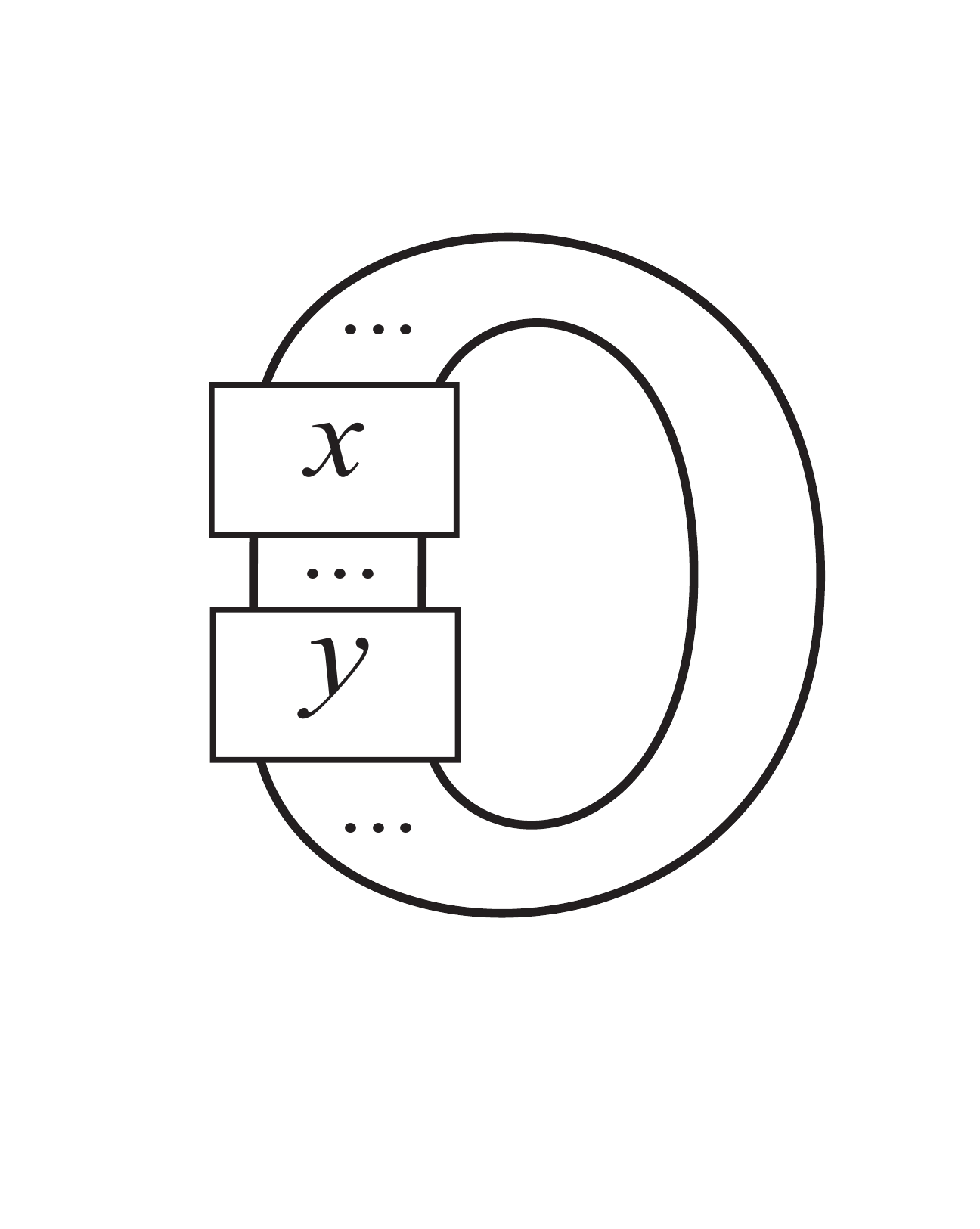}.$$
We call it the (unnormalized) Markov trace.
\end{proposition}

\begin{proof}
It is enough to prove the equation for any homogenous $x$ and $y$. When the grading $|x|+|y|$ is not 0 mod $N$, both sides are zeros. When the grading $|x|+|y|$ is 0 mod $N$, applying the para isotopy and the $2\pi$ rotation of $x$, we obtain the equality.
\end{proof}

The normalized Markov trace $tr$ on $m$-boxes is given by $\frac{1}{\delta^m}\grb{tr}$.
The inclusion from $\mathscr{P}_{m,\pm}$ to $\mathscr{P}_{m+1,\pm}$ by adding one string to the right preserves the normalized trace.

\subsection{String Fourier transforms}
The string Fourier transform\footnote{We originally called this transformation  the ``quantum Fourier transform.''  Afterwards we realized that in quantum information one gives this name to usual Fourier series on $\Z_{d}$. So we replace ``quantum'' by ``string'' to reflect the geometric nature of this transformation.} $\mathfrak{F}_{\text{s}}$ 
is an important ingredient in planar (para) algebras. It behaves as a rotation in planar algebras.
The SFT is defined as the action of the following tangle,
$$\graa{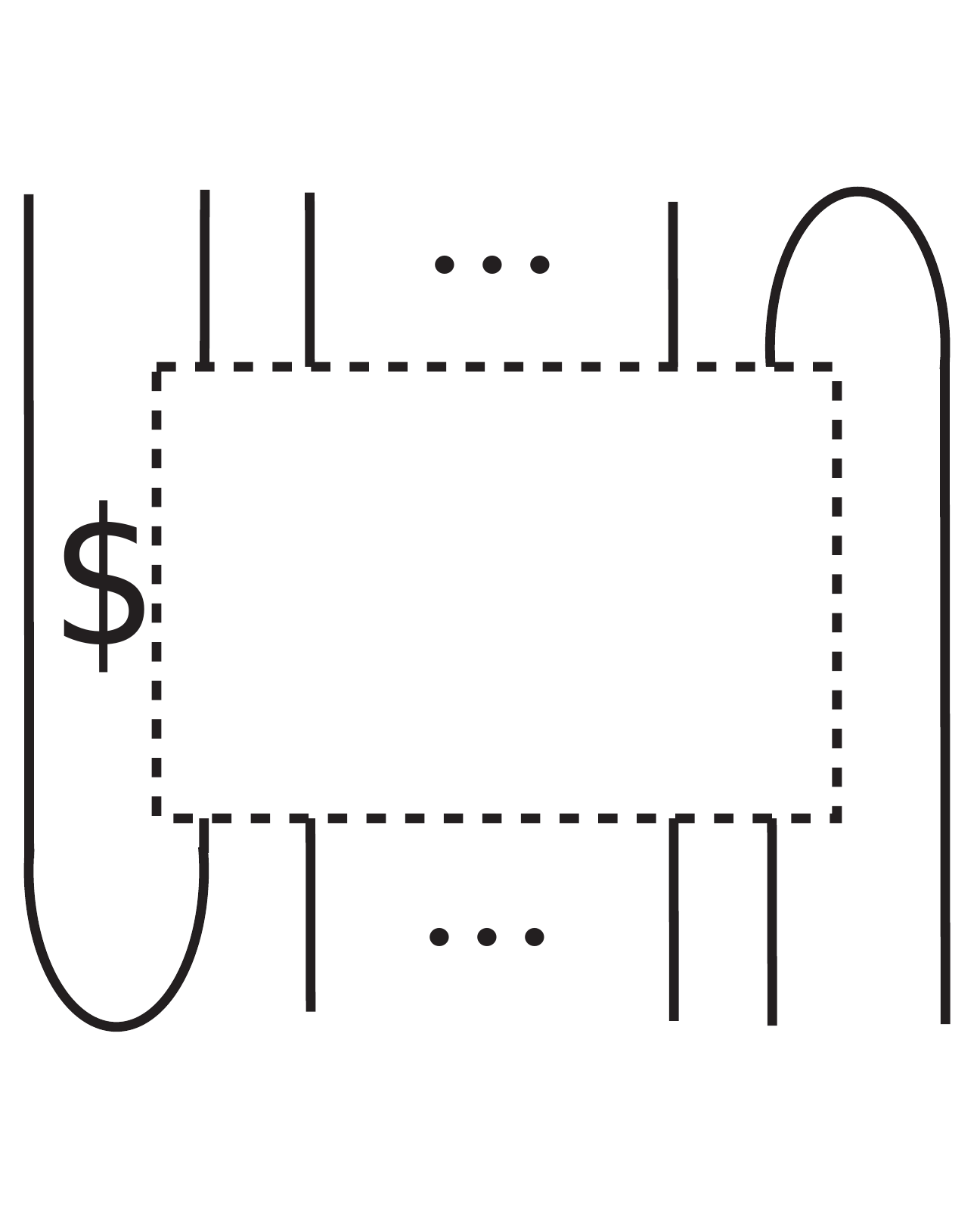}.$$
This definition is motivated by the Fourier transform on paragroups introduced by Ocneanu \cite{Ocn88}.
We also use other rotations on the $m$-box space, which are powers of the SFT:
\begin{itemize}
\item  Denote the $2\pi$ rotation by $\rho_{2\pi}=\mathfrak{F}_{\text{s}}^{2m}$.
\item Denote the $\pi$ rotation by $\rho_{\pi}=\mathfrak{F}_{\text{s}}^{m}$, which one also calls the contragredient map.  Note
	\be\label{PiRotation}
	\rho_{\pi}(xy)=\rho_{\pi}(y) \rho_{\pi}(x)\;,
	\quad\text{and}\quad
	\rho_{\pi}(x\otimes y)=\rho_{\pi}(y) \otimes \rho_{\pi}(x)\;.	
	\ee
\item For even $m$, denote the $\frac{\pi}{2}$ rotation by $\rho_{\frac{\pi}{2}}=\mathfrak{F}_{\text{s}}^{\frac{m}{2}}$.  This can also be considered as the string Fourier transform.
\end{itemize}
We refer the readers to Section 4 in \cite{Liuex} and \cite{JiaLiuWu} on the study of the string Fourier transform on subfactor planar algebras.

\subsection{Reflections} 
\label{sect:Reflections}
Two reflections  that play distinct roles are reflections about a vertical or horizontal line. The vertical reflection defines the usual adjoint in a planar para algebra.
\begin{definition}
We say a planar para algebra $\mathscr{P}_{\bullet}$ is a *-algebra, if there is an anti-linear involution $*: \mathscr{P}_{m,\pm, g} \rightarrow \mathscr{P}_{m,\pm,-g}$, for each $m$ and $g\in G$;
and $Z_{T^*}(x^*)=Z_T(x)^*$, for any $x$ in the tensor power of $\mathscr{P}_{n,\pm}$, where the tangle $T^*$ is the vertical reflection of the tangle $T$.
\end{definition}

\begin{definition}
An anti-linear involution $\Theta$ on the unshaded planar para algebra $\mathscr{P}_{\bullet}$ is called a horizontal reflection,
if $\Theta: \mathscr{P}_{m,\pm, g} \rightarrow \mathscr{P}_{m,\pm,-g}$, for each even $m$ and $g\in G$;
$\Theta: \mathscr{P}_{m,\pm, g} \rightarrow \mathscr{P}_{m,\mp,-g}$, for each odd $m$ and $g\in G$;
and $\Theta(Z_T(x))= Z_{\Theta(T)}(\Theta(x))$, where the tangle $\Theta(T)$ is the horizontal reflection of the tangle $T$. In particular, the reflection $\Theta$ acts as  $\Theta(x\otimes_+ y)= \Theta(y) \otimes_- \Theta(x)$ and $\Theta(xy)=\Theta(x)\Theta(y)$.
\end{definition}

Consider the example of the group $G=\mathbb{Z}_{N}$, and  the bicharacter $\chi(j,k)=q^{jk}$, where $q=e^{\frac{2\pi i}{N}}$. Choose $\zeta$ to be a square root of $q$ such that $\zeta^{N^2}=1$. Then
	\be\label{ZetaChoice}
	\zeta
	= \left\{
	\begin{matrix}
	-e^{\frac{\pi i}{N}}\;, \text{ if $N$ is odd}\hfill\\
	\pm e^{\frac{\pi i}{N}}\;, \text{ if $N$ is even}\hfill
	\end{matrix}
	\right.\;.
	\ee
In the odd case with one solution, also $\zeta^{N}=1$.  In the even case one must choose one of the two solutions throughout, and also $\zeta^{N}=-1$.

\begin{proposition}\label{gauss sum}
Let $\zeta$ be a square root of $q=e^{\frac{2\pi i}{N}}$, such that $\zeta^{N^2}=1$. Define
	\be\label{Sum}
		\omega
		=
		\frac{1}{\sqrt{ N}}\sum_{j=0}^{N-1} \zeta^{j^{2}}\;.
		\quad \text{Then}\quad
		\abs{\omega}=1\;.
	\ee
\end{proposition}

\begin{proof}
The Fourier transform $F$ on $\Z_{N}$ is
	\be\label{FourierInversion}
		(Ff)(j)
		= \frac{1}{\sqrt{N}} \sum_{i=0}^{N-1} q^{ij} f(i)\;,
		\quad\text{with inverse}\quad(F^{2}f)(-i)=f(i)\;.
	\ee
%Let $\omega=\frac{1}{\sqrt{ N}}\sum_{j=0}^{N-1} \zeta^{j^{2}}$, so we wish to show $\abs{\omega}=1$. Also
Let $f(i)=\zeta^{i^2}$ and  $g(j)= \zeta^{-j^2}$. Then
\[
	(Ff)(j)
	= \frac{1}{\sqrt{N}} \sum_{i=0}^{N-1} q^{ij}\,\zeta^{i^2}
	= \frac{1}{\sqrt{N}} \sum_{i=0}^{N-1} \zeta^{(i+j)^2} \,\zeta^{-j^2}
	=\omega\, g(j)\;.
\]
In the last equality,  we use that $\zeta^{N^{2}}=1$, so the sum of $\zeta^{(i+j)^{2}}$ over $i$ is independent of $j$. Similarly $(Fg)(i)=\overline{\omega}f(i)$.   But using the Fourier inversion identity of \eqref{FourierInversion}, as well as $f(-i)=f(i)$ in our case, we infer $\abs{\omega}^{2}=1$.
\end{proof}

Recall that the $2\pi$ rotation is not the identity on a $(\mathbb{Z}_N,\chi)$ planar para algebra.
 Also recall that for $\zeta^{2N}=\zeta^{N^2}=1$, the power $\zeta^{-|x|_{}^2}$ is well-defined.
 \begin{definition}
Define the reflection $\Theta=\Theta_{\zeta}$ as an antilinear extension of the operator on homogeneous elements $x$ given by
 	\be\label{ReflectionDef}
	\Theta(x)= \zeta^{-|x|_{}^2} \rho_{\pi}(x^*)\;.
	\ee
\end{definition}

\begin{proposition}\label{reflections}
On a $(\mathbb{Z}_N,\chi)$ planar para *-algebra,
the map $\Theta$ defined in \eqref{ReflectionDef}  is a horizontal reflection.
\end{proposition}

\begin{proof}
The horizontal reflection is the composition of an anti-clockwise $\pi$ rotation, a vertical reflection and a complex conjugation.
Suppose $T(x)$ is a labelled tangle for a regular planar tangle $T$ and $\displaystyle x=\otimes_i x_i$. Assume that the $i$th label $x_i$ is graded by $g_i$. Then the label $\Theta(x_i)$ in $\Theta(x)$ is graded by $-g_i$, and $\Theta(x_i)=\zeta^{-g_i^2} \rho_{\pi}(x_i)$.
The para isotopy of each pair of labels contributes a scalar $q^{(-g_i)(-g_j)}$. Therefore
\begin{align*}
Z_{\Theta(T)}(\Theta(x))&= \prod_i \zeta^{-g_i^2} \times \prod_{i,i'} q^{-g_ig_{i'}} \rho_{\pi}(Z_{T^*}(x^*))\\
&= \zeta^{-|x|_{}^2} \rho_{\pi}(Z_T(x)^*)\\
&=\Theta(Z_T(x))\\
\end{align*}
\end{proof}

\subsection{The Twisted Tensor Product}\label{Sect:TwistedTensorProduct}
To introduce reflection positivity, the reflection $\Theta(x)$ should be the horizontal reflection of $x$; it should be represented as a box beside $x$, namely on the same level, with also the $\$$ signs on the same horizontal level.
But equal levels are not permitted in planar para algebras.

In order to avoid this difficulty, we introduce the twisted tensor product, which plays the same role as the twisted product for parafermion algebras in \cite{JafFab15,JafJan2016}.   For any homogenous $x$, we have $|\Theta(x)|_{}=-|x|_{}$. By para isotopy,
$$\Theta(x)\otimes_+ x=q^{-|x|_{}^2}\Theta(x)\otimes_- x.$$

\begin{definition}[\bf Twisted tensor product]\label{TwistedProduct-1}
Let  the twisted tensor product of $\Theta(x)$ and $x$ be
	\be\label{TT}
	\Theta(x) \otimes_{t} x:= \zeta^{|x|_{}^2 } \Theta(x)\otimes_+ x=\zeta^{-|x|_{}^2} \Theta(x)\otimes_- x\;,
	\ee
pictorially denoted by putting $x$ and $\Theta(x)$ on the same level,
	\be
	\raisebox{-.6cm}{
\tikz{
\draw (0-1/6,0) rectangle (1/2+1/6,1/2);
\node at (1/4,1/4) {\size{$\Theta(x)$}};
\node at (1/4,3/4) {\size{$\cdots$}};
\node at (1/4,-1/4) {\size{$\cdots$}};
\draw (0,0)--(0,-2/6);
\draw (1/2,0)--(1/2,-2/6);
\draw (0,1/2)--(0,5/6);
\draw (1/2,1/2)--(1/2,5/6);
\draw (1-1/6,0) rectangle (1+1/2+1/6,1/2);
\node at (1+1/4,1/4) {\size{$x$}};
\node at (1+1/4,3/4) {\size{$\cdots$}};
\node at (1+1/4,-1/4) {\size{$\cdots$}};
\draw (1+0,0)--(1+0,-2/6);
\draw (1+1/2,0)--(1+1/2,-2/6);
\draw (1+0,1/2)--(1+0,5/6);
\draw (1+1/2,1/2)--(1+1/2,5/6);
}}\;.
	\ee
\end{definition}

\begin{proposition}
For homogenous $x$ and $y$ in $\mathscr{S}_{m,\pm}$, we have
\begin{align*}
 (\Theta(x) \otimes_{t} x)(\Theta(y) \otimes_{t} y)&=\Theta(xy) \otimes_{t} xy.
\end{align*}
\end{proposition}

\begin{proof}
It follows from the equality
$\zeta^{|x|_{}^2 }\zeta^{|y|_{}^2} q^{|x|_{}|y|_{}}=\zeta^{(|x|_{}+|y|_{})^2 }.$
\end{proof}

\begin{proposition}\label{rotation isotopy}
  For any homogenous $m$-box $x$, we have
  $$\mathfrak{F}_{\text{s}}^{-m}\left(
  \raisebox{-.6cm}{
\tikz{
\draw (0-1/6,0) rectangle (1/2+1/6,1/2);
\node at (1/4,1/4) {\size{$\Theta(x)$}};
\node at (1/4,3/4) {\size{$\cdots$}};
\node at (1/4,-1/4) {\size{$\cdots$}};
\draw (0,0)--(0,-2/6);
\draw (1/2,0)--(1/2,-2/6);
\draw (0,1/2)--(0,5/6);
\draw (1/2,1/2)--(1/2,5/6);
\draw (1-1/6,0) rectangle (1+1/2+1/6,1/2);
\node at (1+1/4,1/4) {\size{$x$}};
\node at (1+1/4,3/4) {\size{$\cdots$}};
\node at (1+1/4,-1/4) {\size{$\cdots$}};
\draw (1+0,0)--(1+0,-2/6);
\draw (1+1/2,0)--(1+1/2,-2/6);
\draw (1+0,1/2)--(1+0,5/6);
\draw (1+1/2,1/2)--(1+1/2,5/6);
}}
  \right)=\grb{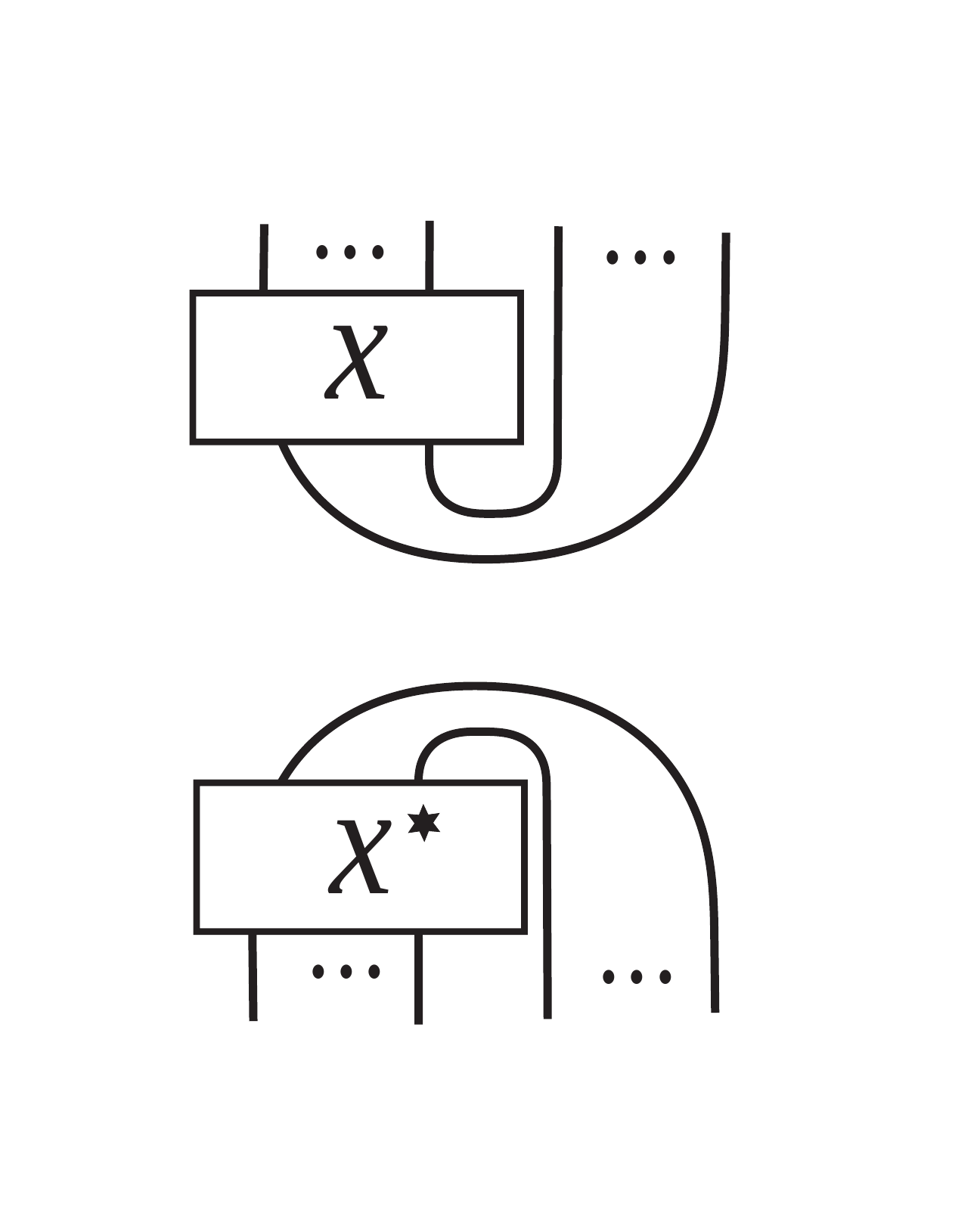}.$$
\end{proposition}

\begin{proof}
Note that $\Theta(x) \otimes_{t} x=\rho_{\pi}(x^*)\otimes_+ x$. Applying $\mathfrak{F}_{\text{s}}^{-m}$, we obtain the equality by isotopy.
\end{proof}

\begin{definition}[General twisted tensor product]\label{Def:twist tensor general}
To define the twisted tensor product in general, we lift the grading from $\Z_{N}$ to $\Z$ and define the twisted tensor product based on the lifted grading.
Suppose $x$ and $y$ are homogenous,
$i,j\in \mathbb{Z}$, such that $|x|,|y|$ are $i,j$ mod $N$.
We define the twisted tensor product of $(x,i)$ and $(y,j)$ as
\begin{align}
(x,i) \otimes_{t} (y,j) &:=  (\zeta^{-ij}x\otimes_+ y, i+j)  =(\zeta^{ij}  x\otimes_- y, i+j) \;.
\end{align}
We define $\Pi$ as $\Pi((x,i))=x$.
Then $\Pi((x,i) \otimes_{t} (y,j))$ is an interpolation between $x\otimes_+ y$ and $x\otimes_- y$.
\end{definition}
Note that when $N$ is odd, the interpolation $\zeta^{-ij}x\otimes_+ y$ is independent of the choice of $i,j$.
When $N$ is even, there are two  interpolations depending on the choice of $i,j$.
We use 
\be
\raisebox{-.6cm}{
\tikz{
\draw (0-1/6,0) rectangle (1/2+1/6,1/2);
\node at (1/4,1/4) {\size{$x,i$}};
\node at (1/4,3/4) {\size{$\cdots$}};
\node at (1/4,-1/4) {\size{$\cdots$}};
\draw (0,0)--(0,-2/6);
\draw (1/2,0)--(1/2,-2/6);
\draw (0,1/2)--(0,5/6);
\draw (1/2,1/2)--(1/2,5/6);
\draw (1-1/6,0) rectangle (1+1/2+1/6,1/2);
\node at (1+1/4,1/4) {\size{$y,j$}};
\node at (1+1/4,3/4) {\size{$\cdots$}};
\node at (1+1/4,-1/4) {\size{$\cdots$}};
\draw (1+0,0)--(1+0,-2/6);
\draw (1+1/2,0)--(1+1/2,-2/6);
\draw (1+0,1/2)--(1+0,5/6);
\draw (1+1/2,1/2)--(1+1/2,5/6);
}}
\;,
\quad\text{to denote}\quad \zeta^{-ij}x\otimes_+ y\;.
\ee
Moreover, we define $(x,i)^*:=(x^*,-i)$ and $\Theta((x,i)):=(\Theta(x),-i)$.
We can draw multiple diagrams on the same vertical level by the following proposition.

\begin{proposition}
For $x,y,z$ in $\mathscr{S}_{\bullet}$, we have
\be
((x,i)\otimes_{t}(y,j))\otimes_{t} (z,k)=(x,i)\otimes_{t}((y,j)\otimes_{t}(z,k)) \;.
\ee
\end{proposition}
\begin{proof}
It follows from the fact that
$ij+(i+j)k=i(j+k)+jk$.
\end{proof}

\begin{proposition}
For $x,y$ in $\mathscr{S}_{m,\pm}$, we have
\begin{align}
\Theta(\Theta((x,i)) \otimes_{t} (y,j))&=\Theta((y,j)) \otimes_{t} (x,i)\;,\\
(\Theta((x,i)) \otimes_{t} (y,j))^*&=\Theta((x,i)^*) \otimes_{t} (y,j)^*.
\end{align}

\end{proposition}
\begin{proof}
They follow from the definitions of $\otimes_t,\Theta,*$.
\iffalse
Since $\Theta$ is a horizontal reflection, we have that
\begin{align*}
\Theta(\Theta(x) \otimes_{t} y)&=\Theta(\zeta^{|x|_{+} |y|_{+}}\Theta(x) \otimes_{+} y)\\
&=\zeta^{-|x|_{+} |y|_{+}} \Theta(y) \otimes_{-} x\\
&=\Theta(y) \otimes_{t} x.
\end{align*}

\begin{align*}
\Theta(\Theta(x) \otimes_{t} y)&=\Theta(\zeta^{|x|_{+} |y|_{+}}\Theta(x) \otimes_{+} y)\\
&=\zeta^{-|x|_{+} |y|_{+}} \Theta(y) \otimes_{-} x\\
&=\Theta(y) \otimes_{t} x.
\end{align*}

Since $*$ is a vertical reflection, we have that
\begin{align*}
(\Theta(x) \otimes_{t} y)^*&=(\zeta^{|x|_{+} |y|_{+}} \Theta(x) \otimes_{t} y)^*\\
&=\zeta^{-|x|_{+} |y|_{+}} \Theta(x^*) \otimes_{-} y^*\\
&=\Theta(x^*) \otimes_{t} y^*\;.
\end{align*}
\fi
\end{proof}

\begin{notation}
  In the parafermion planar para algebras (PAPPA) (see \S \ref{positivity}), if $x$ is a homogenous 1-box, then $(x,i)$ is determined by $i$. Thus we simply use $i$ to denote $(x,i)$.
\end{notation}

\subsection{Subfactor planar para algebras}
The $m$-box space of a planar para *-algebra has an inner product $tr(x^*y)$ for $m$-boxes $x$ and $y$.

\begin{definition}
A subfactor planar para algebra $\mathscr{P}_{\bullet}$ will be a spherical planar para *-algebra with $\dim \mathscr{P}_{m,\pm}<\infty$ for all $m$, and such that the inner product is positive.
\end{definition}

\begin{remark}
We call it a subfactor planar para algebra, because a subfactor planar para algebra is the graded standard invariant of a $G$ graded subfactor.
The general theory will be discussed in a coming paper.
Motivated by the deep work of Popa \cite{Pop90,Pop94}, we conjecture that strongly amenable graded hyperfinite subfactors of type II$_1$ are classified by subfactor planar para algebras.
\end{remark}

When $\chi=1$, the subfactor planar para algebra $\mathscr{P}_{\bullet}$ is a ($G$ graded) subfactor planar algebra.
The zero graded part of a subfactor planar para algebra is a subfactor planar algebra.

Many notions of subfactor planar algebras are inherited for subfactor planar para algebras, such as the Jones projections, the basic construction, principal graphs, depths. We refer the readers to \cite{Jon83,JonPA} for the planar algebra case.

\begin{definition}
A subfactor planar para algebra $\mathscr{P}_{\bullet}$ is called irreducible, if $\dim \mathscr{P}_{1,\pm,0}=1$.
\end{definition}

\subsection{Examples}\label{construction}
Skein theory is a presentation theory for planar algebras in terms of generators and (algebraic and topological) relations. One can study the skein theory for planar para algebras in a similar way.
We refer the reader to \cite{JonPA} for the skein theory of planar algebras (in Section 1) and many interesting examples (in Section 2). Also see \cite{BMPS,LiuYB} for the skein-theoretic construction of the extended Haagerup planar algera and a new family of planar algebras.

Let us construct a spherical unshaded planar para algebra with the para symmetry $(\mathbb{Z}_N,\chi)$. We take the same bicharacter that we considered in \S\ref{sect:Reflections}, namely $\chi(j,k)=q^{jk}$, where $q=e^{\frac{2\pi i}{N}}$ and choose  $\zeta$ to be a square root of $q$ given in \eqref{ZetaChoice}, such that $\zeta^{N^2}=1$.
This planar para algebra plays the role of the Temperley-Lieb-Jones planar algebra among planar para algebras with para symmetry $(\mathbb{Z}_N,\chi)$.

Let $\mathscr{P}_{\bullet}$ be the unshaded planar algebra over the field $\mathbb{C}(\delta)$ generated by a 1-box $c$, graded by $1$, and satisfying the following relations:
\begin{itemize}
\item[(1)]$\graa{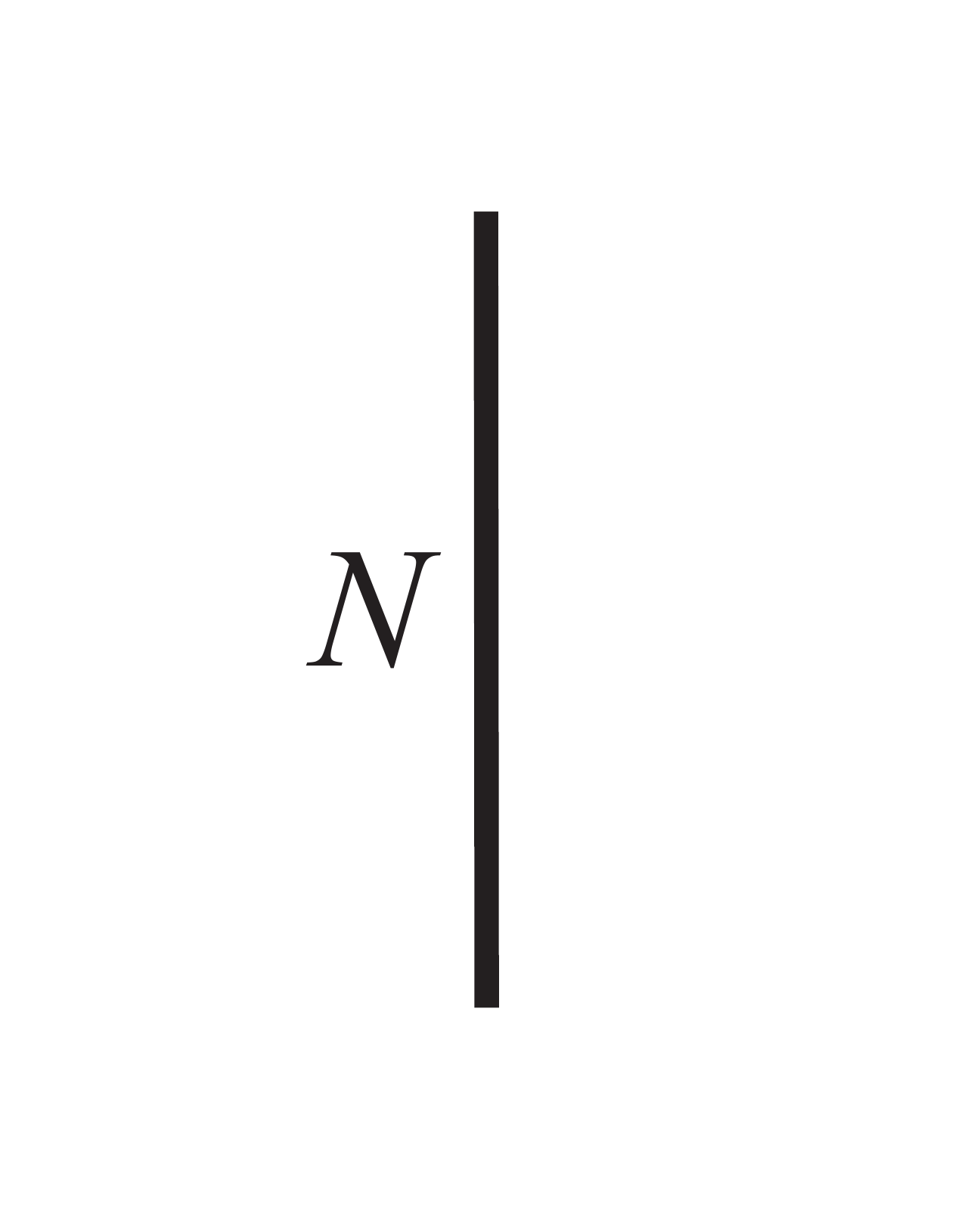}=\graa{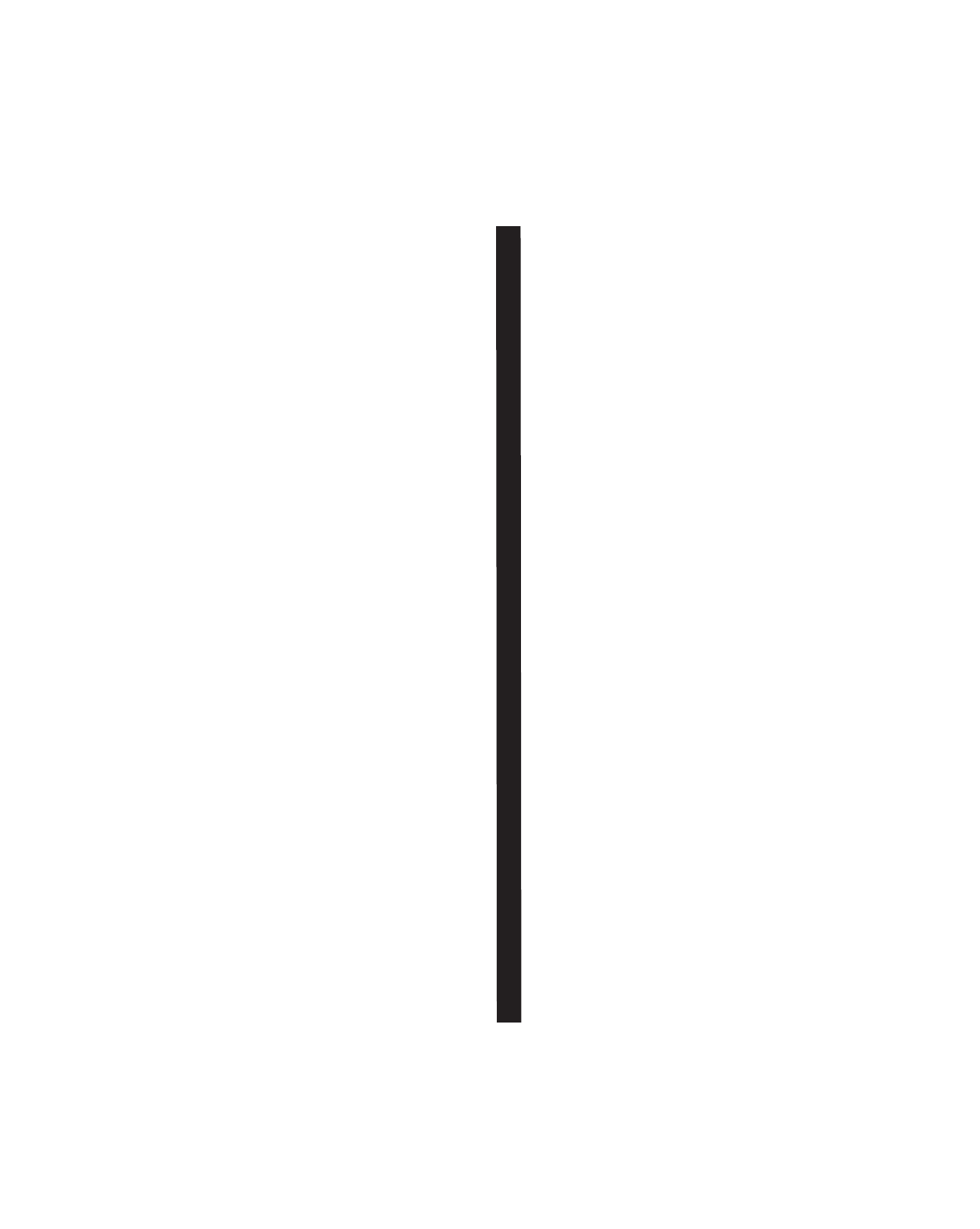}$,

\item[(2)]$\graa{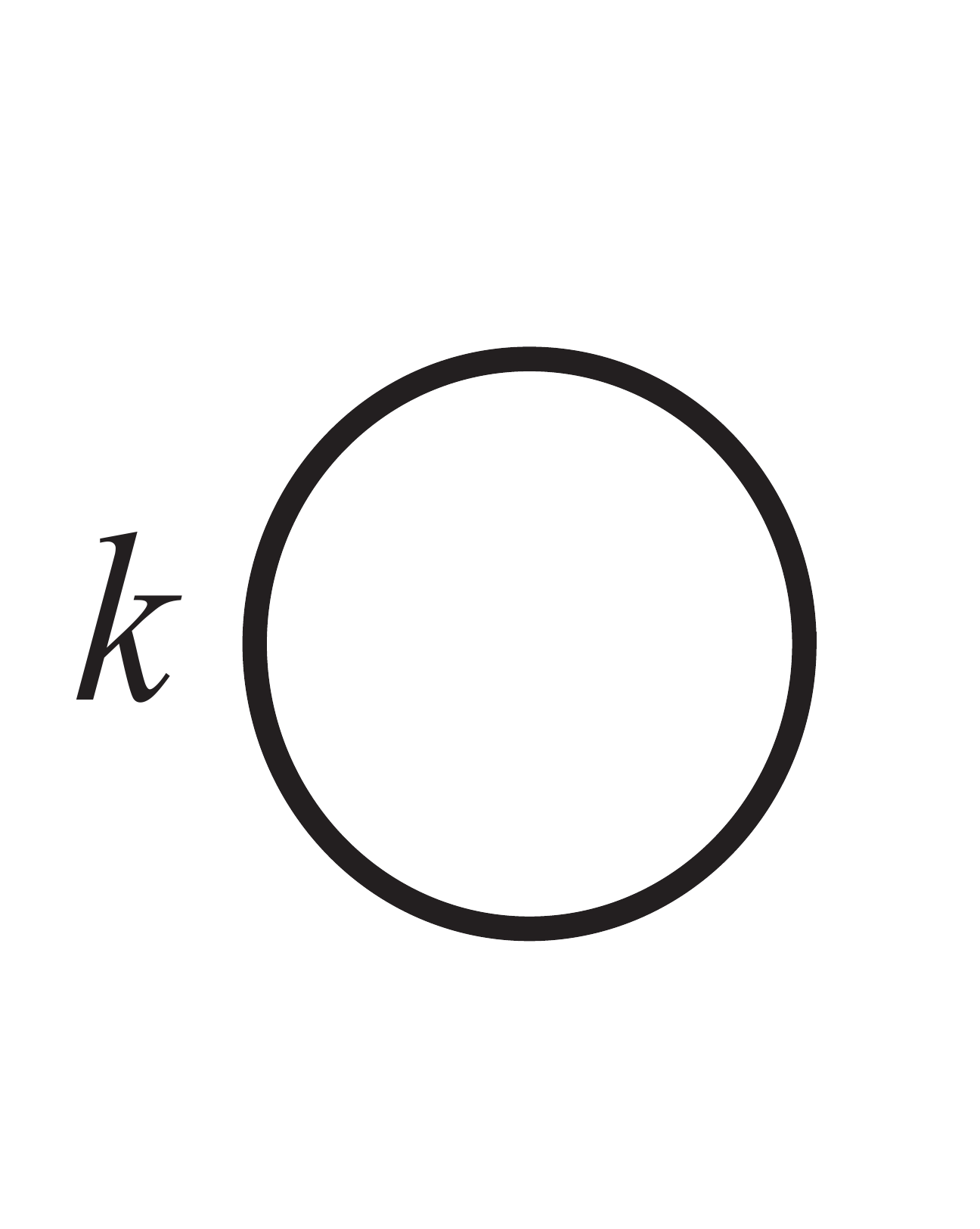}=0$, for $1\leq k\leq N-1$,

\item[(3)]$\graa{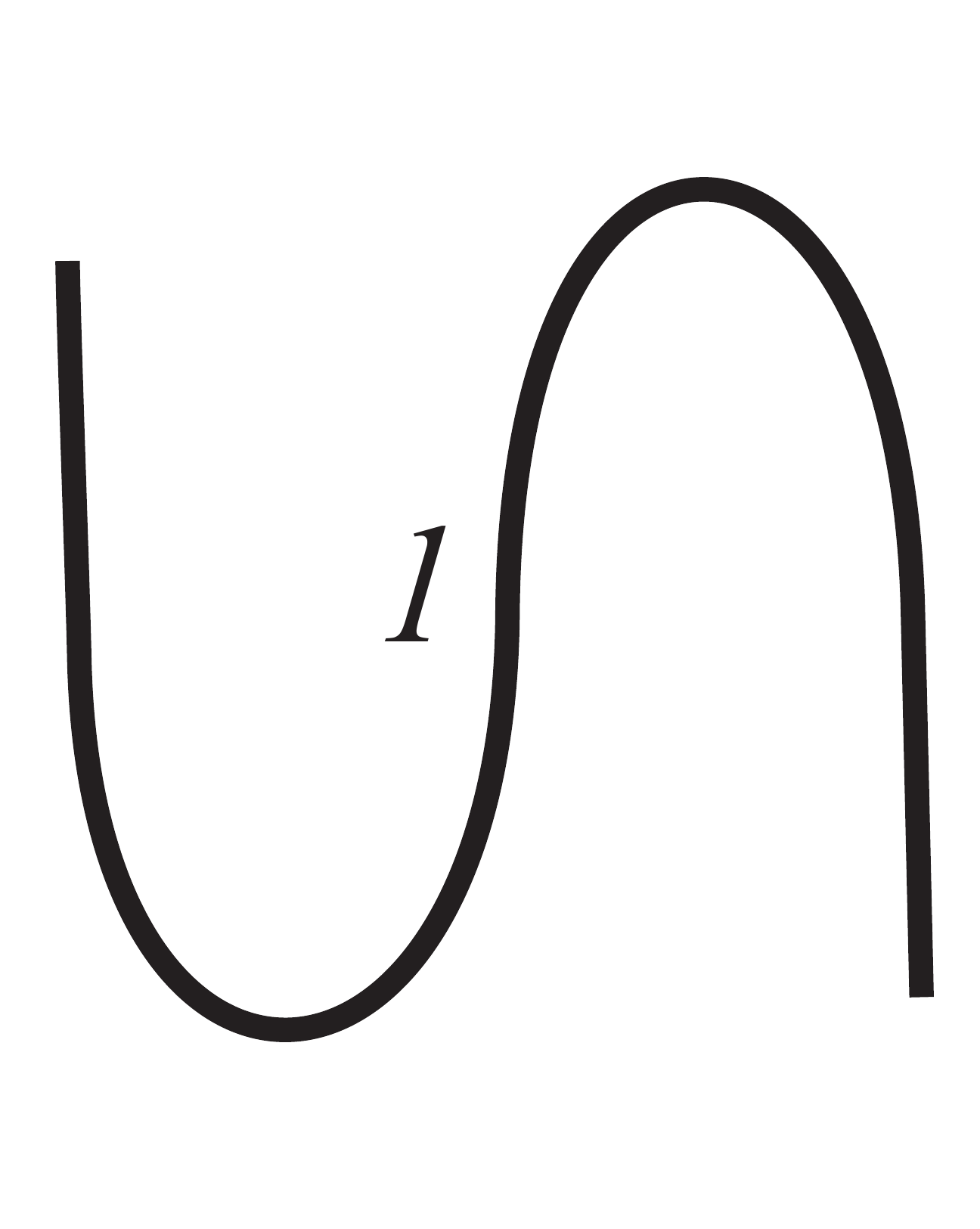}=\zeta \graa{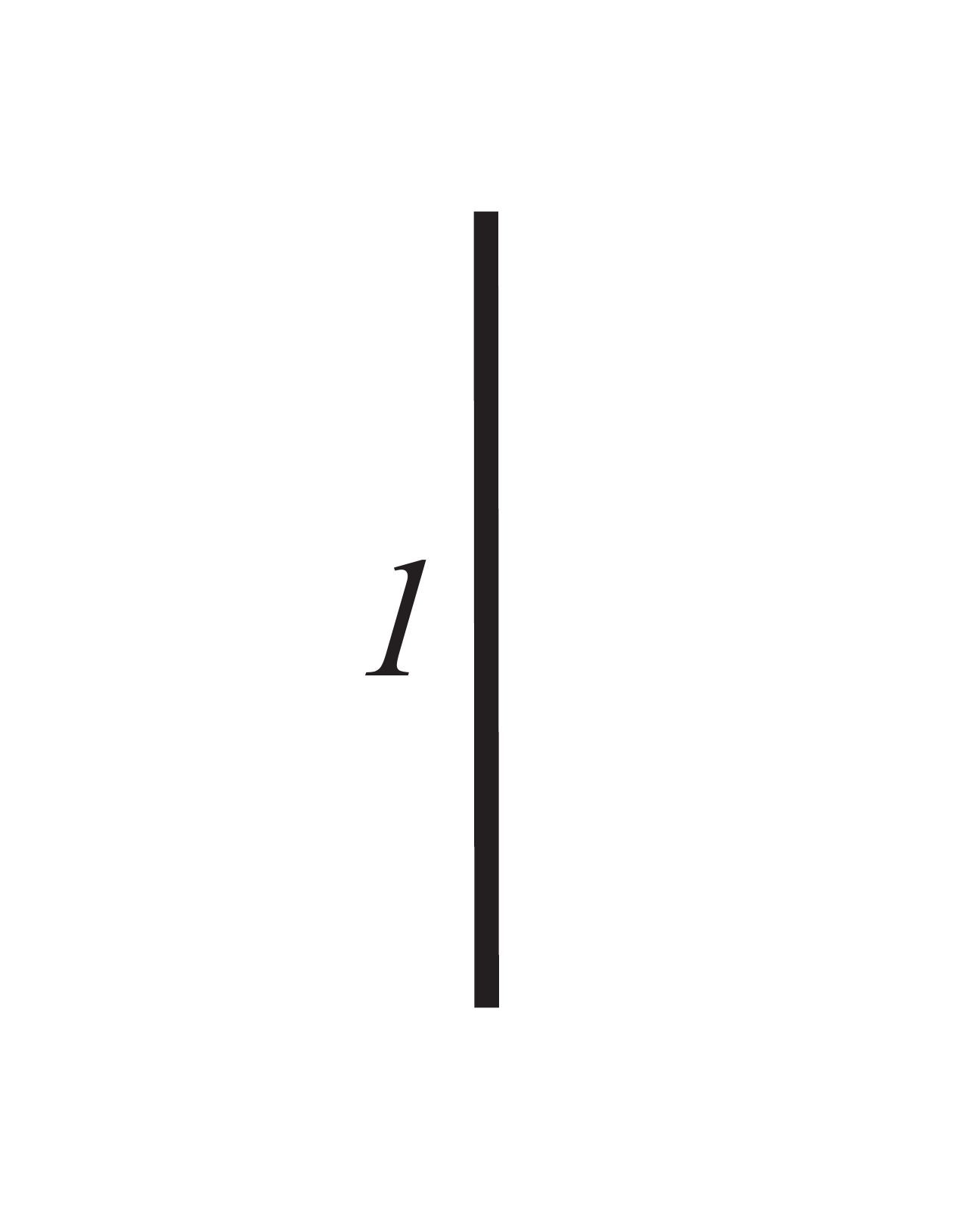}$, \quad\quad\quad\quad\quad\quad \text{namely Fourier-parafermion relation,}
\end{itemize}
where $\delta$ is the circle parameter and $\graa{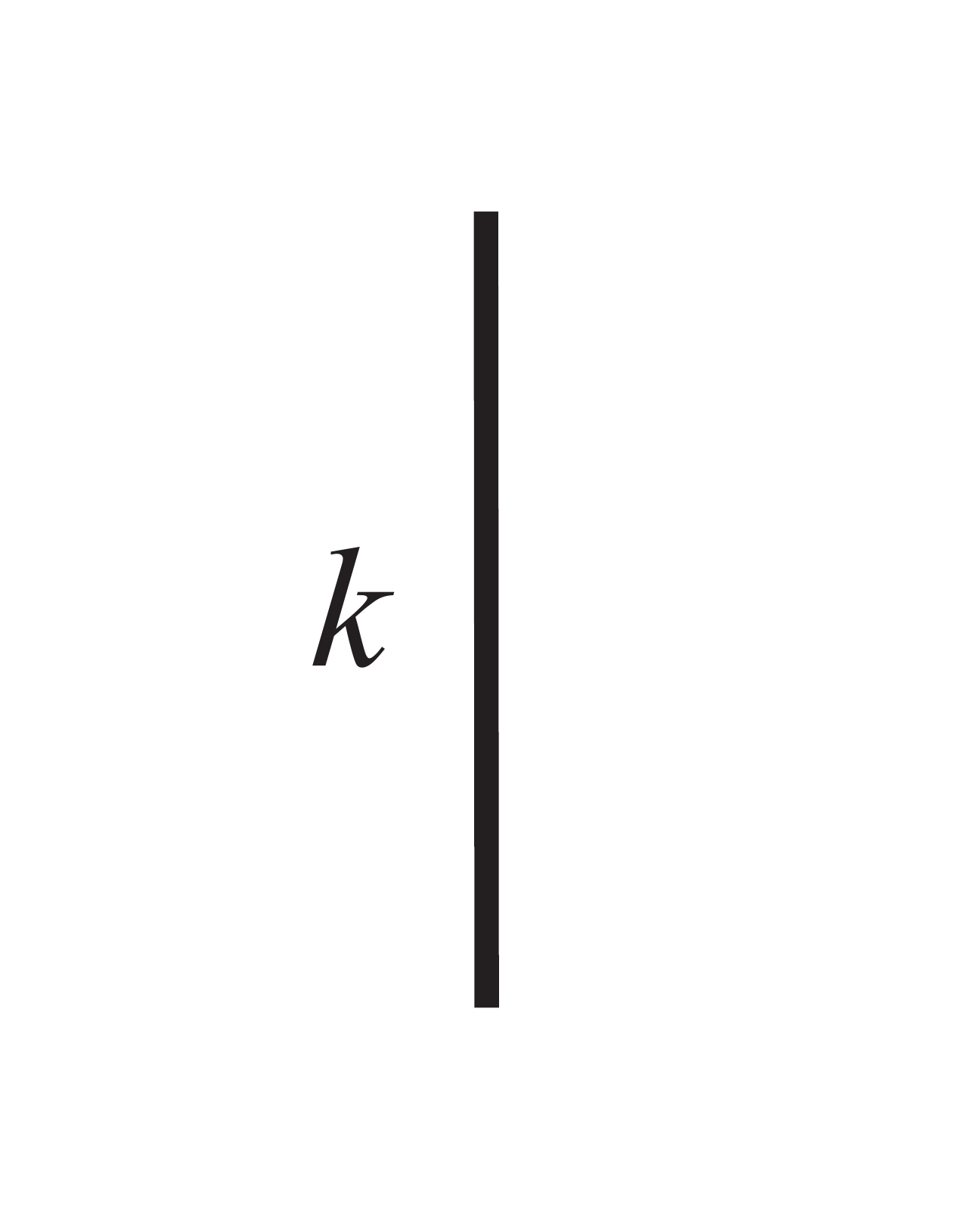}$ denotes a through string with $k$ labels $c$.

\begin{figure}[h]
$$\grd{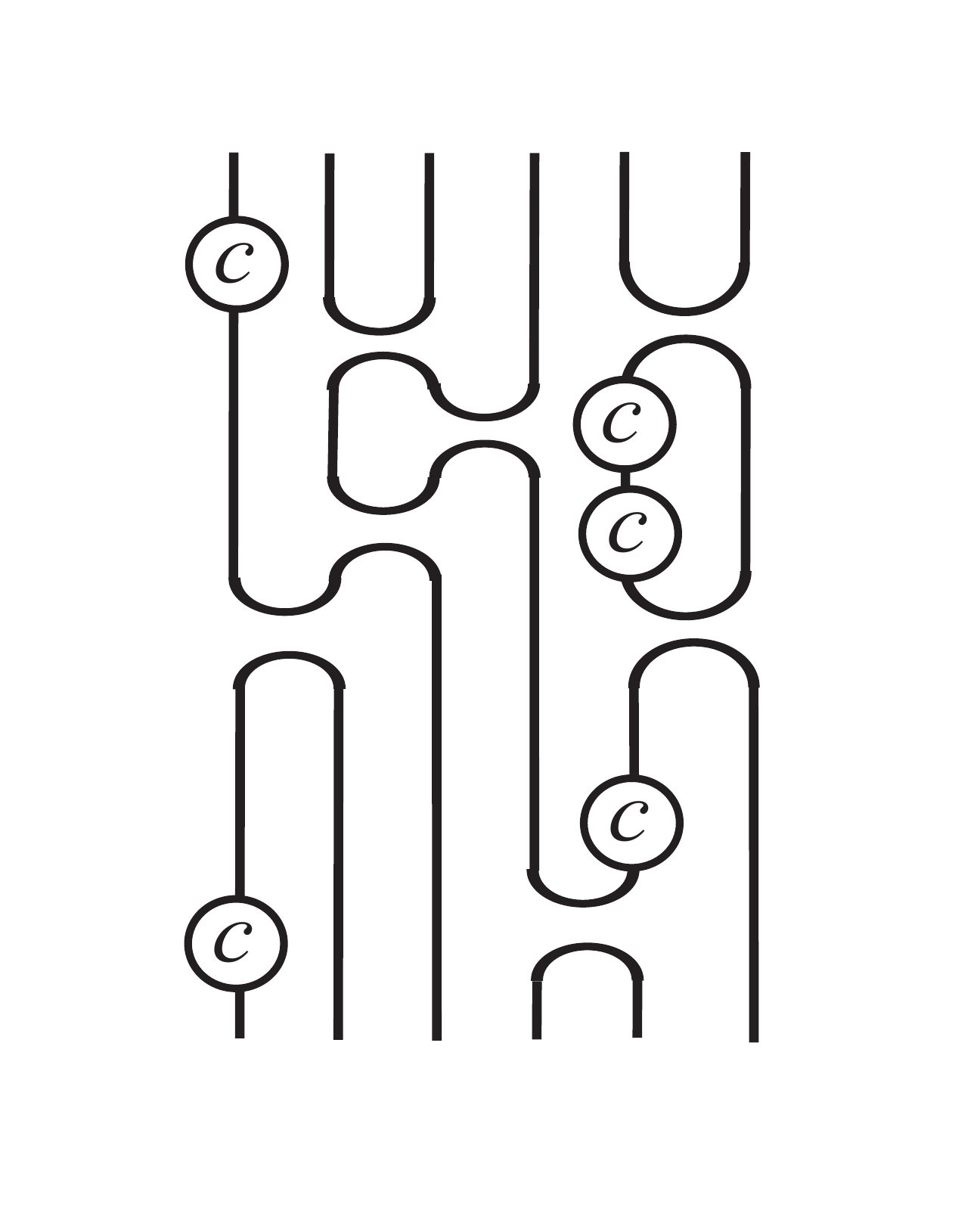}$$
\caption{A regular planar 6-tangle labelled by $c$}
\end{figure}

Precisely, the vectors in $\mathscr{P}_{m}$ are linear sums of regular planar $m$-tangles labelled by $c$ modulo the relations.
The para isotopy can also be viewed as relations:
$$\graa{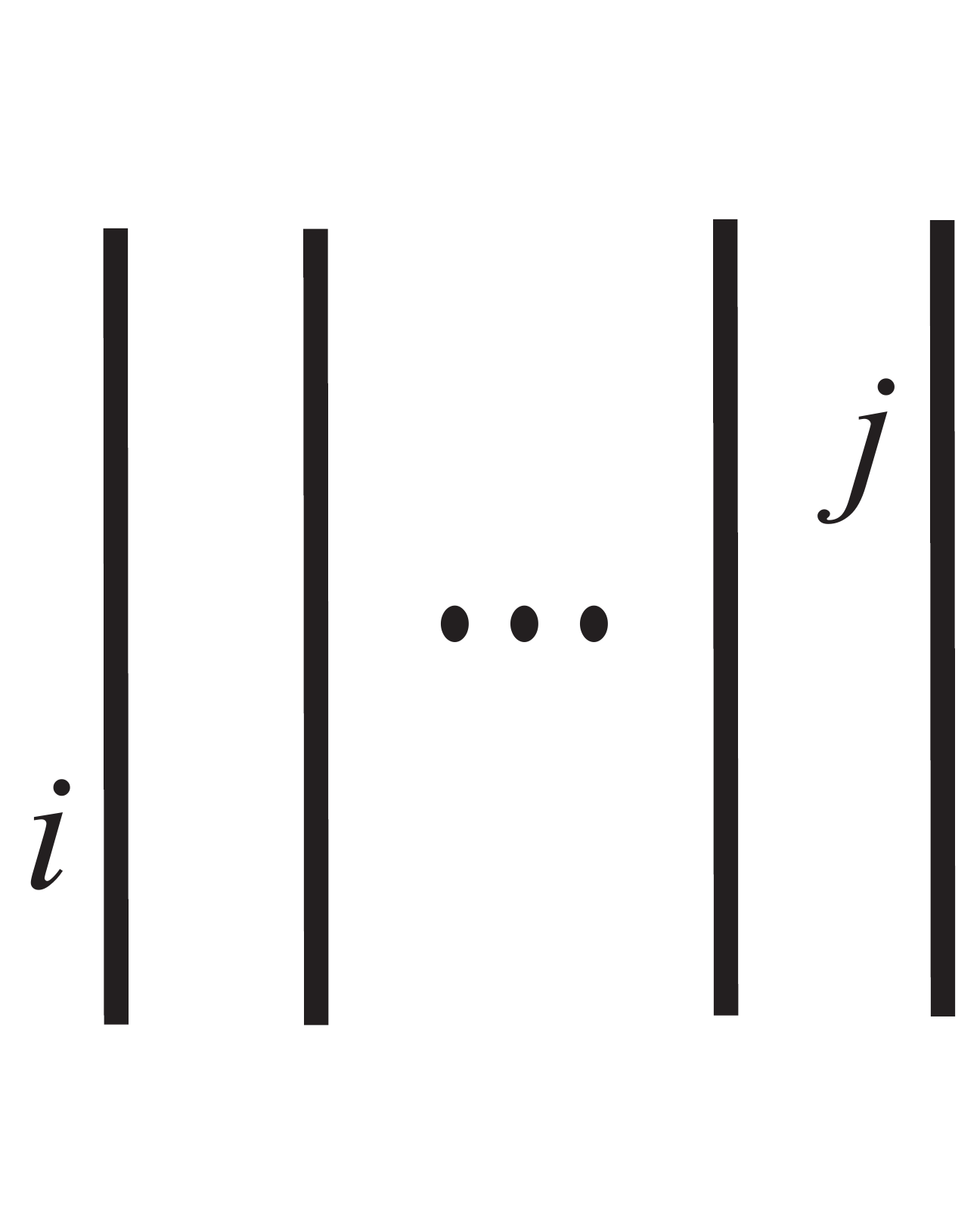}=q^{ij}~\graa{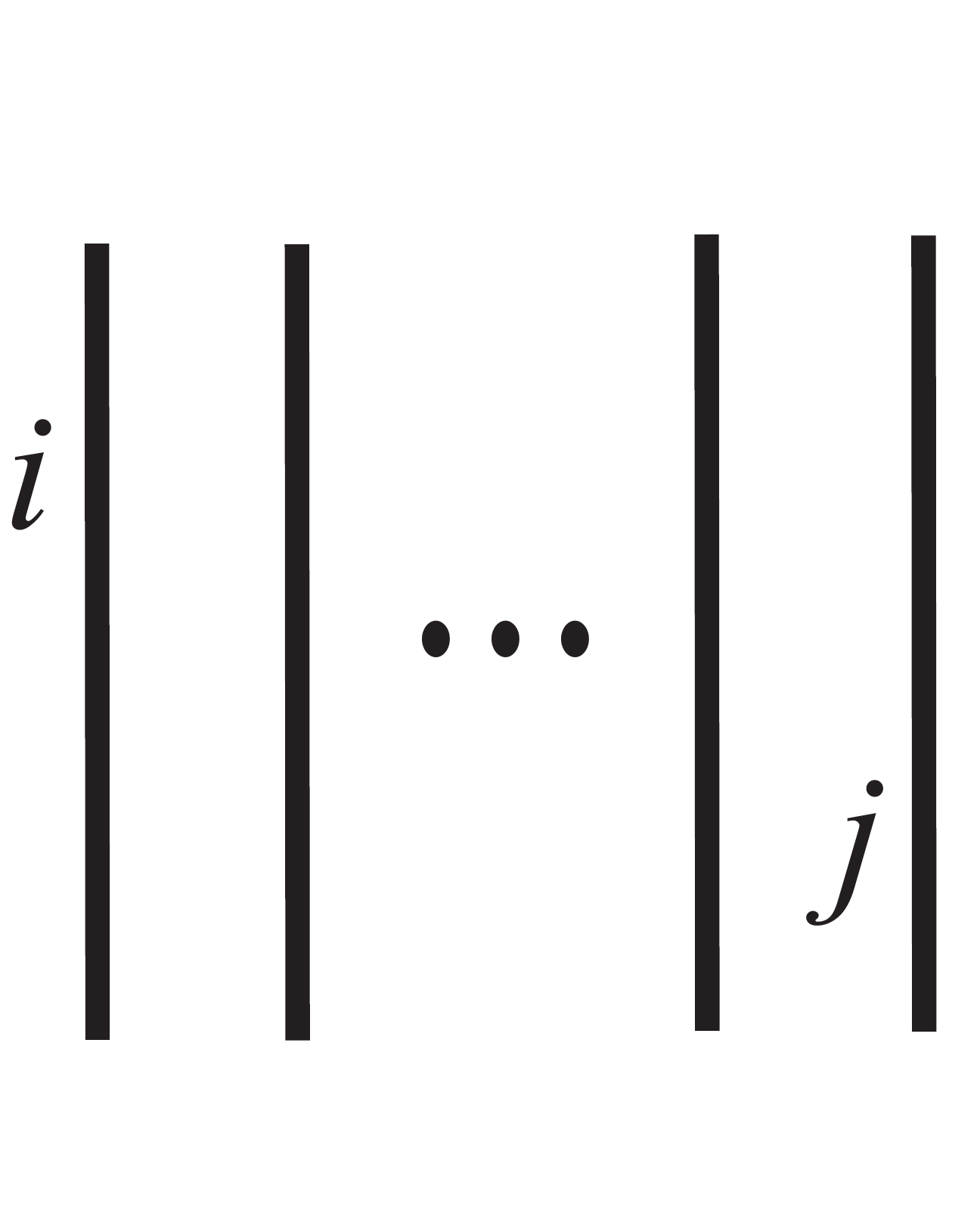}.$$
Regular planar tangles act on labelled regular planar tangles by gluing the diagrams.

The planar para algebra is called {\it evaluable} by the relations, if $\dim(\mathscr{P}_{0})\leq 1$, i.e., any regular labelled planar $0$-tangles is reduced to the ground field; and $\dim(\mathscr{P}_{m})< \infty$.

The relations are called {\it consistent}, if $\dim(\mathscr{P}_{0})=1$, i.e., different processes of evaluating a regular labelled planar $0$-tangle give the same value in the ground field.
In this case, the map from regular labelled planar $0$-tangles to the ground field is called the {\it partition function}, denoted by $Z$.

\begin{theorem}\label{Theorem: construction}
The above relations of the generator $c$ are consistent and the unshaded planar para algebra $\mathscr{P}_{\bullet}$ is evaluable and spherical over the field $\mathbb{C}(\delta)$.
\end{theorem}

\begin{proof}
  See Appendix A.
\end{proof}

When $\delta$ is a real number, we introduce the vertical reflection on $\mathscr{P}_{\bullet}$ mapping $c$ to $c^{-1}(=c^{N-1})$. Note that the involution preserves the relations of $c$, thus it is well-defined on the planar para algebra $\mathscr{P}_{\bullet}$. So $\mathscr{P}_{\bullet}$ is a planar para *-algebra over $\mathbb{C}$.
We will prove that the partition function $Z$ is positive semi-definite with respect to * in Sections \ref{positivity} and \ref{positivity2} and construct subfactor planar para algebras by taking a proper quotient.

Note that the 1-box space of a $(G,\chi)$ planar para algebra forms a finite dimensional $G$ graded algebra with a $G$ graded trace. (Here a $G$ graded trace means that the trace of any non-zero graded vector is zero.)
On the other hand, given an Abelian group $G$, a bicharacter $\chi$ of $G$, and any finite dimensional $G$ graded algebra $A$ with a $G$ graded trace $\tau$, we can construct a shaded $(G,\chi)$ planar para algebra $\mathscr{P}(A)$ with the circle parameter $\delta$ over the field $\mathbb{C}(\delta)$. The generators of $\mathscr{P}(A)$ are 1-boxes \graa{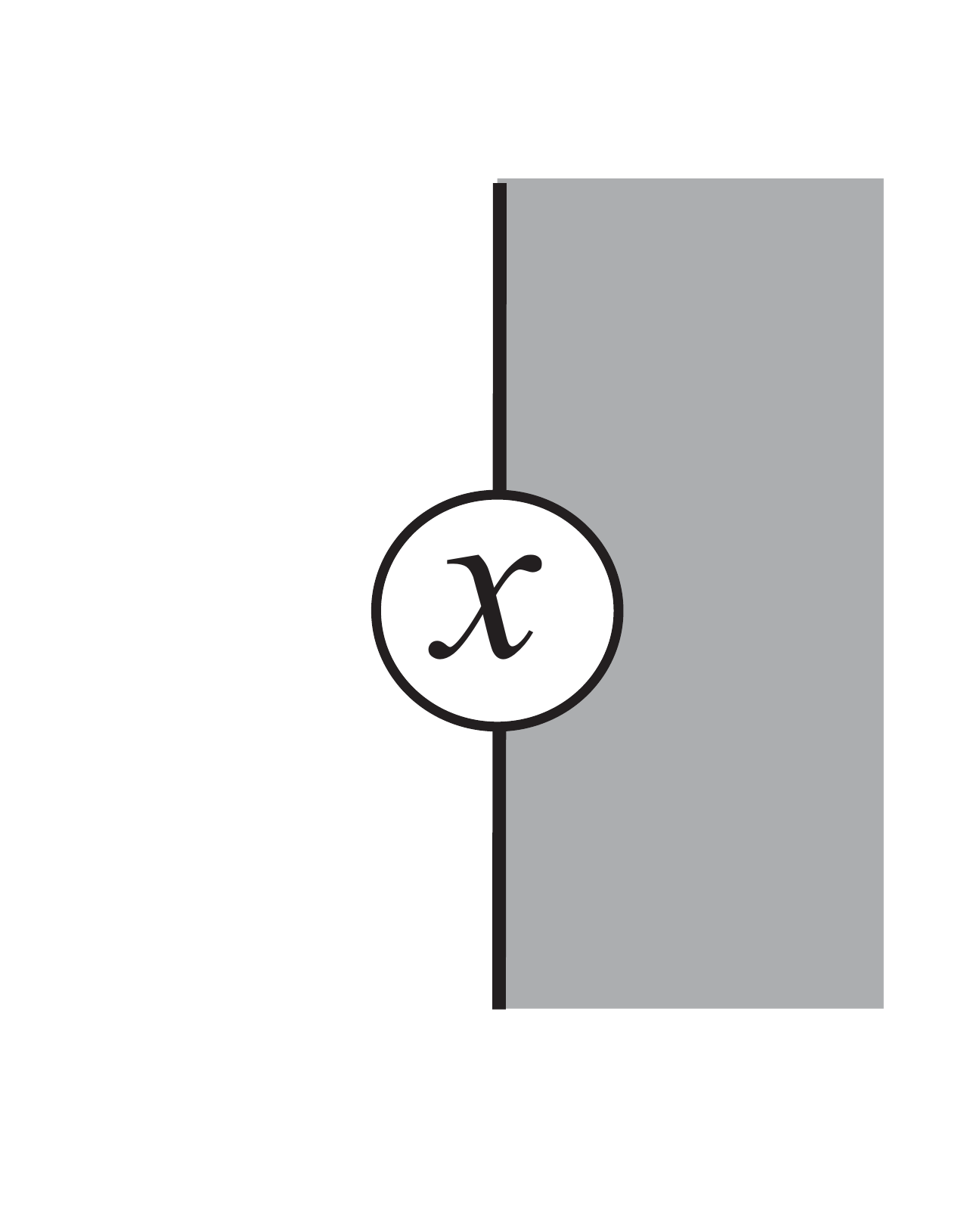}, for all $x\in A$. The relations are given by
\begin{itemize}
  \item[(1)] $\graa{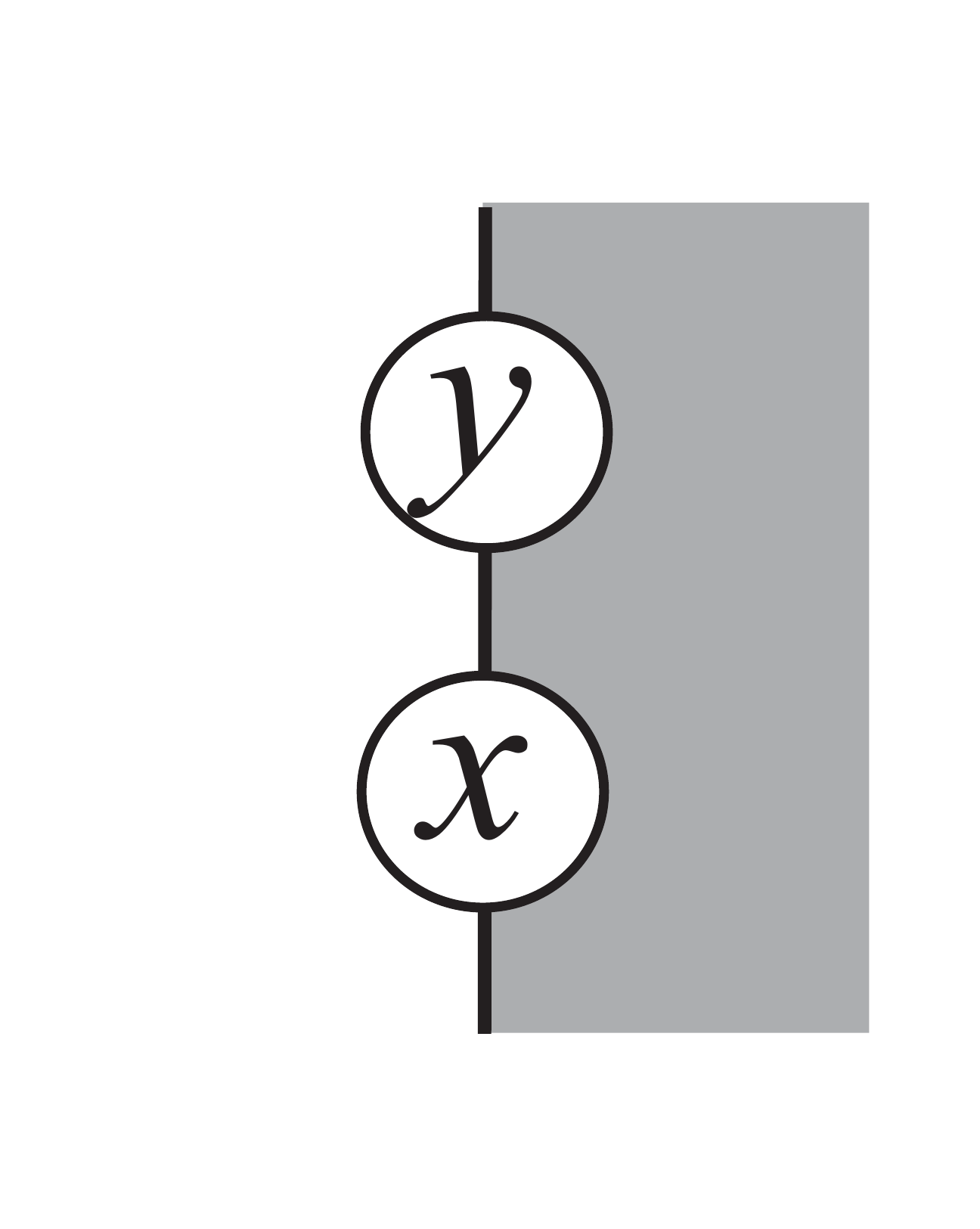}=\graa{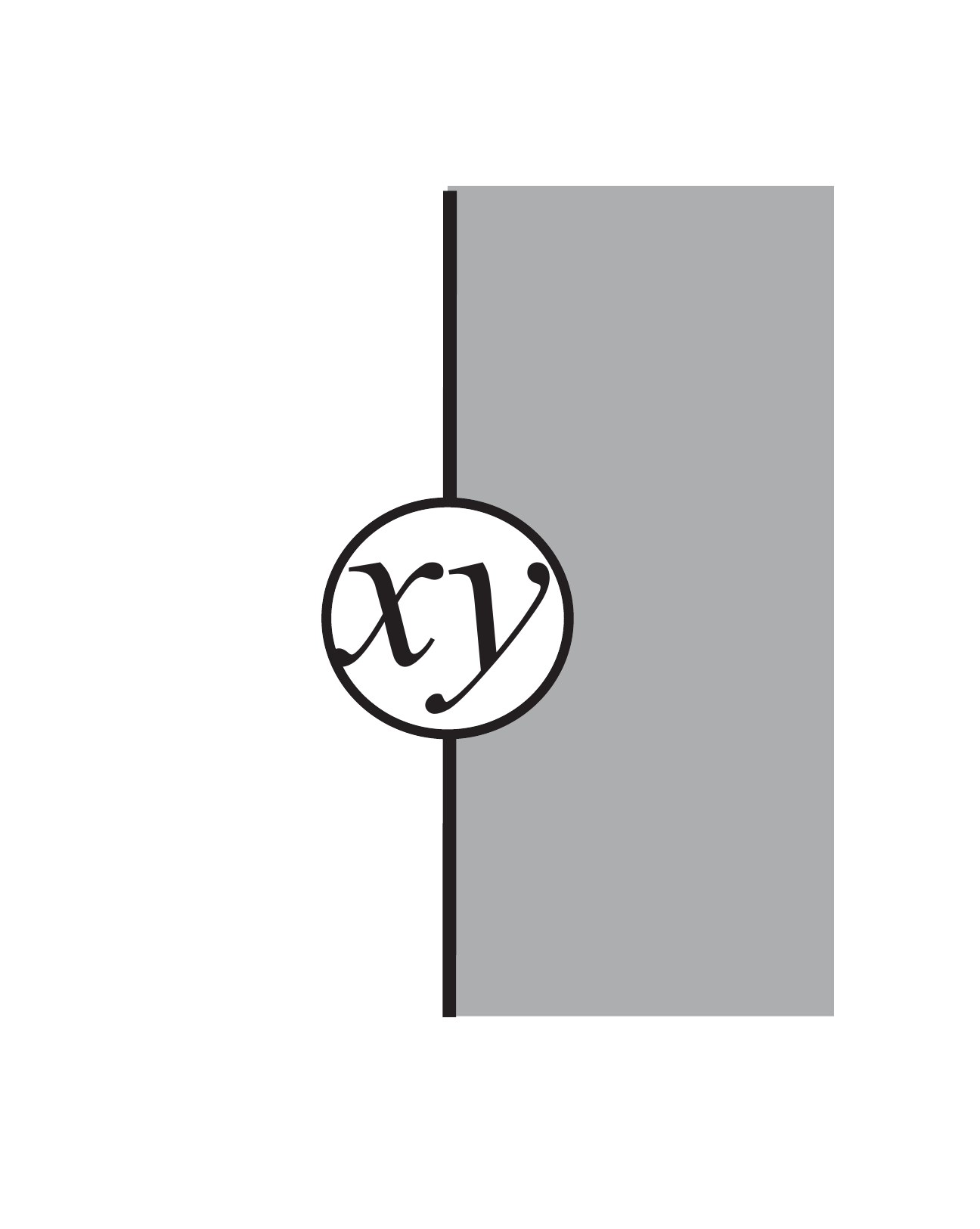}$;
  \item[(2)] $\graa{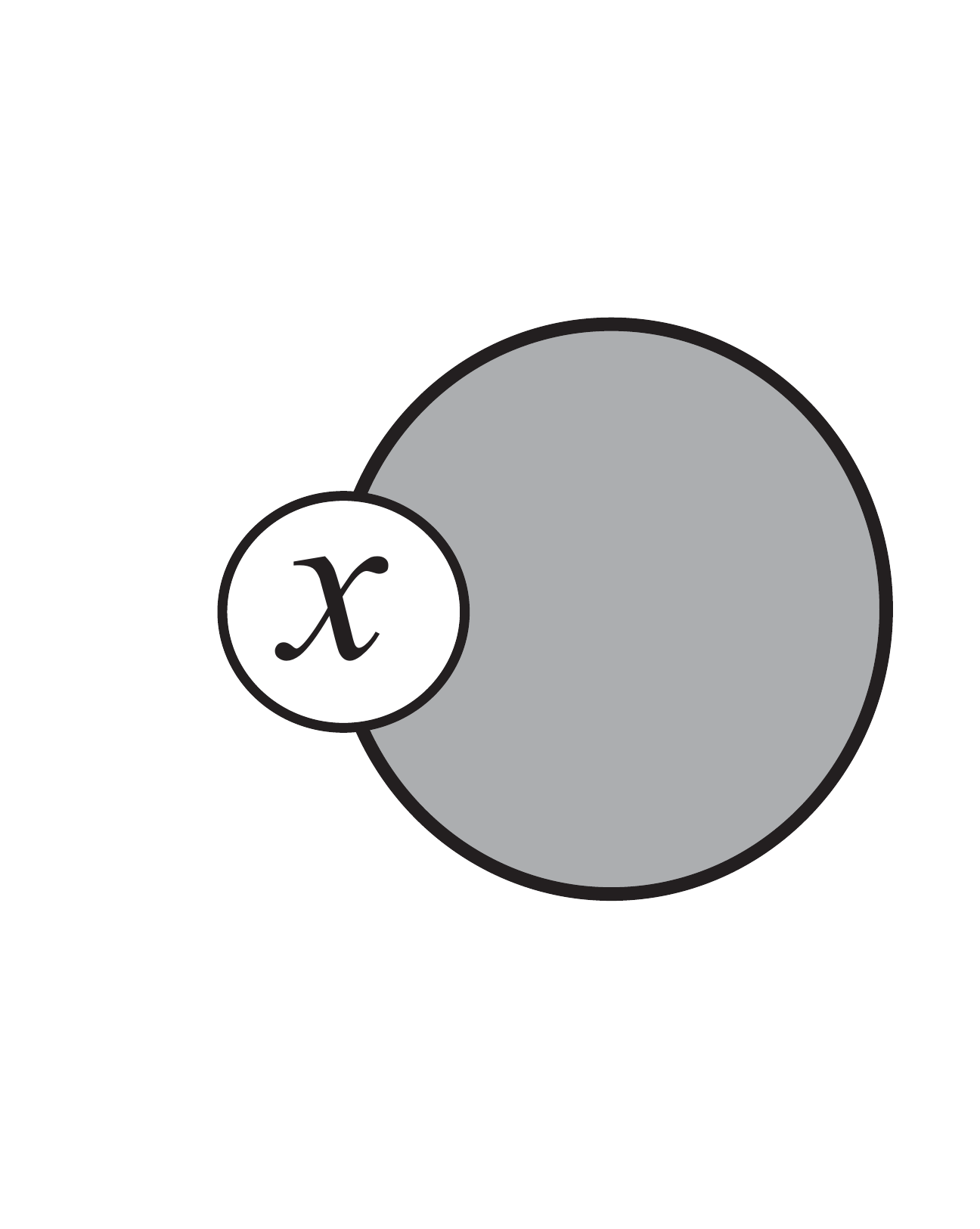}=\tau(x)$ and $\graa{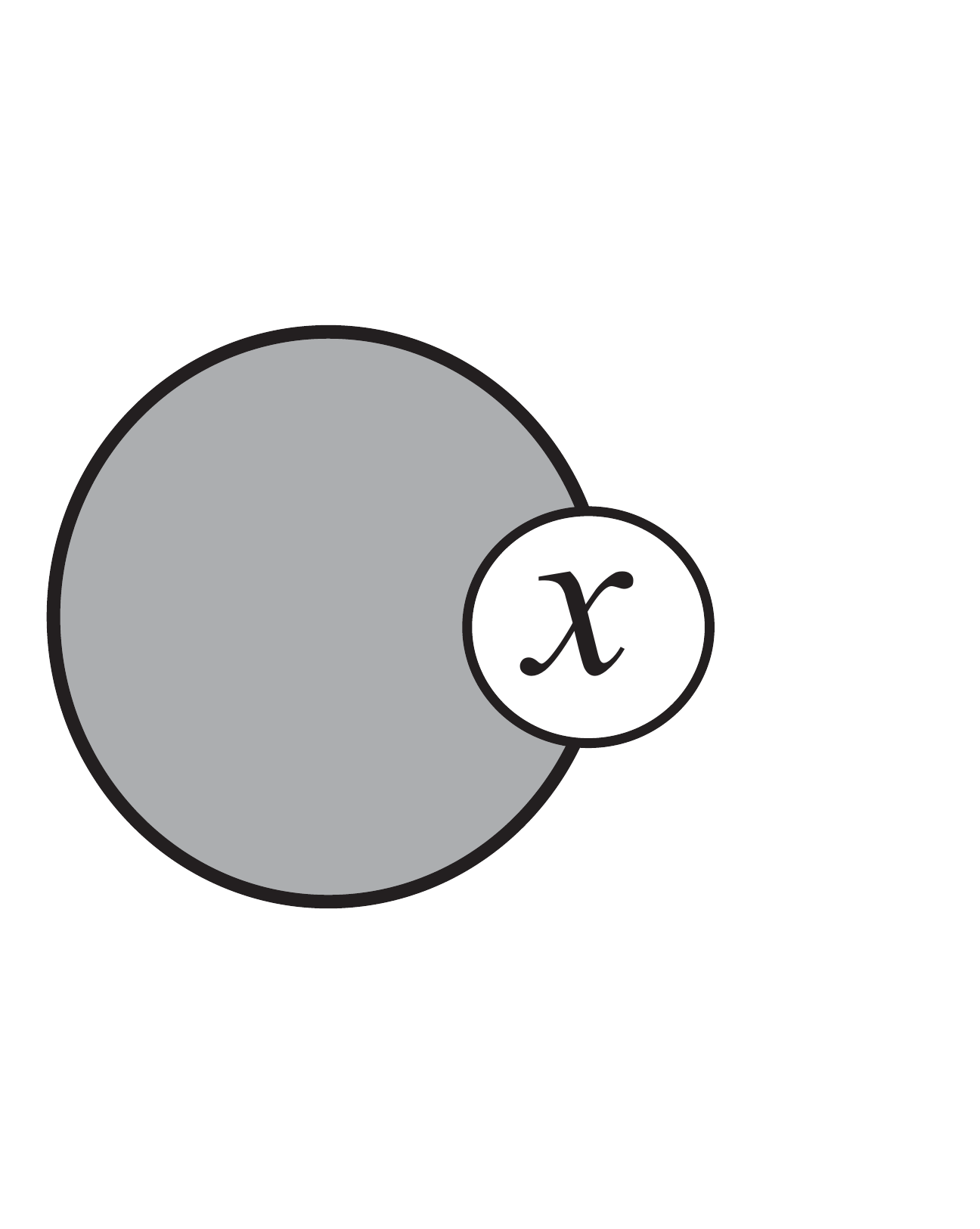}=\tau(x)$,
\end{itemize}
for any $x, y \in A$.

\begin{theorem}
The above relations are consistent and the shaded $(G,\chi)$ planar para algebra $\mathscr{P}(A)$ is evaluable and spherical over the field $\mathbb{C}(\delta)$.
\end{theorem}
\begin{proof}
  The proof is similar to that of Theorem \ref{Theorem: construction}.
\end{proof}

In the case of parafermion planar para algebras, the $\mathbb{Z}_N$ graded algebra $A$ is given by the $N$ dimensional algebra generated by $c$ and $c^N=1$.
The $G$ graded trace is given by $\tau(c^k)=0$, for $1\leq k\leq N-1$, and $tr(1)=1$.

\setcounter{equation}{0}
\section{Parafermion Planar Para Algebras\label{positivity}}
In this section, we take $\delta=\sqrt{N}$ and study the planar para algebra $\mathscr{P}_{\bullet}$ over the field $\mathbb{C}$.
Recall that $\mathscr{P}_{\bullet}$ is a planar para *-algebra with the vertical reflection * defined as an extension of $c^*=c^{-1}$.

The kernel of the partition function $\ker(Z)=\bigcup_{m,\pm}\{x\in \mathscr{P}_{m,\pm} |Z(tr(xy))=0, \forall y\in \mathscr{P}_{m,\pm}\}$ is {\it an ideal} of $\mathscr{U}(P)$, in the sense that any fully labelled regular planar tangle with a label in $\ker(Z)$ is in $\ker(Z)$.
Thus action of regular planar tangles is well defined on the quotient $\mathscr{P}/\ker(Z)$.

We prove that the following relation holds in $\mathscr{P}/\ker(Z)$:
\begin{align}\label{e=sumgi}
\graa{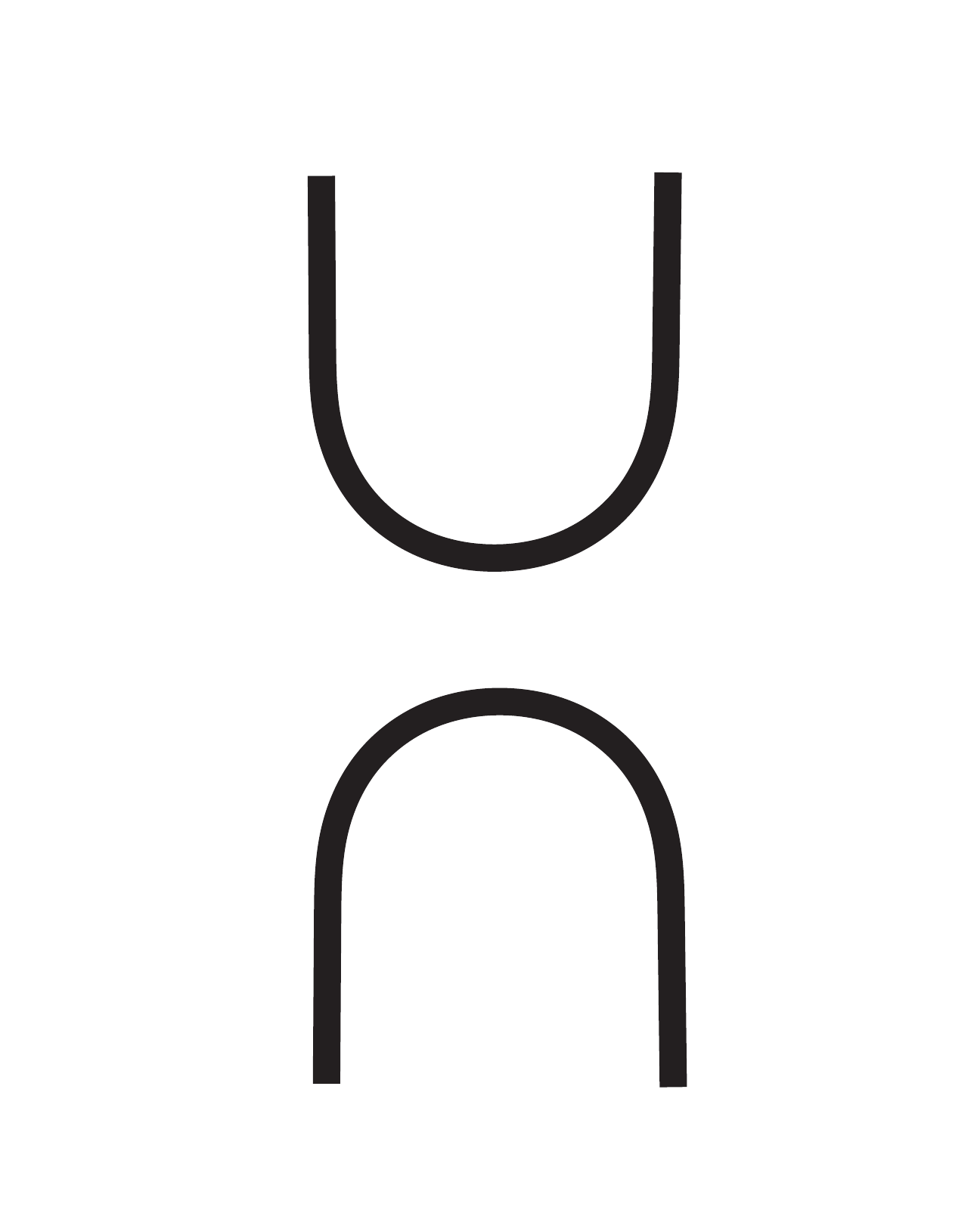}=\frac{1}{\sqrt{N}}\sum_{i=0}^{N-1} \graa{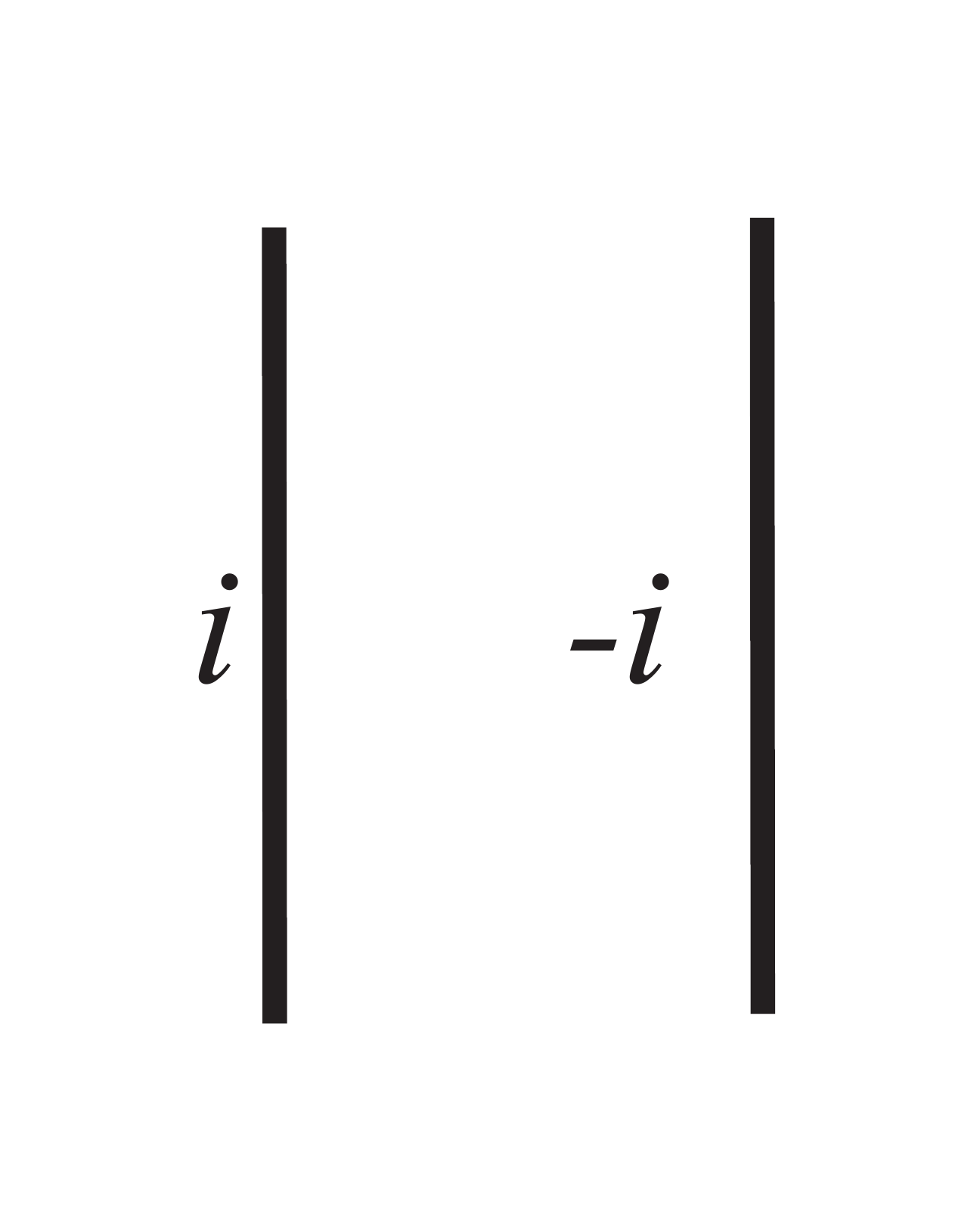}.
\end{align}
Recall that the twisted tensor product $\graa{i=-i}$ is defined as
$$\graa{i=-i}=\zeta^{-i^2}\graa{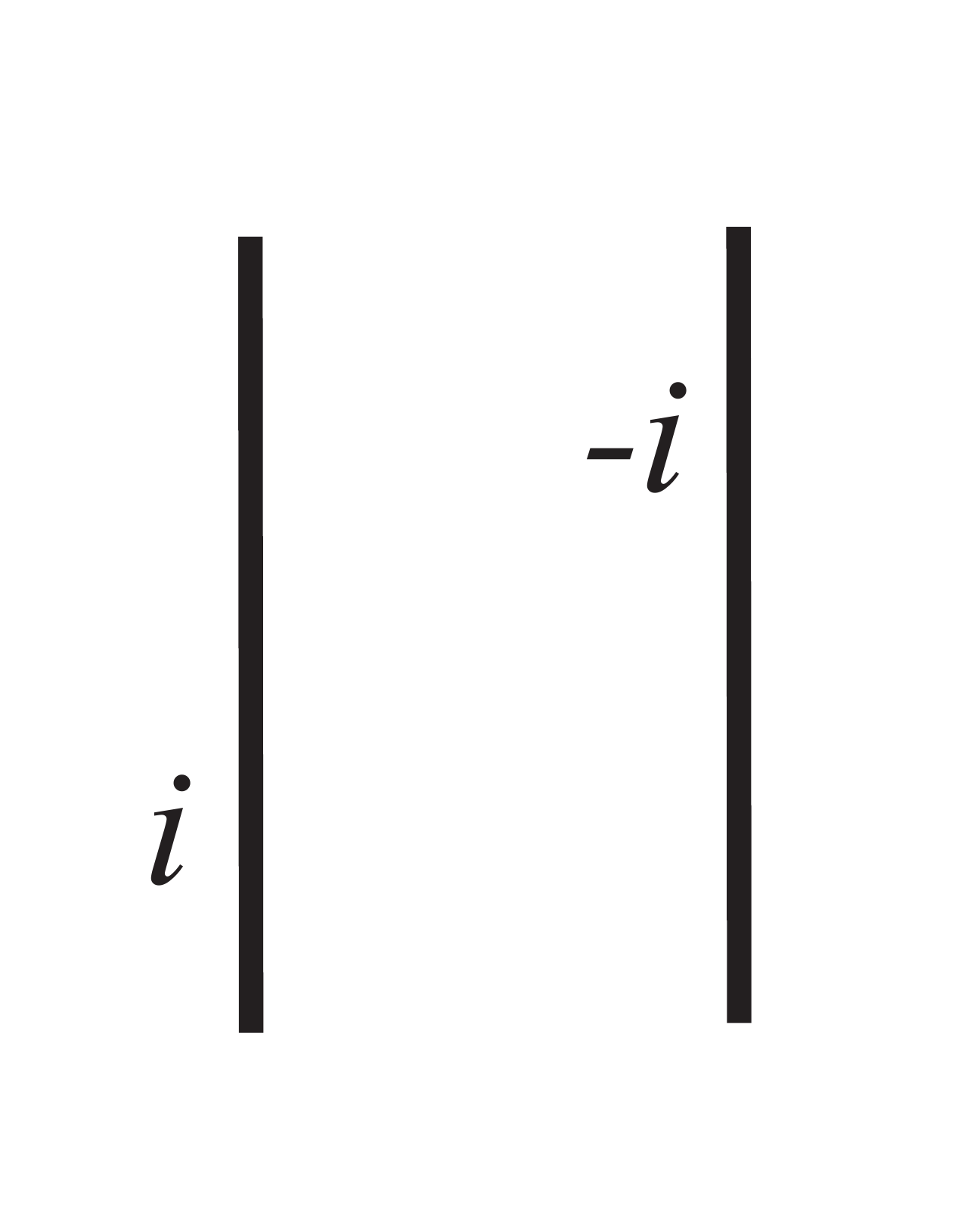}=\zeta^{i^2}\graa{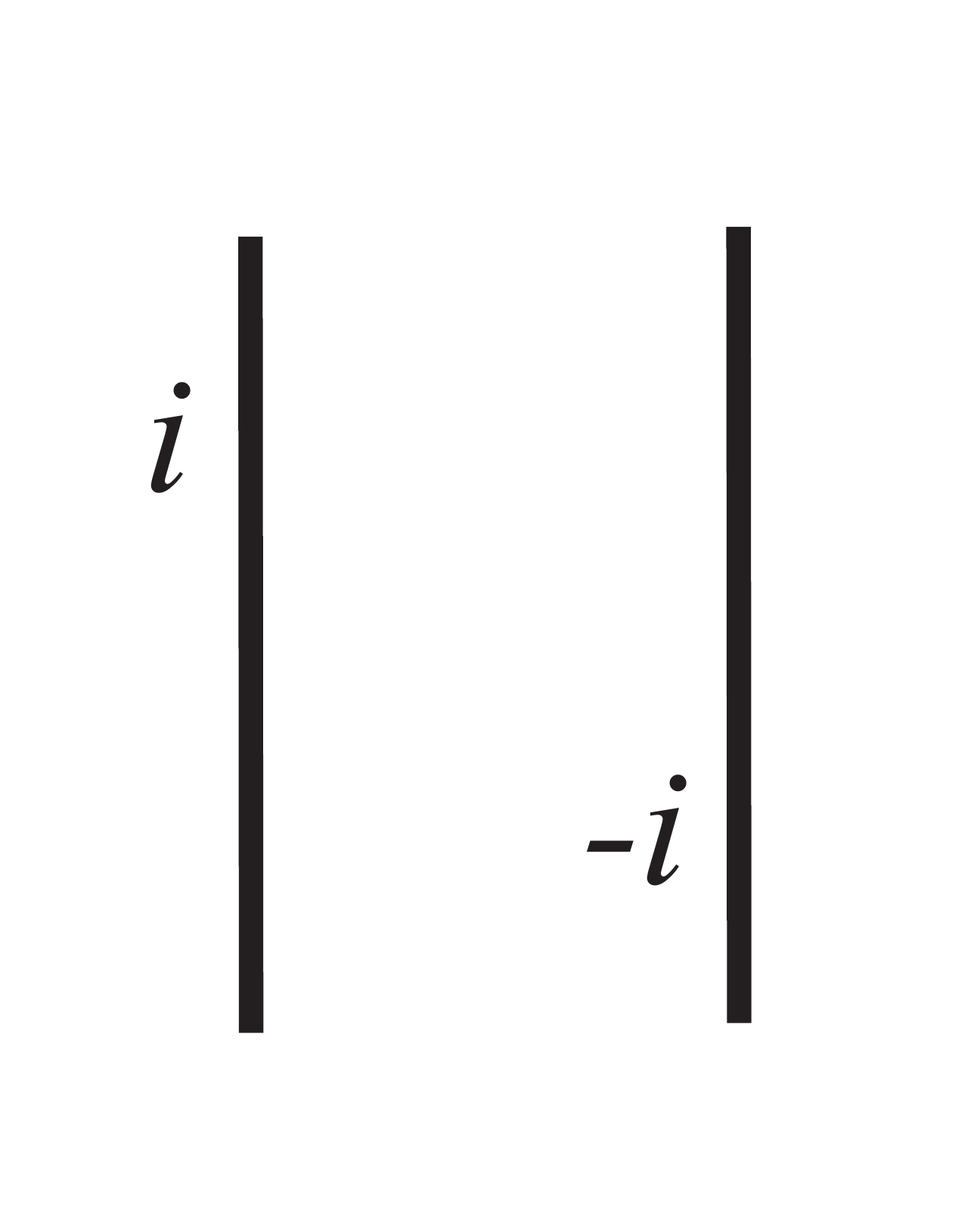}\;.$$
By relation \ref{e=sumgi}, any fully labelled regular planar tangle is a linear sum of labelled regular planar tangles with only labelled vertical strings.
The algebra generated by labelled vertical strings is a parafermion algebra, see Section \ref{parafermion algebra} for the definition of parafermion algebras.
Therefore we call the planar para algebra $\mathscr{P}/\ker(Z)$ the parafermion planar para algebra (PAPPA), 
% \underline{parafermion planar para algebra}, 
denoted by $PF_{\bullet}$.
We prove that $PF_{\bullet}$ is a subfactor planar para algebra. We use this $C^*$ positivity condition to prove reflection positivity in Section \ref{sect:RP}.
We give some interesting properties of the parafermion planar para algebra in Section \ref{Pauli} and \ref{Pictorial presentation}.
Further applications in quantum information of these topological isotopy and braided relations in Section \ref{braided relations} are discussed in \cite{JafLiuWoz}.

\begin{notation}
Take $$v_i^j=\frac{1}{\delta}\graa{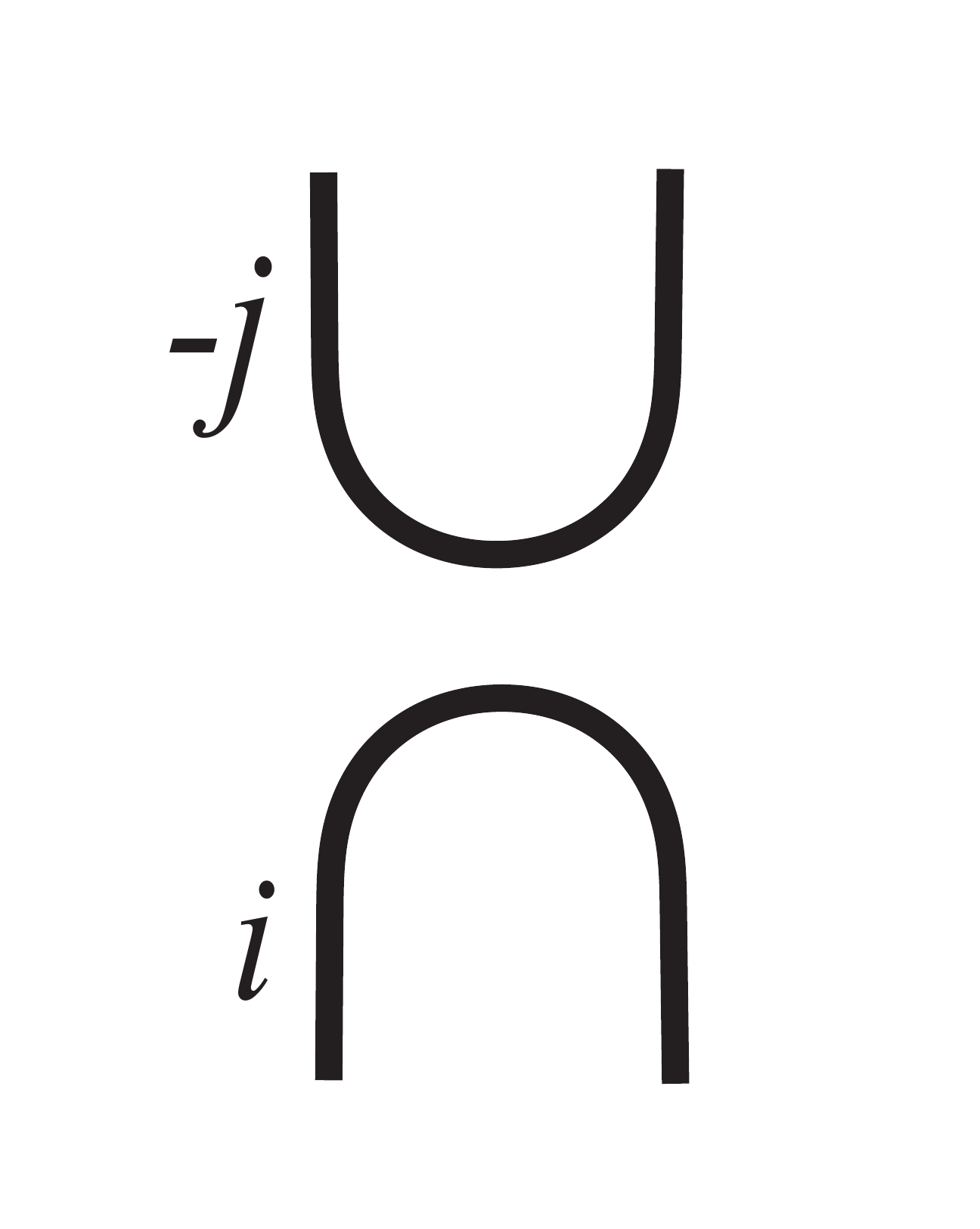}\;.$$
\end{notation}
\noindent Then it is easy to check that $v_i^jv_k^l=\delta_{j,k}v_i^l$, and $(v_i^j)^*=v_j^i$.
In particular, $\{v_i^i\}$ are pairwise orthogonal idempotents.

\begin{lemma}\label{kernel}
The vector $I_2-\sum_{i\in \mathbb{Z}_N} v_i^i$ is in the kernel of the partition function of $\mathscr{P}_{\bullet}$.
\end{lemma}

\begin{proof}
The 2-box space has a generating set
	\[
	\left\{\graa{pij}\;,\qquad \graa{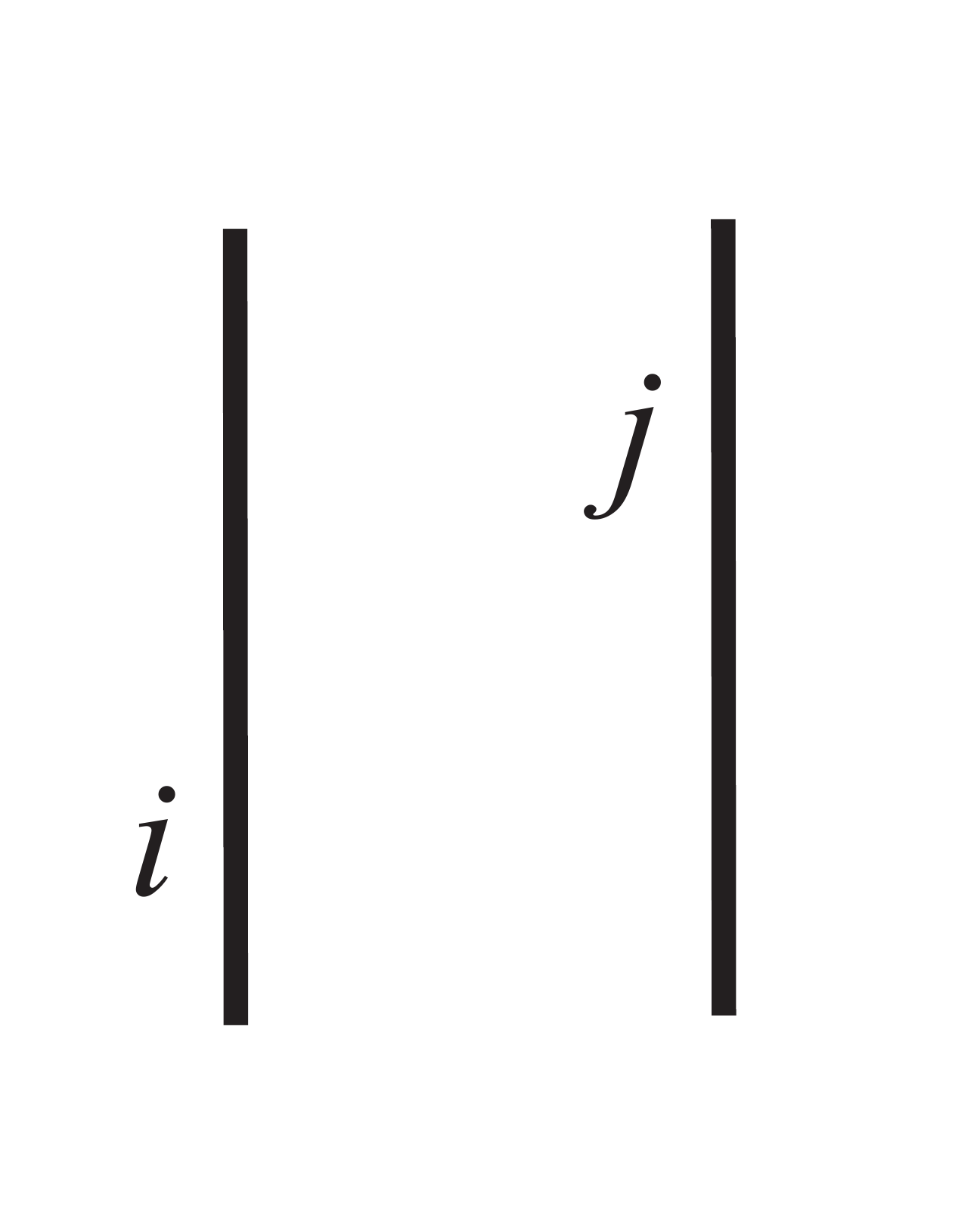}\right\}_{0\leq i,j\leq N-1}\;.
	\]
Take $x=id-\sum_{i\in \mathbb{Z}_N} v_i^i \in \ker(Z)$. It is easy to check that $tr(xy)=0$ for any 2-box $y$.
By the spherical property, we have that any 0-tangle labelled by $x$ is isotopic to $tr(xy)$ for some 2-box $y$. So $x$ is in the kernel of the partition function.
\end{proof}

Thus we have the relation $I_2=\sum_{g\in G} v_g$ in the quotient $(\mathscr{P}/\ker{Z})_{\bullet}$, i.e.,
\begin{align}\label{i2=sumpi}
\graa{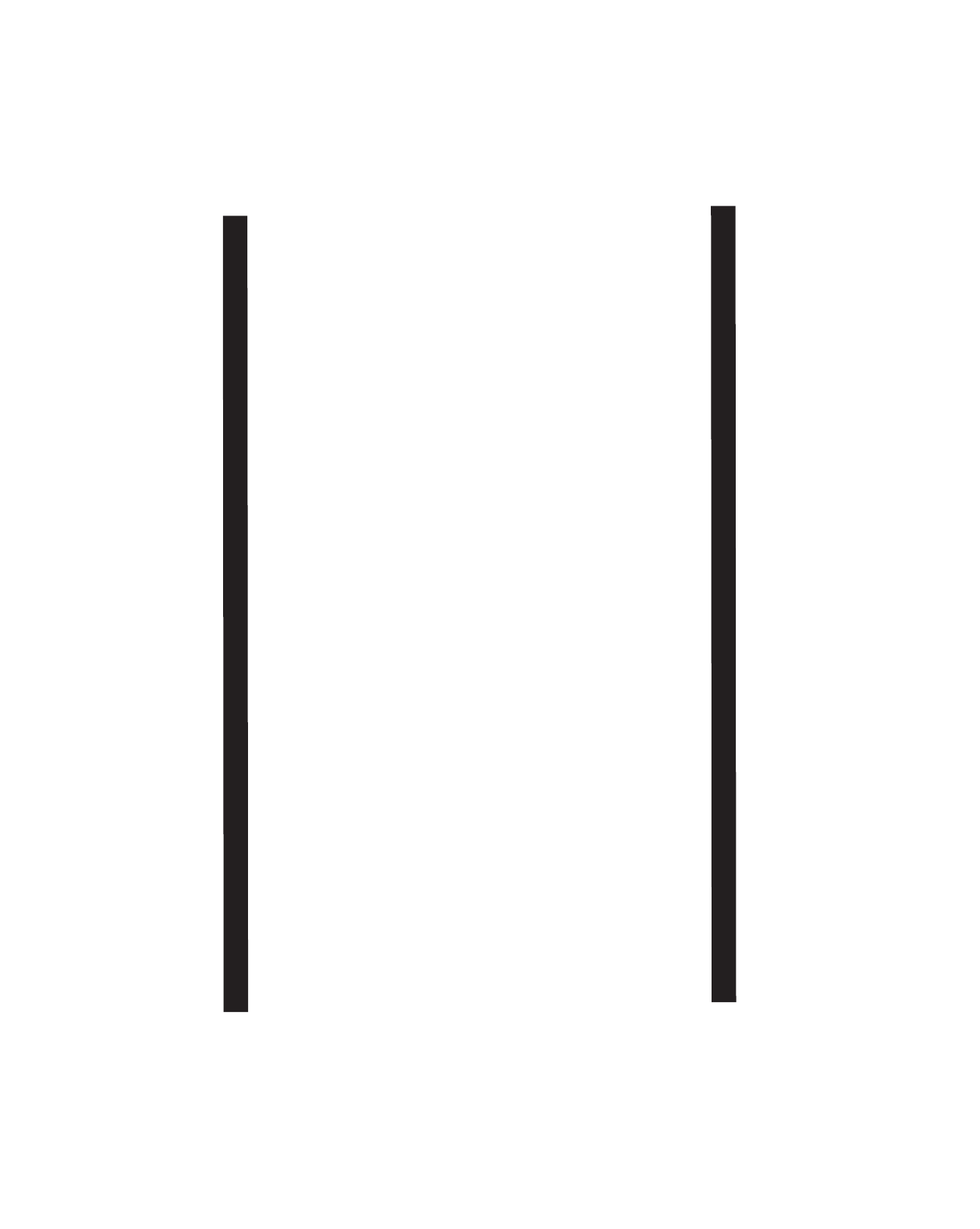}=\frac{1}{\sqrt{N}} \sum_{i=0}^{N-1} \graa{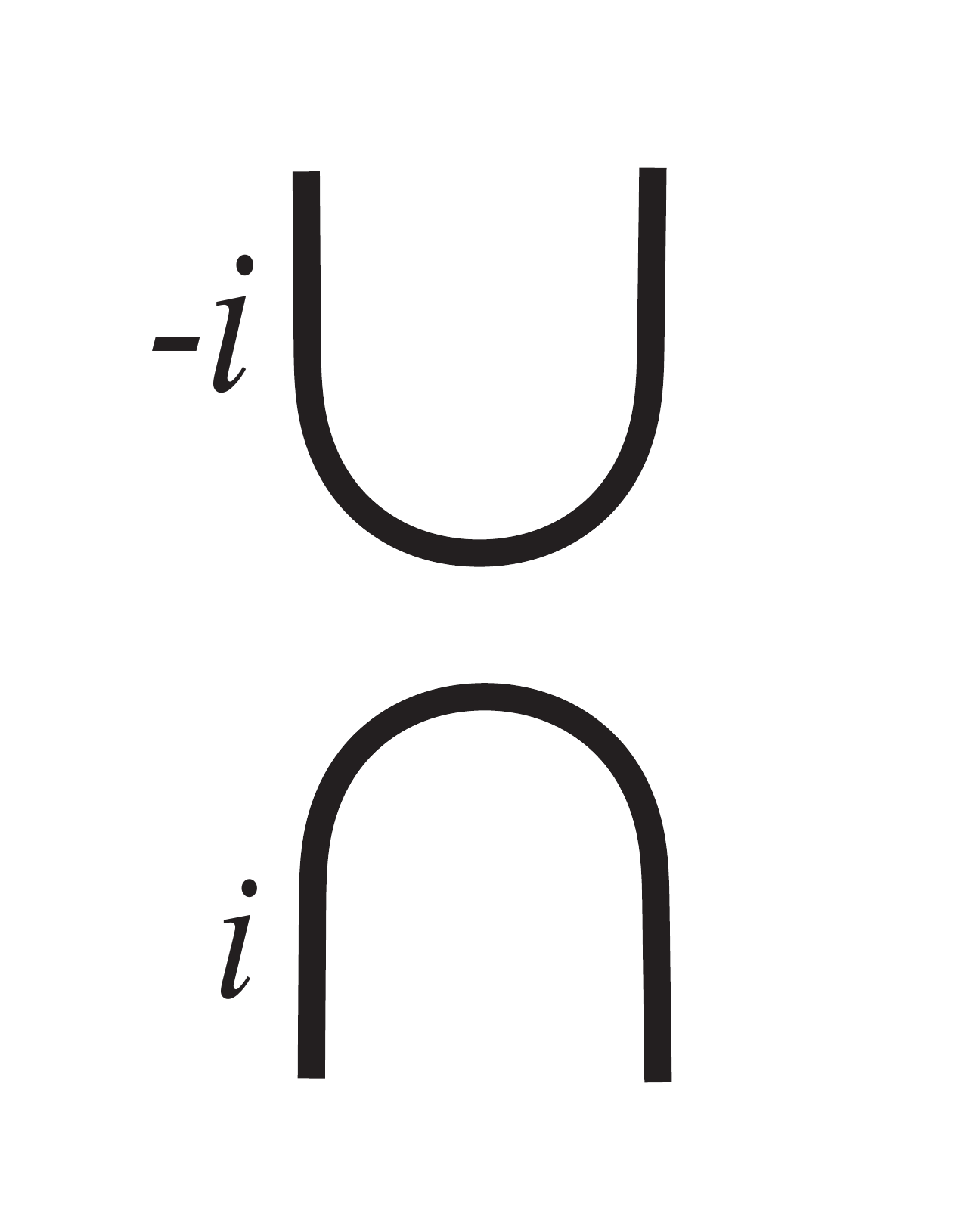}
\end{align}
Take the string Fourier transform $\mathfrak{F}_{\text{s}}$ on both sides, i.e., the $\frac{\pi}{2}$ rotation. We obtain Equation~\ref{e=sumgi}.

\begin{lemma}\label{basis}
The vectors $c^{i_1}\otimes_+c^{i_2}\cdots\otimes_+c^{i_m}$, i.e.,
	\[
	\graa{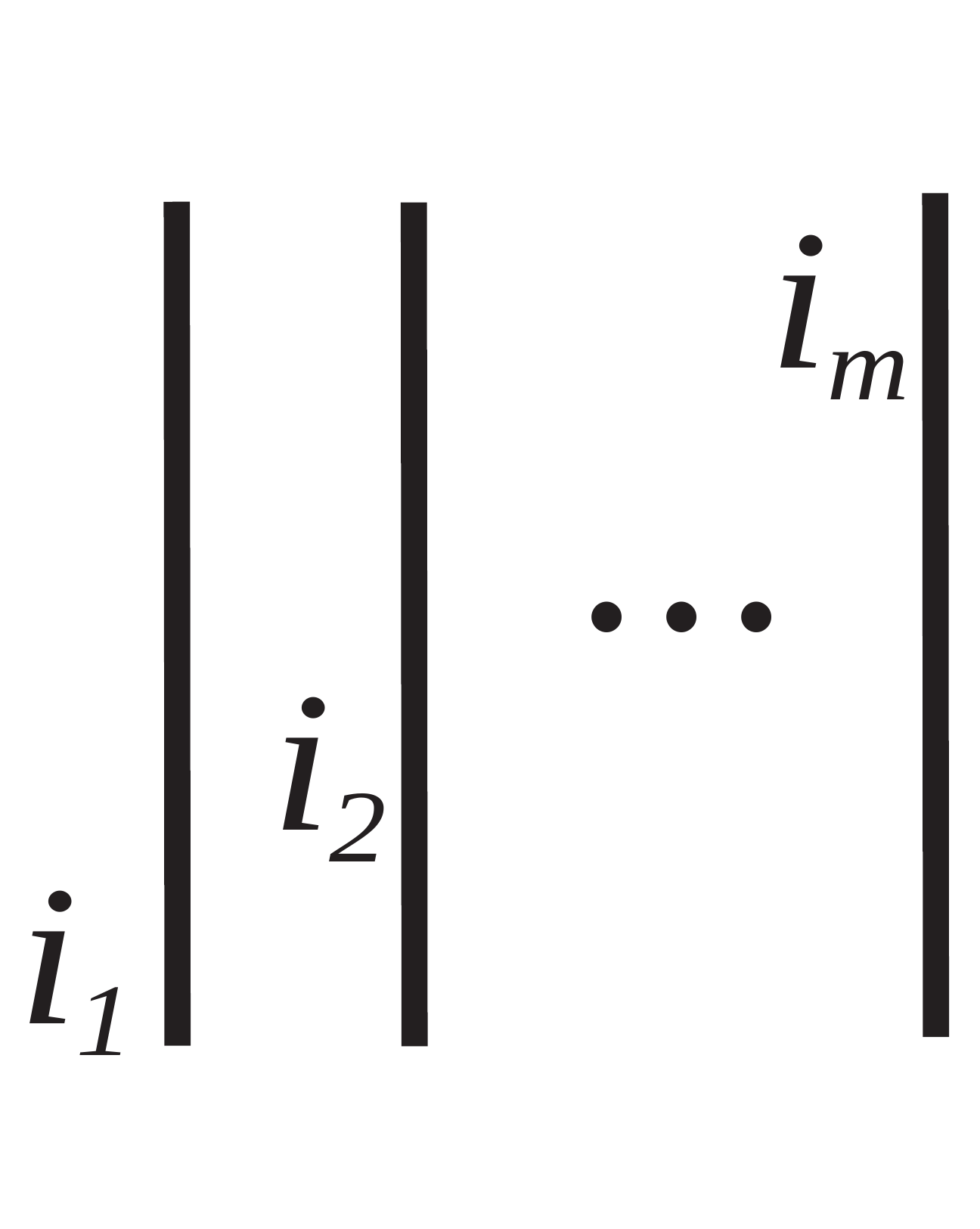}\;,
	\quad\text{for}\quad
	0\leq n_1, n_2 \cdots n_m\leq N-1\;,
	\]
form an orthonormal basis of $(\mathscr{P}/\ker{Z})_{m}$.
\end{lemma}

\begin{proof}
Any $m$-box is a linear sum of labelled Temperley-Lieb diagrams. Applying the relation \ref{e=sumgi}, any labelled Temperley-Lieb diagram is a linear sum of the vectors $c^{n_1}\otimes_+c^{n_2}\cdots\otimes_+c^{n_m}$, $0\leq n_1, n_2 \cdots n_m\leq N-1$. Thus these vectors form a generating set of $\mathscr{P}/\ker{Z}_{\bullet}$.
It is easy to check that these vectors form an orthonormal basis with respect to the Markov trace.
%Replace the Jones projections by the string Fourier transform of the relation $id=\sum_{g\in G} v_g$, we have that any $m$-box is a multiplication of shifts of $g$. Applying the para symmetric relation and $c^N=1$, we have that, The vectors $c^{n_1}\otimes_+c^{n_2}\cdots\otimes_+c^{n_m}$, $0\leq n_1, n_2 \cdots n_m\leq N-1$ for a generating set.
\end{proof}

\begin{theorem}
When $\delta=\sqrt{N}$, the kernel of the partition function $\ker{Z}$ is generated by
$id-\sum_{i\in \mathbb{Z}_N} v_i^i,$
and $\mathscr{P}/\ker{Z}$ is a subfactor planar para algebra.
\end{theorem}

\begin{proof}
The proof is a consequence of Lemmas \ref{kernel} and \ref{basis}.
\end{proof}

\setcounter{equation}{0}
\section{Parafermion Pauli Matrices}\label{Pauli}
We define unitary $N\times N$ matrices $X,Y,Z$ that play the role in the parafermion algebra of the $2\times2$ Pauli matrices $\sigma_{x}, \sigma_{y}, \sigma_{z}$ for fermions. We call the matrices $X,Y,Z$ the  \textit{parafermion Pauli matrices}.   They act on an $N$-dimensional Hilbert space with basis vectors indexed by $\mathbb{Z}_{N}$. In \S\ref{sec:PauliDefinition} we define these matrices and determine some of their properties.

%In  \S\ref{sec:quarternion-parafermion} we discuss the relation between the parafermion Pauli matrices and algebra.
In \S\ref{QuadraticPauliRep} we give different ways to represent $X,Y,Z$ as quadratic functions of parafermions, acting on a larger space.  Each matrix $X,Y,Z$ is defined as a particle-anti particle pair, namely a twisted product of one parafermion and the inverse of another.
Restricted to a subspace $\gamma=1$, with $\gamma$ defined in \eqref{gamma-grading}, we obtain a representation of the matrices $X,Y,Z$.

\subsection{Parafermion Pauli matrices: Version $q$\label{sec:PauliDefinition-I}}
Let us use Dirac notation for vectors, and take the ortho-normal basis for an $N$-dimensional Hilbert space: $\{\,\ket{k} \, |\, k\in \mathbb{Z}_{N}\}$. Choose $q=e^{\frac{2\pi i}{N}}$ and its square root $\zeta$ such that $\zeta^{N^2}=1$, as in \eqref{ZetaChoice}.  Define the Pauli matrices $X,Y,Z$ by their action on the basis,
	\be
		X\ket{k} = \ket{k+1}\;, \quad
		Y\ket{k} = \zeta^{1-2k} \ket{k-1}\;,\quad\text{and}\quad
		Z\ket{k} = q^{k} \ket{k}\;.
	\ee
Clearly $X,Y,Z$ are unitary.
For any $N\in \mathbb{N}$, these matrices satisfy a first set of parafermion Pauli matrix relations,
	\be\label{PauliRelations1q}
	X^N=Y^N=Z^N=1\;,\quad
	XY=q\,YX\;,\quad
	YZ=q\,ZY\;, \quad\text{and}\quad
	ZX=q\,XZ\;.
	\ee
They also satisfy a second set of parafermion Pauli matrix relations  that involve $\zeta$,
	\be\label{PauliRelations2q}
	 XYZ = YZX = ZXY =\zeta \;.
	\ee
	
In case $N=2$, the choices
	\be
	\zeta=i\;,\quad
	\ket{0} = \bv{1}{0}\;,\quad\text{and}\quad
	\ket{1} = \bv{0}{1}\;,
	\ee
yield the standard representation of the Pauli matrices; the choice $\zeta=-i$ yields the complex conjugate of the usual representation.
	
\subsubsection{Quaternion relations: Version $q$\label{sec:quarternion-parafermion}}
One has a parafermion quaternion algebra related to the matrices $X,Y,Z$.  These are $N\times N$ matrices $\boldsymbol{i}$, $\boldsymbol{j}$, $\boldsymbol{k}$ satisfy the \textit{parafermion quaternion relations}:
	\be\label{PFQR}
	\boldsymbol{i}^{N}=\boldsymbol{j}^{N}=\boldsymbol{k}^{N}
	= -1\;,
	\quad
	\boldsymbol{i}\boldsymbol{j}
	= q^{-1}\, \boldsymbol{j}\boldsymbol{i}\;,\quad
	\boldsymbol{j}\boldsymbol{k}
	= q^{-1}\, \boldsymbol{k}\boldsymbol{j}\;,\quad
	\boldsymbol{k}\boldsymbol{i}
	= q^{-1}\, \boldsymbol{i}\boldsymbol{k}\;,\quad
	\text{and}\quad
	\boldsymbol{i}\,\boldsymbol{j}\,\boldsymbol{k}=-1\;.
	\ee
This algebra arises from the three unitary transformations $\boldsymbol{i}$, $\boldsymbol{j}$, $\boldsymbol{k}$ by
	\be
	\boldsymbol{i}=-\zeta Y\;,\quad
	\boldsymbol{j}=-\zeta X\;,\quad
	\boldsymbol{k}=-\zeta^{-1} Z\;.
	\ee
The desired relations \eqref{PFQR} are a consequence of $(-\zeta^{\pm1})^{N}=-1$.  The matrices $\boldsymbol{i}$, $\boldsymbol{j}$, $\boldsymbol{k}$ generate the algebra of $N\times N$ matrices.	
	
\subsection{Parafermion Pauli matrices: Version $q^{-1}$ \label{sec:PauliDefinition}}
As in version I, take the ortho-normal basis for an $N$-dimensional Hilbert space: $\{\,\ket{k} \, |\, k\in \mathbb{Z}_{N}\}$. Choose $q=e^{\frac{2\pi i}{N}}$ and its square root $\zeta$ such that $\zeta^{N^2}=1$, as in \eqref{ZetaChoice}.  Define the Pauli matrices $X,Y,Z$ by their action on the basis,
	\be
		X\ket{k} = \ket{k-1}\;, \quad
		Y\ket{k} = \zeta^{-2k-1} \ket{k+1}\;,\quad\text{and}\quad
		Z\ket{k} = q^{k} \ket{k}\;.
	\ee
Clearly $X,Y,Z$ are unitary.
For any $N\in \mathbb{N}$, these matrices satisfy a first set of parafermion Pauli matrix relations,
	\be\label{PauliRelations1q-}
	X^N=Y^N=Z^N=1\;,\quad
	XY=q^{-1}\,YX\;,\quad
	YZ=q^{-1}\,ZY\;, \quad\text{and}\quad
	ZX=q^{-1}\,XZ\;.
	\ee
They also satisfy a second set of parafermion Pauli matrix relations  that involve $\zeta$,
	\be\label{PauliRelations2q-}
	 XYZ = YZX = ZXY =\zeta^{-1} \;.
	\ee
	
 In case $N=2$, the choices
	\be
	\zeta=-i\;,\quad
	\ket{0} = \bv{1}{0}\;,\quad\text{and}\quad
	\ket{1} = \bv{0}{1}\;,
	\ee
yield the standard representation of the Pauli matrices; the choice $\zeta=i$ yields the complex conjugate of the usual representation.

\subsubsection{Quaternion relations: Version $q^{-1}$\label{sec:quarternion-parafermion}}
One has a parafermion quaternion algebra related to the matrices $X,Y,Z$.  These are $N\times N$ matrices $\boldsymbol{i}$, $\boldsymbol{j}$, $\boldsymbol{k}$ satisfy the \textit{parafermion quaternion relations}:
	\be\label{PFQR}
	\boldsymbol{i}^{N}=\boldsymbol{j}^{N}=\boldsymbol{k}^{N}
	= -1\;,
	\quad
	\boldsymbol{i}\boldsymbol{j}
	= q\, \boldsymbol{j}\boldsymbol{i}\;,\quad
	\boldsymbol{j}\boldsymbol{k}
	= q\, \boldsymbol{k}\boldsymbol{j}\;,\quad
	\boldsymbol{k}\boldsymbol{i}
	= q\, \boldsymbol{i}\boldsymbol{k}\;,\quad
	\text{and}\quad
	\boldsymbol{i}\,\boldsymbol{j}\,\boldsymbol{k}=-1\;.
	\ee
This algebra arises from the three unitary transformations $\boldsymbol{i}$, $\boldsymbol{j}$, $\boldsymbol{k}$ by
	\be
	\boldsymbol{i}=-\zeta^{-1} Y\;,\quad
	\boldsymbol{j}=-\zeta^{-1} X\;,\quad
	\boldsymbol{k}=-\zeta Z\;.
	\ee
As in version $q$, we use $(-\zeta^{\pm1})^{N}=-1$.  Again, the matrices $\boldsymbol{i}$, $\boldsymbol{j}$, $\boldsymbol{k}$ generate the algebra of $N\times N$ matrices.

\subsection{Quadratic representations by parafermions\label{QuadraticPauliRep}}
In this subsection we introduce the representation of $X,Y,Z$ by quadratic expressions in four parafermion operators. However these operators may satisfy different relations with respect to $q$, an $N^{\rm th}$ root of unity and $\zeta=q^{1/2}$.  We introduce a common grading transformation $\gamma$ that is shared among the different subcases discussed in the following sub subsections.

Let $c_{1}, c_{2}, c_{3}, c_{4}$ denote four parafermion operators that satisfy the relations
	\be
		c_{i}c_{j} = q\, c_{j}c_{i}\;,
		\quad\text{for}\quad i<j\;,
		\quad\text{and}\quad
		c_{i}^{N}=1\;,
		\quad\text{where}\quad
		q=e^{\frac{2\pi i}{N}}\;.
	\ee
Let
	\be\label{zeta-defn}
		\zeta=q^{1/2}\;,\quad\text{with}\quad\zeta^{N^{2}}=1\;.
	\ee
Define the grading operator as
	\be\label{gamma-grading}
	\gamma =q c_{1}c_{2}^{-1}c_{3}c_{4}^{-1}=(qc_{1}^{-1}c_{2}c_{3}^{-1}c_{4})^{-1}\;.
	\ee

\subsubsection{\bf Model $(q,1)$}\label{sec:XYZ-PModel-II-Equations}
This model comes from $X,Y,Z$ sharing $c_1$, and taking
	\be
	\widehat{X} = \zeta\, c_{1}c_{4}^{-1} \;,\quad
	\widehat{Y} = \zeta\, c_{1}^{-1}c_{3} \;,\quad
	\widehat{Z} = \zeta\, c_{1}c_{4}^{-1} \;.
	\ee
These matrices have the property that they satisfy the first set of parafermion Pauli relations given in \eqref{PauliRelations1q} for $X,Y,Z$, namely
	\be
		\widehat{X}^{N}
		= \widehat{Y}^{N}
		= \widehat{Z}^{N}
		=1\;,\quad
		\widehat{X}\widehat{Y}
		=q\,\widehat{Y}\widehat{X}\;,\quad
		\widehat{Y}\widehat{Z}
		=q\,\widehat{Z}\widehat{Y}\;,\quad
		\widehat{Z}\widehat{X}
		=q\,\widehat{X}\widehat{Z}\;.
	\ee
In this case one also finds that
	\be
	\widehat{X}\widehat{Y}\widehat{Z} =
	\widehat{Y}\widehat{Z}\widehat{X} =
	\widehat{Z}\widehat{X}\widehat{Y} = \zeta \gamma\;,
	\ee
with $\gamma$ given in \eqref{gamma-grading}.  This shows that $\gamma$ commutes with $\widehat{X}$, $\widehat{Y}$, and $\widehat{Z}$.	
Thus one achieves the desired Pauli relation $\widehat{X}\widehat{Y}\widehat{Z} =\zeta$ representing \eqref{PauliRelations2q} on the subspace for which the unitary $\gamma=1$.

\subsubsection{\bf Model $(q^{-1},4)$}\label{sec:XYZ-PModel-I-Equations-2}
In this model $q^{-1}$ replaces $q$, and the matrices share $c_{4}$.  Define
	\be
	\widehat{X} = \zeta\, c_{1}^{-1}c_{4} \;,\quad
	\widehat{Y} = \zeta\, c_{2}c_{4}^{-1} \;,\quad
	\widehat{Z} = \zeta\, c_{3}^{-1}c_{4} \;.
	\ee
These matrices have the property that they satisfy the first set of parafermion Pauli relations given in \eqref{PauliRelations1q-} for $X,Y,Z$, namely
	\be
		\widehat{X}^{N}
		= \widehat{Y}^{N}
		= \widehat{Z}^{N}
		=1\;,\quad
		\widehat{X}\widehat{Y}
		=q^{-1}\,\widehat{Y}\widehat{X}\;,\quad
		\widehat{Y}\widehat{Z}
		=q^{-1}\,\widehat{Z}\widehat{Y}\;,\quad
		\widehat{Z}\widehat{X}
		=q^{-1}\,\widehat{X}\widehat{Z}\;.
	\ee
Furthermore the product $\widehat{X}\widehat{Y}\widehat{Z}$ has the form
	\be
	\widehat{X}\widehat{Y}\widehat{Z} =
	\widehat{Y}\widehat{Z}\widehat{X} =
	\widehat{Z}\widehat{X}\widehat{Y} = \zeta^{-1} \gamma^{-1}\;.
	\ee
This relation shows that $\gamma$ commutes with $\widehat{X}$, $\widehat{Y}$, and $\widehat{Z}$.  So the product  $\widehat{X}\widehat{Y}\widehat{Z} =\zeta^{-1}$ gives the desired \eqref{PauliRelations2q-}, on the eigenspace for which the unitary operator $\gamma=1$.
%In the $N=2$ case, this is a well-known transformation in condensed matter physics.

\subsubsection{\bf Model $(q,4)$}\label{sec:XYZ-PModel-I-Equations-3}
Here we have $q$ in the parafermion relation and $X,Y,Z$ share $c_{4}$.  Define
	\be
	\widehat{X} = \zeta\, c_{3}^{-1}c_{4} \;,\quad
	\widehat{Y} = \zeta\, c_{2}c_{4}^{-1} \;,\quad
	\widehat{Z} = \zeta\, c_{1}^{-1}c_{4} \;.
	\ee
These operators satisfy
	\be
		\widehat{X}^{N}
		= \widehat{Y}^{N}
		= \widehat{Z}^{N}
		=1\;,\quad
		\widehat{X}\widehat{Y}
		=q\,\widehat{Y}\widehat{X}\;,\quad
		\widehat{Y}\widehat{Z}
		=q\,\widehat{Z}\widehat{Y}\;,\quad
		\widehat{Z}\widehat{X}
		=q\,\widehat{X}\widehat{Z}\;.
	\ee
Furthermore
	\be\label{XYZ-1}
	\widehat{X}\widehat{Y}\widehat{Z} =
	\widehat{Y}\widehat{Z}\widehat{X} =
	\widehat{Z}\widehat{X}\widehat{Y} = \zeta \gamma^{-1}\;,
	\ee
with the same $\gamma$ as in \eqref{gamma-grading}.  Again $\gamma$ commutes with $X,Y,Z$ and the relationship \eqref{XYZ-1} reduces to the desired $\widehat{X}\widehat{Y}\widehat{Z} =\zeta$ in \eqref{PauliRelations2q}  on the eigenspace $\gamma=1$.

%In the $N=2$ case, this is a well-known transformation in condensed matter physics.

\subsubsection{\bf Model $(q^{-1},1)$}\label{sec:XYZ-PModel-II-Equations-4}
This model has $q^{-1}$ and  $X,Y,Z$ share $c_{1}$. Take
	\be
	\widehat{X} = \zeta\, c_{1}^{-1}c_{2} \;,\quad
	\widehat{Y} = \zeta\, c_{1}c_{3}^{-1} \;,\quad
	\widehat{Z} = \zeta\, c_{1}^{-1}c_{4} \;.
	\ee
Then these operators represent \eqref{PauliRelations1q-}, namely
	\be
		\widehat{X}^{N}
		= \widehat{Y}^{N}
		= \widehat{Z}^{N}
		=1\;,\quad
		\widehat{X}\widehat{Y}
		=q^{-1}\,\widehat{Y}\widehat{X}\;,\quad
		\widehat{Y}\widehat{Z}
		=q^{-1}\,\widehat{Z}\widehat{Y}\;,\quad
		\widehat{Z}\widehat{X}
		=q^{-1}\,\widehat{X}\widehat{Z}\;.
	\ee
In this case one also finds that
	\be\label{XYZ-1}
	\widehat{X}\widehat{Y}\widehat{Z} =
	\widehat{Y}\widehat{Z}\widehat{X} =
	\widehat{Z}\widehat{X}\widehat{Y} = \zeta^{-1} \gamma^{-1}\;,
	\ee
with $\gamma$ as in \eqref{gamma-grading}.  Again $\gamma$ commutes with $X,Y,Z$.   So the relationship represents the desired $\widehat{X}\widehat{Y}\widehat{Z} =\zeta^{-1}$ in \eqref{PauliRelations2q-} on the eigenspace $\gamma=1$.

\setcounter{equation}{0}
\section{Pictorial Representations of Parafermion Pauli Matrices\label{sec:PauliDiagrams}}
Here we give several different diagrammatic representations for  the matrices $X, Y, Z$ introduced in \S\ref{Pauli}.  We give alternative representations as different ones can be helpful in different situations. We call these two string and four string models, as the transformations $X,Y,Z$ are given by diagrams with two strings or four strings respectively.  We number the subsections here to correspond as well as possible the numbering of subsections  in \S\ref{Pauli}.

In the two-string models, we represent  $X, Y, Z$ as $N\times N$ matrices.  In the four-string models, we represent $\widehat{X},\widehat{Y},\widehat{Z}$ as $N^{2}\times N^{2}$ matrices.  They are zero-graded, reflecting their definition as particle-anti particle products in \S\ref{Pauli}.  Hence they leave invariant subspaces of dimension $N$ which have fixed grading.  On the zero-graded subspace, the matrices $\widehat{X},\widehat{Y},\widehat{Z}$ represent the matrices $X,Y,Z$.

The diagrams give a simple interpretation to the matrices $\widehat{X},\widehat{Y},\widehat{Z}$ in the four-string models, and show how they leave the appropriate subspace invariant. Throughout this section take $\delta = \sqrt{N}$.

 We call these the different types of four string models the \textit{QI Model}  and the \textit{OA Model}.  They correspond naturally to the representations that arise from the zero-particle state (vectors as two  caps)  and the Markov trace (vectors as nested caps), mentioned in the introduction.  These representations are especially suitable for quantum information (QI) and operator algebras (OA), respectively.
The second model is a generalization of the well-known representation for Pauli matrices by Majoranas,  commonly used in condensed-matter physics.  The first model is different.  In both cases the Pauli matrices are products of parafermion particle operators with their anti-particle operators.

\subsection{The two-string model version $q$}\label{sec:Two-StringModel I}
In this model, we deal with the $X,Y,Z$ directly.  We represent the vector $\ket{k}$ by the cap diagram
	\be
	\ket{k}=N^{-\frac{1}{4}}
\raisebox{-.2cm}{
\tikz{
\fqudit{0}{0}{1/3}{1/3}{k}
}}
\;.
	\ee
The vertical reflection gives the adjoint, or dual vector $\bra{k}$, which we represent as the cup diagram
	\be
		\bra{k}=N^{-\frac{1}{4}}
\raisebox{-.2cm}{
\tikz{
\fmeasure{0}{0}{1/3}{1/3}{-k}
}} \;,
	\ee
so that $\lra{k,k'}= \braket{k}{k'} = \delta_{kk'}$.

The parafermion Pauli matrices $X$, $Y$ and $Z$ act on these vectors.  We represent them as
\begin{align}
X&=
\raisebox{-.4cm}{\tikz{
\draw (0,0)--(0,1);
\draw (2/3,0)--(2/3,1);
\node at (1/3,1/2) {\size{$1$}};
}}
\;,&
Y&=
\raisebox{-.4cm}{\tikz{
\draw (0,0)--(0,1);
\draw (2/3,0)--(2/3,1);
\node at (-1/3,1/2) {\size{$-1$}};
}}\;,&
Z&=
\raisebox{-.4cm}{\tikz{
\draw (0,0)--(0,1);
\draw (2/3,0)--(2/3,1);
\node at (1/3,1/2) {\size{$-1$}};
\node at (-1/3,1/2) {\size{$1$}};
}}\;.
\end{align}
In the diagram for $Z$, we place the labels on the same vertical level, using the notation for the twisted tensor product in \S\ref{Sect:TwistedTensorProduct}.
From the diagrams it is clear that
	\be
	XYZ=\zeta\;.
	\ee

If we represent the basis in the $n$-fold tensor product $\vec{\ket{k}}=\ket{k_1,k_2,\cdots, k_n}$ as
\be
\vec{\ket{k}}
=\frac{1}{d^{n/4}}\ \raisebox{-.5cm}{\tikz{
\fqudit{4/-3}{0\nn}{1/3}{3\nn}{\size{k_1}}
\fqudit{0/-3}{0\nn}{1/3}{2\nn}{\size{k_2}}
\fqudit{-6/-3}{0\nn}{1/3}{1\nn}{\size{k_n}}
\node at (-4/-3,1\nn) {$\cdots$};
}}\quad,
\ee
then we obtain the Jordan-Wigner transformation as:

\begin{align}
1 \otimes \cdots \otimes 1 \otimes X \otimes  1 \otimes \cdots \otimes 1
&= ~~~\raisebox{-.4cm}{
\tikz{
\draw (0-1/6,0)--(0-1/6,1);
\draw (-1/-3-1/6,0)--(-1/-3-1/6,1);
\node at (1/6-1/6,1/2) {\size{1}};
\draw (-2/-3,0)--(-2/-3,1);
\draw (-3/-3,0)--(-3/-3,1);
\draw (-6/-3,0)--(-6/-3,1);
\draw (-7/-3,0)--(-7/-3,1);
\node at (-4/-3,1/2) {$\cdots$};
\node at (-1.5/-3,1/2) {\size{-1}};
\node at (-2.5/-3,1/2) {\size{1}};
\node at (-5.5/-3,1/2) {\size{-1}};
\node at (-6.5/-3,1/2) {\size{1}};
\draw (-2/3,0)--(-2/3,1);
\draw (-3/3,0)--(-3/3,1);
\draw (-5/3,0)--(-5/3,1);
\draw (-6/3,0)--(-6/3,1);
\node at (-4/3,1/2) {$\cdots$};
}} \;,\\
1 \otimes \cdots \otimes 1 \otimes Y \otimes  1 \otimes \cdots \otimes 1
&= ~\raisebox{-.4cm}{
\tikz{
\draw (0-1/6,0)--(0-1/6,1);
\draw (-1/-3-1/6,0)--(-1/-3-1/6,1);
\node at (-1/6-1/6,1/2) {\size{-1}};
\draw (-2/-3,0)--(-2/-3,1);
\draw (-3/-3,0)--(-3/-3,1);
\draw (-6/-3,0)--(-6/-3,1);
\draw (-7/-3,0)--(-7/-3,1);
\node at (-4/-3,1/2) {$\cdots$};
\node at (-1.5/-3,1/2) {\size{1}};
\node at (-2.5/-3,1/2) {\size{-1}};
\node at (-5.5/-3,1/2) {\size{1}};
\node at (-6.5/-3,1/2) {\size{-1}};
\draw (-2/3,0)--(-2/3,1);
\draw (-3/3,0)--(-3/3,1);
\draw (-5/3,0)--(-5/3,1);
\draw (-6/3,0)--(-6/3,1);
\node at (-4/3,1/2) {$\cdots$};
}} \;,\\
1 \otimes \cdots \otimes 1 \otimes Z \otimes 1 \otimes \cdots \otimes 1
&= \; \raisebox{-.4cm}{
\tikz{
\draw (0-1/6,0)--(0-1/6,1);
\draw (-1/-3-1/6,0)--(-1/-3-1/6,1);
\node at (-1/6-1/6,1/2) {\size{1}};
\node at (0,1/2) {\size{-1}};
\draw (-2/-3,0)--(-2/-3,1);
\draw (-3/-3,0)--(-3/-3,1);
\draw (-6/-3,0)--(-6/-3,1);
\draw (-7/-3,0)--(-7/-3,1);
\node at (-4/-3,1/2) {$\cdots$};
\draw (-2/3,0)--(-2/3,1);
\draw (-3/3,0)--(-3/3,1);
\draw (-5/3,0)--(-5/3,1);
\draw (-6/3,0)--(-6/3,1);
\node at (-4/3,1/2) {$\cdots$};
}} \;.
\end{align}

If we represent the basis in the $n$-fold tensor product $\vec{\ket{k}}=\ket{k_1,k_2,\cdots, k_n}$ as
\be
\vec{\ket{k}}
=\frac{\size{1}}{d^{n/4}}\ \raisebox{-.5cm}{\tikz{
\fqudit{4/-3}{0\nn}{1/3}{1\nn}{k_1}
\fqudit{0/-3}{0\nn}{1/3}{2\nn}{k_2}
\fqudit{-6/-3}{0\nn}{1/3}{3\nn}{k_n}
\node at (-4/-3,1\nn) {$\cdots$};
}}\quad,
\ee
then we obtain the Jordan-Wigner transformation as:
\begin{align}
1 \otimes \cdots \otimes 1 \otimes X \otimes 1 \otimes \cdots \otimes 1
=& \raisebox{-.4cm}{
\tikz{
\draw (0-1/6,0)--(0-1/6,1);
\draw (1/3-1/6,0)--(1/3-1/6,1);
\node at (1/6-1/6,1/2) {\size{1}};
\draw (-2/3,0)--(-2/3,1);
\draw (-3/3,0)--(-3/3,1);
\draw (-6/3,0)--(-6/3,1);
\draw (-7/3,0)--(-7/3,1);
\node at (-5/3,1/2) {$\cdots$};
\node at (-5/6,1/2) {\size{-1}};
\node at (-7/6,1/2) {\size{1}};
\node at (-13/6,1/2) {\size{-1}};
\node at (-15/6,1/2) {\size{1}};
\draw (2/3,0)--(2/3,1);
\draw (3/3,0)--(3/3,1);
\draw (5/3,0)--(5/3,1);
\draw (6/3,0)--(6/3,1);
\node at (4/3,1/2) {$\cdots$};
}} \;, \\
1 \otimes \cdots \otimes 1 \otimes Y \otimes 1 \otimes \cdots \otimes 1
=& \raisebox{-.4cm}{
\tikz{
\draw (0-1/6,0)--(0-1/6,1);
\draw (1/3-1/6,0)--(1/3-1/6,1);
\node at (-1/6-1/6,1/2) {\size{-1}};
\draw (-2/3,0)--(-2/3,1);
\draw (-3/3,0)--(-3/3,1);
\draw (-6/3,0)--(-6/3,1);
\draw (-7/3,0)--(-7/3,1);
\node at (-5/3,1/2) {$\cdots$};
\node at (-5/6,1/2) {\size{1}};
\node at (-7/6,1/2) {\size{-1}};
\node at (-13/6,1/2) {\size{1}};
\node at (-15/6,1/2) {\size{-1}};
\draw (2/3,0)--(2/3,1);
\draw (3/3,0)--(3/3,1);
\draw (5/3,0)--(5/3,1);
\draw (6/3,0)--(6/3,1);
\node at (4/3,1/2) {$\cdots$};
}} \;,\\
1 \otimes \cdots \otimes 1 \otimes Z \otimes 1 \otimes \cdots \otimes 1
=&\quad  \raisebox{-.4cm}{
\tikz{
\draw (0-1/6,0)--(0-1/6,1);
\draw (1/3-1/6,0)--(1/3-1/6,1);
\node at (0,1/2) {\size{-1}};
\node at (-1/6-1/6,1/2) {\size{1}};
\draw (-2/3,0)--(-2/3,1);
\draw (-3/3,0)--(-3/3,1);
\draw (-6/3,0)--(-6/3,1);
\draw (-7/3,0)--(-7/3,1);
\node at (-5/3,1/2) {$\cdots$};
\draw (2/3,0)--(2/3,1);
\draw (3/3,0)--(3/3,1);
\draw (5/3,0)--(5/3,1);
\draw (6/3,0)--(6/3,1);
\node at (4/3,1/2) {$\cdots$};
}} \;.
\end{align}

\subsection{The two-string model version $q^{-1}$}\label{sec:Two-StringModel II}
In this model, we deal with the $X,Y,Z$ directly.  We represent the vector $\ket{k}$ by the cap diagram
	\be
	\ket{k}=N^{-\frac{1}{4}}
\raisebox{-.2cm}{
\tikz{
\fqudit{0}{0}{1/3}{1/3}{}
\node at (-1/6,1/3) {\size{$k$}};
}}
\;.
	\ee
The vertical reflection gives the adjoint, or dual vector $\bra{k}$, which we represent as the cup diagram
	\be
		\bra{k}=N^{-\frac{1}{4}}
\raisebox{-.2cm}{
\tikz{
\fmeasure{0}{0}{1/3}{1/3}{}
\node at (-1/3,-1/3) {\size{$-k$}};
}} \;,
	\ee
so that $\lra{k,k'}= \braket{k}{k'} = \delta_{kk'}$.

The parafermion Pauli matrices $X$, $Y$ and $Z$ act on these vectors.  We represent them as
\begin{align}
X&=
\raisebox{-.4cm}{\tikz{
\draw (0,0)--(0,1);
\draw (2/3,0)--(2/3,1);
\node at (-1/3,1/2) {\size{$-1$}};
}}
\;,&
Y&=
\raisebox{-.4cm}{\tikz{
\draw (0,0)--(0,1);
\draw (2/3,0)--(2/3,1);
\node at (1/3,1/2) {\size{$1$}};
}}\;,&
Z&=
\raisebox{-.4cm}{\tikz{
\draw (0,0)--(0,1);
\draw (2/3,0)--(2/3,1);
\node at (1/3,1/2) {\size{$-1$}};
\node at (-1/3,1/2) {\size{$1$}};
}}\;.
\end{align}
From the diagrams it is clear that
	\be
	XYZ=\zeta^{-1}\;.
	\ee

We have the Jordan-Wigner transformation similarly.

\subsection{Four-string models}
We have the diagram yielding the zero-particle state used in quantum information (QI),
\be\label{Zero-particle state}
	\ket{k}=N^{-\frac{1}{2}}
\raisebox{-.2cm}{
\tikz{
\fqudit{0}{0}{1/3}{1/3}{}
\fqudit{4/3}{0}{1/3}{1/3}{}
}}\;.
\ee
We also have the diagram yielding  the Markov tracial state used in operator algebra theory (OA),
\be
	\ket{k}=
	N^{-\frac{1}{2}}
\raisebox{-.5cm}{
\tikz{
\fdoublequdit{0}{0}{1/3}{1/3}{}{}
}}
\;.
\ee
Also we have the diagram for the grading operator
	\be\label{GradingDiagram}
	\gamma	=
\raisebox{-.4cm}{
\tikz{
\draw (-2/3,0)--(-2/3,1);
\draw (-3/3,0)--(-3/3,1);
\draw (-4/3,0)--(-4/3,1);
\draw (-5/3,0)--(-5/3,1);
\node at (-5/6,1/2) {\size{-1}};
\node at (-7/6,1/2) {\size{1}};
\node at (-9/6,1/2) {\size{-1}};
\node at (-11/6,1/2) {\size{1}};
}}
\;.
	\ee
We call $\gamma$ the \textit{grading operator}, since it detects the grading when acting on the underling vectors.
\be
\raisebox{-.4cm}{
\tikz{
\fqudit {-3/3}{1}{1/6}{0}{j}
\fqudit {-5/3}{1}{1/6}{1/3}{i}
\draw (-2/3,0)--(-2/3,1);
\draw (-3/3,0)--(-3/3,1);
\draw (-4/3,0)--(-4/3,1);
\draw (-5/3,0)--(-5/3,1);
\node at (-5/6,1/2) {\size{-1}};
\node at (-7/6,1/2) {\size{1}};
\node at (-9/6,1/2) {\size{-1}};
\node at (-11/6,1/2) {\size{1}};
}}
=q^{i+j}
\raisebox{-.4cm}{
\tikz{
\fqudit {-3/3}{1}{1/6}{0}{j}
\fqudit {-5/3}{1}{1/6}{1/3}{i}
\draw (-2/3,0)--(-2/3,1);
\draw (-3/3,0)--(-3/3,1);
\draw (-4/3,0)--(-4/3,1);
\draw (-5/3,0)--(-5/3,1);
}}
\;.
\ee
Hence the eigenspace $\gamma=1$ defines the zero-graded subspace.

\subsubsection{\bf The (QI,$q$) four string model
\label{Sect:QIq4string}}  This model corresponds to the $\widehat{X}$, $\widehat{Y}$ and $\widehat{Z}$ with $q$ and with the shared first string.  This corresponds to the algebraic model of  \S\ref{sec:XYZ-PModel-II-Equations} with the shared operator $c_{1}$.  Here we use the quantum information (QI) representation for vectors.

The basis states  $\ket{k}$ belong to the zero-graded part of the tensor product of two copies of the two-string model.
We represent the vector $\ket{k}$ and its adjoint by
	\be\label{Zero-GradedBasis}
	\ket{k}=N^{-\frac{1}{2}}
\raisebox{-.2cm}{
\tikz{
\fqudit{0}{0}{1/3}{1/3}{}
\fqudit{4/3}{0}{1/3}{1/3}{-k}
\node at (-1/6,1/3) {\size{$k$}};
}}
\;,
	\qquad\text{and}\qquad
	\bra{k} = N^{-\frac{1}{2}}
\raisebox{-.2cm}{
\tikz{
\fmeasure{0}{0}{1/3}{1/3}{}
\fmeasure{4/3}{0}{1/3}{1/3}{-k}
\node at (-1/6,-1/3) {\size{$k$}};
}}
\;.
	\ee	
Furthermore the parafermion Pauli matrices $\widehat{X}$, $\widehat{Y}$ and $\widehat{Z}$  are
\begin{align}
\widehat{X}&=
\raisebox{-.4cm}{
\tikz{
\draw (-2/3,0)--(-2/3,1);
\draw (-3/3,0)--(-3/3,1);
\draw (-4/3,0)--(-4/3,1);
\draw (-5/3,0)--(-5/3,1);
\node at (-5/6,1/2) {\size{-1}};
\node at (-11/6,1/2) {\size{1}};
}} \;,&
\widehat{Y}&=
\raisebox{-.4cm}{
\tikz{
\draw (-2/3,0)--(-2/3,1);
\draw (-3/3,0)--(-3/3,1);
\draw (-4/3,0)--(-4/3,1);
\draw (-5/3,0)--(-5/3,1);
\node at (-7/6,1/2) {\size{1}};
\node at (-11/6,1/2) {\size{-1}};
}} \;,&
\widehat{Z}&=
\raisebox{-.4cm}{
\tikz{
\draw (-2/3,0)--(-2/3,1);
\draw (-3/3,0)--(-3/3,1);
\draw (-4/3,0)--(-4/3,1);
\draw (-5/3,0)--(-5/3,1);
\node at (-9/6,1/2) {\size{-1}};
\node at (-11/6,1/2) {\size{1}};
}} \;.
\end{align}
Note that here the labels are at the same vertical level, so they correspond to the twisted tensor product of \S\ref{Sect:TwistedTensorProduct}.
Multiplying these representations, we see
	\[
\widehat{X}\widehat{Y}\widehat{Z}=\zeta\gamma\;,
	\]
with $\gamma$ in \eqref{GradingDiagram}.
Hence $\gamma=1$ on the zero-graded subspace on which the matrices $\widehat{X}$, $\widehat{Y}$, and $\widehat{Z}$ satisfy the correct algebraic relations.  The diagrams show that  $\widehat{X}$, $\widehat{Y}$, and $\widehat{Z}$ preserve the grading and that they commute with $\gamma$.

\subsubsection{\bf The (QI, $q^{-1}$) four-string model} This model corresponds to the parafermion representation of $\widehat{X}$, $\widehat{Y}$ and $\widehat{Z}$ given in \S\ref{sec:XYZ-PModel-I-Equations-2} with $q^{-1}$ and a shared fourth string.  Again we use the quantum information (QI) representation for vectors.

The vectors $\ket{k}$ belong to the zero-graded part of the tensor product of two copies of the two-string model.
We represent the vector $\ket{k}$ and its adjoint by
	\be\label{Zero-GradedBasis}
	\ket{k}=N^{-\frac{1}{2}}
\raisebox{-.2cm}{
\tikz{
\fqudit{0}{0}{1/3}{1/3}{}
\fqudit{4/3}{0}{1/3}{1/3}{-k}
\node at (-1/6,1/3) {\size{$k$}};
}}
\;,
	\qquad\text{and}\qquad
	\bra{k} = N^{-\frac{1}{2}}
\raisebox{-.2cm}{
\tikz{
\fmeasure{0}{0}{1/3}{1/3}{}
\fmeasure{4/3}{0}{1/3}{1/3}{-k}
\node at (-1/6,-1/3) {\size{$k$}};
}}
\;.
	\ee
The parafermion Pauli matrices $\widehat{X}$, $\widehat{Y}$ and $\widehat{Z}$
\begin{align}
\widehat{X}&=
\raisebox{-.4cm}{
\tikz{
\draw (-2/3,0)--(-2/3,1);
\draw (-3/3,0)--(-3/3,1);
\draw (-4/3,0)--(-4/3,1);
\draw (-5/3,0)--(-5/3,1);
\node at (-5/6,1/2) {\size{1}};
\node at (-11/6,1/2) {\size{-1}};
}} \;,&
\widehat{Y}&=
\raisebox{-.4cm}{
\tikz{
\draw (-2/3,0)--(-2/3,1);
\draw (-3/3,0)--(-3/3,1);
\draw (-4/3,0)--(-4/3,1);
\draw (-5/3,0)--(-5/3,1);
\node at (-5/6,1/2) {\size{-1}};
\node at (-9/6,1/2) {\size{1}};
}} \;,&
\widehat{Z}&=
\raisebox{-.4cm}{
\tikz{
\draw (-2/3,0)--(-2/3,1);
\draw (-3/3,0)--(-3/3,1);
\draw (-4/3,0)--(-4/3,1);
\draw (-5/3,0)--(-5/3,1);
\node at (-5/6,1/2) {\size{1}};
\node at (-7/6,1/2) {\size{-1}};
}} \;.
\end{align}
Also
	\[
	\gamma^{-1}=\zeta \widehat{X}\widehat{Y}\widehat{Z}
	=
\raisebox{-.4cm}{
\tikz{
\draw (-2/3,0)--(-2/3,1);
\draw (-3/3,0)--(-3/3,1);
\draw (-4/3,0)--(-4/3,1);
\draw (-5/3,0)--(-5/3,1);
\node at (-5/6,1/2) {\size{1}};
\node at (-7/6,1/2) {\size{-1}};
\node at (-9/6,1/2) {\size{1}};
\node at (-11/6,1/2) {\size{-1}};
}}
\;.
	\]
This acts on vectors as
\be
\raisebox{-.4cm}{
\tikz{
\fqudit {-3/3}{1}{1/6}{0}{j}
\fqudit {-5/3}{1}{1/6}{1/3}{i}
\draw (-2/3,0)--(-2/3,1);
\draw (-3/3,0)--(-3/3,1);
\draw (-4/3,0)--(-4/3,1);
\draw (-5/3,0)--(-5/3,1);
\node at (-5/6,1/2) {\size{1}};
\node at (-7/6,1/2) {\size{-1}};
\node at (-9/6,1/2) {\size{1}};
\node at (-11/6,1/2) {\size{-1}};
}}
=q^{-i-j}
\raisebox{-.4cm}{
\tikz{
\fqudit {-3/3}{1}{1/6}{0}{j}
\fqudit {-5/3}{1}{1/6}{1/3}{i}
\draw (-2/3,0)--(-2/3,1);
\draw (-3/3,0)--(-3/3,1);
\draw (-4/3,0)--(-4/3,1);
\draw (-5/3,0)--(-5/3,1);
}}
\;.
\ee

\subsubsection{\bf The (OA, $q$)  four-string model}
This model corresponds to the parafermion representation of $\widehat{X}$, $\widehat{Y}$ and $\widehat{Z}$ given in \S\ref{sec:XYZ-PModel-I-Equations-3} with $q$ and a shared fourth string.  Here we use the operator algebra (OA) representation for vectors.

We represent the vector $\ket{k}$ and its dual $\bra{k}$ by
	\be
	\ket{k}=
	N^{-\frac{1}{2}}
\raisebox{-.5cm}{
\tikz{
\fdoublequdit{0}{0}{1/3}{1/3}{}{}
\node at (-1/6,1/3) {\size{$k$}};
\node at (1/3,1/3) {\size{$-k$}};
}}
\;,
	\qquad\text{and}\qquad
	\bra{k}=
	N^{-\frac{1}{2}}
\raisebox{-.5cm}{
\tikz{
\fdoublemeasure{0}{0}{1/3}{1/3}{}{}
\node at (-1/3,-1/3) {\size{$-k$}};
\node at (1/3+1/6,-1/3) {\size{$k$}};
}}
\;.
	\ee
The Pauli matrices $X$, $Y$ and $Z$ are represented by
\begin{align}
\widehat{X}&=
\raisebox{-.4cm}{
\tikz{
\draw (-2/3,0)--(-2/3,1);
\draw (-3/3,0)--(-3/3,1);
\draw (-4/3,0)--(-4/3,1);
\draw (-5/3,0)--(-5/3,1);
\node at (-7/6,1/2) {\size{-1}};
\node at (-5/6,1/2) {\size{1}};
}} \;,&
\widehat{Y}&=
\raisebox{-.4cm}{
\tikz{
\draw (-2/3,0)--(-2/3,1);
\draw (-3/3,0)--(-3/3,1);
\draw (-4/3,0)--(-4/3,1);
\draw (-5/3,0)--(-5/3,1);
\node at (-5/6,1/2) {\size{-1}};
\node at (-9/6,1/2) {\size{1}};
}} \;,&
\widehat{Z}&=
\raisebox{-.4cm}{
\tikz{
\draw (-2/3,0)--(-2/3,1);
\draw (-3/3,0)--(-3/3,1);
\draw (-4/3,0)--(-4/3,1);
\draw (-5/3,0)--(-5/3,1);
\node at (-5/6,1/2) {\size{1}};
\node at (-11/6,1/2) {\size{-1}};
}} \;.
\end{align}

The grading operator that occurs is represented by
	\be
	\gamma^{-1}=\zeta^{-1} \widehat{X}\widehat{Y}\widehat{Z}
	=
\raisebox{-.4cm}{
\tikz{
\draw (-2/3,0)--(-2/3,1);
\draw (-3/3,0)--(-3/3,1);
\draw (-4/3,0)--(-4/3,1);
\draw (-5/3,0)--(-5/3,1);
\node at (-5/6,1/2) {\size{1}};
\node at (-7/6,1/2) {\size{-1}};
\node at (-9/6,1/2) {\size{1}};
\node at (-11/6,1/2) {\size{-1}};
}}
\;.
	\ee

\subsubsection{\bf The (OA,  $q^{-1}$) four-string model}
This model corresponds to the parafermion representation of $\widehat{X}$, $\widehat{Y}$ and $\widehat{Z}$ given in \S\ref{sec:XYZ-PModel-II-Equations-4} with $q^{-1}$ and a shared first string.  Here we use the operator algebra (OA) representation for vectors.

We represent the vector $\ket{k}$ and its dual $\bra{k}$ by
	\be
	\ket{k}=
	N^{-\frac{1}{2}}
\raisebox{-.5cm}{
\tikz{
\fdoublequdit{0}{0}{1/3}{1/3}{}{}
\node at (-1/6,1/3) {\size{$k$}};
\node at (1/3,1/3) {\size{$-k$}};
}}
\;,
	\qquad\text{and}\qquad
	\bra{k}=
	N^{-\frac{1}{2}}
\raisebox{-.5cm}{
\tikz{
\fdoublemeasure{0}{0}{1/3}{1/3}{}{}
\node at (-1/3,-1/3) {\size{$-k$}};
\node at (1/3+1/6,-1/3) {\size{$k$}};
}}
\;.
	\ee
The Pauli matrices $X$, $Y$ and $Z$ are presented by
\begin{align}
\widehat{X}&=
\raisebox{-.4cm}{
\tikz{
\draw (-2/3,0)--(-2/3,1);
\draw (-3/3,0)--(-3/3,1);
\draw (-4/3,0)--(-4/3,1);
\draw (-5/3,0)--(-5/3,1);
\node at (-11/6,1/2) {\size{-1}};
\node at (-9/6,1/2) {\size{1}};
}} \;,&
\widehat{Y}&=
\raisebox{-.4cm}{
\tikz{
\draw (-2/3,0)--(-2/3,1);
\draw (-3/3,0)--(-3/3,1);
\draw (-4/3,0)--(-4/3,1);
\draw (-5/3,0)--(-5/3,1);
\node at (-7/6,1/2) {\size{-1}};
\node at (-11/6,1/2) {\size{1}};
}} \;,&
\widehat{Z}&=
\raisebox{-.4cm}{
\tikz{
\draw (-2/3,0)--(-2/3,1);
\draw (-3/3,0)--(-3/3,1);
\draw (-4/3,0)--(-4/3,1);
\draw (-5/3,0)--(-5/3,1);
\node at (-5/6,1/2) {\size{1}};
\node at (-11/6,1/2) {\size{-1}};
}} \;.
\end{align}

The grading operator $\gamma$ is represented by
	\be
	\gamma^{-1}=\zeta \widehat{X}\widehat{Y}\widehat{Z}
	=
\raisebox{-.4cm}{
\tikz{
\draw (-2/3,0)--(-2/3,1);
\draw (-3/3,0)--(-3/3,1);
\draw (-4/3,0)--(-4/3,1);
\draw (-5/3,0)--(-5/3,1);
\node at (-5/6,1/2) {\size{1}};
\node at (-7/6,1/2) {\size{-1}};
\node at (-9/6,1/2) {\size{1}};
\node at (-11/6,1/2) {\size{-1}};
}}
\;.
	\ee

%{\color{red}
%In this case, the matrices $b_1$ and $b_2$ defined in Equations \eqref{b1} and \eqref{b2} are presented pictorially (up to a phase) as
%\be
%b_{1}=\graa{b1+}\;,\qquad\quad b_{2}= \graa{b2+}\;.
%\ee
%One can also derive braiding relations in the Type II model for the parafermion Pauli matrices, such as \eqref{conjugation-1}--\eqref{conjugation-2} in the Type I model, by using the braiding relations in \S\ref{braided relations}.
%}
\setcounter{equation}{0}
\section{A Pictorial Interpretation of Parafermion Algebras}\label{Pictorial presentation}
\subsection{Parafermion algebras}\label{parafermion algebra}
The parafermion algebra is defined by
generators: $c_i$, $i=1,2,\cdots$
and relations,
	\be
	c_i^N=1\;,\quad c_i \,c_j=q\,c_j\,c_i\;, \quad\text{for}\quad i<j\;, \quad\text{with}\quad q=e^{\frac{2\pi i}{N}}\;.
	\ee
Denote the parafermion algebra generated by $c_i$, $1\leq i\leq m$ as $PF_m$.  It has a basis
%$C_I=c_{i_1}^{n_1}c_{i_2}^{n_2}\cdots c_{i_k}^{n_k}$, for $i_1<i_2<\cdots<i_k$, $0\leq n_1, n_2, \cdots, n_k\leq N-1$.
$C_I=c_1^{i_1}c_2^{i_2}\cdots c_m^{i_m}$, for $0\leq i_1, i_2, \cdots, i_m\leq N-1$.
The expectation on $PF_m$ is defined as $tr(1)=1$, $tr(C_I)=0$, if $C_I\neq 1$. It is a tracial state.
The inclusion from $PF_m$ to $PF_{m+1}$ is trace preserving.

\begin{remark}
If we apply the Gelfand-Naimark-Segal construction to the inductive limit $\displaystyle \lim_{m \to\infty} PF_m$ with respect to the tracial state, then we obtain a hyperfinite factor $\mathcal{R}$ of type II$_1$. The Bernoulli shift $c_i \to c_{i+1}$ is an endomorphism $\rho$ of the factor $\mathcal{R}$.
The graded standard invariant of the corresponding subfactor $\mathcal{R}\supset\rho(\mathcal{R})$ is exactly the subfactor planar para algebra for parafermions with $\delta=\sqrt{N}$. \end{remark}

\subsection{Actions of planar tangles on parafermion algebras}
In planar para algebras, the labelled regular planar $m$-tangle \graa{12m} is presented by the vector $c_1^{i_1}c_2^{i_2}\cdots c_m^{i_m}$ in $PF_m$.

The Markov trace $\frac{1}{\delta^m}\grb{tr}$ is the expectation on $PF_m$.

The multiplication tangle gives the usual multiplication on $PF_m$:
$xy=\grb{xy}.$

The tangle
$\frac{1}{\delta}\graa{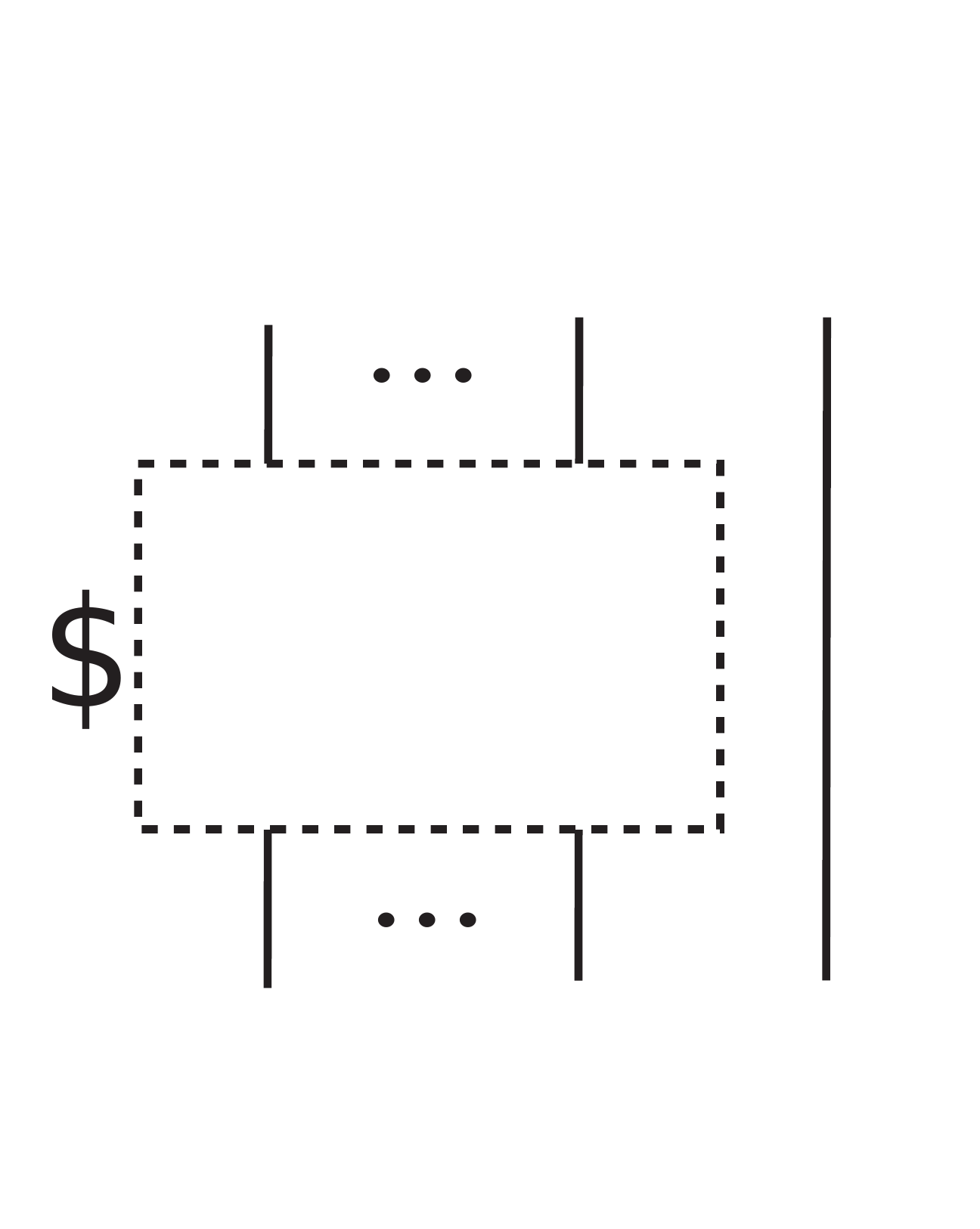}$  is the trace preserving inclusion from $PF_m$ to $PF_{m+1}$.

The tangle
$\frac{1}{\delta}\graa{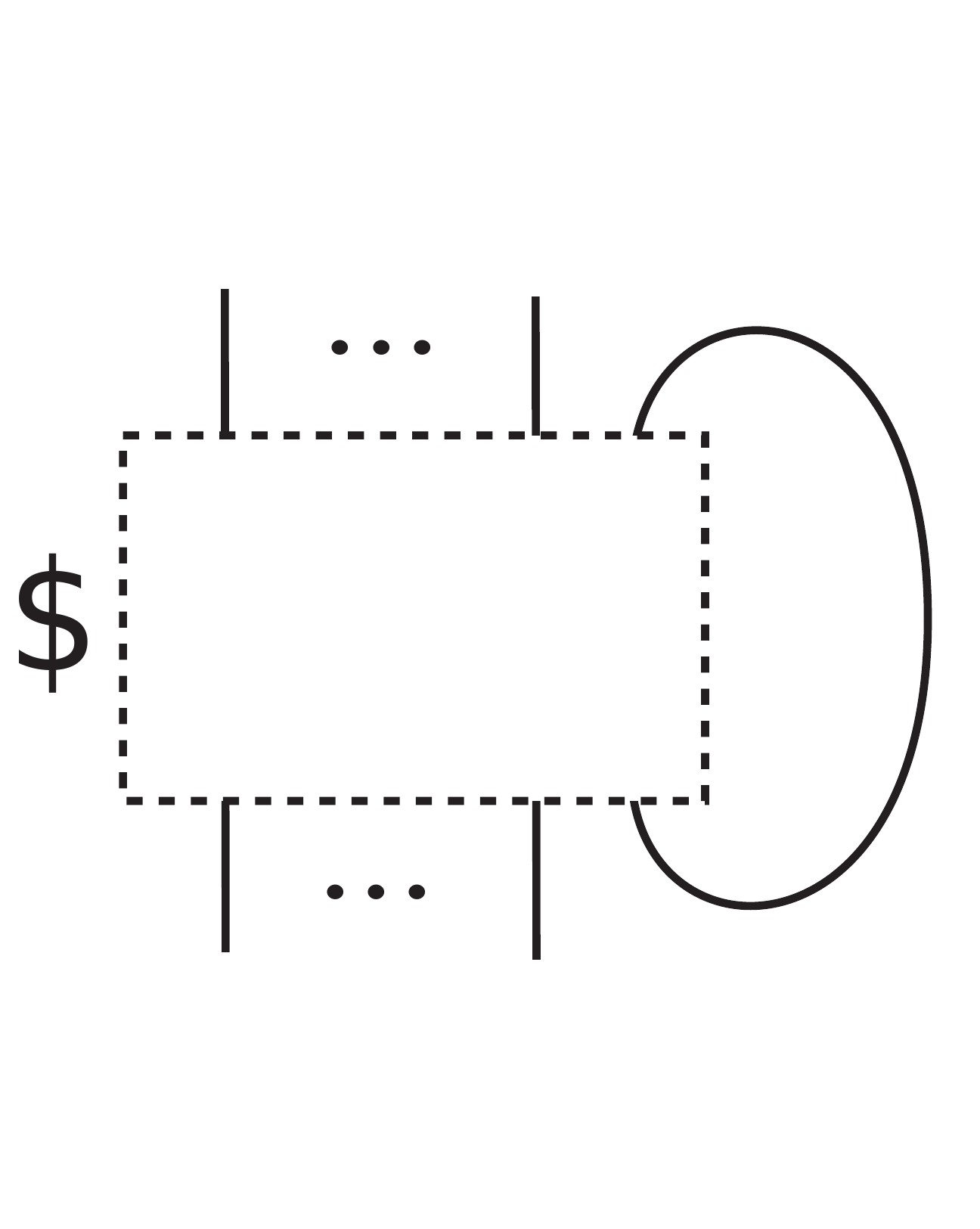}$  is the trace preserving conditional expectation from $PF_m$ to $PF_{m-1}$.

We also have the graded tensor products from $PF_m\hat{\otimes}PF_n$ to $PF_{m+n}$ given by
$$\grb{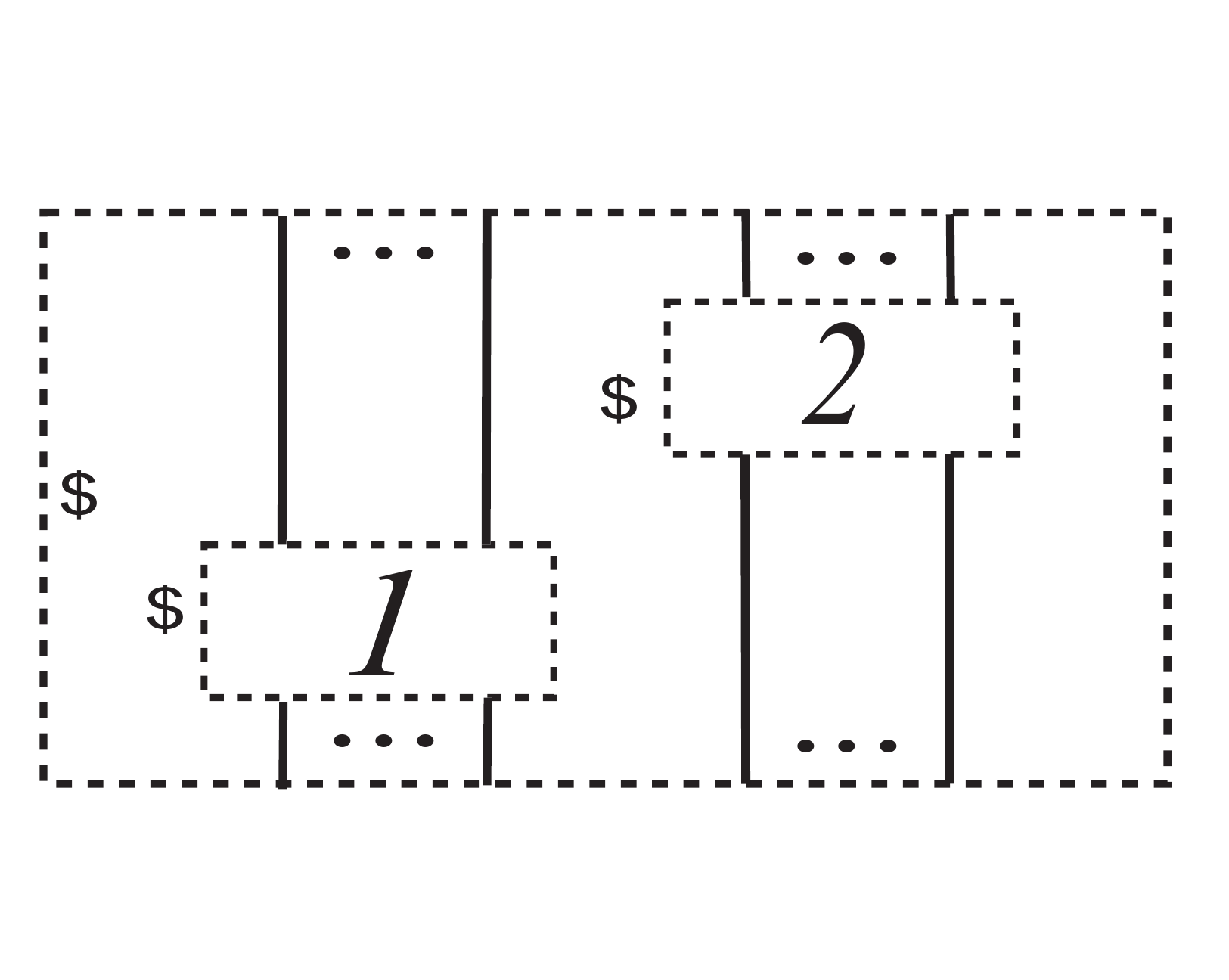}.$$

\subsection{Temperley-Lieb subalgebras}
Take
	\[
	E_i
	=\frac{1}{\sqrt{N}} \sum_{k=0}^{N-1}
		q^{\frac{k^2}{2}} \,c_i^{k} \, c_{i+1}^{-k}
	=\frac{1}{\sqrt{N}} \sum_{k=0}^{N-1}
		q^{-\frac{k^2}{2}} \,c_{i+1}^{-k} \,c_i^{k}\;.
	\]
From Equation \eqref{e=sumgi}, we infer that  $E_i$
 is presented by \graa{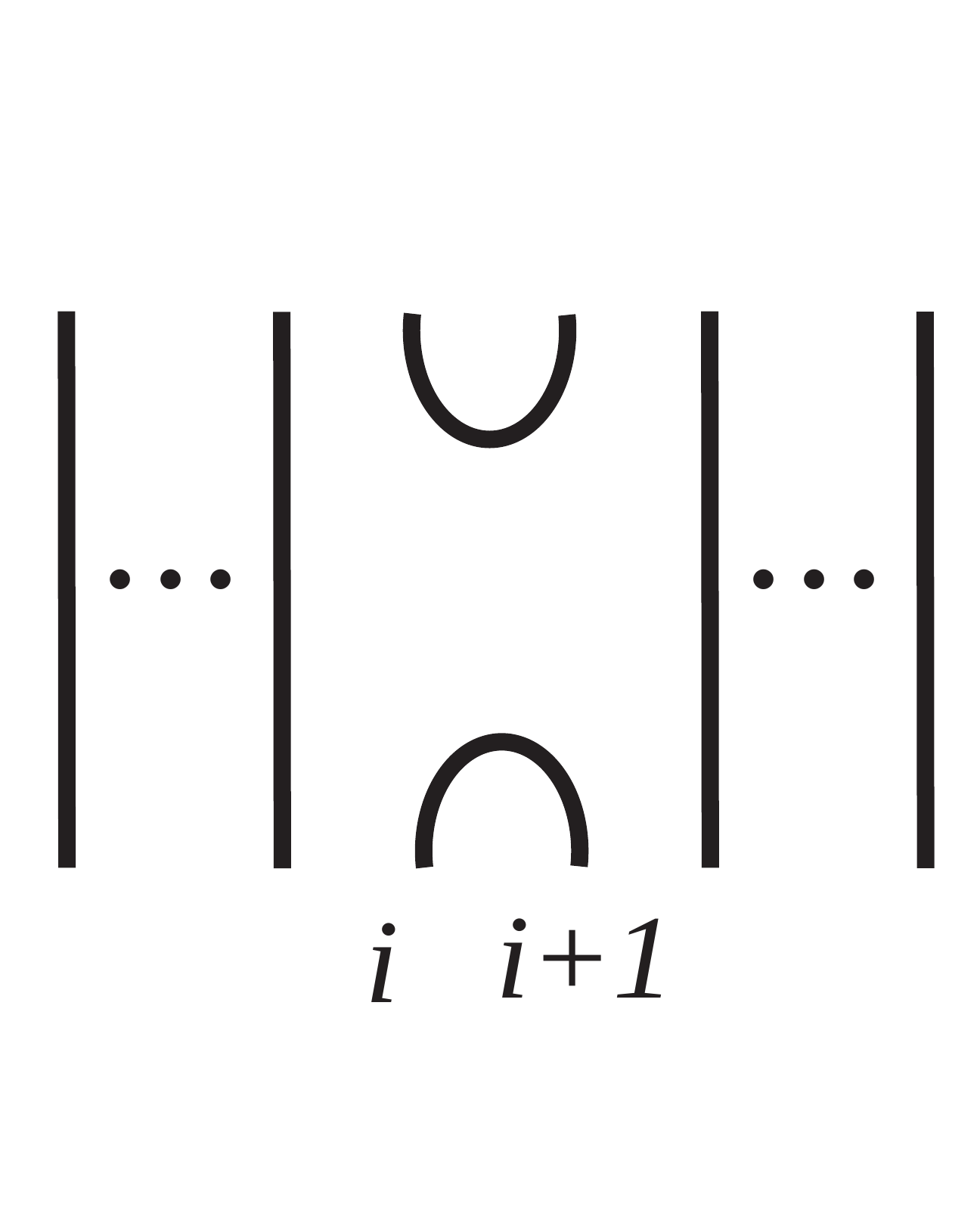}.  The $E_{i}$ satisfy the following relations and generate a Temperley-Lieb subalgebra in the parafermion algebra:
\begin{itemize}
\item[(1)] $E_i=E_i^*=\frac{1}{\sqrt{N}}\, E_i^2$.

\item[(2)] $E_i\,E_j=E_j\,E_i$, for  $|i-j|\geq2$.

\item[(3)] $E_i\,E_{i\pm1}\,E_i=E_i$.
\end{itemize}
One can check the following joint relations for $E_i$ and $c_i$ algebraically,
$$E_i\,c_i^k=q^{-\frac{k^2}{2\phantom{k}}}E_i\,c_{i+1}^k\;,$$
$$c_i^k E_i=q^{\frac{1^2}{2\phantom{1}}}c_{i+1}^kE_i\;.$$
They can also be derived from the relation
$\graa{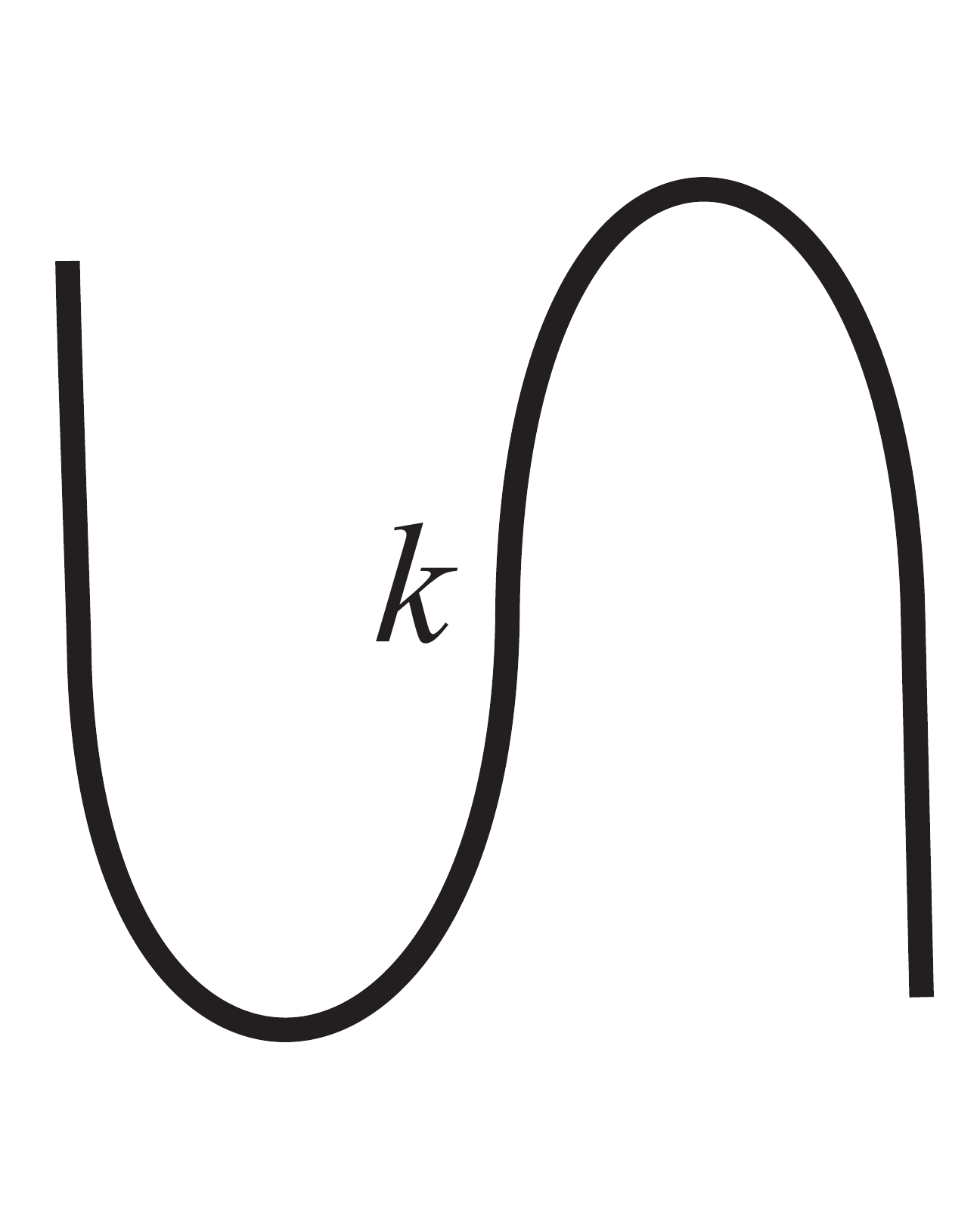}=\zeta^{k^2} \graa{ck}$.

\subsection{String Fourier transform\label{sec:Fourier}}
The string Fourier transform is an important ingredient in subfactor planar (para) algebras. Since the parafermion algebra forms a subfactor planar para algebra. we can introduce the string Fourier transform on parafermion algebras. Its algebraic definition is complicated, but its topological definition is simply a rotation.
The string Fourier transform $\mathfrak{F}_{\text{s}}$ is given by the action of the following tangle,
$$\graa{Fourier}.$$

Algebraically the string Fourier transform on $PF_m$ is defined as follows: We first embed $PF_m$ in $PF_{m+1}$ by mapping $c_i$ to $c_{i+1}$. The inclusion is denoted by $\iota_l$.
Let $\Phi_r$ be the trace preserving conditional expectation from $PF_{m+1}$ to the subalgebra $PF_m$ generated by $c_1,c_2, \cdots, c_m$.
Then the string Fourier transform of $x \in PF_m$ is defined as
$\mathfrak{F}_{\text{s}}(x)=\sqrt{N}\Phi_r(E_mE_{m-1}\cdots E_1\iota_l(x))$.

In particular, the zero graded part of the 2-box space has a basis $$\left\{\graa{i=-i}\right\}_{i\in\mathbb{Z}_N}\;.$$
Moreover, the basis forms the group $\mathbb{Z}_N$:
$$\graa{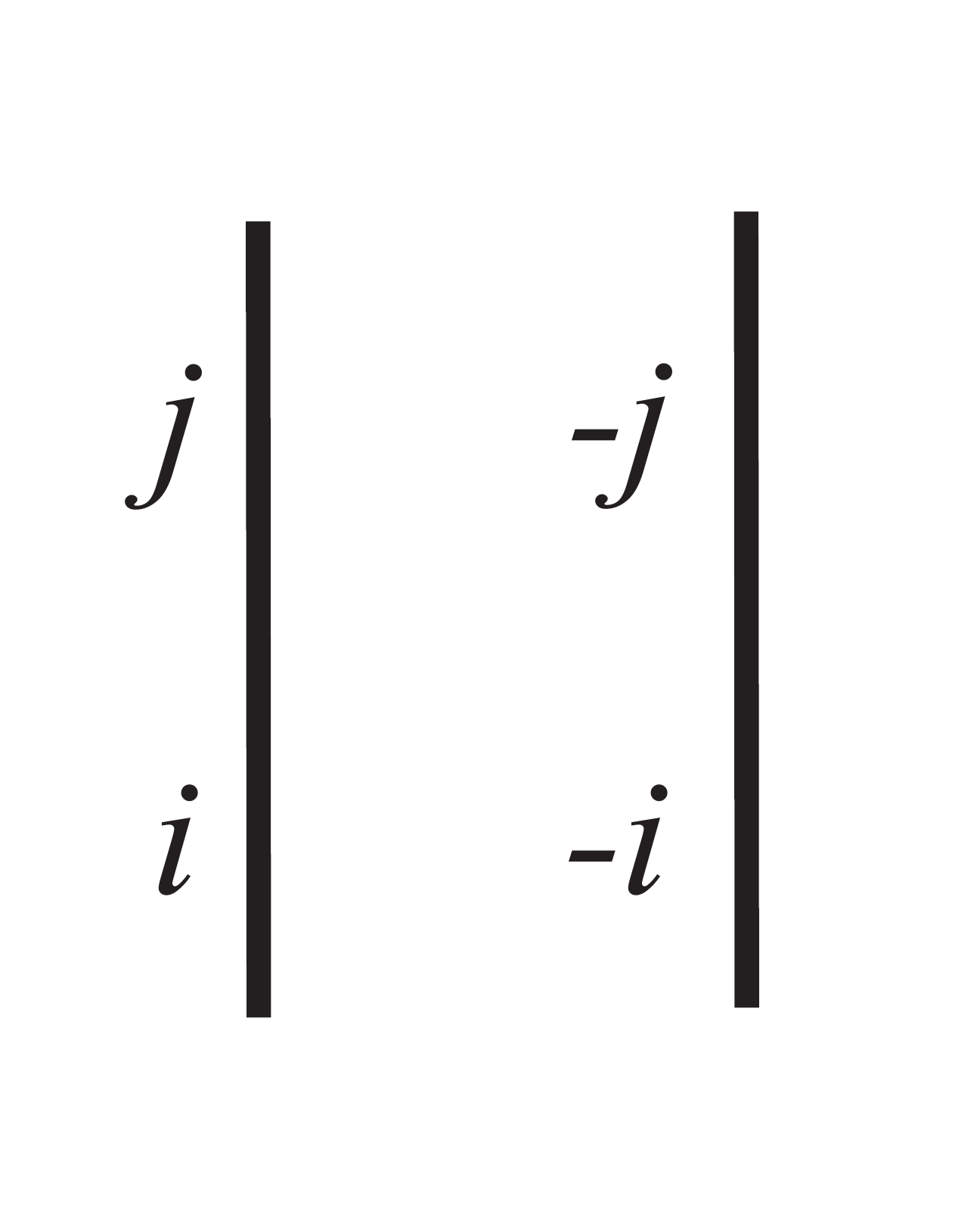}=\graa{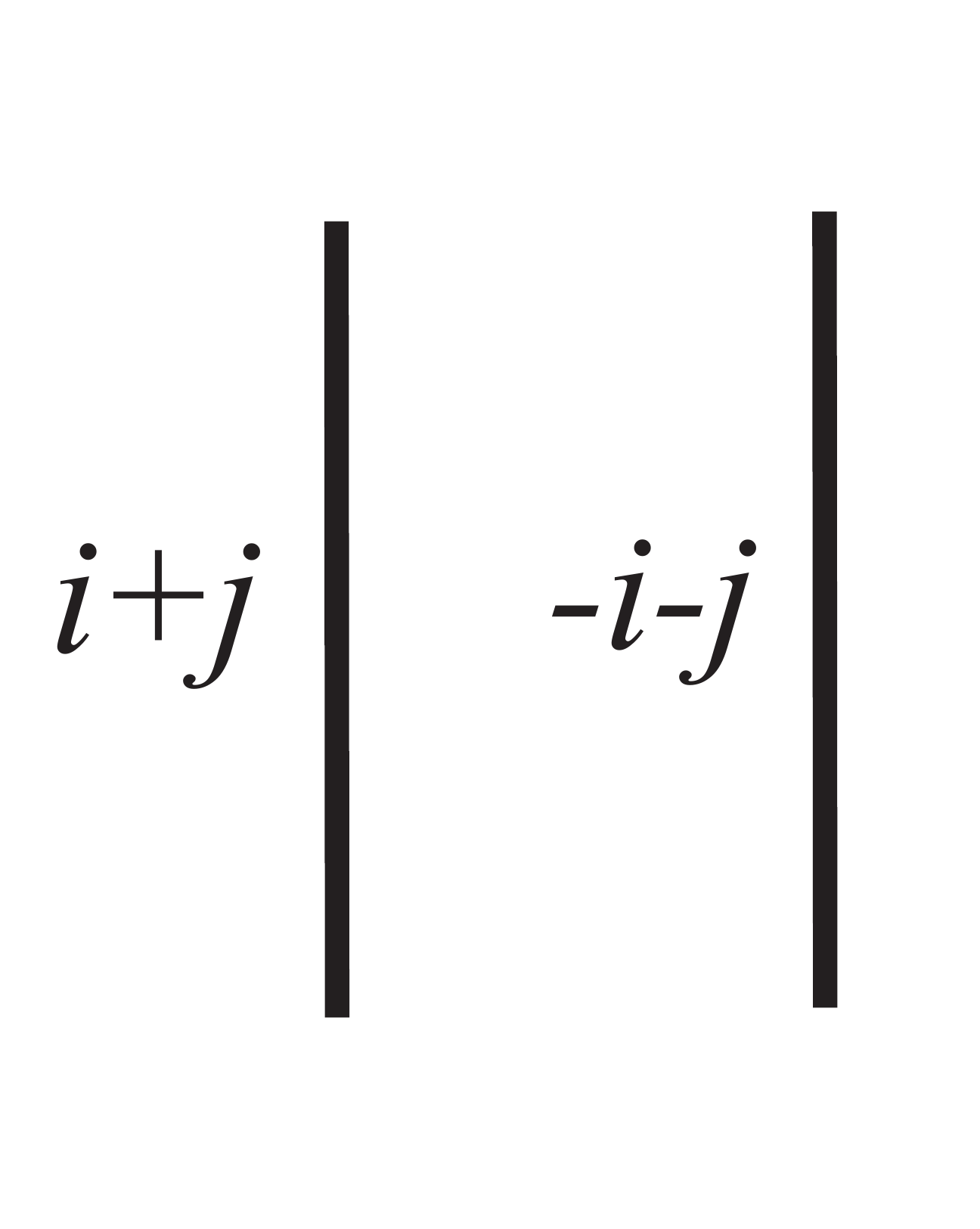}.$$

\begin{proposition}\label{Prop:SFT=FT}
The restriction of the string Fourier transform on the zero graded part of 2-box space is the discrete Fourier transform on the group $\mathbb{Z}_N$:
\begin{equation}\label{Fouriertransform}
  \mathfrak{F}_{\text{s}}\left(\graa{i=-i}\right)=\frac{1}{\sqrt{N}}\sum_{j=0}^{N-1} q^{ij} \graa{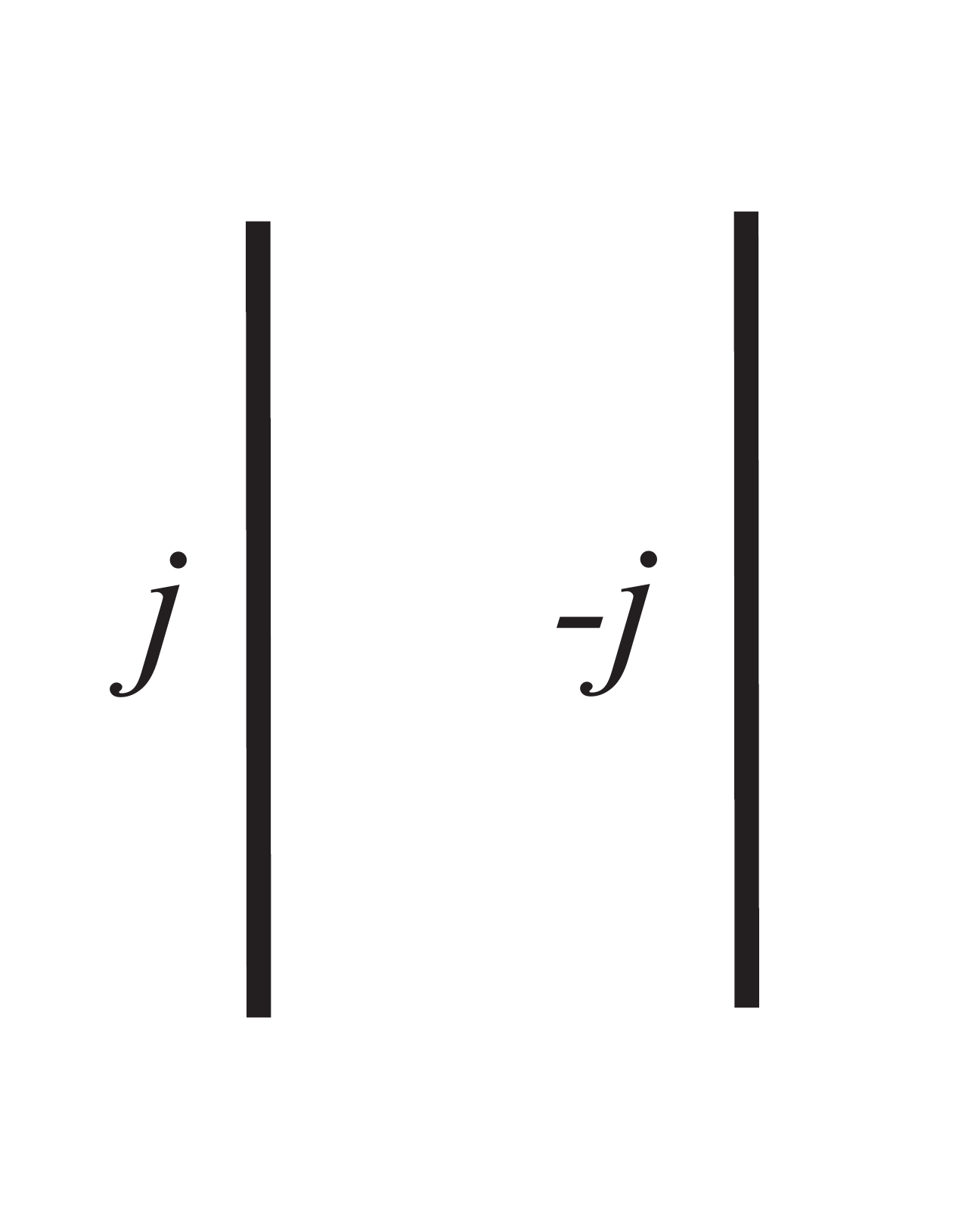}.
\end{equation}
\end{proposition}

\begin{proof} Diagrammatically,
\begin{align*}
\mathfrak{F}_{\text{s}}\left(\graa{i=-i}\right)
&=\graa{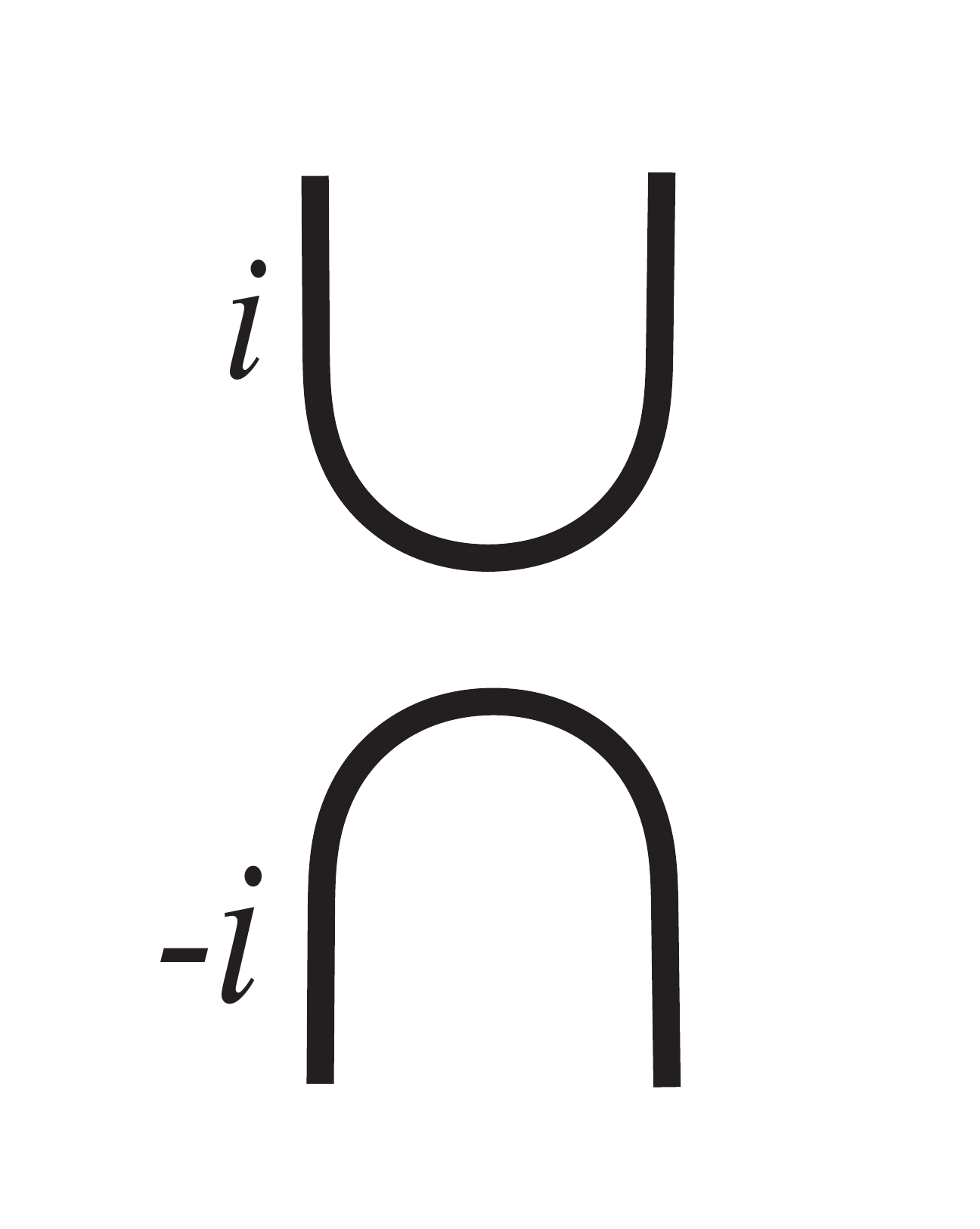}\;, && \text{by Proposition \ref{rotation isotopy},}\\
&=\frac{1}{\sqrt{N}}\sum_{j=0}^{N-1} \graa{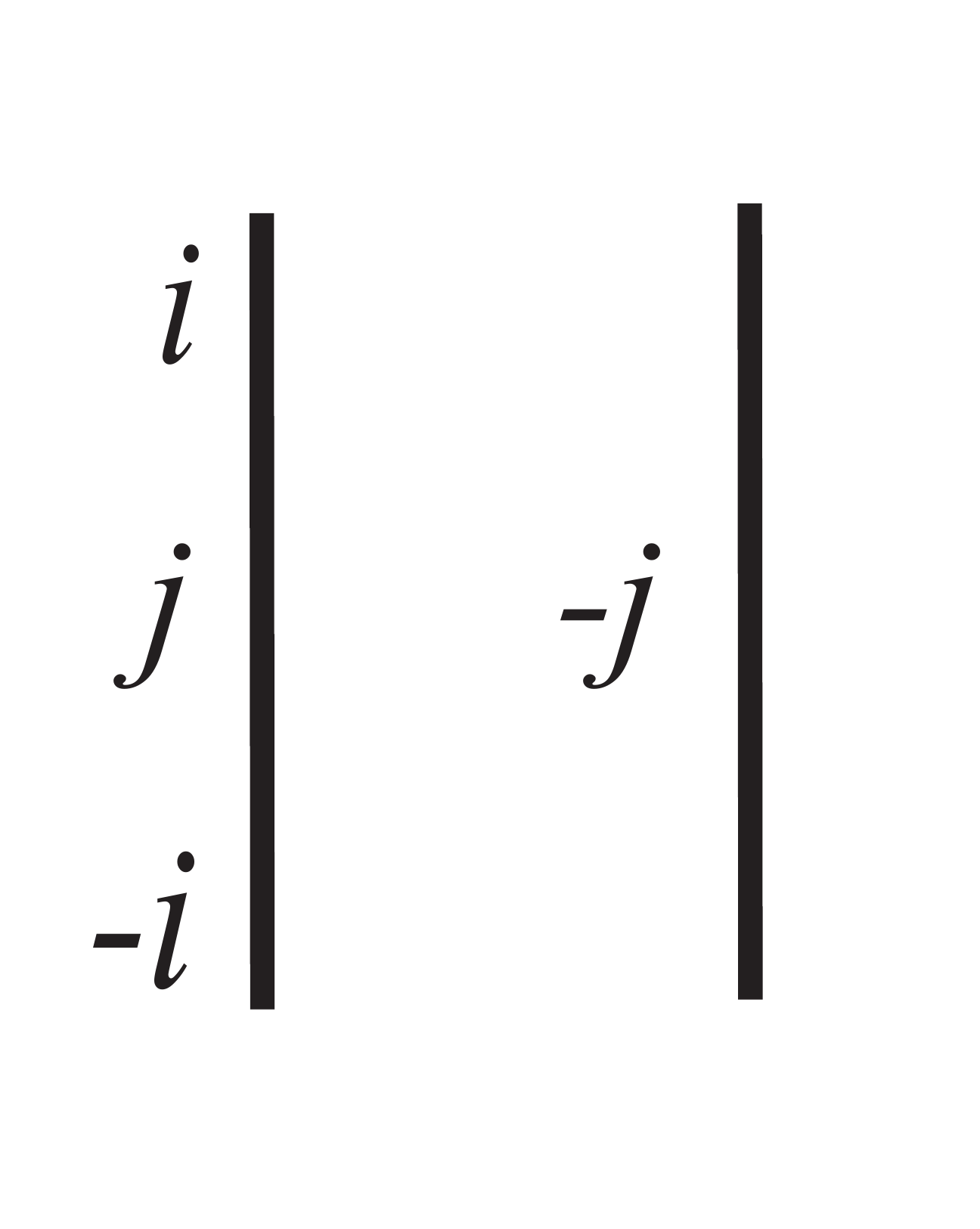} \;, && \text{by Equation \ref{e=sumgi},}\\
&=\frac{1}{\sqrt{N}}\sum_{j=0}^{N-1} q^{ij} \graa{j-j} \;, && \text{by para isotopy}.
\end{align*}
\end{proof}
Note that the 2-box space forms an $N$ by $N$ matrix algebra. Thus we can extend the Fourier transform on the group $\mathbb{Z}_N$ to the string Fourier transform on $N\times N$ matrices.

\subsection{Matrix units}
With the help of the pictures, we construct matrix units of parafermion algebras. The matrix units of $PF_{2m}$ are given by
$$N^{-\frac{m}{2}} \grc{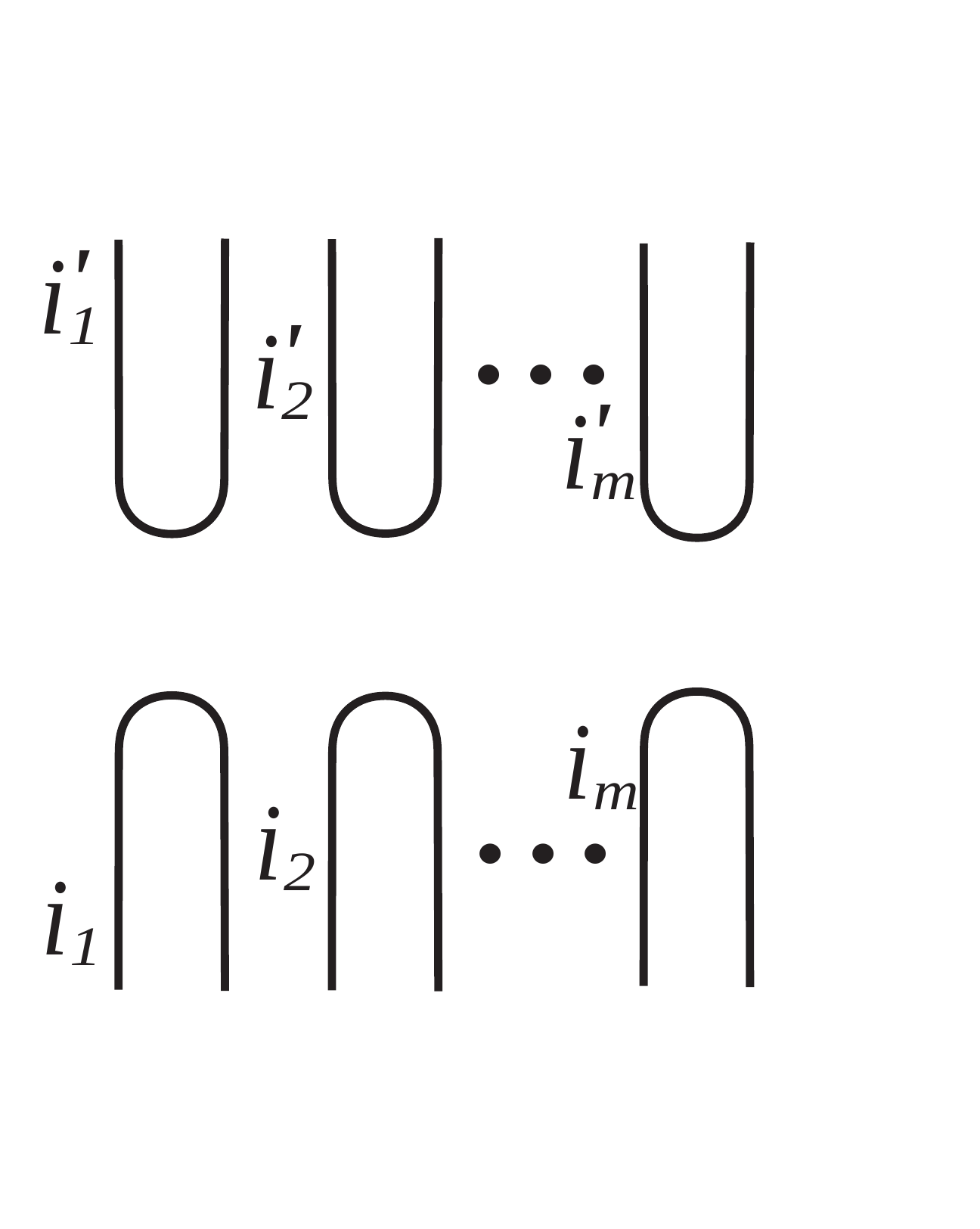},$$
for $0\leq i_1, i'_1, i_2, i'_2,  \cdots, i_m, i'_m \leq N-1$.
Note that $PF_1$ is the group algebra for $\mathbb{Z}_N$. The $N$ minimal projections of $PF_1$ are given by $\displaystyle Q_i=\frac{1}{N}\sum_{j=0}^{N-1} q^{ij} c_1^j$, for $0\leq i\leq N-1$.

The matrix units of $PF_{2m+1}$ are given by
$$N^{-\frac{m}{2}} \grc{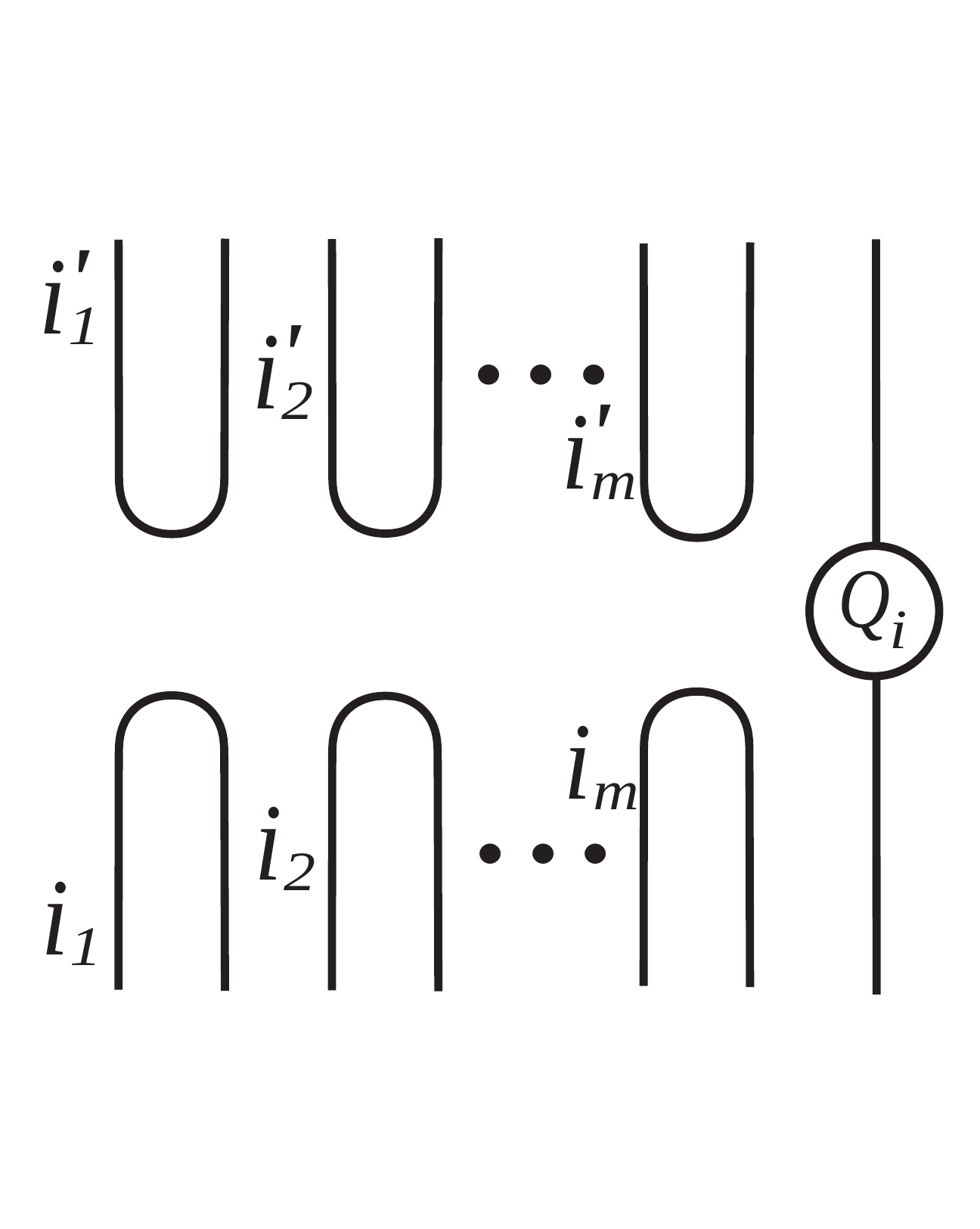}\;,$$
for $0\leq i, i_1, i'_1, i_2, i'_2,  \cdots, i_m, i'_m \leq N-1$.
If we apply the relation \ref{e=sumgi}, then the matrix units can be expressed in terms of the usual basis $\graa{12m}$ of the parafermion algebra $PF_m$.

\setcounter{equation}{0}
\section{Reflection positivity\label{sect:RP}}
In this section, we will apply the string Fourier transform on subfactor planar para algebras to prove the reflection positivity.
\subsection{General case}
Suppose $\mathscr{S}$ is a $(\mathbb{Z}_N,\chi)$ subfactor planar para *-algebra, where $\chi(i,j)=q^{ij}$, $q=e^{\frac{2\pi i}{N}}$.
Recall that $\zeta$ is a square root of $q$ and $\zeta^{N^2}=1$. Then $\zeta^{|x|_{}^2}$ is well-defined for any homogenous $x$.
By Proposition \ref{reflections}, the map $\Theta(x)= \zeta^{-|x|_{}^2} \rho_{\pi}(x^*)$ extends anti-linearly to a horizontal reflection on a subfactor planar para algebra.

In Proposition \ref{rotation isotopy}, we proved that $$\mathfrak{F}_{\text{s}}^{-m}\left(\raisebox{-.6cm}{
\tikz{
\draw (0-1/6,0) rectangle (1/2+1/6,1/2);
\node at (1/4,1/4) {\size{$\Theta(x)$}};
\node at (1/4,3/4) {\size{$\cdots$}};
\node at (1/4,-1/4) {\size{$\cdots$}};
\draw (0,0)--(0,-2/6);
\draw (1/2,0)--(1/2,-2/6);
\draw (0,1/2)--(0,5/6);
\draw (1/2,1/2)--(1/2,5/6);
\draw (1-1/6,0) rectangle (1+1/2+1/6,1/2);
\node at (1+1/4,1/4) {\size{$x$}};
\node at (1+1/4,3/4) {\size{$\cdots$}};
\node at (1+1/4,-1/4) {\size{$\cdots$}};
\draw (1+0,0)--(1+0,-2/6);
\draw (1+1/2,0)--(1+1/2,-2/6);
\draw (1+0,1/2)--(1+0,5/6);
\draw (1+1/2,1/2)--(1+1/2,5/6);
}}\right)=\grb{xex},$$
for any homogenous $m$-box $x$.
Note that the lower half of $\grb{xex}$ is the adjoint of the upper half. Thus
	$$\grb{xex}\geq 0$$
as an operator in the $C^*$ algebra $\mathscr{S}_{2m,\pm}$.
Reflection positivity is related to the $C^*$ positivity by the string Fourier transform $\mathfrak{F}_{\text{s}}^{-m}$, the anti-clockwise $\displaystyle \frac{\pi}{2}$ rotation.

\begin{theorem}[\bf Reflection Positivity: General Case]\label{RPgeneral}
Consider a subfactor planar para algebra $\mathscr{S}$, and a Hamiltonian $H\in \mathscr{S}_{2m,\pm,0}$.  Let  $\mathfrak{F}_{\text{s}}^{-m}(-H)$ be a positive operator in $\mathscr{S}_{2m,\pm}$. Then $H$ has reflection positivity on $\mathscr{S}_{m,\pm}$, for all $\beta\geq0$.  That is
	$$tr(e^{-\beta H} (\Theta(x)\otimes_{t} x))\geq0,$$
for any homogenous $x\in \mathscr{S}_{m,\pm}$.
\end{theorem}

\begin{proof}
If $\mathfrak{F}_{\text{s}}^{-m}(-H)$ is positive, then we take its square root $T=T^*=(\mathfrak{F}_{\text{s}}^{-m}(-H))^{\frac{1}{2}}$.
For any homogenous $x\in \mathscr{S}_{m,\pm}$, $\Theta(x)\otimes_t x$ is zero graded. Applying anti-clockwise $\displaystyle \frac{\pi}{2}$ rotation, we have
\begin{equation}\label{RPequation}
\grd{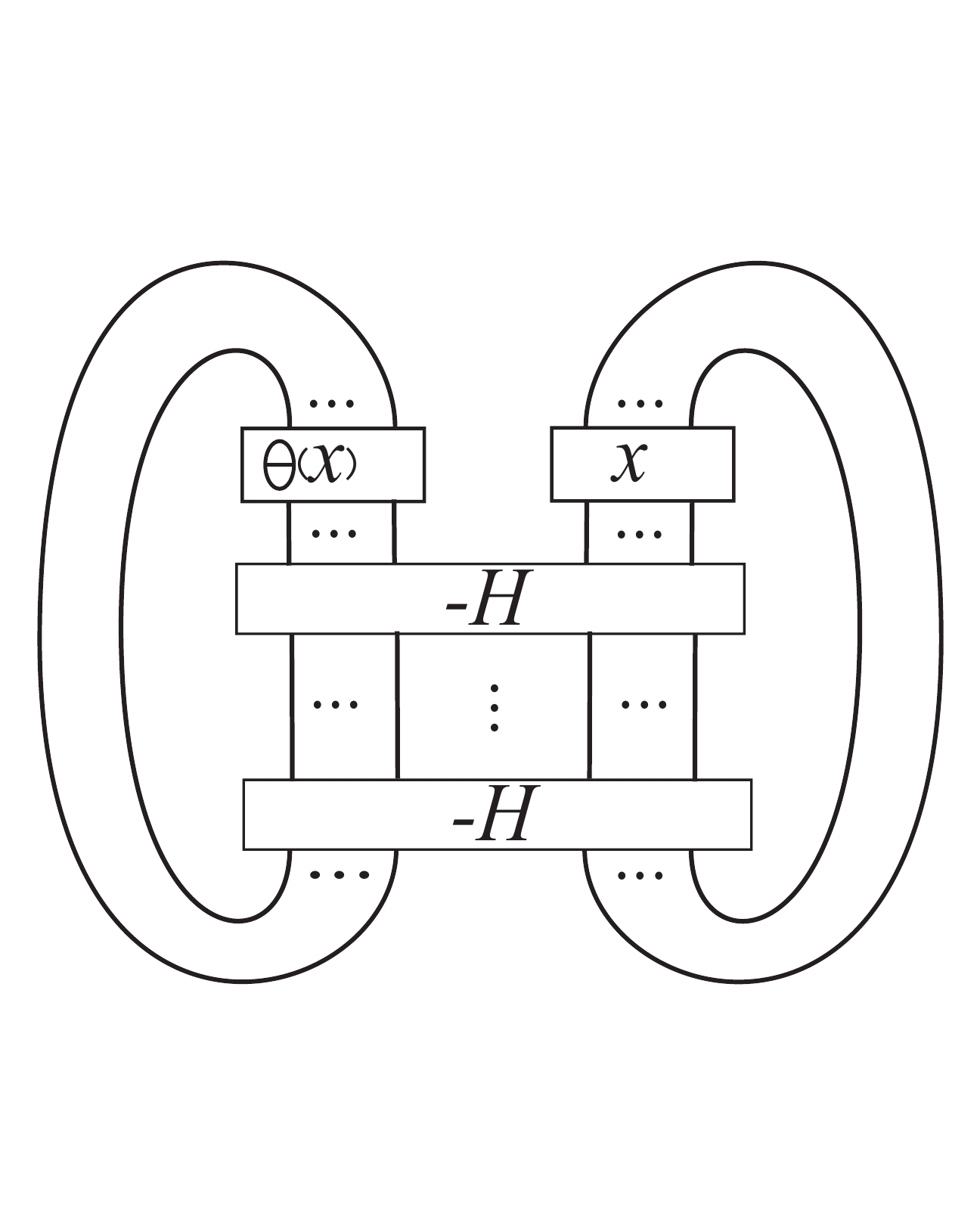}=\grd{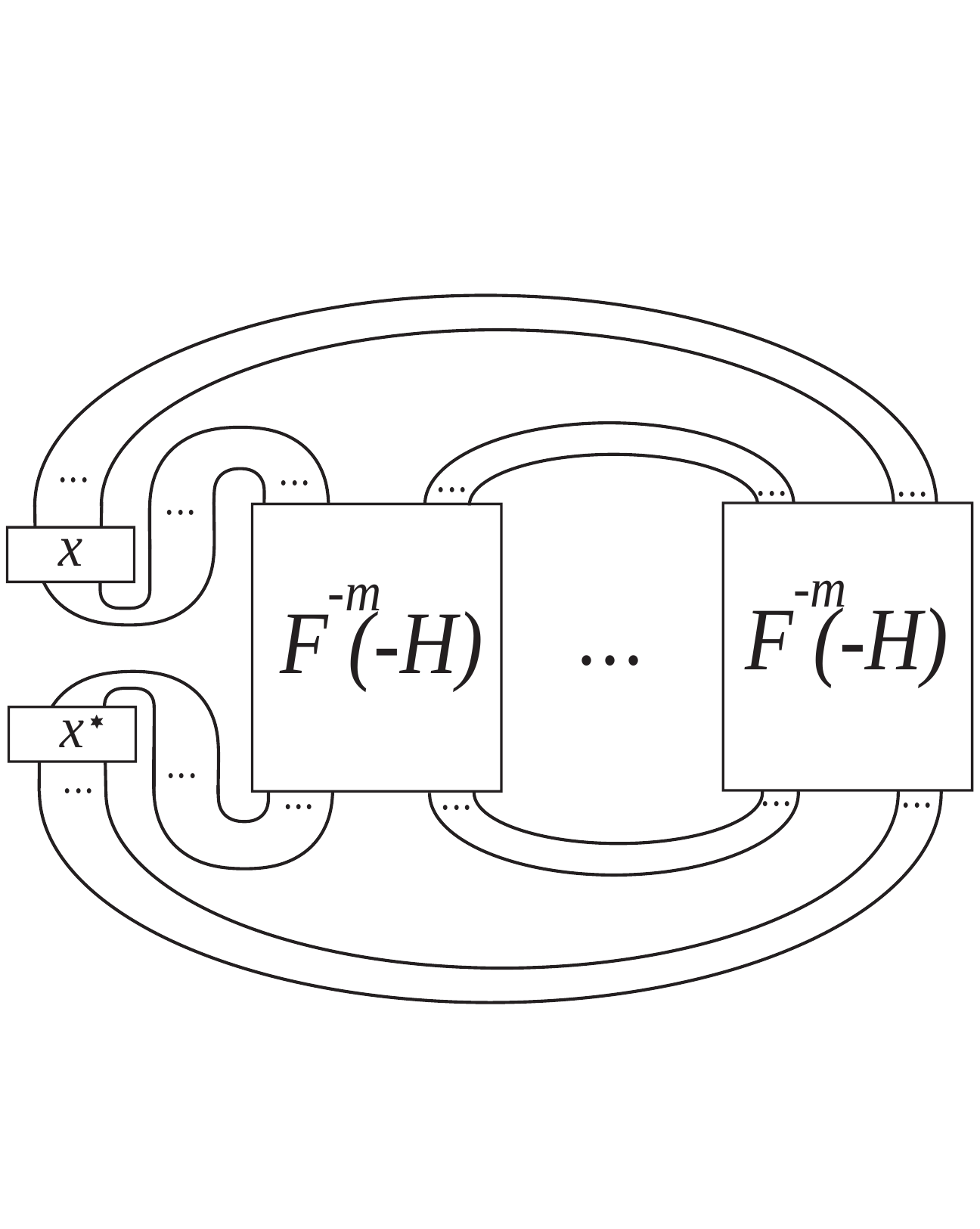}=\grd{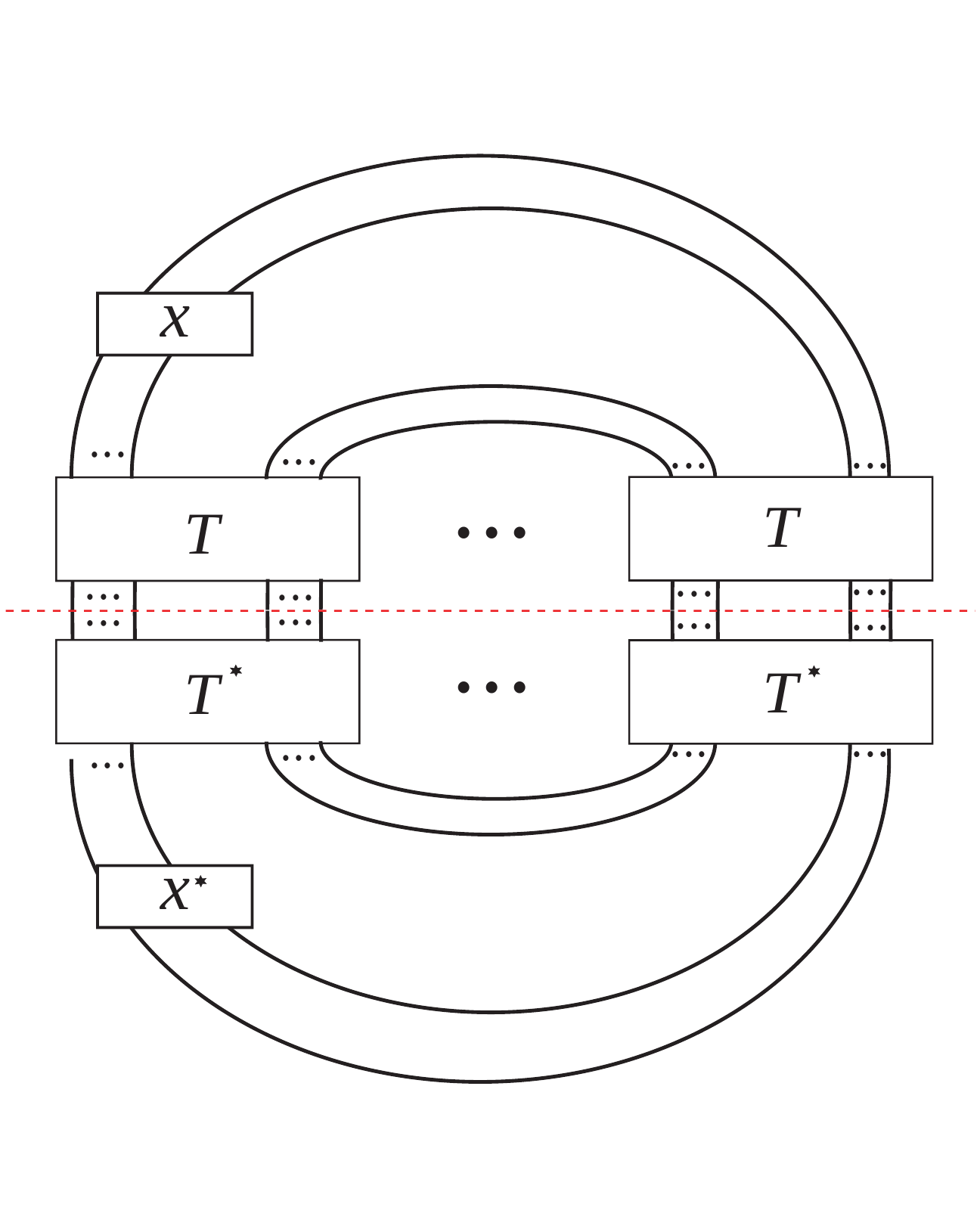}\geq0.
\end{equation}
The last inequality holds, since the lower half is adjoint of the upper half.
Algebraically, for any $k\geq 0$, $$\delta^{2m} tr( (-H)^k \Theta(x)\otimes_t x)\geq0.$$
For any $\beta>0$, we have that
\begin{align*}
tr(e^{-\beta H} (\Theta(x)\otimes_{t} x))&=\sum_{k=0}^{\infty} \beta^k tr( (-H)^k \Theta(x)\otimes_t x)\geq0\\
\end{align*}
and $H$ has reflection positivity.
\end{proof}

\begin{remark}
The string Fourier transform as the anti-clockwise $\frac{\pi}{2}$ rotation changes the trace to vacuum state, the multiplication to the convolution.
The positivity of the convolution positive operators is known as the Schur product theorem, proved in \cite{Liuex} for subfactor planar algebras.  For the parafermion algebra case, the Schur product of $\mathfrak{F}_{\text{s}}^{-m}(-H)$ corresponds to the Hadamard product of the coupling constant matrix of $H$.
\end{remark}

\subsection{Quantized vectors}
The homogenous condition for $x$ in Theorem \ref{RPgeneral} is not necessary. Recall that we can lift the grading and the twisted tensor product in general in Definition \ref{Def:twist tensor general}.

Suppose $\displaystyle x=\sum_{i=0}^{N-1}  x_i$ and $\displaystyle y=\sum_{i=0}^{N-1} y_i$, and $x_i$, $y_i$ are graded by $i$.
Let $i\to i'$ be a lift of the grading to $\Z$ and we define $\displaystyle \hat{x}=\sum_{i=0}^{N-1}  (x_i,i')$ and $\displaystyle \hat{y}=\sum_{i=0}^{N-1}  (y_i,i')$.

For a Hamiltonian $H\in\mathscr{S}_{2m,\pm}$ , we define the inner product
$$<x,y>_{\Theta}=tr(e^{-\beta H} \Pi(\Theta(\hat{x})\otimes_{t} \hat{y})),$$
for $x, y\in \mathscr{S}_{m,\pm}$.
Then $$<x,y>_{\Theta}=\sum_{i=0}^{N-1} tr(e^{-\beta H} \zeta^{i^2} \Theta(x_i) \otimes_+ y_i),$$
which is independent of the choice of the lift.

If $H$ has reflection positivity, then $\mathscr{S}_{m,\pm}$ forms a Hilbert space with respect to the inner product $<\cdot,\cdot>_{\Theta}$, called the quantized space.
The image of $x$ in the quantized space is denoted by $\hat{x}$. We give a presentation of the quantized vector $\hat{x}$ in the subfactor planar para algebra $\mathscr{S}$.

\begin{theorem}\label{quantized vector}
Suppose $\mathfrak{F}_{\text{s}}^{-m}(-H)$ is positive, and $T$ is its square root. We construct the quantized vector
$$ \hat{x}:=\oplus_{k=0}^{\infty} \frac{\beta^{\frac{k}{2}}}{\delta^{k}}\grb{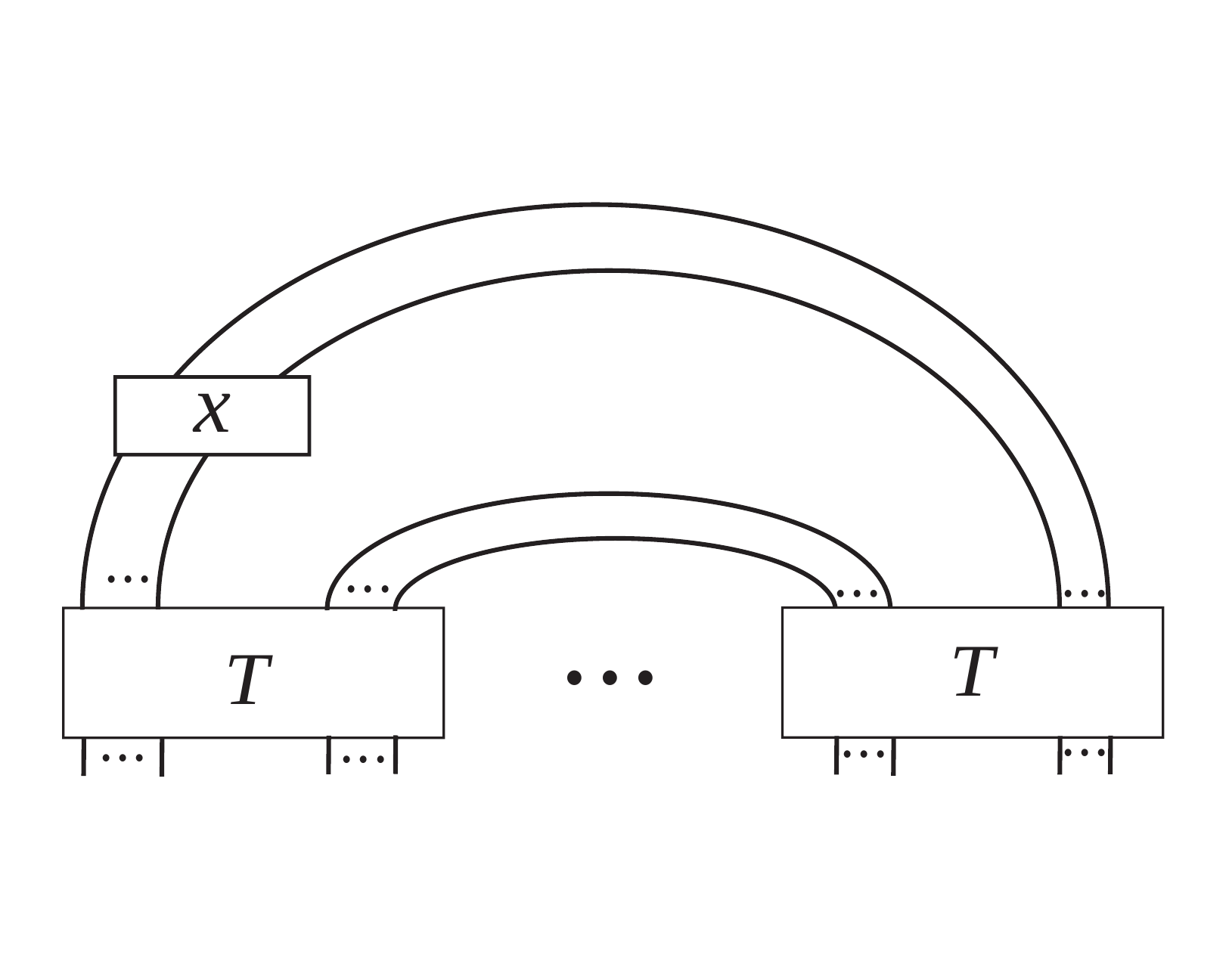}.$$
(There are $k$ copies of $T$ in the diagram.) Then
$$\lra{x,x}_{\Theta}=\hat{x}^*\hat{x}\geq0.$$
\end{theorem}

\begin{proof}
Suppose $\displaystyle x=\sum_{i=0}^{N-1}  x_i$ and $x_i$ is graded by $i$. Then $\lra{x_{i},x_{i}}_{\Theta}=\hat{x_i}^*\hat{x_i}$ by Equation \ref{RPequation}.
Since $\lra{x_i,x_j}_{\Theta}$ and $\hat{x_i}^*\hat{x_j}$ are graded by $j-i$, we infer that they are zero if $i\neq j$. Therefore
$$<x,x>_{\Theta}=\sum_{i=0}^{N-1} <x_i,x_i>=\sum_{i=0}^{N-1} \hat{x_i}^*\hat{x_i}=\hat{x}^*\hat{x}\geq0.$$
\end{proof}
\goodbreak

\subsection{Parafermion algebras}
Recall that the basis of $PF_{m}$ is given by $c_1^{i_1}c_2^{i_2}\cdots c_m^{i_m}$,
%$c^{i_1}\otimes_+c^{i_2}\cdots\otimes_+c^{i_m}$,
$0\leq i_1, i_2 \cdots i_m\leq N-1$.

Let $A_+$ be the sub algebra of $PF_{2m}$ that consists of $I_m \otimes x$, for $x\in PF_{m}$.  Let $A_-$ be the sub algebra of $PF_{2m}$ that consists of $y\otimes I_m$, for  $y\in PF_{m}$. Then the graded tensor product $A=A_- \hat{\otimes} A_+$ is $PF_{2m}$.

Note that $\Theta(c)=\zeta\mathfrak{F}_{\text{s}}^{-1}(c^*)=c^{-1}$.
The reflection $\Theta$ from $A_\pm\cong PF_{m}$ to $A_\mp\cong PF_{m}$ is the anti-linear extension of $\Theta(c^{i_1}\otimes_+c^{i_2}\cdots\otimes_+c^{i_m})= c^{-i_m} \otimes_- \cdots \otimes_- c^{-i_2} \otimes_- c^{-i_1}$.
Therefore $A=\theta(A_+) \hat{\otimes} A_+$. We call the graded tensor product $A$ the double algebra of $A_+$.

Take the Hamiltonian $H$ in $PF_{m}$. In terms of the basis $C_I$, we have
$$-H= \sum_{I,I'} J_{I}^{I'} \Theta(C_I) \otimes_{t} C_{I'}$$
for some coupling constants $J_{I}^{I'}$.  The Hamiltonian $H$ is called reflection invariant, if $\Theta(H)=H$, or equivalently $J_{I}^{I'}=\overline{J_{I'}^{I}}$ for all $I$, $I'$, or equivalently $J$ is a Hermitian matrix.

Let $J_0$ be the sub matrix of $J$, whose coordinates $I$ and $I'$ are both non-empty, i.e., the matrix of coupling constants crossing the reflection plane.   The following theorem is formulated and proved in \cite{JafJan2016} by a different method. Here we give a diagrammatic interpretation that gives special insight and understanding.

\begin{theorem}[\bf Reflection Positivity for Parafermions]
\label{RPparafermion}
Suppose the Hamiltonian $H$ is reflection invariant and $|H|_{+}=0$. Then $H$ has reflection positivity, i.e.,
$$tr(e^{-\beta H} (\Theta(x)\otimes_{t} x))\geq0,$$
for any $x\in PF_m$, for all $\beta\geq0$, if and only if $J_0\geq0$.
\end{theorem}

\begin{proof}
Take
$$v_{I}^{I'}=N^{-\frac{m}{2}}\grb{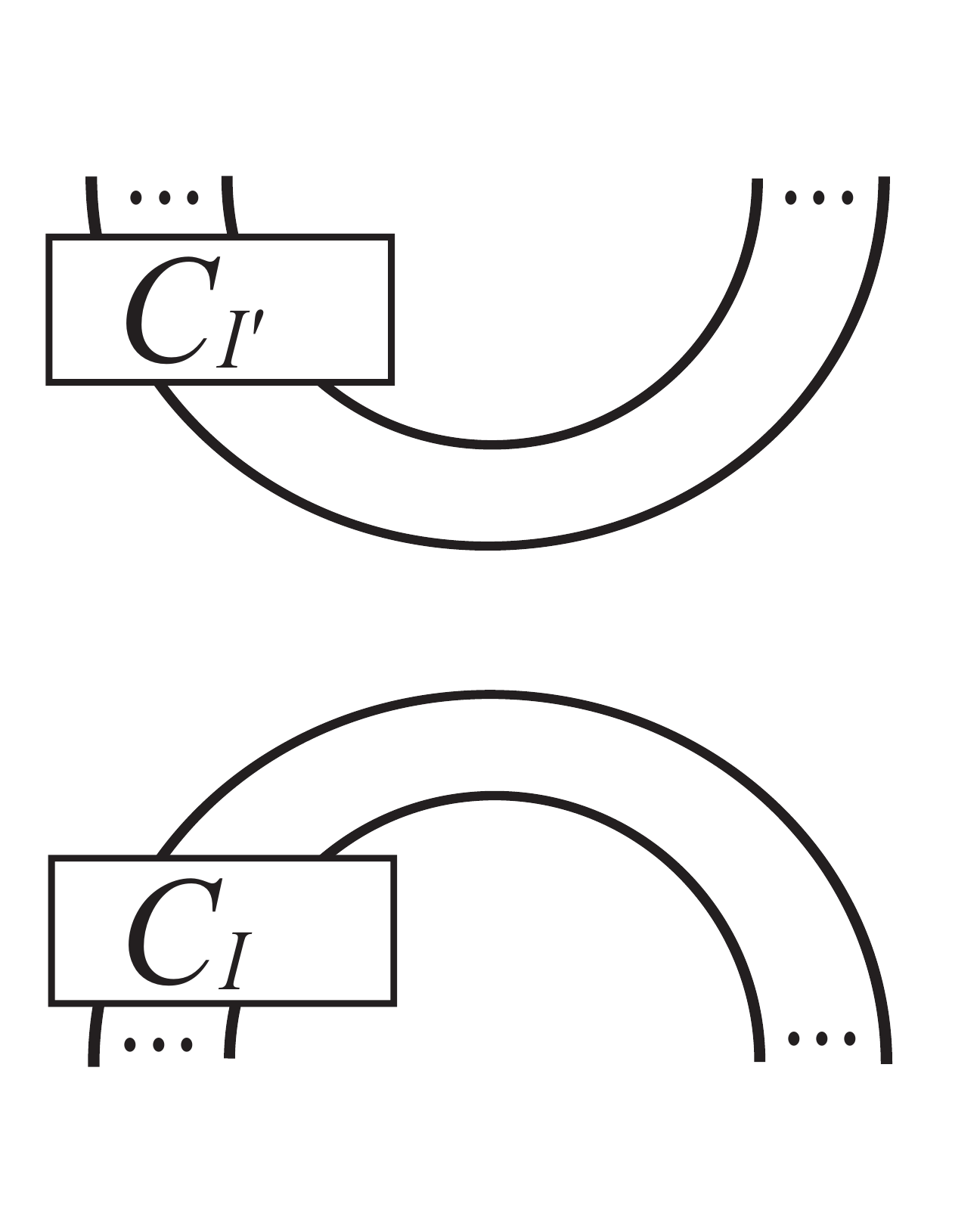}.$$
Then $v_{I}^{I'}$ are matrix units acting on the Hilbert space $V=\{\grb{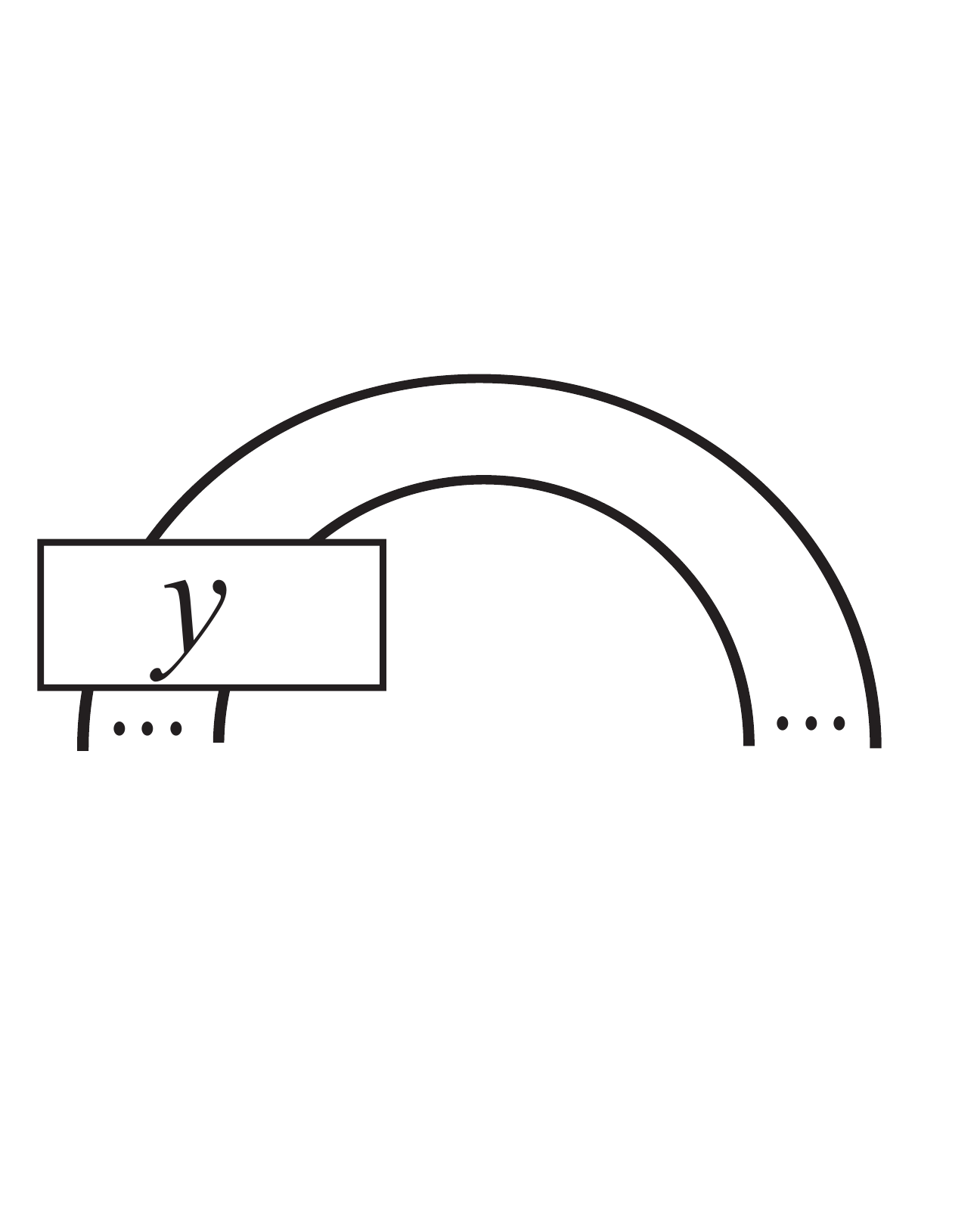}|y\in PF_m\}$.
By Proposition \ref{rotation isotopy},
\be\label{fm}
\mathfrak{F}_{\text{s}}^{-m}(-H)= N^{\frac{m}{2}} \sum_{I,I'} J_{I}^{I'} v_{I}^{I'}.
\ee

\iffalse
Take $\cap_m$ to be $\graa{capm}$ and $\cup_m$ to be $\graa{cupm}$.

Rotate the left half side of $H$ anticlockwise to the right bottom, by Proposition \ref{rotation isotopy}, we obtain the operator
$\mathfrak{F}_{\text{s}}^{-m}(-H)=\sum_{I,I'} J_{I}^{I'} (C_I^* \otimes 1^{\otimes m}) \cap_m \cup_m (C_{I'} \otimes 1^{\otimes m}).$

Case 1: When $J\geq0$, we have $\mathfrak{F}_{\text{s}}^{-m}(-H)\geq0$.

Let $*_m$ be the $m$-string coproduct on $PF_{2m}$, and $(\mathfrak{F}_{\text{s}}^{-m}(-H))^{*k}$ be the $k$th power under the $m$-string coproduct. Then $(\mathfrak{F}_{\text{s}}^{-m}(-H))^{*k}\geq0$ by Schur Product theorem proved in \cite{Liuex}.

When $x$ is homogenous, $\Theta(x)\otimes_{t}x$ is zero graded. Applying the $\displaystyle \frac{\pi}{2}$ rotation, we have
$$\grd{ref1}=\grd{ref3}=\grd{ref2}\geq0.$$

Algebraically, $$\delta^{2m} tr( (-\beta H)^k \Theta(x)\otimes_t x)\geq 0.$$

\begin{align*}
&\delta^{2m} tr( (-\beta H)^k \Theta(x)\otimes_t x)\\
=& \beta^k \cup_m (x^* \otimes 1^{\otimes m}) (\mathfrak{F}_{\text{s}}^{-m}(-H))^{*k} (x \otimes 1^{\otimes m}) \cap_m\\
\geq&0.
\end{align*}

When $x=\sum_g x_g$, we have
\begin{align*}
&\delta^{2m} tr( (-\beta H)^k \Theta(x)\otimes_t x)\\
=&\sum_g \delta^{2m} tr( (-\beta H)^k (\Theta(x_g)\otimes_t x_g))\\
\geq&0
\end{align*}
Sum over $k$, we have
$$tr(e^{-\beta H} (\Theta(x)\otimes_{t} x))\geq0.$$
\fi
Note that $e^{-\beta (H+rI_{2m})}=e^{-\beta r}e^{-\beta H}$, so the scalar $r$ will not affect the reflection positivity condition of $H$.
Without loss of generality, we assume that $J_{\emptyset}^{\emptyset}=0$.

When $J_0\geq 0$, for any $s>0$, take
\begin{align*}
-H(s)&=-H+s\sum_{I,I'\neq \emptyset} J_{I}^{\emptyset} J_{\emptyset}^{I'}  \Theta(C_I) \otimes_{t} C_{I'})+s^{-1} I_{2m}\;.
%&=\sum_{I,I'\neq \emptyset} (J_{I}^{I'} \Theta(C_I) \otimes_{t} C_{I'} +s J_{I}^{\emptyset} J_{\emptyset}^{I'}  \Theta(C_I) \otimes_{t} C_{I'})\\
%&+ \sum_{I} (J_{I}^{\emptyset}  \Theta(C_I) \otimes_{t} C_{\emptyset} + \overline{J_{\emptyset}^{I}}  \Theta(C_\emptyset) \otimes_{t} C_{I}) + s^{-1} I_{2m}.
\end{align*}
Since $J$ is Hermitian, we have
\begin{align*}
\mathfrak{F}_{\text{s}}^{-m}(-H(s))&=N^{\frac{m}{2}}\sum_{I\neq\emptyset,I'\neq\emptyset}J_{I}^{I'}v_{I}^{I'}
+N^{\frac{m}{2}}s^{-1}(v_{\emptyset}^{\emptyset}+s\sum_{I\neq\emptyset}J_{I}^{\emptyset}v_{I}^{\emptyset})(v_{\emptyset}^{\emptyset}+s\sum_{I\neq\emptyset}J_{I}^{\emptyset}v_{I}^{\emptyset})^*\\
\geq0
\end{align*}
By Theorem \ref{quantized vector}, $H(s)$ has reflection positivity,
$$tr(e^{-\beta H(s)} (\Theta(x)\otimes_{t} x))\geq0,$$
so does $H(s)-s^{-1}I_{2m}$.
Take $s\rightarrow 0$. This shows that  $H$ has reflection positivity,  $$tr(e^{-\beta H} (\Theta(x)\otimes_{t} x))\geq0\;.$$

On the other hand, if $H$ has reflection positivity for all $\beta\geq0$, then for any homogenous $x$ in $PF_m$ orthogonal to $I_m$, we have
$$tr(e^{-\beta H} (\Theta(x)\otimes_{t} x))\geq0\;,$$
and the equality holds when $\beta=0$.
Take the first derivative with respect to $\beta$. Then we have
\be\label{RPH}
tr(-H (\Theta(x)\otimes_{t} x))\geq0\;.
\ee
Apply the anti-clockwise $\frac{\pi}{2}$ rotation to Equation \ref{RPH}, and use Equation \ref{fm}. This shows that
we have
$$
\sum_{I,I'}J_I^{I\prime}\gre{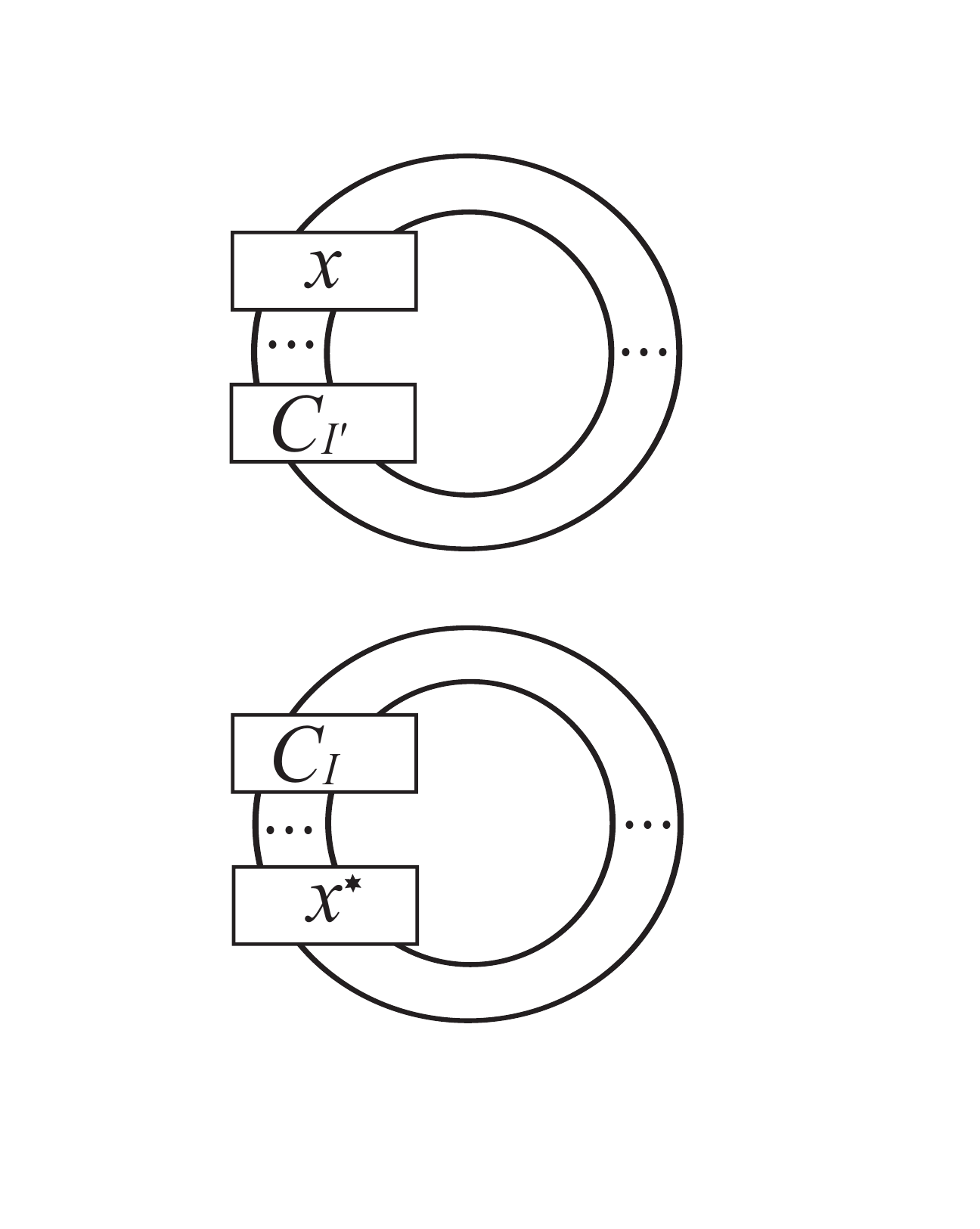}\geq0\;,
$$
for any $m$-box $x$ orthogonal to $I_m$.
Therefore the matrix $J_0$ as the restriction of $J$ on the subspace $V\setminus\mathbb{C}\{\graa{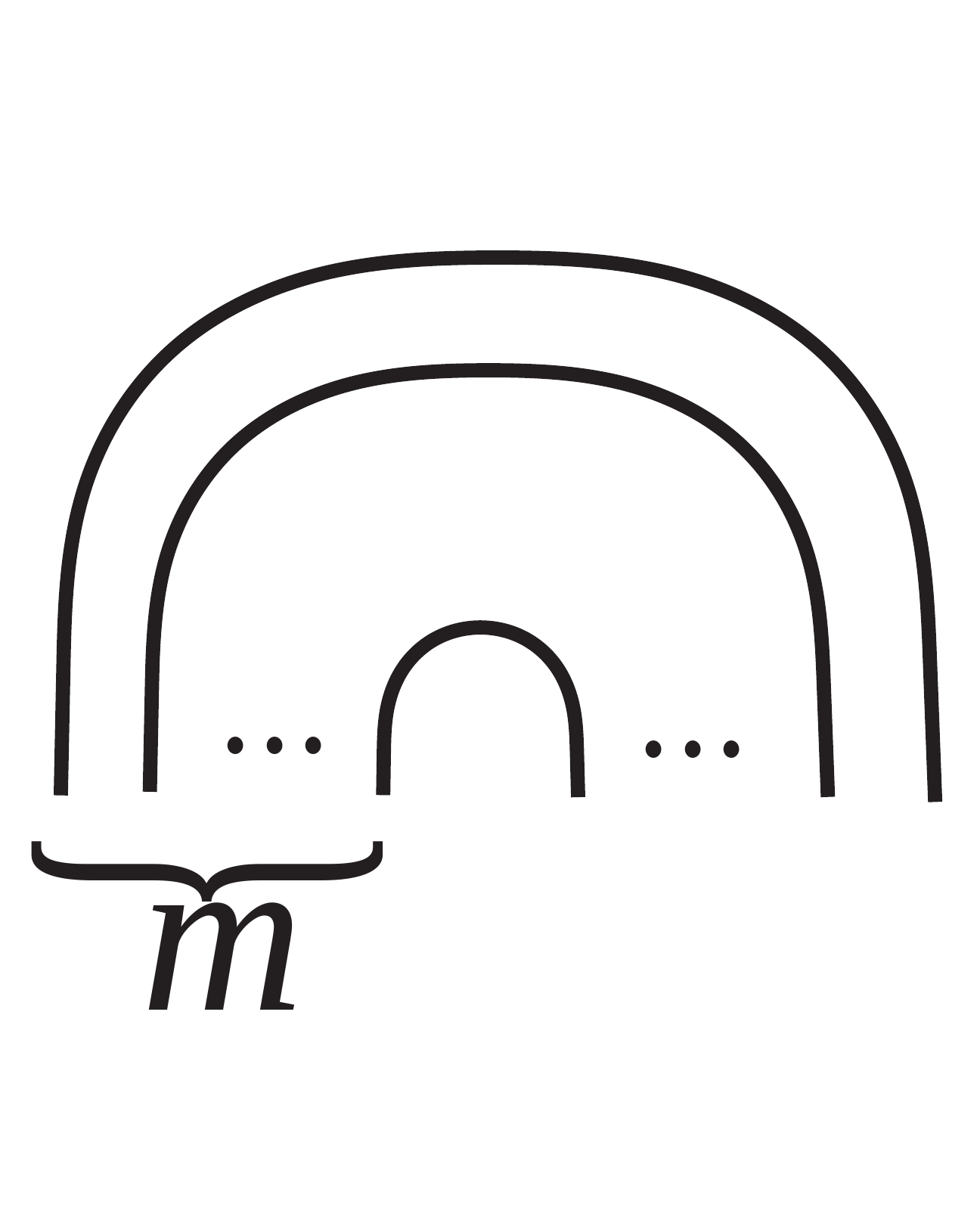}\}$ is positive.
\end{proof}

\setcounter{equation}{0}
\section{Braid relations}\label{braided relations}
In this section, we construct braids for parafermion algebras which behave well in a diagrammatic way, so that the strings can act over the parafermion planar para algebra $PF_{\bullet}$ in the 3-dimensional space.

Take $\omega=\frac{1}{\sqrt{N}}\sum_{i=0}^{N-1}\zeta^{i^2}$, so $|\omega|=1$ by Proposition \ref{gauss sum}. Let $\omega^{\frac{1}{2}}$ be a square root of $\omega$.
Let us construct the braids as
\begin{align}
\graa{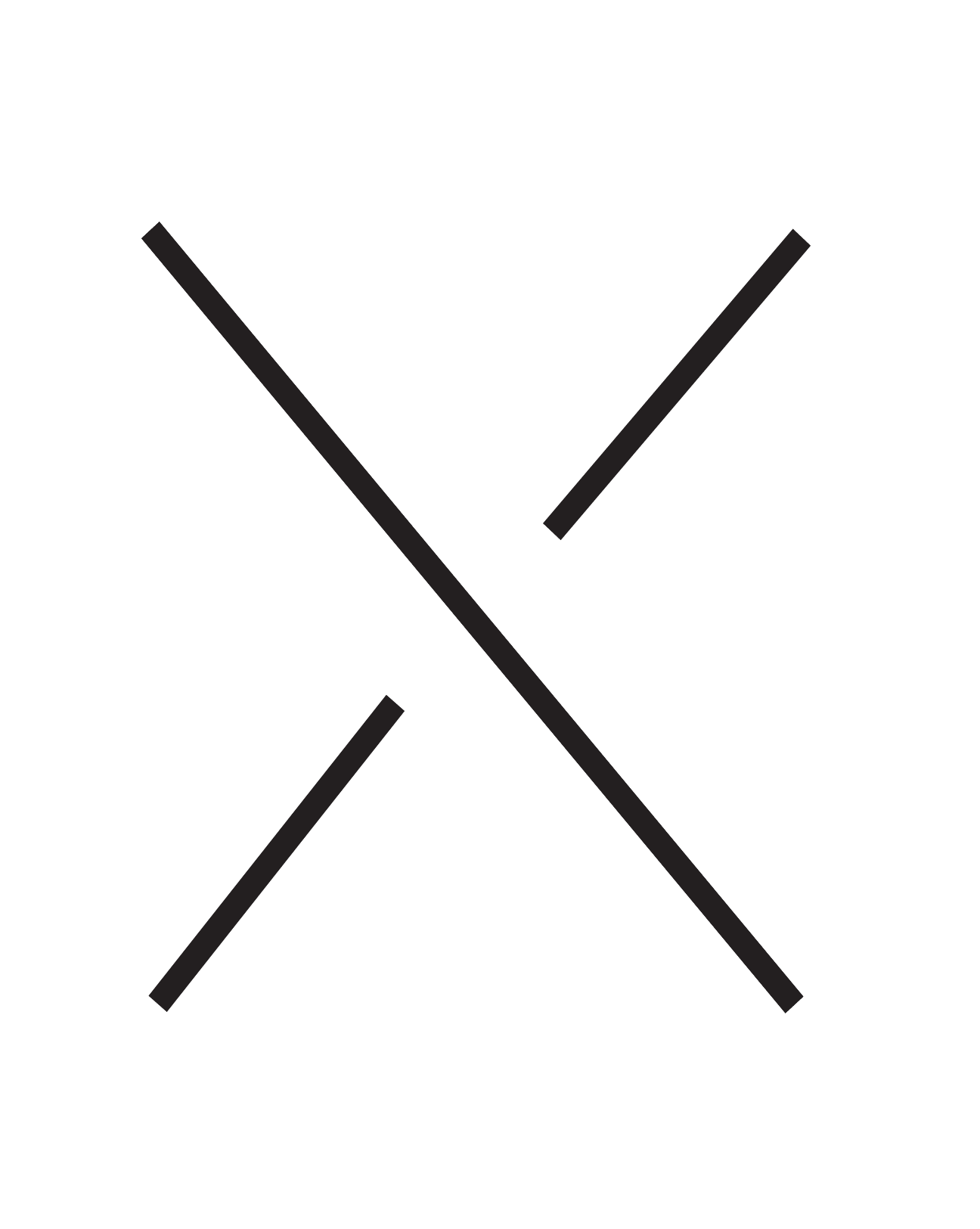}&=\frac{\omega^{\frac{1}{2}}}{\sqrt{N}}\sum_{i=0}^{N-1}\graa{i+-i}\label{b-1}\\
&=\frac{\omega^{\frac{1}{2}}}{\sqrt{N}}\sum_{i=0}^{N-1}\zeta^{-i^2}\graa{i=-i}\;,\label{b-2}\\
\graa{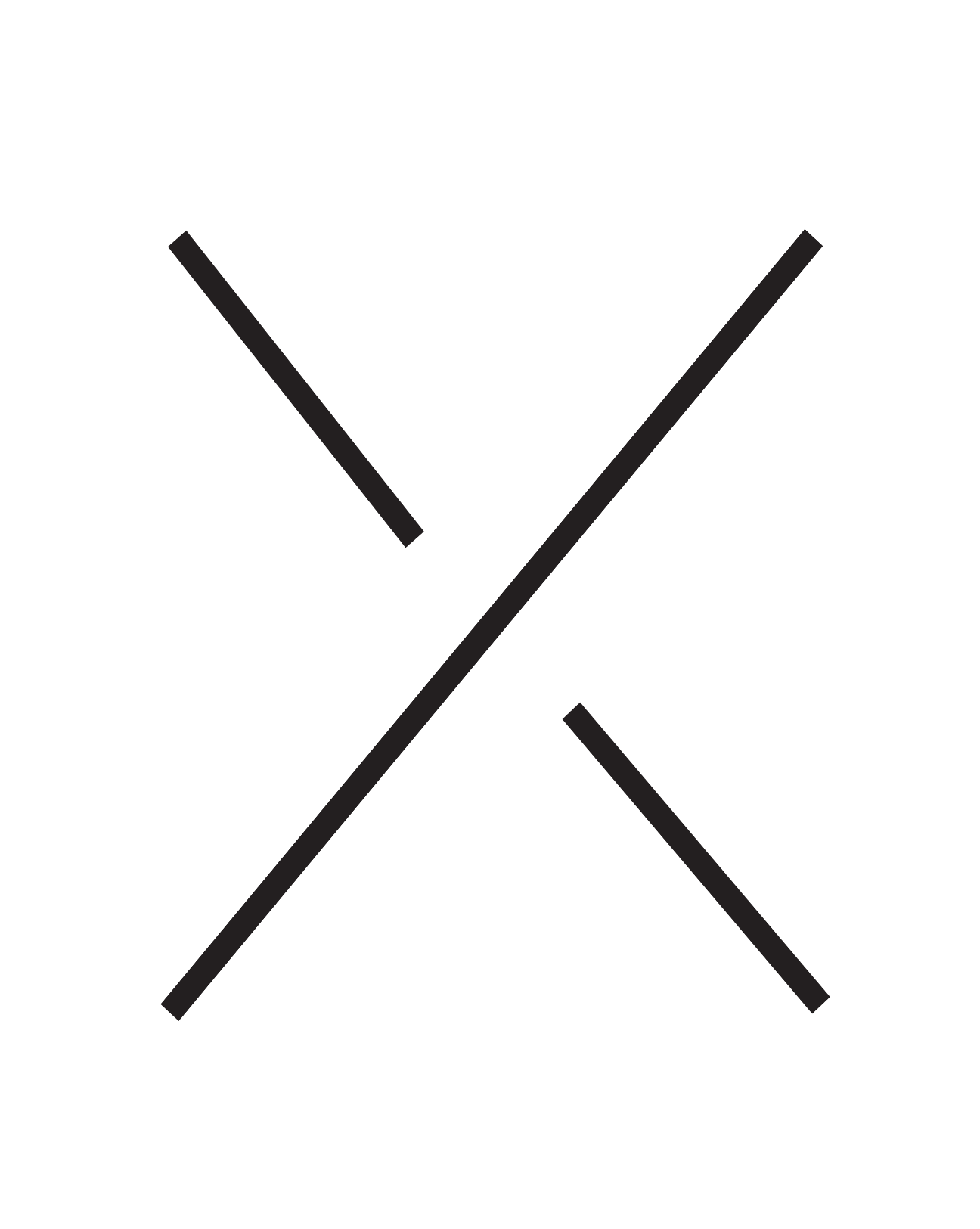}&=\frac{\omega^{-\frac{1}{2}}}{\sqrt{N}}\sum_{i=0}^{N-1}\graa{i--i}\label{b+1}\\
&=\frac{\omega^{\frac{1}{2}}}{\sqrt{N}}\sum_{i=0}^{N-1}\zeta^{i^2}\graa{i=-i}.\label{b+2}
\end{align}
Since $\zeta^{i^2}=\zeta^{(-i)^2}$, the two braids behave well under the vertical reflection $*$ and also under the horizontal reflection $\Theta$:
\be
\left(\graa{b-}\right)^*=\graa{b+}\;,
\qquad
\Theta\left(\graa{b+}\right)=\graa{b-}\;.\\
\ee

\begin{proposition}\label{fourierbraid}
With the above notion of $\omega^{\frac{1}{2}}$, the two braids behave well under $\frac{\pi}{2}$~rotation:
\be\label{fourierbraid1}
  \mathfrak{F}_{\text{s}}\left(\graa{b+}\right)=\graa{b-}\;,
  \qquad
%\ee
%\begin{align}\label{fourierbraid2}
  \mathfrak{F}_{\text{s}}\left(\graa{b-}\right)=\graa{b+}\;.
\ee
\end{proposition}

\begin{proof}
The computation has been done in the proof of Proposition \ref{gauss sum}
\iffalse
\begin{align*}
\mathfrak{F}_{\text{s}}(\graa{b+})
&=\frac{\omega^{-\frac{1}{2}}}{\sqrt{N}}\sum_{i=0}^{N-1} \zeta^{i^2}  \mathfrak{F}_{\text{s}}(\graa{i=-i}) && \text{by Equation} \ref{b+2}\\
&=\frac{\omega^{-\frac{1}{2}}}{N}\sum_{i,j=0}^{N-1} \zeta^{i^2} q^{ij}  \graa{j-j} &&  \text{by Equation} \ref{Fouriertransform}\\
&=\frac{\omega^{-\frac{1}{2}}}{N}\sum_{i,j=0}^{N-1} \zeta^{(i+j)^2} \zeta^{-j^2}  \graa{j-j}\\
&=\frac{\omega^{\frac{1}{2}}}{\sqrt{N}}\sum_{j=0}^{N-1}  \zeta^{-j^2}  \graa{j-j}\\
&=\graa{b-} &&  \text{by Equation} \ref{b-2}
\end{align*}
The other equation can be proved similarly.
\fi
\end{proof}

Recall that $\mathfrak{F}_{\text{s}}\left(\graa{i=-i}\right)=\graa{p-i}$. Thus
\begin{align}
\graa{b-}&=\frac{\omega^{-\frac{1}{2}}}{\sqrt{N}}\sum_{i=0}^{N-1}\zeta^{i^2}\graa{p-i}\;,\label{b-3}\\
\graa{b+}&=\frac{\omega^{\frac{1}{2}}}{\sqrt{N}}\sum_{i=0}^{N-1}\zeta^{-i^2}\graa{p-i}\;.\label{b+3}
\end{align}
Therefore the braids are unitary and we have the Reidemeister move of type II:
\begin{align}\label{moveII}
\graa{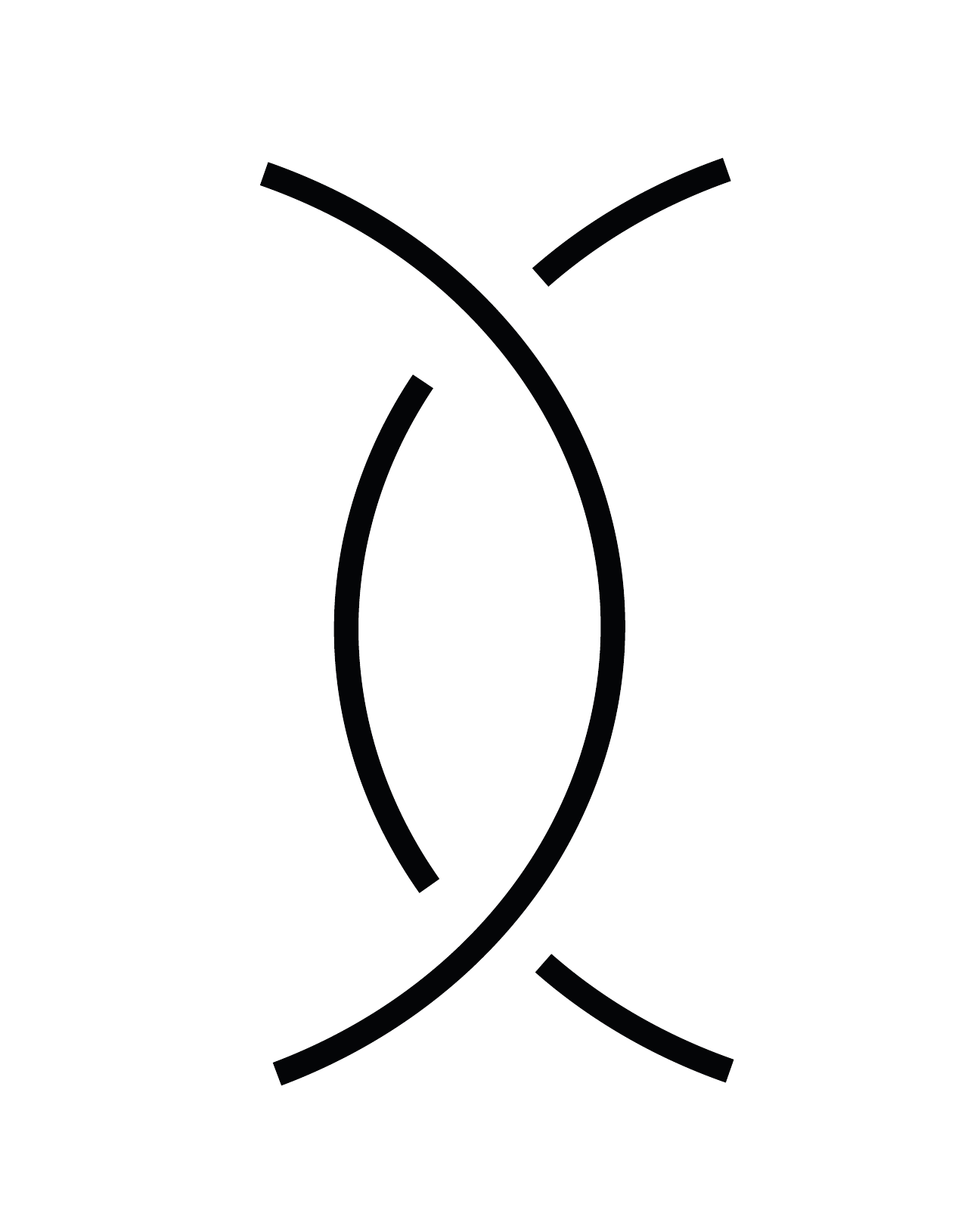}&=\graa{I2}.
\end{align}
Moreover, we have the following Reidemeister moves of type I:
\begin{align*}
\graa{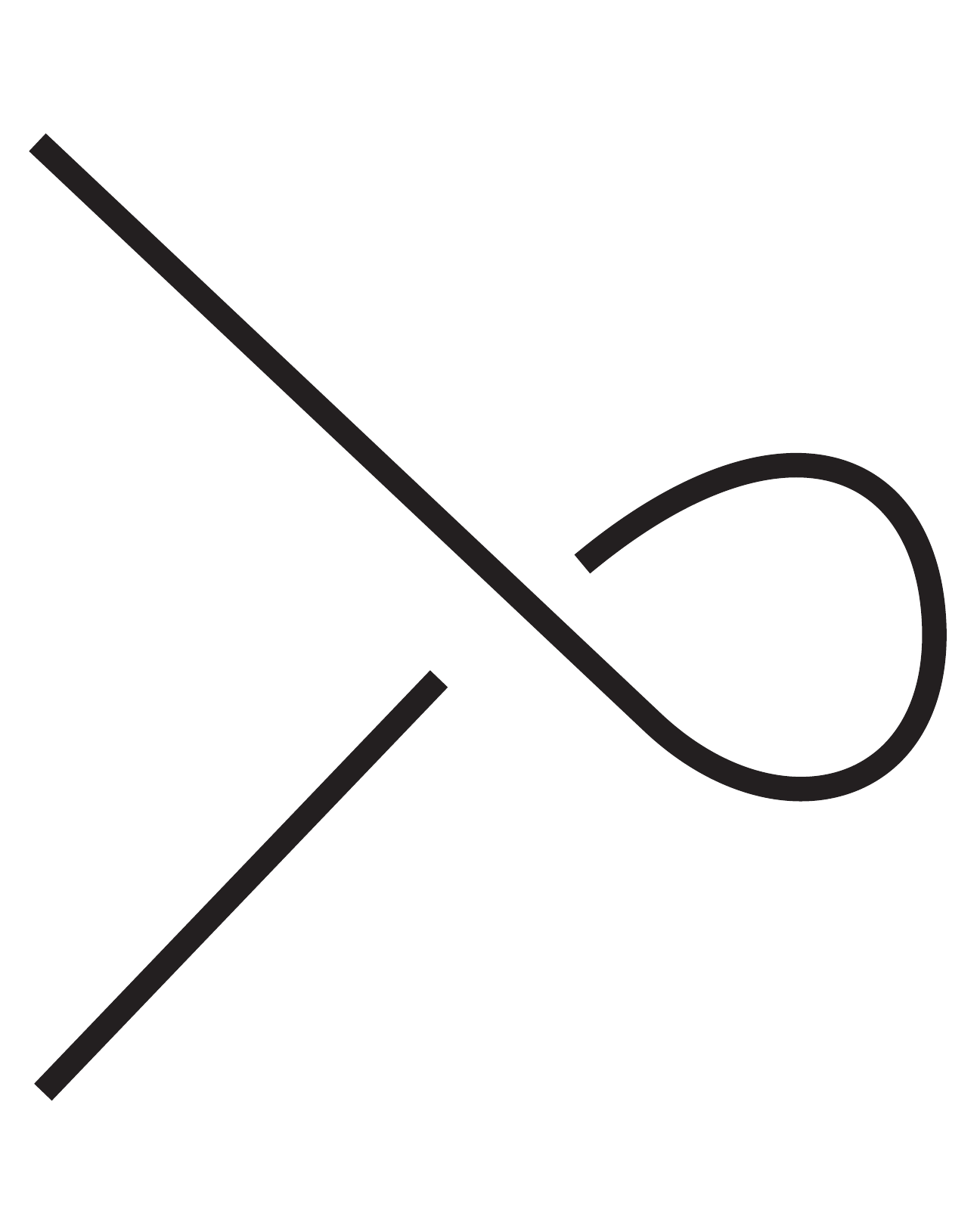}&=\omega^{\frac{1}{2}}\graa{id}\;,\quad\qquad
\graa{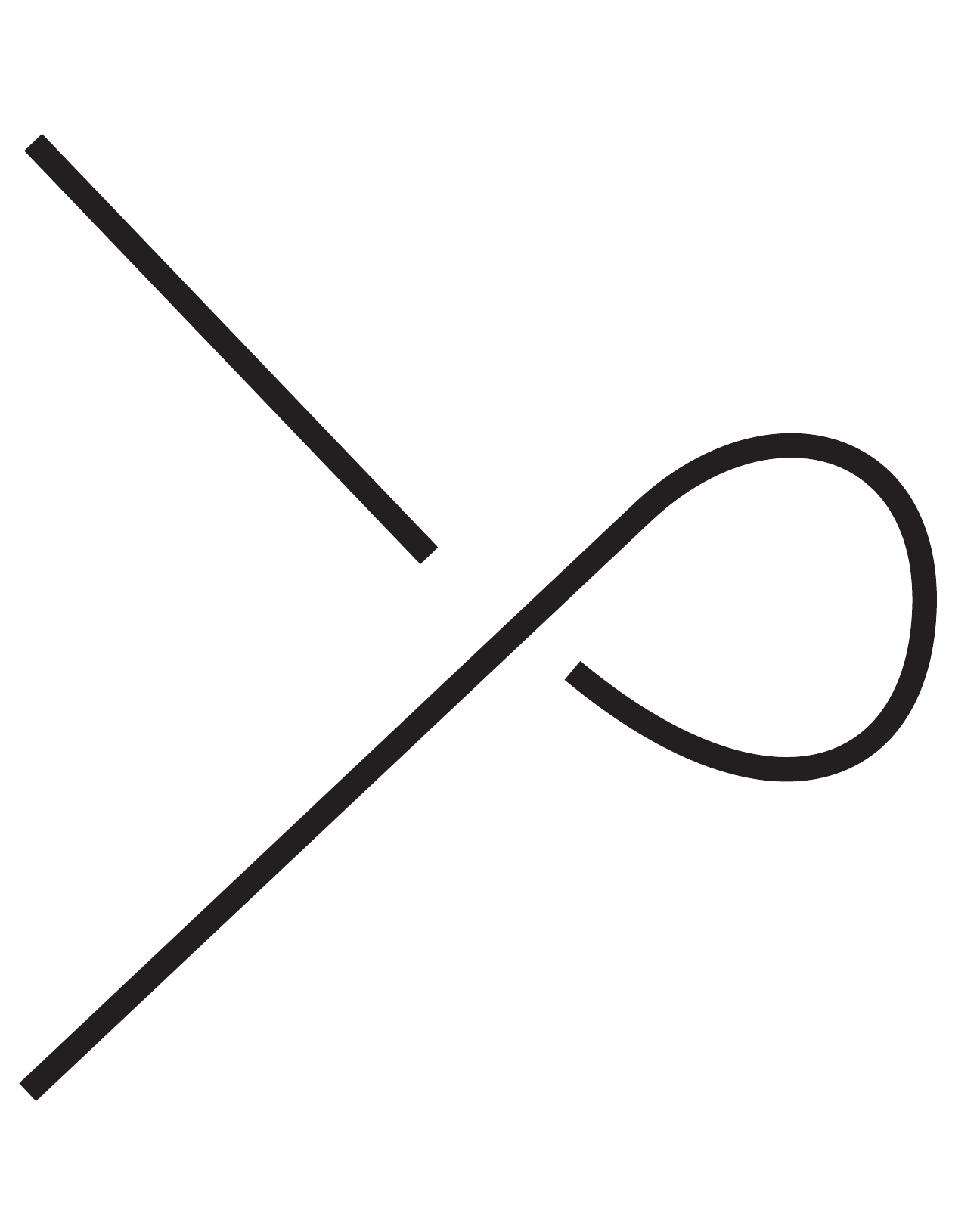}=\omega^{-\frac{1}{2}}\graa{id}\;.
%\graa{yb1}&=\graa{yb2}.
\end{align*}
The Reidemeister move of type III is also known as the Yang-Baxter equation:
\begin{align*}
\graa{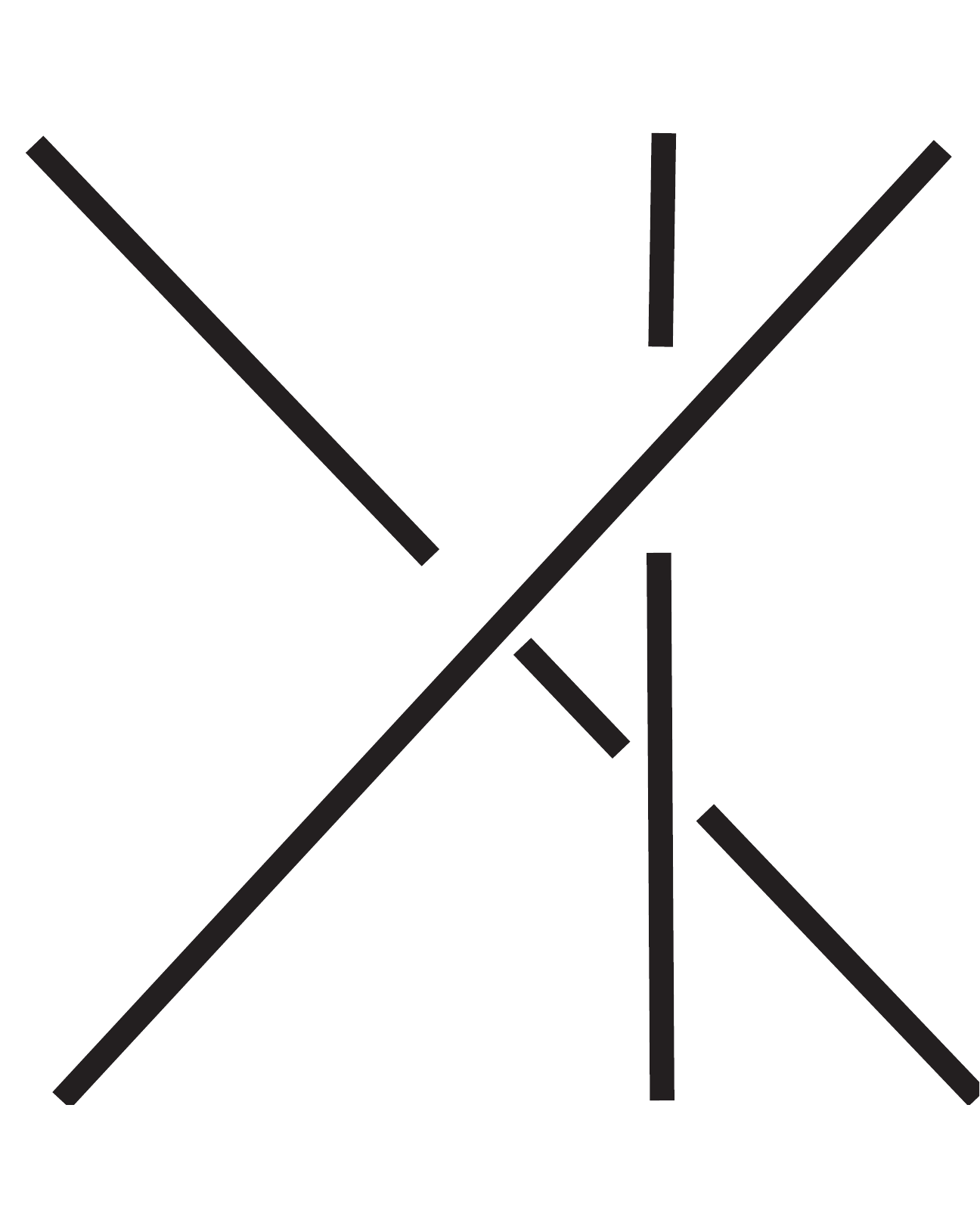}\ &=\ \graa{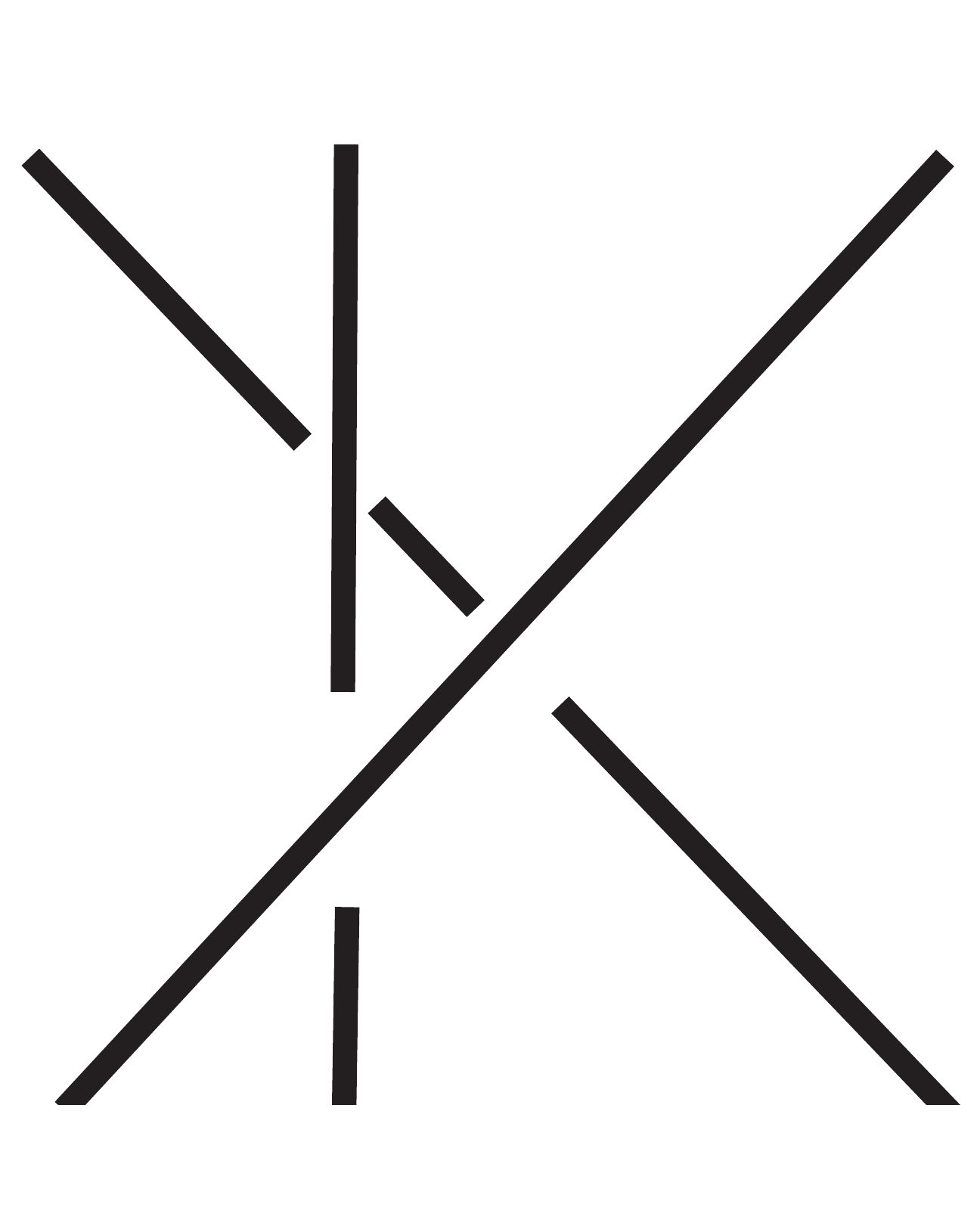}\ .
\end{align*}
This is a consequence of:

\begin{theorem}[\bf Braid-Parafermion Relation]
We have the relation:
\begin{align}
\graa{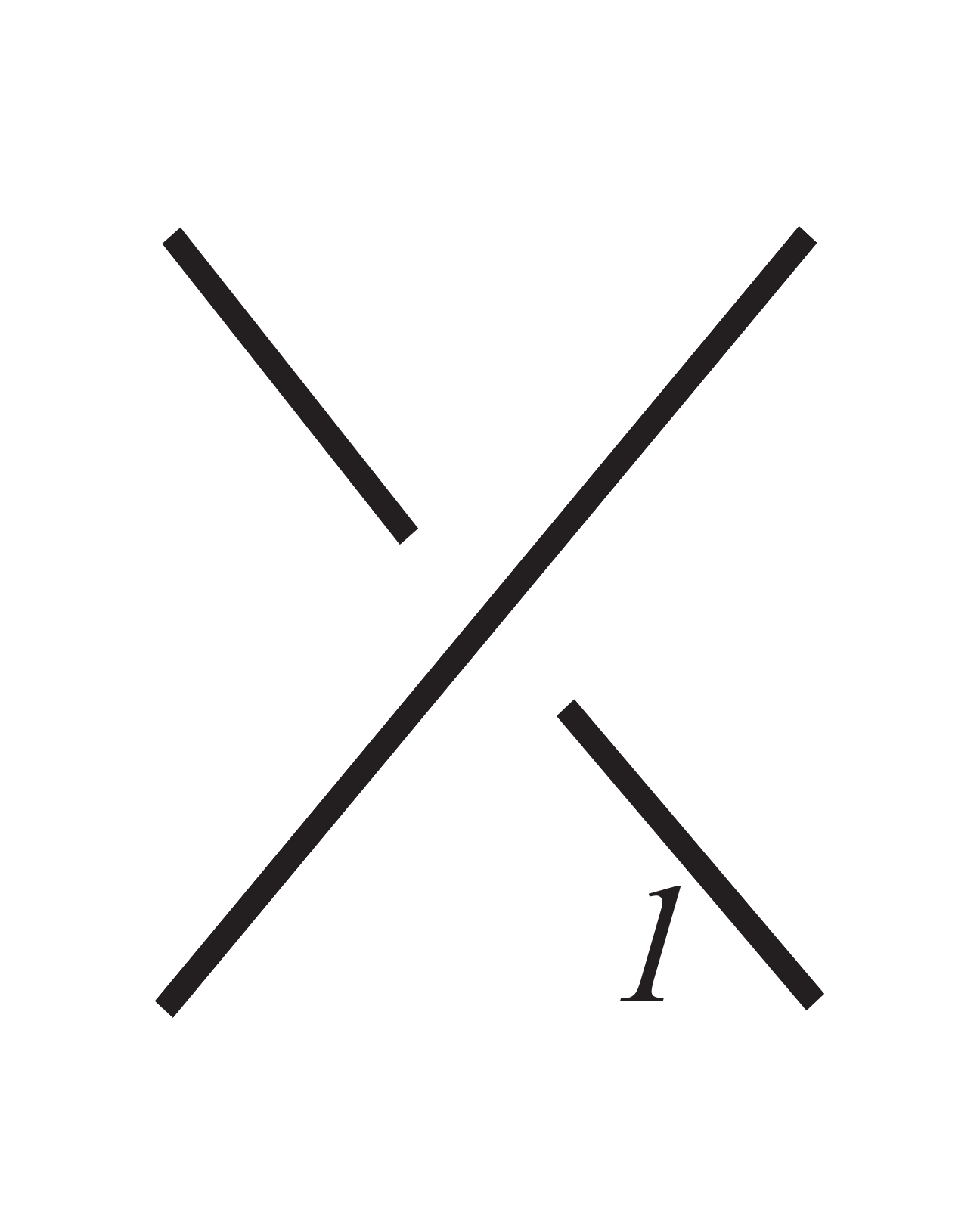}&=\graa{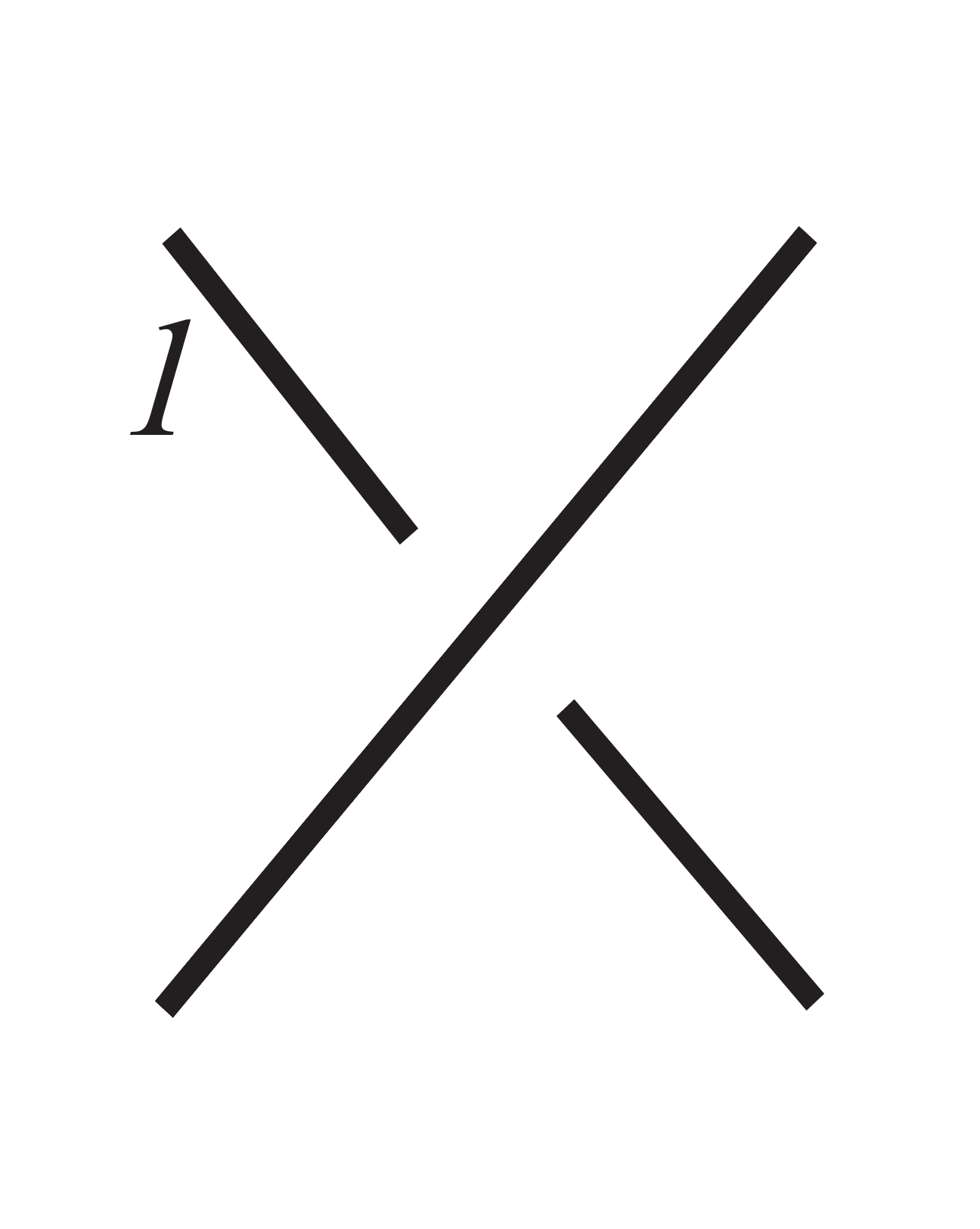}\;. \label{braidparafermion}
\end{align}
\end{theorem}

\begin{proof}
By Equation \ref{b+1},
\begin{align*}
  \graa{1b+}&=\frac{\omega^{\frac{1}{2}}}{\sqrt{N}} \sum_{i=0}^{N-1}\graa{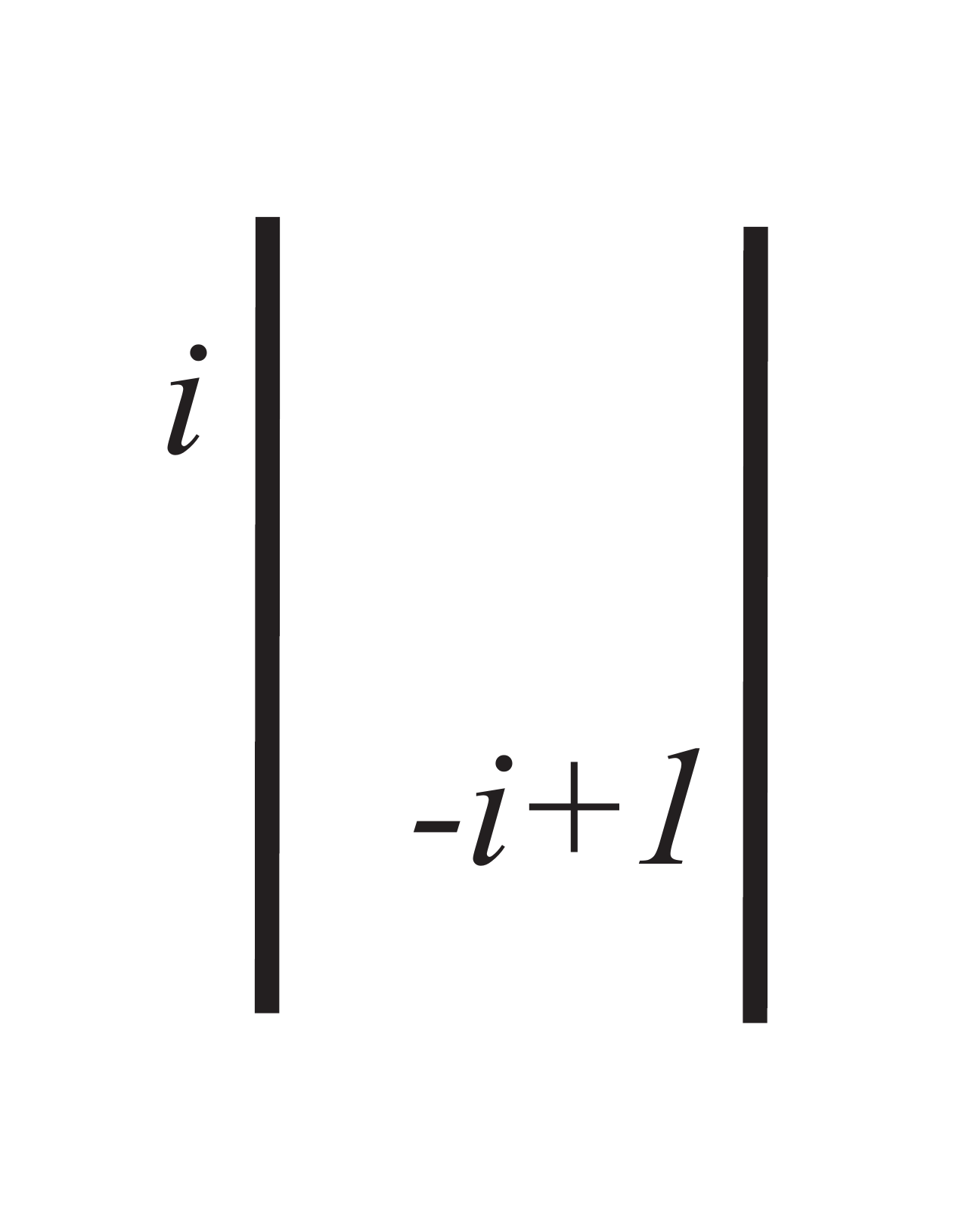}
            =\frac{\omega^{\frac{1}{2}}}{\sqrt{N}} \sum_{i=0}^{N-1}\graa{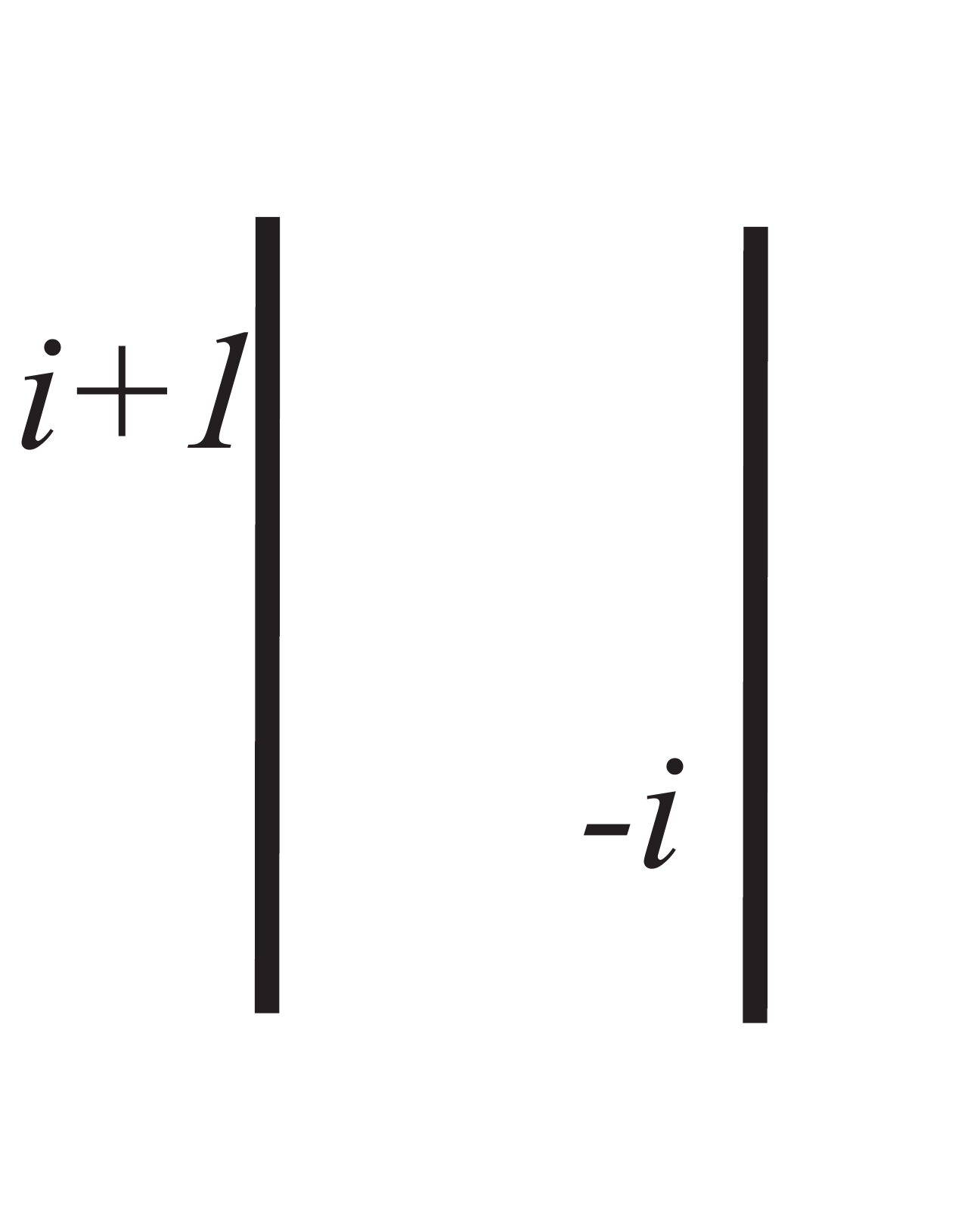}
            =\graa{b+1}\;.
\end{align*}
Here we translate the sum in $\mathbb{Z}_{N}$.
\end{proof}
\goodbreak

The braid-parafermion relation \ref{braidparafermion} says that the generator $c$ can move under the string. Combining this with the Reidemeister move of type II in \eqref{moveII}, any $m$-box $x$ can move under the string:
\begin{align}\label{underflat}
  \graa{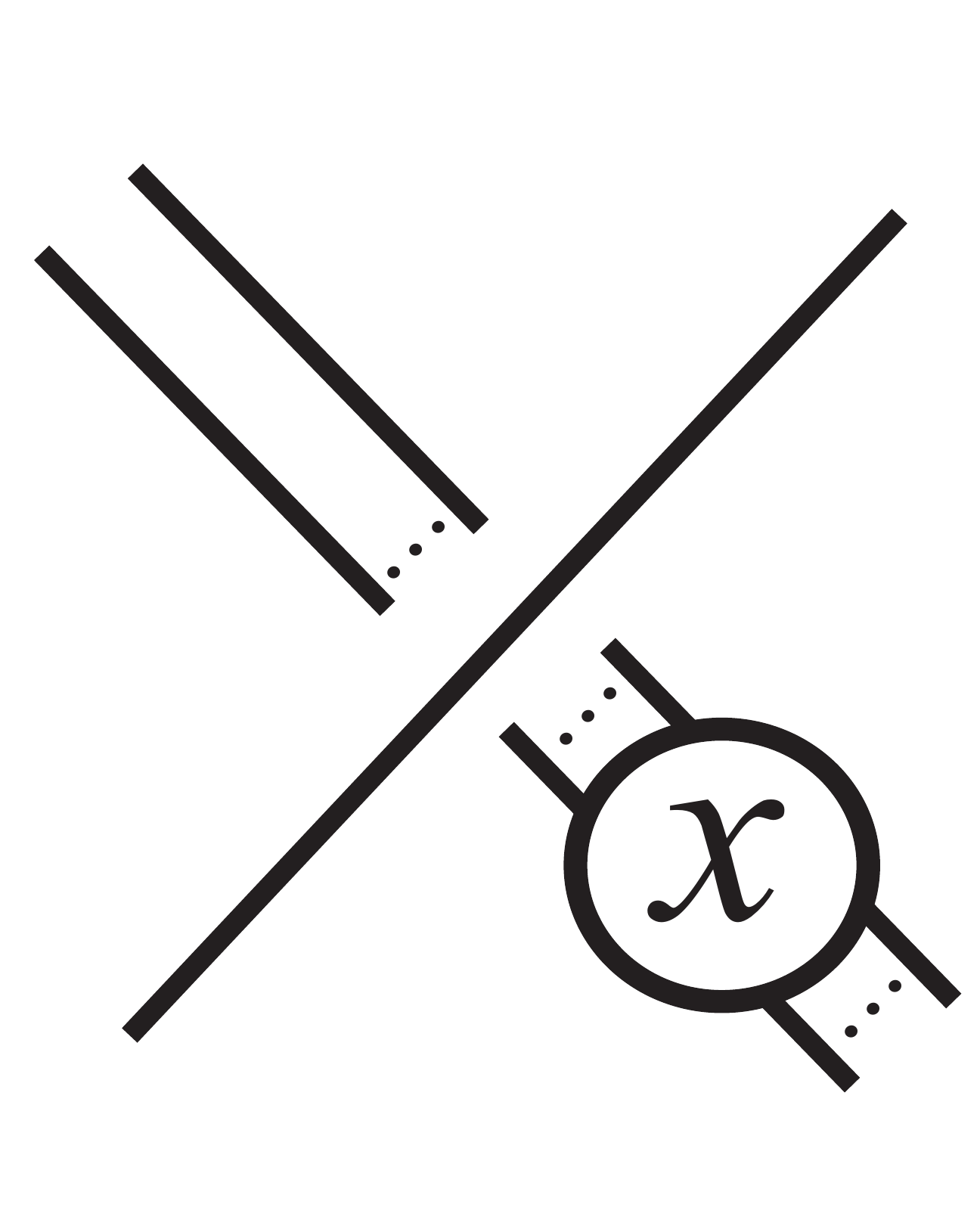}&=\graa{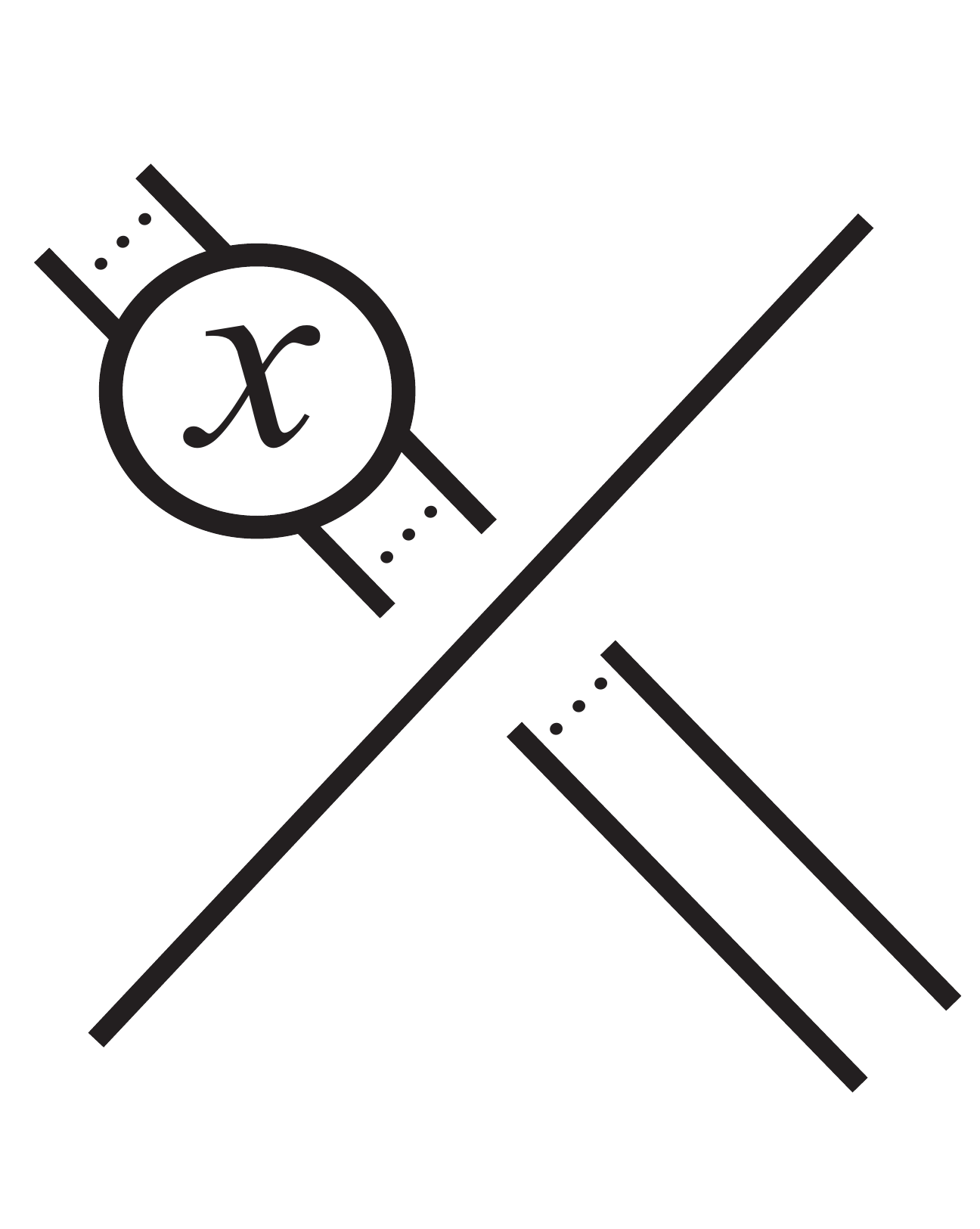}\;.
\end{align}
Therefore the strings can be lifted to the three dimensional space acting over the planar para algebra.
We call this property of the parafermion planar para algebra $PF_{\bullet}$ the {\it half-braided} property.
We call a planar algebra {\it braided}, if the string acts both over and under it.

\begin{definition}
An unshaded planar (para) algebra is called half braided, if there are (zero-graded) 2-boxes $\graa{b+}$ and $\graa{b-}$, such that
Equations \eqref{fourierbraid1}, \eqref{moveII} hold, and for any $m$-box $x$ Equation \eqref{underflat} holds.
Furthermore, it is called braided, if Equation \eqref{underflat} holds while switching $\graa{b+}$ to $\graa{b-}$.
\end{definition}

\begin{remark}
The zero graded part of the parafermion planar para algebra $PF_{\bullet}$ is the group $\mathbb{Z}_N$ subfactor planar algebra $P^{\mathbb{Z}_N}$. It is generated by 2-boxes $\left\{\graa{i=-i}\right\}_{i\in\mathbb{Z}_N}$ which form the group $\mathbb{Z}_N$.
The bosonic generator $\graa{i=-i}$ is decomposed as the twisted tensor product of the parafermion $c^i$ and its antiparticle $\Theta(c^i)$. We interpret the decomposition as the parasymmetry of the parafermion planar para algebra $PF_{\bullet}$.
The proof the Yang-Baxter equation takes advantage of the parasymmetry.
\end{remark}

\begin{theorem}
The string moves under the zero-graded planar subalgebra $P^{\mathbb{Z}_N}$ of $PF_{\bullet}$ as a $\mathbb{Z}_2$ flip:
  $$\graa{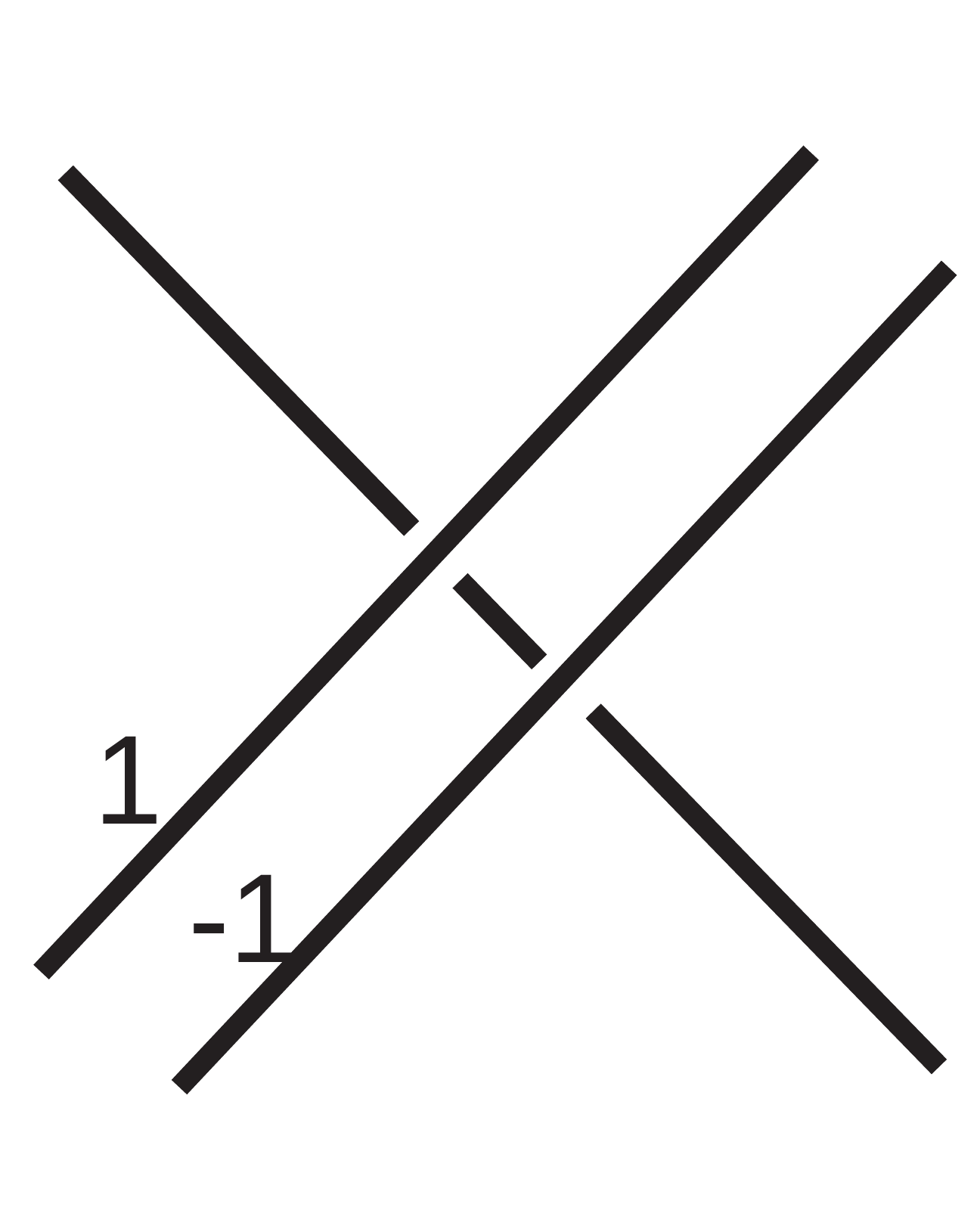}=\graa{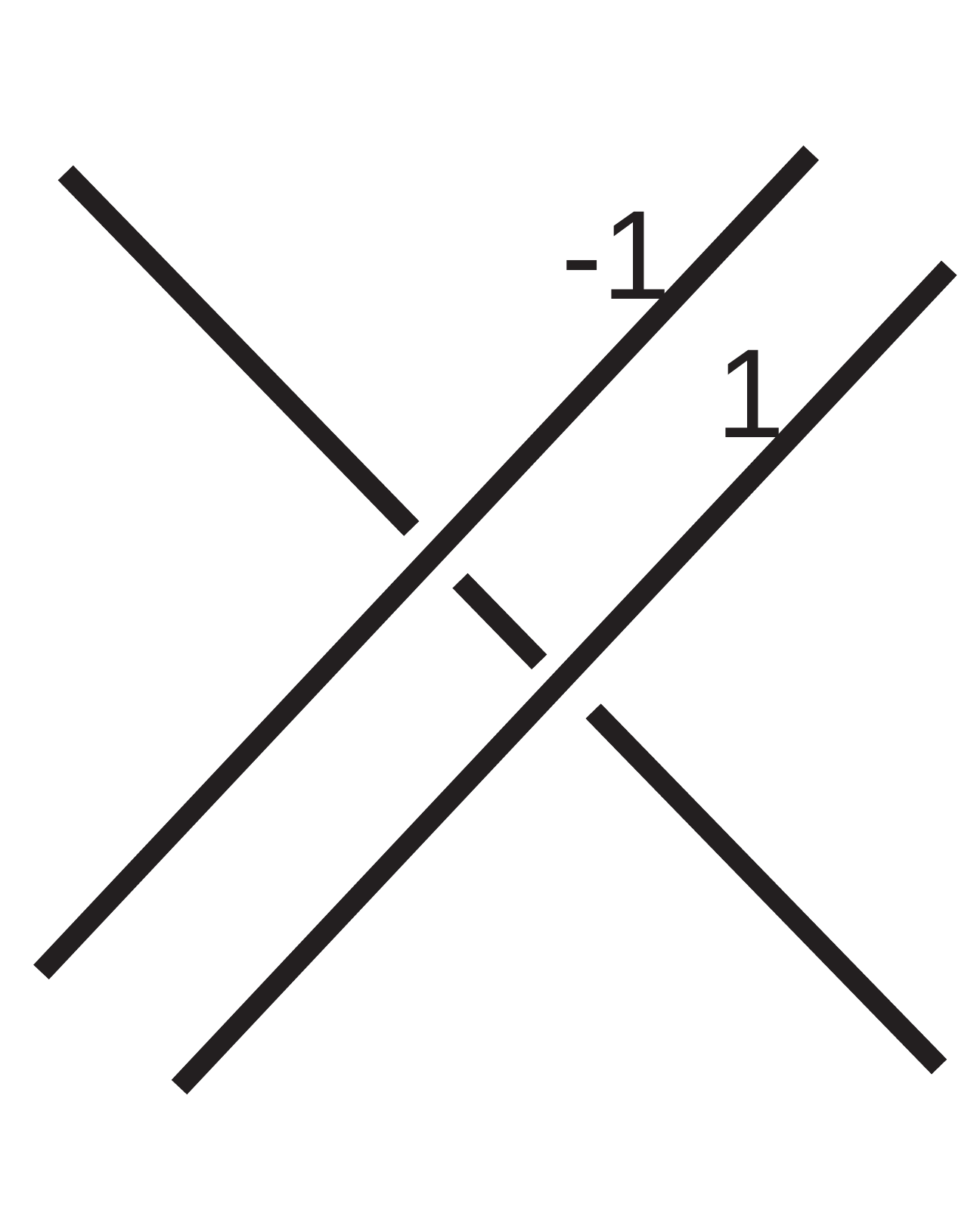}\;.$$
\end{theorem}

\begin{proof}
By Equation \ref{b+2},
\begin{align*}
  \graa{flip1}&=\frac{\omega^{-1}}{N}\sum_{i,j=0}^{N-1}\graa{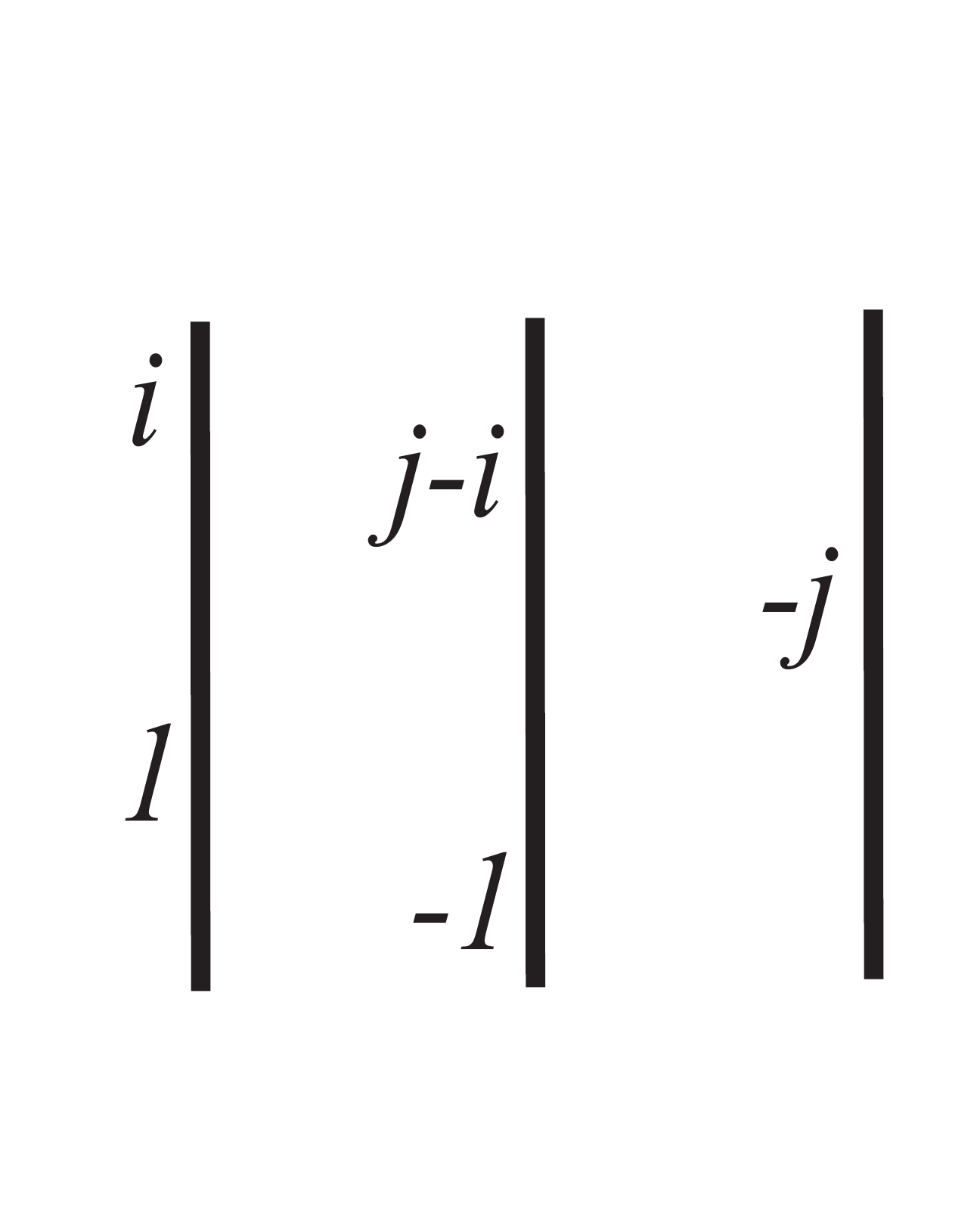}\\
            &=\frac{\omega^{-1}}{N}\sum_{i,j=0}^{N-1}q^{j-i}\graa{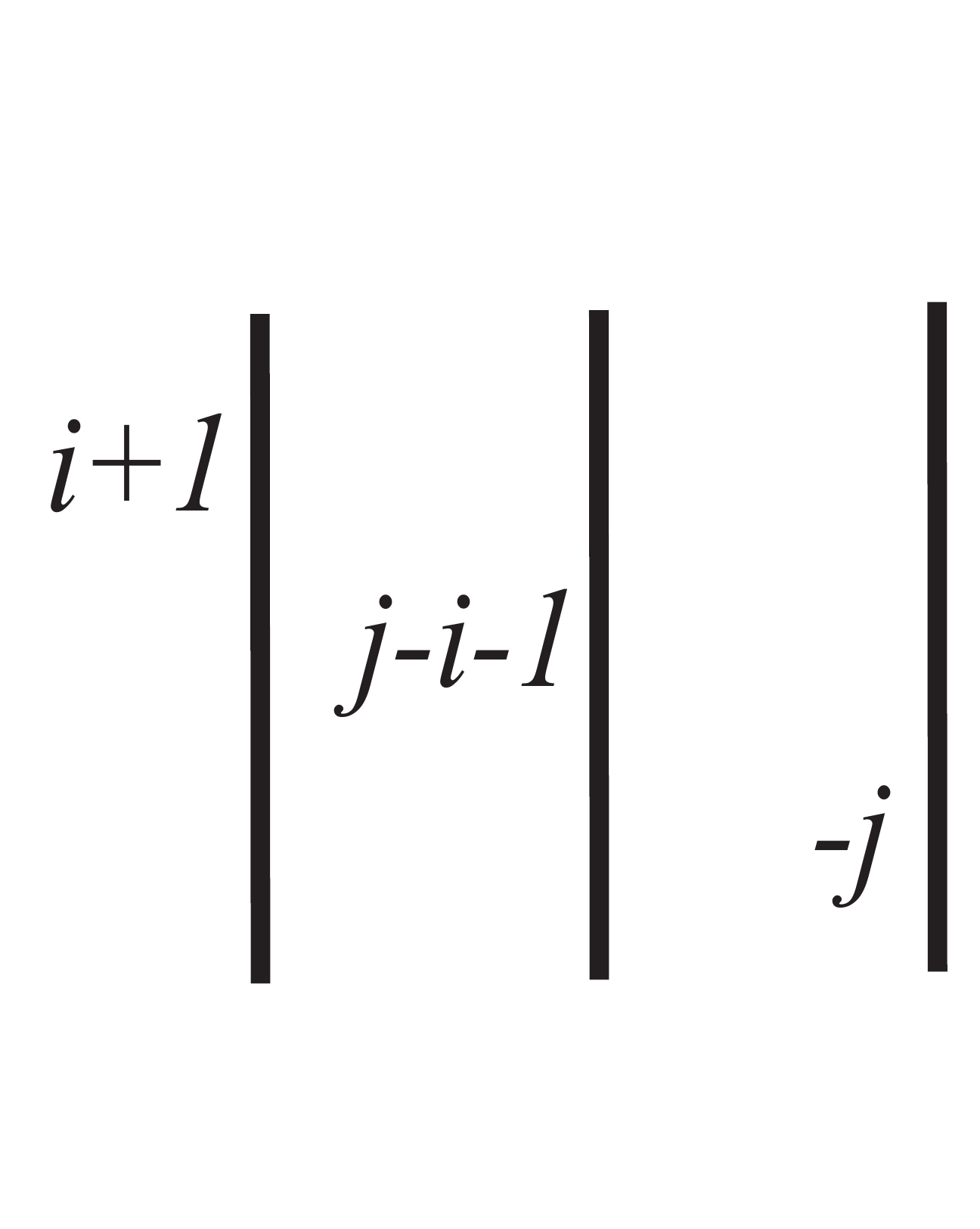}\;.
\end{align*}

\begin{align*}
  \graa{flip2}&=\frac{\omega^{-1}}{N}\sum_{i,j=0}^{N-1}\graa{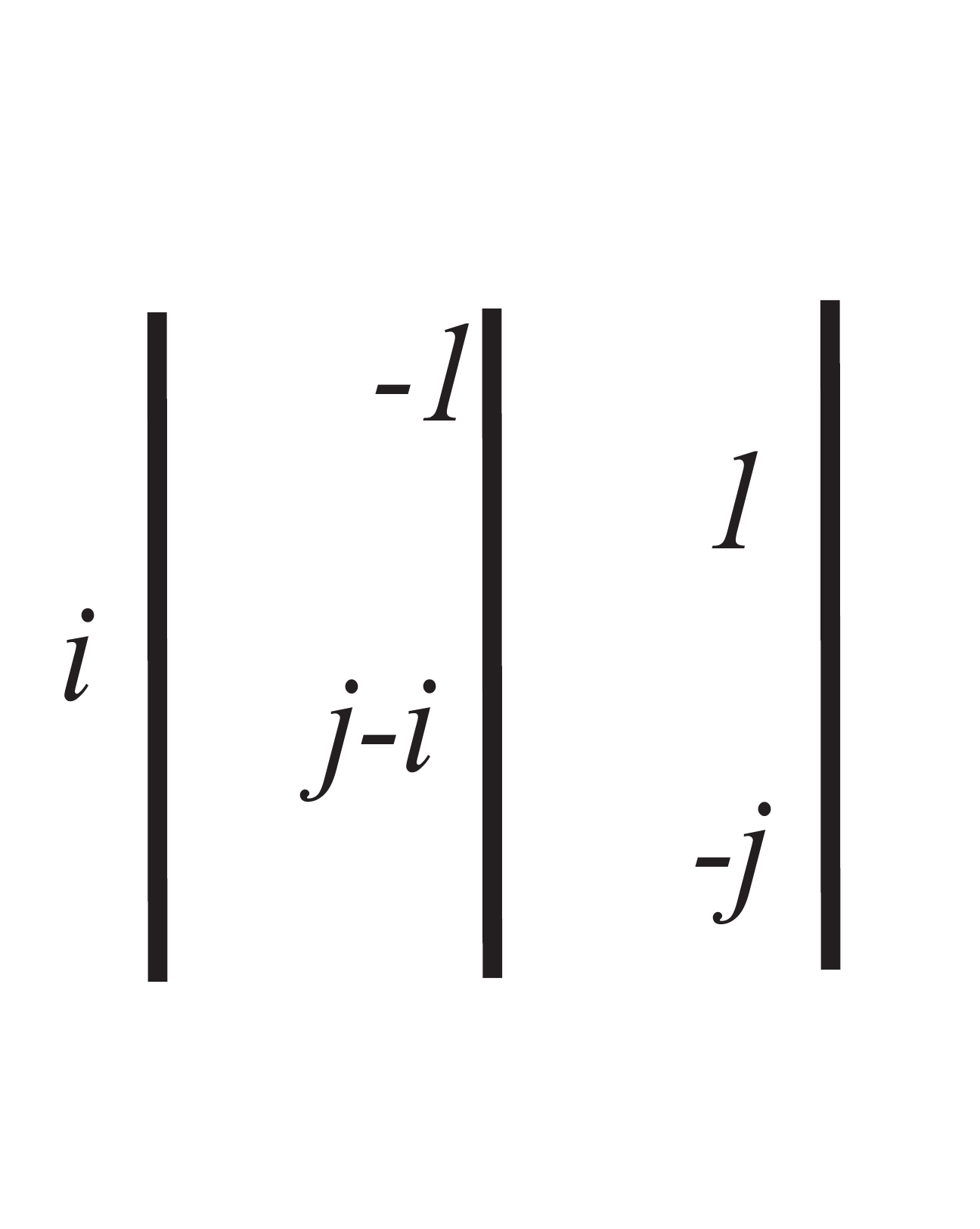}\\
            &=\frac{\omega^{-1}}{N}\sum_{i,j=0}^{N-1}q^{j-i}\graa{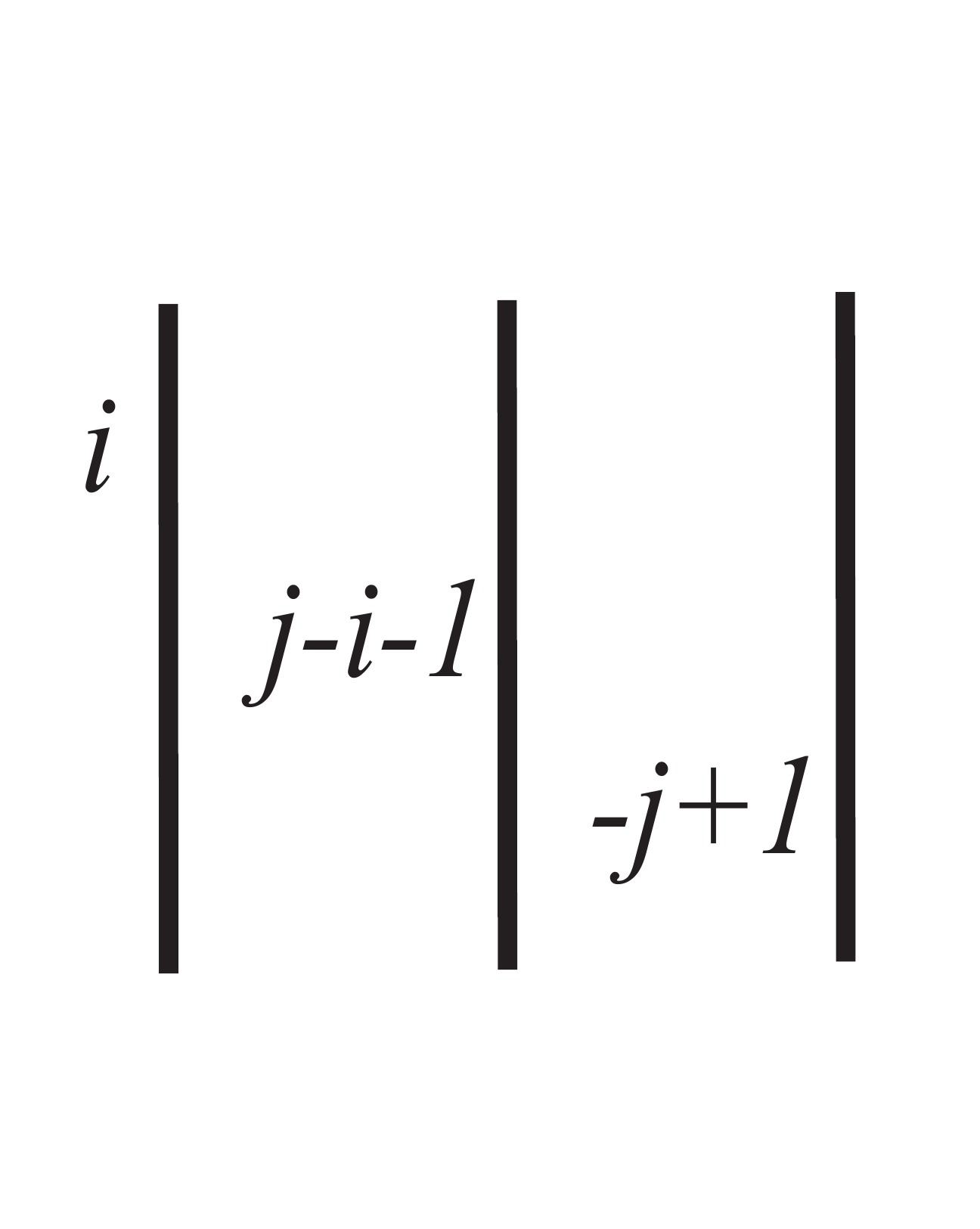}\;.
\end{align*}
Substitute $i,j$ by $i+1,j+1$ in the above equation. Then we have that
  $$\graa{flip1}=\graa{flip2}\;.$$
\end{proof}

\begin{corollary}
  The $\mathbb{Z}_2$ fixed point planar subalgebra of $P^{\mathbb{Z}_N}$ is braided.
\end{corollary}

\begin{corollary}\label{overflat2}
Any element $x$ in $P^{\mathbb{Z}_N}$ can move above double strings.
  $$\graa{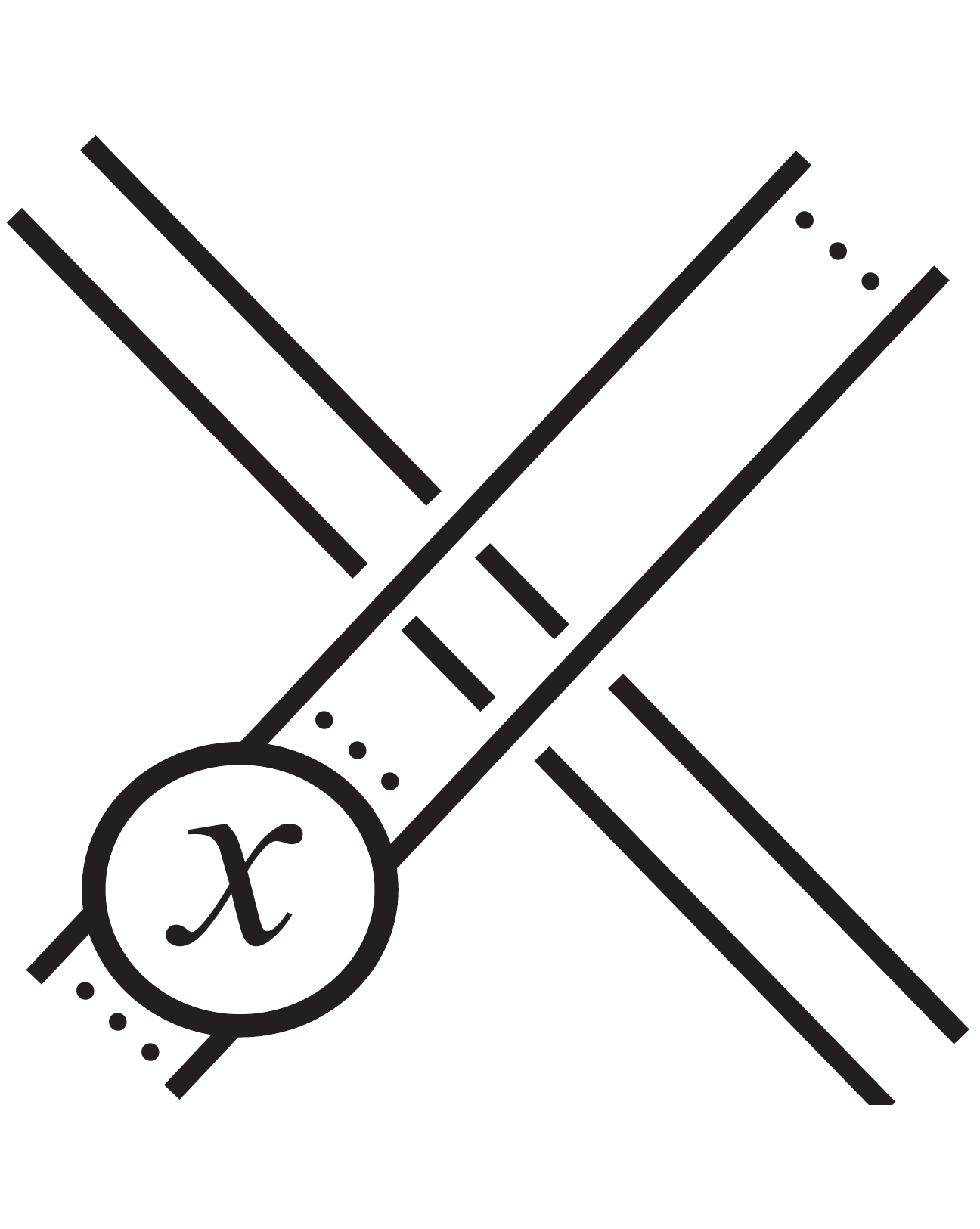}=\graa{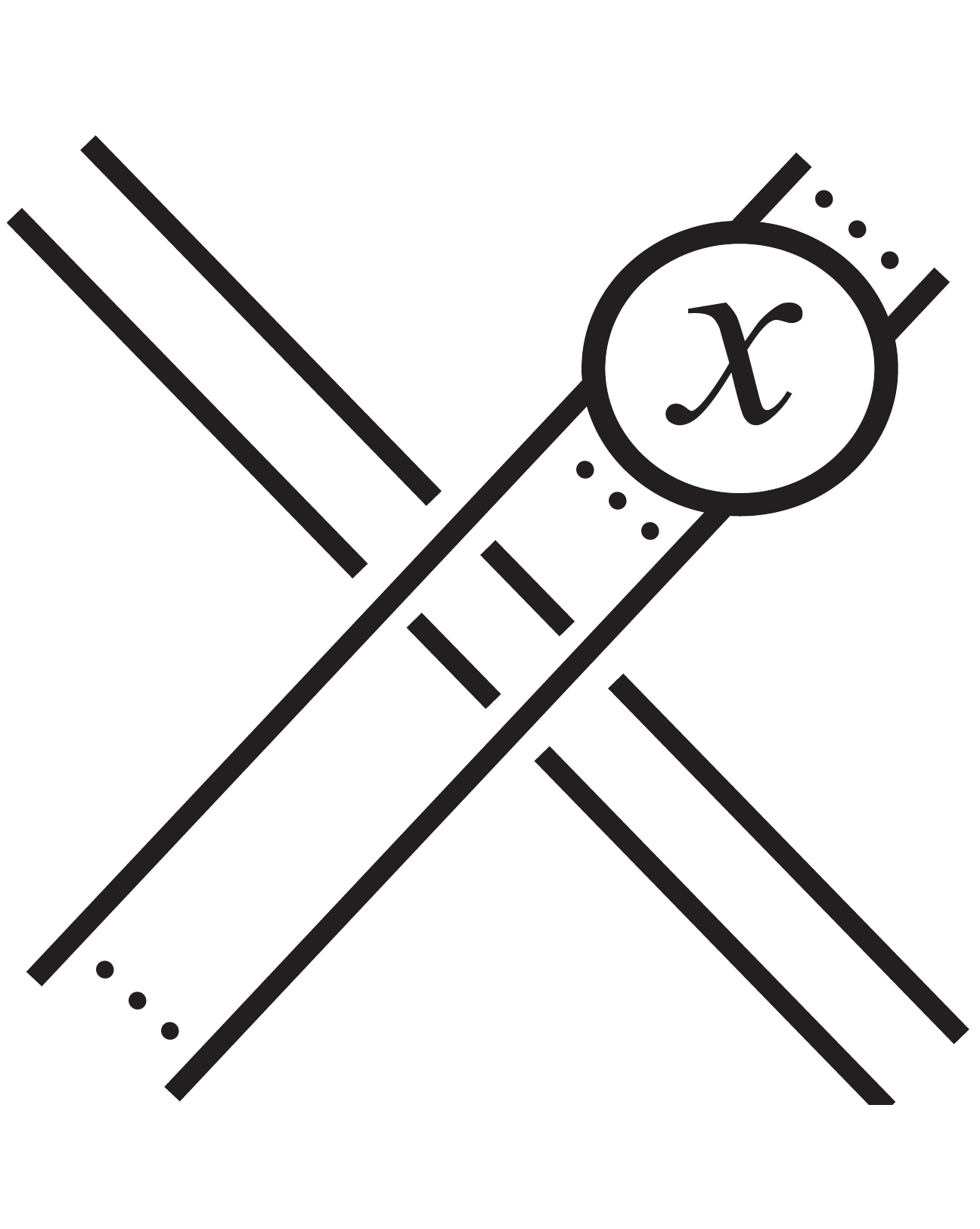}\;.$$
\end{corollary}

Therefore the even, zero-graded part of $PF_{\bullet}$ can move both above and under double strings,
and we recover the modular tensor category $Rep(\mathbb{Z}_N)$.
The simple objects are given by projections $\left\{\graa{p-i}\right\}_{i\in\mathbb{Z}_N}.$ The morphisms are given by zero-graded elements of $PF_{\bullet}$. The braids are derived from $\graa{b+}$ given in \eqref{b+2}. The multiplication and the tensor product are given by the action of corresponding tangles in Section \ref{section: planar para algebra}.
Moreover, $PF_{\bullet}$ turns out to be a module category over the modular tensor category. We refer the readers to \cite{Lon95,Xu98,BocEva98,Ocn00,Ost03} on the general theory of module categories over modular tensor categories.

One can relate our braid \eqref{b+2} to the solution \cite{FatZam82} with parameter $\la$ in the limit $\la\to-i\infty$, by taking $\zeta=-e^{\frac{\pi i}{N}}$.
Thus this braid can be ``Baxterized'' in the sense of Jones \cite{Jon91}. The braid has been studied previously in \cite{GolJon89}. 
In the case $N=2$, the braid gives the Jones polynomial 
associated with the index-two subfactor \cite{Jon85}.  
The case $N=5$ was studied extensively in \cite{Jon89}. The braid gives a Kauffman polynomial \cite{Kau90}, that is associated with a certain Birman-Wenzl-Murakami algebra \cite{BirWen,Mur87}.
We refer the reader to a recent survey of related work in \cite{AYP16}.

In another work, we discuss applications of the braid relations \eqref{b-1}--\eqref{b+2} for arbitrary $N$ to quantum information \cite{JafLiuWoz}.  

\setcounter{equation}{0}
\section{Clifford group}

In the $(QI,q)$ four string model of \S\ref{Sect:QIq4string}, the basis vector $\ket{k}$ is represented by a pair of caps with opposite grading. By Proposition \ref{Prop:SFT=FT}, the  SFT $\FS$ acting on the pair of caps is the discrete fourier transform $F$
\be
F\ket{k}=\frac{1}{\sqrt{N}}\sum_{i=0}^{N-1} q^{kl}\ket{l} \;.
\ee
The SFT on one cap gives the Gaussian $G$.
\be
G\ket{k}=\zeta^{k^2}\ket{k} \;.
\ee
We use the notation for Pauli $X$, $Y$, $Z$ defined in this model.
Then $X,Y,Z,F,G$ satisfy the following relations: 
\beq
	X^N=Y^N=Z^N=1\;, \\
	XY=q\,YX\;,\quad
	YZ=q\,ZY\;, \quad
	ZX=q\,XZ\;, \\
XYZ=\zeta \;; \\
\nonumber\\
FXF^{-1}=Z \;, 
FZF^{-1}=X^{-1} \;, \\
GXG^{-1}=Y^{-1}(=\zeta XZ) \;, 
GZG^{-1}=Z \;;\\
\nonumber\\
(FG)^3=\omega \;, \\
F^4=1 \;, \\
G^N=1 \;, \\
F^2 G= GF^2 \;.
\eeq

The Pauli matrices $X,Y,Z$ generate a projective linear group $(\Z_N)^2$. 
The Clifford group is defined to be the normalizer group of Pauli matrices.
If we express the element $X^iZ^j$ in $(\Z_N)^2$ as a vector 
        $\left(\begin{array}{c}
        i  \\
        j 
        \end{array}\right)$        
Then, $Ad_F=S=\left[\begin{array}{cc}
        0 & -1 \\
        1 & 0
      \end{array}\right]$
and   $Ad_G=T=\left[\begin{array}{cc}
        1 & 1 \\
        0 & 1
      \end{array}\right]$.
      
Suppose a unitary $U$ is in the Clifford group. 
If $Ad_U$ fixes $X$ and $Z$ projectively, then one can show that $U=X^iZ^j$ projectively by a simple computation in $N \times N$ matrices.

In general, $Ad_U$ acts as an element in $GL(2,\Z_{N})$.
Note that $Ad_{F},Ad_{G}$ generate  $SL(2,\Z_{N})$.
One can prove that $\left[\begin{array}{cc}
        1 & 0 \\
        0 & -1
      \end{array}\right]$ can not be realized as $Ad_U$ by a simple computation in $N \times N$ matrices. 
Therefore, we have the following proposition:      
\begin{proposition}
The projective linear group generated by $X,Y,Z,F,G$ is the Clifford group.  This group is $(\Z_N)^2 \rtimes SL(2,\Z_N)$.
\end{proposition}
As a consequence,  the above relations for $X,Y,Z,F,G$ are complete.
%
%\subsection{Multiple-qudit case}
%In the two string model, the string Fourier transformation on 2-qudit generates the controlled transformation $C_Z$.
%$C_Z=(GF^{-1}\otimes FG^{-1})\FS (1\otimes F^{-1}G^{-1})$.
%Thus, $X,Y,Z,F,G,\FS$ and $\sigma$ generated the $n$-qudit Clifford group.
%
%In particular, the string Fourier is topological quantum computer, as it generates any unitary transformations with 1-qudit transformations. 

\setcounter{equation}{0}
\section{Positivity for the General Circle Parameter\label{positivity2}}
We constructed the planar para algebra $\mathscr{P}_{\bullet}$ over the field $\mathbb{C}(\delta)$ in Section \ref{construction}. The $m$-box space has a sub algebra generated by labelled tangles with only vertical strings which is isomorphic to the parafermion algebra $PF_m$.

Similar to the Temperley-Lieb-Jones planar algebra case, we can construct matrix units of $\mathscr{P}_m$ over the field $\mathbb{C}(\delta)$ inductively by the matrix units of parafermion algebras constructed in Section \ref{Pictorial presentation},
the basic construction and the general Wenzl's formula \cite{Wen87,LiuYB}.
If a labelled tangle is not in the basic construction ideal, then it is in the parafermion algebra. Therefore, the principal graph of the planar para algebra $\mathscr{P}_{\bullet}$ is the same as the Bratteli diagram of parafermion algebras, i.e.,
$$\graa{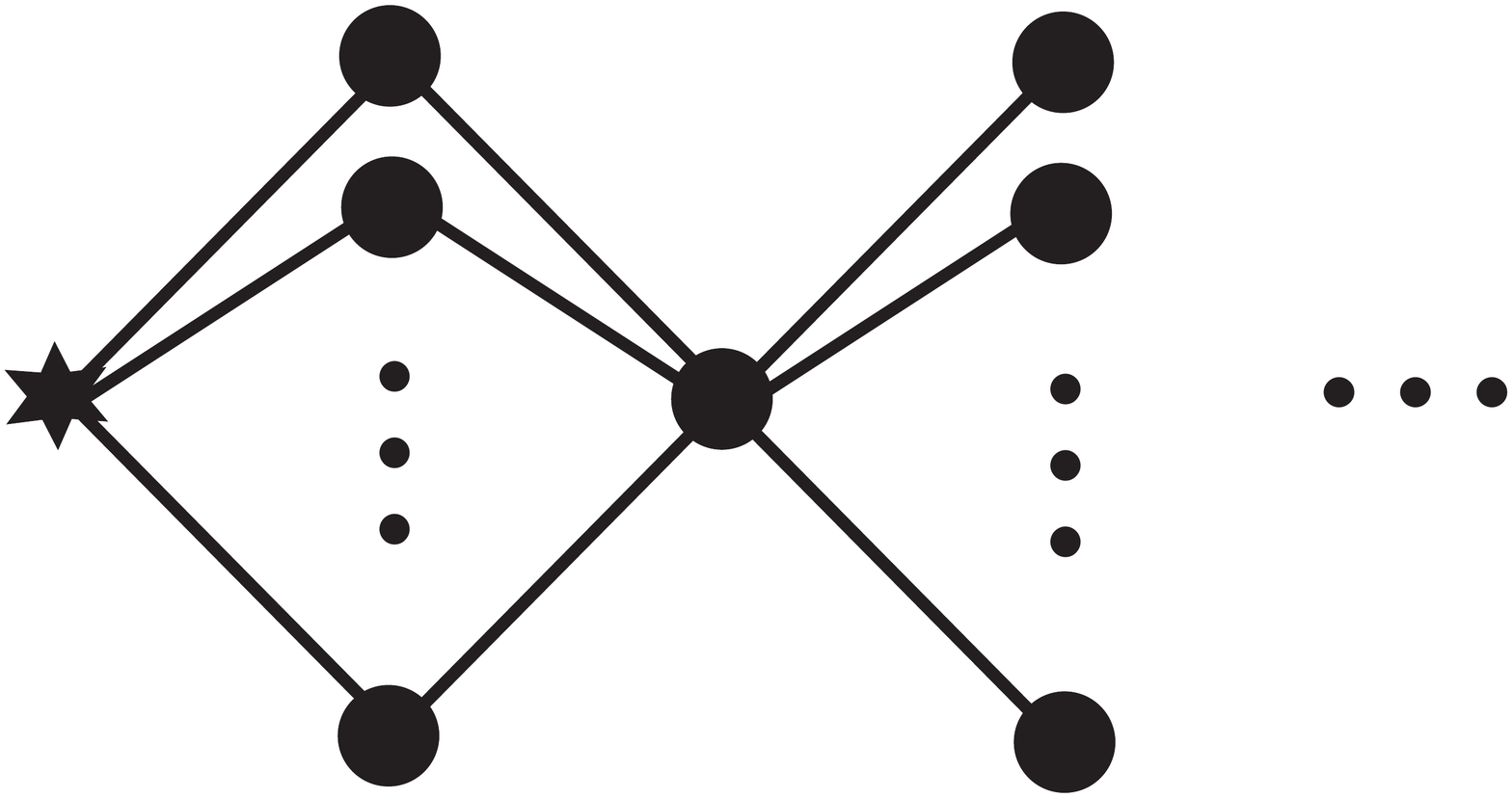},$$
assuming the quantum dimensions of vertices in the principal graph are non-zero.
This assumption can be avoided by the bi-induction argument in \cite{LiuYB}.
Moreover, we obtain the formula of the quantum dimensions of these vertices.
There is one depth $2m$ vertex. Its quantum dimension is $\sqrt{N}[2m]$.
There are $N$ depth $2m+1$ vertices. Any of them has quantum dimension $\frac{\sqrt{N}}{N}[2m+1]$. Here $[m]$ is the quantum number $\frac{q^m-q^{-m}}{q-q^{-1}}$, and $\delta=\sqrt{N}[2]$.

Jones' remarkable rigidity theorem \cite{Jon83} says that all possible values of the circle parameter of a subfactor planar algebra are given by
$$\{2\cos\frac{\pi}{n} | n=3,4,\cdots \}\cup [2,\infty).$$
These values are realized by Temperley-Lieb-Jones subfactor planar algebras.

To obtain the positivity for the planar para algebra $\mathscr{P}_{\bullet}$, $\delta$ has to be positive. In this case, we can define the (unique) vertical reflection $*$ on the planar algebra induced by $c^*=c^{-1}$.

\begin{theorem}\label{Theorem: positivity general}
The planar para algebra $\mathscr{P}_{\bullet}$ has positivity if and only if $\frac{\delta}{\sqrt{N}}$ is in
$$\{2\cos\frac{\pi}{k} | k=3,4,\cdots \}\cup [2,\infty).$$
\end{theorem}

\begin{proof}
The matrix units of $\mathscr{P}_m$ are constructed over the field $\mathbb{C}(\delta)$. When $\delta$ is a scalar, the matrix units of $\mathscr{P}_m$ are well-defined by Wenzl's formula, if the Markov trace is non-degenerated on $\mathscr{P}_{m-1}$.

If $2\cos\frac{\pi}{k-1}<\delta<2\cos\frac{\pi}{k}$, then $[i]>0$ for all $i<k$. Thus the matrix units of $\mathscr{P}_k$ are still well-defined. Since $[k]<0$, the positivity fails.

If $\delta=2\cos\frac{\pi}{k}$, then $[i]>0$ for all $i<k$. Thus the matrix units of $\mathscr{P}_k$ are still well-defined. Since $[k]=0$, any minimal idempotent orthogonal to the basic construction ideal has trace 0. Thus it is in the kernel of the partition function. Therefore, $\mathscr{P}_{\bullet}$ modulo the kernel of the partition function is a depth $k-1$ subfactor planar para algebra. It has the following principal graph for $k=3, 4, 5, \cdots$

$$\graa{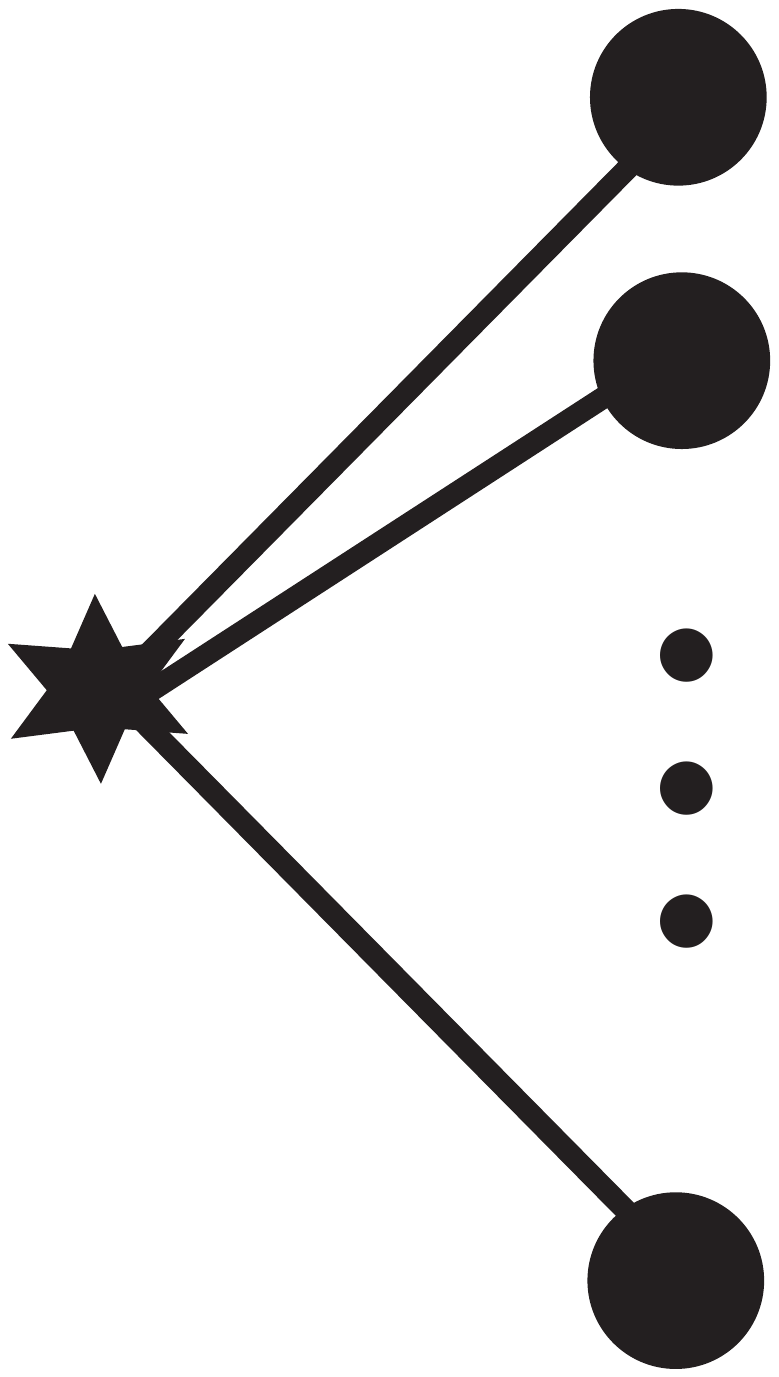};
\quad
\graa{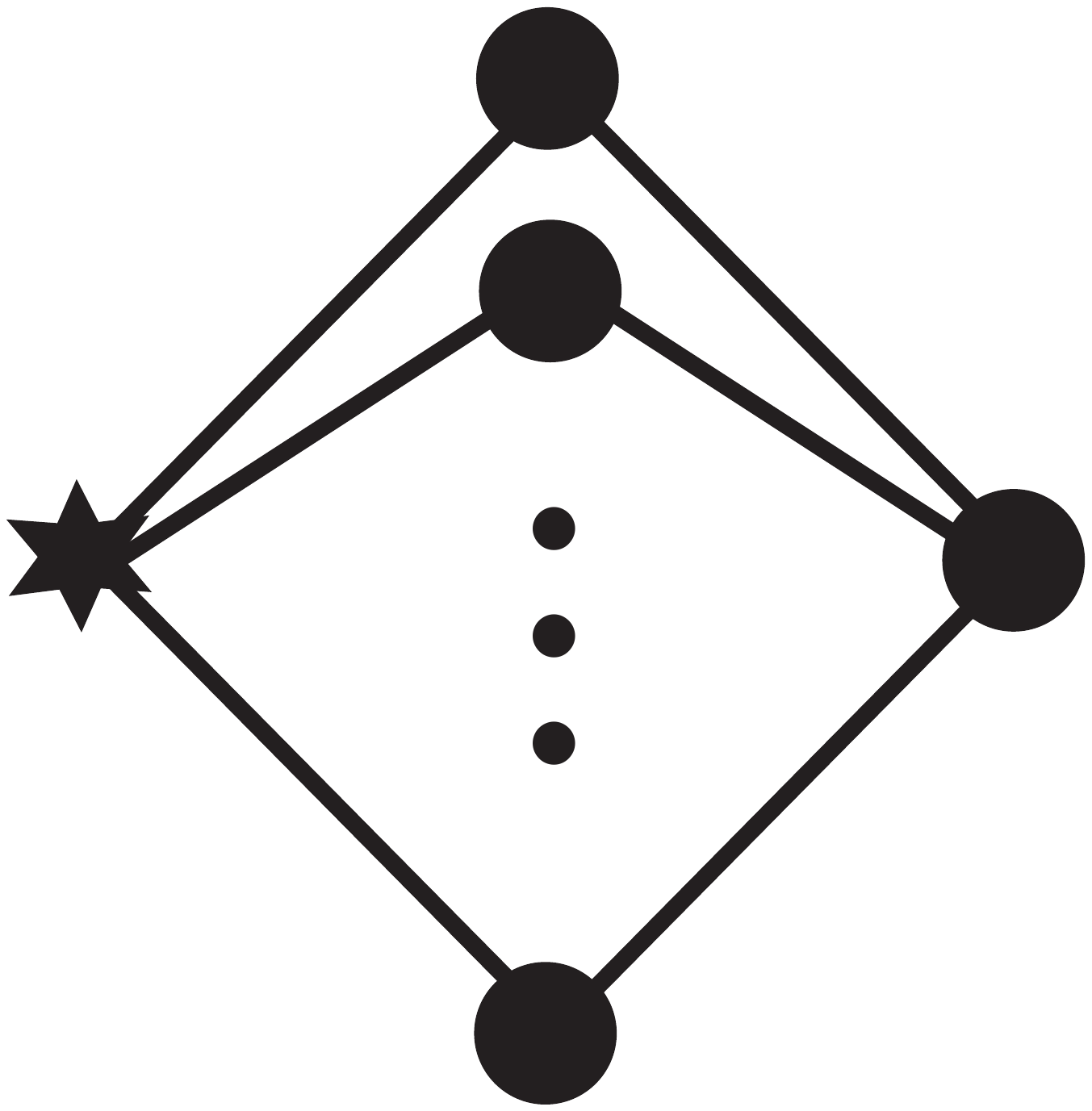};
\quad
\graa{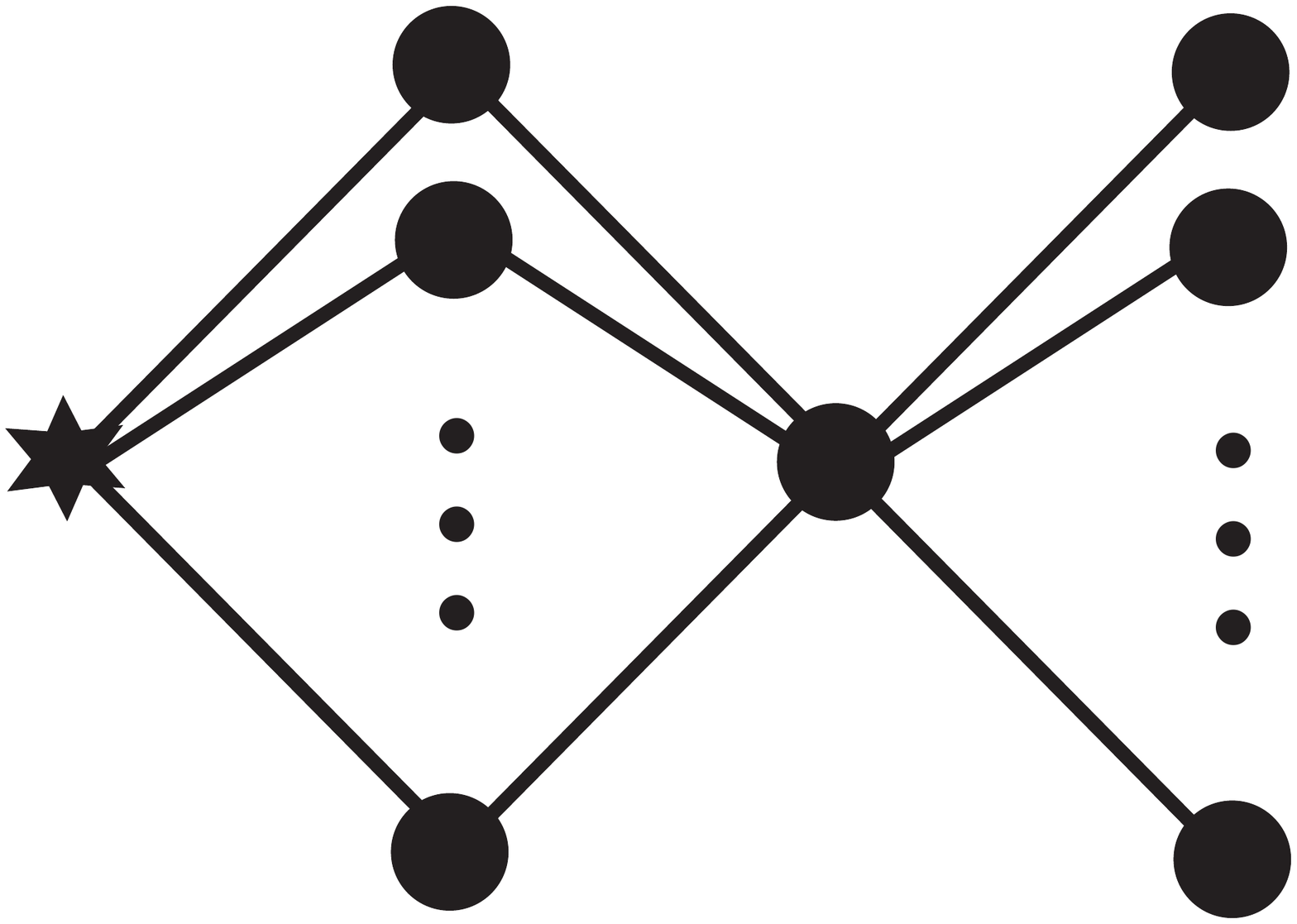};
\quad
\cdots$$

If $\delta\geq 2$, then $[i]>0$ for all $i$. Thus the matrix units of $\mathscr{P}_{\bullet}$ are still well-defined. Moreover, $\mathscr{P}_{\bullet}$ is a subfactor planar para algebra with the following principal graph
$$\graa{pginf}.$$
\end{proof}

\section*{Acknowledgement}
\thanks
This research was supported in part by a grant from the Templeton Religion Trust.  We are also grateful for hospitality at the FIM of the ETH-Zurich, where part of this work was carried out.

\renewcommand{\theequation}{A.\arabic{equation}}
\setcounter{equation}{0}

\renewcommand{\theequation}{B.\arabic{equation}}
\setcounter{equation}{0}
\section*{Appendix A. The Construction of Planar Para Algebras for Parafermions}
Since the generators are 1-boxes, any labelled 0-tangle is a disjoint union of closed strings labelled by generators. For an innermost closed string, we can move all its labels toward on point by isotopy. Then we can reduce the labelled closed string to a scalar by the relations given by the multiplication and the trace. Note that only the zero graded part on the closed string is evaluated as a non-zero scalar, since the trace is graded.
The para isotopy and the $2\pi$ rotation reduce to the usual isotopy of planar algebras on zero-graded part. Thus the evaluation of different labeled closed strings are independent modulo para isotopy. Essentially we only need the consistency condition on a single labelled closed string which indicates the associativity of the multiplication and the tracial condition of the expectation.

The above argument can be formalized by the method in Section 5 in \cite{LiuYB} which was motivated by the work of Kauffman \cite{Kau90}.
The idea is first constructing the planar algebras generated by the generators without relations, namely the {\it universal planar algebra}. Then one can define a partition function on the universal planar para algebra as the average of complexity reducing evaluations and prove that the relations are in the kernel of the partition function.

\begin{proof}[Proof of Theorem \ref{Theorem: construction}]
For the group $\mathbb{Z}_N$ and a bicharacter $\chi(i,j)=q^{ij}$, $q=e^{\frac{2\pi i}{N}}$,
first let us construct a shaded $(\mathbb{Z}_N,\chi)$ planar para algebra generated by the 1-box $c$ with grading 1 and relations $c^N=1$,
$\graa{circk}=0$, for $1\leq k\leq N-1$.
The para isotopy and the $2\pi$ rotation for the generator $c$ can also be viewed as relations of $c$.

Let $\mathscr{U}$ be the $(\mathbb{Z}_N,\chi)$ universal planar para algebra generated by $c$.
Let us define the partition function $Z$ inductively by the number of labelled circles of labelled 0-tangles as follows.

The partition function of the empty diagram is 1.
We assume that the partition function for diagrams with at most $n-1$ labelled circles is defined. Let us define the partition function of a labelled 0-tangle $T$ with $n$ labelled circles.
Let $IC$ be the set of innermost (labelled) circles of $T$. Take one circle $L$ in $IC$, let us define $Z(T,L)$.

If $L$ has no label $c$, then $Z(T,L):=\delta Z(T\setminus L)$.
If the number of labels of $L$ is not divisible by $N$, then $Z(T,L):=0$.
If $L$ has $Nk$ labels, we count the labels in $L$ anti-clockwise starting from the top label $c$, denoted by $c_i$, $0\leq i\leq Nk-1$. Let us move $c_i$ clockwise to $c_0$ one by one by RT isotopy and para isotopy. While applying the para isotopy to $c_i$ and another label, we obtain a scalar $q$ or $q^{-1}$ each time. While moving $c_i$ to $c_0$, if $c_i$ is rotated clockwise by $2k_i\pi$, then we obtain a scalar $q^{k_i}$.
Let $q^L$ be the multiplication of all these scalars. Then $Z(T,L):=q^L \delta Z(T\setminus L)$.

Let us define
$$Z(T)=\frac{1}{|IC|}\sum_{L\in IC} Z(T,L).$$
By an inductive argument and the fact that $q^N=1$, it is easy to check that $Z(T)$ is well-defined on the universal planar para algebra. The most complex case is to show the $Z(T,L)$ is well-defined while applying the para isotopy to $c_0$ and $c_1$. Under this isotopy, the top label becomes $c_1$. In this case, we need to move $c_0$ clockwise along $L$. We obtain $Nk-1$ scalars $q$ from the para isotopy, and one scalar $q$ from the $2\pi$ rotation of $c_0$. Their multiplication is 1. So $Z(T,L)$ does not change.

Moreover, it is easy to check that all the relations are in the kernal of the partition function $Z$. Therefore the relations are consistent. The identity is the only 0 graded 1-box, so $\mathscr{P}/\mathscr{I}$ is a spherical planar para algebra.

Take $\zeta$ to be a square root of $q$ such that $\zeta^{N^2}=1$.
Note that $\zeta \mathfrak{F}_{\text{s}}(c)$ satisfies the  relations as $c$.
Therefore, we can lift the shading of $\mathscr{P}/\mathscr{I}$ and by introducing the relation $\mathfrak{F}_{\text{s}}(c)=\zeta c$.
Then $\mathscr{P}/\mathscr{I}$ is an unshaded planar algebra.
\end{proof}

  \bibliography{bibliography}

\providecommand{\bysame}{\leavevmode\hbox to3em{\hrulefill}\thinspace}
\providecommand{\MR}{\relax\ifhmode\unskip\space\fi MR }
% \MRhref is called by the amsart/book/proc definition of \MR.
\providecommand{\MRhref}[2]{%
  \href{http://www.ams.org/mathscinet-getitem?mr=#1}{#2}
}
\providecommand{\href}[2]{#2}
\begin{thebibliography}{BMPS12}
\bibitem[Ati88]{Ati88}
M.~F. Atiyah, \emph{Topological quantum field theory}, Publications
  Math{\'e}matiques de l'IH{\'E}S \textbf{68} (1988), 175--186.

\bibitem[AYP16]{AYP16}
H.~Au-Yang and J.~H.~H.~Perk, 
About 30 years of integrable chiral Potts model, quantum groups at 
roots of unity, and cyclic hypergeometric functions,
Proc. Centre Math. Appl., in press.
{\color{blue}\url{http://arxiv.org/abs/1601.01014}}

\bibitem[BE98]{BocEva98}
J.~B{\"o}ckenhauer and D.~E. Evans, \emph{Modular invariants, graphs and
  $\alpha$-induction for nets of subfactors {I}}, Commun. Math. Phys.
  \textbf{197} (1998), no.~2, 361--386.

\bibitem[BMPS12]{BMPS}
S.~Bigelow, S.~Morrison, E.~Peters, and N.~Snyder, \emph{Constructing the
  extended {H}aagerup planar algebra}, Acta Math. (2012), 29--82.

\bibitem[BW89]{BirWen}
J.~Birman and H.~Wenzl, \emph{Braids, link polynomials and a new algebra},
  Trans. AMS \textbf{313(1)} (1989), 249--273.

\bibitem[CO89]{CO89}
E.~Cobanera and G.~Ortiz,
\emph{Fock parafermions and self-dual representations of the braid group}, 
Phys. Rev. A. \textbf{89} (2014), 012328. 

\bibitem[ENO05]{ENO}
P.~Etingof, D.~Nikshych, and V.~Ostrik, \emph{On fusion categories}, Annals of
  Mathematics (2005), 581--642.

\bibitem[FZ82]{FatZam82}
V.~A. Fateev and A.~B. Zamolodchikov, \emph{Self-dual solutions of the
  star-triangle relations in zn-models}, Physics Letters A \textbf{92} (1982),
  no.~1, 37--39.

\bibitem[GJ89]{GolJon89}
D.~M. Goldschmidt and V.~F.~R. Jones, \emph{Metaplectic link invariants},
  Geometriae Dedicata \textbf{31} (1989), no.~2, 165--191.

\bibitem[JJ16a]{JafJan2015}
A.~Jaffe and B.~Janssens, \emph{Characterization of reflection positivity:
  Majoranas and spins}, Commun. Math. Phys. (to appear) (2016),
  {\color{blue}\url{http://dx.doi.org/10.1007/s00220-015-2545-z}}.

\bibitem[JJ16b]{JafJan2016}
\bysame, \emph{Reflection positive doubles}, 2016, in preparation.


\bibitem[JLW16a]{JafLiuWoz}
A.~Jaffe, Z.~Liu, and A.~Wozniakowski, 
\emph{Holographic software for quantum networks}, 2016,
  {\color{blue}\url{}}.

\bibitem[JLW16b]{JafLiuWoz-Tele}
\bysame,
\emph{Compressed teleportation}, 2016,
  {\color{blue}\url{}}.


\bibitem[JLW16c]{JiaLiuWu}
C.~Jiang, Z.~Liu, and J.~Wu, \emph{Noncommutative uncertainty principles},
  Journal of Functional Analysis \textbf{270} (2016), 264--311.

\bibitem[Jon83]{Jon83}
V.~F.~R. Jones, \emph{Index for subfactors}, Invent. Math. \textbf{72} (1983),
  1--25.

\bibitem[Jon85]{Jon85}
\bysame, \emph{A polynomial invariant for knots via von neumann algebras},
  Mathematical Sciences Research Institute, 1985.

\bibitem[Jon89]{Jon89}
\bysame, \emph{On a certain value of the kauffman polynomial}, Commun. Math.
  Phys. \textbf{125} (1989), no.~3, 459--467.

\bibitem[Jon91]{Jon91}
\bysame, \emph{Baxterization}, Inter. J. Modern Physics A \textbf{6} (1991),
  no.~12, 2035--2043.

\bibitem[Jon98]{JonPA}
\bysame, \emph{Planar algebras, {I}}, New Zealand J. Math. (1998),
  {\color{blue}\url{http://arxiv.org/abs/math/9909027}}.

\bibitem[Jon12]{Jon12}
\bysame, \emph{Quadratic tangles in planar algebras}, Duke Math. J.
  \textbf{161} (2012), no.~12, 2257--2295.

\bibitem[JP15a]{JafFab15}
A.~Jaffe and F.~L. Pedrocchi, \emph{Reflection positivity for majoranas},
  Annales Henri Poincar{\'e}, vol.~16, Springer, 2015, pp.~189--203.

\bibitem[JP15b]{JafPed}
\bysame, \emph{Reflection positivity for parafermions}, Commun. Math. Phys.
  \textbf{337} (2015), 455--472.

\bibitem[Kau90]{Kau90}
L.H. Kauffman, \emph{An invariant of regular isotopy}, Trans. AMS \textbf{318}
  (1990), 417--471.

\bibitem[Kit]{Kitaev06}
A.~Kitaev, \emph{Anyons in an exactly solved model and beyond}, Ann. Phys.
  \textbf{306}, 2--111.

\bibitem[Liua]{Liuex}
Z.~Liu, \emph{Exchange relation planar algebras of small rank}, to appear
  Trans. AMS.

\bibitem[Liub]{LiuYB}
\bysame, \emph{Yang-baxter relation planar algebras},
  {\color{blue}\url{http://arxiv.org/abs/1507.06030}}.

\bibitem[LR95]{Lon95}
R.~Longo and K-H Rehren, \emph{Nets of subfactors}, Reviews in Mathematical
  Physics \textbf{7} (1995), no.~04, 567--597.

\bibitem[MPS10]{MPSD2n}
S.~Morrison, E.~Peters, and N.~Snyder, \emph{Skein theory for the ${D}_{2n}$
  planar algebras}, Journal of Pure and Applied Algebra \textbf{214} (2010),
  117--139.

\bibitem[Mur87]{Mur87}
J.~Murakami, \emph{The kauffman polynomial of links and representation theory},
  Osaka J. Math. \textbf{24(4)} (1987), 745--758.

\bibitem[Ocn88]{Ocn88}
A.~Ocneanu, \emph{Quantized groups, string algebras and {G}alois theory for
  algebras}, Operator algebras and applications, Vol.\ 2, London Math. Soc.
  Lecture Note Ser., vol. 136, Cambridge Univ. Press, Cambridge, 1988,
  pp.~119--172.

\bibitem[Ocn02]{Ocn00}
\bysame, \emph{The classification of subgroups of quantum {SU(N)}},
  Contemporary Mathematics \textbf{294} (2002), 133--160.

\bibitem[OS73a]{OstSch73a}
K.~Osterwalder and R.~Schrader, \emph{Axioms for {E}uclidean {G}reen's
  functions}, Commun. Math. Phys. \textbf{31} (1973), no.~2, 83--112.

\bibitem[OS73b]{OstSch73b}
\bysame, \emph{Euclidean {F}ermi fields and a {F}eynman--{K}ac formula for
  boson--fermion models}, Helv. Phys. Acta \textbf{46} (1973), 277--302.

\bibitem[Ost03]{Ost03}
V.~Ostrik, \emph{Module categories, weak hopf algebras and modular invariants},
  Transformation Groups \textbf{8(2)} (2003), 177--206.

\bibitem[Pop90]{Pop90}
S.~Popa, \emph{Classification of subfactors: reduction to commuting squares},
  Invent. Math. \textbf{101} (1990), 19--43.

\bibitem[Pop94]{Pop94}
\bysame, \emph{Classification of amenable subfactors of type {II}}, Acta Math.
  \textbf{172} (1994), 352--445.

\bibitem[Wen87]{Wen87}
H.~Wenzl, \emph{On sequences of projections}, C. R. Math. Rep. Acad. Sci.
  Canada \textbf{9(1)} (1987), 5--9.

\bibitem[Wit88]{Wit88}
E.~Witten, \emph{Topological quantum field theory}, Commun. Math. Phys.
  \textbf{117} (1988), no.~3, 353--386.

\bibitem[Xu98]{Xu98}
Feng Xu, \emph{New braided endomorphisms from conformal inclusions}, Commun.
  Math. Phys. \textbf{192} (1998), no.~2, 349--403.
\end{thebibliography}
  \bibliographystyle{amsalpha}

\end{document}